\documentclass[oneside,final]{ucr}
\usepackage{etex}
\usepackage[english]{babel}
\usepackage{latexsym,amsfonts,amssymb,exscale,enumerate,amsthm,stmaryrd}
\usepackage{graphicx}
\usepackage{tikz,tikz-cd}
\usepackage{flafter}
\usepackage{float}
\usetikzlibrary{intersections}
\usepackage[centertags,sumlimits,intlimits,namelimits,reqno]{amsmath}
\usepackage{amsfonts}
\usepackage{bondgraph}
\usepackage[all,2cell]{xy} \UseAllTwocells
\usepackage[hyphens,spaces,obeyspaces]{url}

\usepackage[toc,page]{appendix}

\usepackage{tikz,tikz-cd}

\usepackage{arydshln}

\usetikzlibrary{intersections,decorations.markings}
\usetikzlibrary{arrows,positioning,fit,matrix,shapes.geometric,external}
\usetikzlibrary{backgrounds,circuits,circuits.ee.IEC,shapes,fit,matrix}

\pgfdeclarelayer{edgelayer}
\pgfdeclarelayer{nodelayer}
\pgfsetlayers{edgelayer,nodelayer,main}

\tikzset{font=\footnotesize}
\tikzset{->-/.style={decoration={
  markings,
  mark=at position #1 with {\arrow{>}}},postaction={decorate}}}

\tikzstyle{none}=[inner sep=0pt]
\tikzstyle{circ}=[circle,fill=black,draw,inner sep=3pt]

\usepackage{enumitem}

\usepackage[pagebackref=true,
colorlinks=true,linkcolor=blue,citecolor=blue,urlcolor=blue,
pdfauthor={Kenny Courser},
pdftitle={The mathematics of open dynamical systems from a double categorical perspective},
pdfsubject={category theory, network theory},
pdfkeywords={category theory, network theory}]{hyperref}

\usepackage{bm} 
\usepackage{color}

\usepackage[mathscr]{euscript}
\SelectTips{eu}{11}

\newtheorem{theorem}{Theorem}[section]
\newtheorem{proposition}[theorem]{Proposition}

\newtheorem{lemma}[theorem]{Lemma}   
\newtheorem{corollary}[theorem]{Corollary}  
\theoremstyle{definition}
\newtheorem{definition}[theorem]{Definition}

\newcommand{\ten}{\otimes}

\newcommand{\maps}{\colon}

\newcommand{\colim}{\mathrm{colim}}

\newcommand{\inta}{\raisebox{.3\depth}{$\smallint\hspace{-.02in}$}}

\newcommand{\Rel}{\mathsf{Rel}}

\newcommand{\Cospan}{\mathsf{Cospan}}

\newcommand{\FinGraph}{\mathsf{FinGraph}}

\newcommand{\LagRel}{\mathsf{LagRel}}

\newcommand{\lCsp}{\mathbb{C}\mathbf{sp}}

\newcommand{\A}{\mathsf{A}}

\newcommand{\N}{\mathbb{N}}
\newcommand{\R}{\mathbb{R}}

\newcommand{\fs}{\ensuremath{\mathfrak{s}}}

\usetikzlibrary{arrows,positioning,fit,matrix,shapes.geometric,external}

\newcommand*\pgfdeclareanchoralias[3]{%
  \expandafter\def\csname pgf@anchor@#1@#3\expandafter\endcsname
     \expandafter{\csname pgf@anchor@#1@#2\endcsname}}

\definecolor{lblue}{rgb}{0,250,255}
\tikzstyle{species}=[circle,fill=yellow,draw=black,scale=1.15]
\tikzstyle{transition}=[rectangle,fill=lblue,draw=black,scale=1.15]
\tikzstyle{inarrow}=[->, >=stealth, shorten >=.03cm,line width=1]
\tikzstyle{empty}=[circle,fill=none, draw=none]
\tikzstyle{inputdot}=[circle,fill=purple,draw=purple, scale=.25]
\tikzstyle{inputarrow}=[->,draw=purple, shorten >=.05cm]
\tikzstyle{simple}=[-,draw=purple,line width=1.000]
\tikzstyle{none}=[inner sep=0pt]

\tikzset{every picture/.style={line width=0.5pt}}

\newcommand{\define}[1]{{\bf \boldmath #1}}

\usetikzlibrary{intersections,decorations.markings}
\usetikzlibrary{arrows,positioning,fit,matrix,shapes.geometric,external}
\usetikzlibrary{backgrounds,circuits,circuits.ee.IEC,shapes,fit,matrix}

\pgfdeclareanchoralias{regular polygon}{corner 1}{io}
\pgfdeclareanchoralias{regular polygon}{corner 2}{left out}
\pgfdeclareanchoralias{regular polygon}{corner 3}{right out}
\pgfdeclareanchoralias{regular polygon}{corner 2}{left copy}
\pgfdeclareanchoralias{regular polygon}{corner 3}{right copy}
\pgfdeclareanchoralias{regular polygon}{corner 3}{addend}
\pgfdeclareanchoralias{regular polygon}{corner 2}{summand}
\pgfdeclareanchoralias{regular polygon}{corner 3}{left in}
\pgfdeclareanchoralias{regular polygon}{corner 2}{right in}

\usepackage{color}

\newcommand\SWarrow{\mathrel{\rotatebox[origin=c]{-135}{$\Rightarrow$}}}
\newcommand\SEarrow{\mathrel{\rotatebox[origin=c]{-45}{$\Rightarrow$}}}
\newcommand\NWarrow{\mathrel{\rotatebox[origin=c]{135}{$\Rightarrow$}}}
\newcommand\NEarrow{\mathrel{\rotatebox[origin=c]{45}{$\Rightarrow$}}}
\newcommand\vertsimeq{\mathrel{\rotatebox[origin=c]{-90}{$\simeq$}}}

\tikzset{
    overdraw/.style={preaction={draw,white,line width=#1}},
    overdraw/.default=5pt
}

\tikzstyle{main node} =[circle,fill=white!20,draw,font=\sffamily\Large\bfseries]

\newcommand{\lC}{\mathbb{C}}
\newcommand{\lD}{\mathbb{D}}

\newcommand{\lX}{\mathbb{X}}

\newcommand{\Mark}{\mathsf{Mark}}
\newcommand{\FinSet}{\mathsf{FinSet}}

\newcommand{\Dynam}{\mathsf{Dynam}}
\newcommand{\SemiAlgRel}{\mathsf{SemiAlgRel}}
\newcommand{\LinRel}{\mathsf{LinRel}}
\newcommand{\Petri}{\mathsf{Petri}}
\newcommand{\X}{\mathsf{X}}
\newcommand{\SemiAlg}{\mathsf{SemiAlg}}

\newcommand{\MMark}{\mathbb{M}\mathbf{ark}}
\newcommand{\LLinRel}{\mathbb{L}\mathbf{inRel}}

\begin{document}
\title{Open Systems: A Double Categorical Perspective
}

\author{Kenny Allen Courser}

\degreemonth{March}
\degreeyear{2020}
\degree{Doctor of Philosophy}
\chair{Professor John Baez}
\othermembers{Professor Jacob Greenstein \newline Professor David Weisbart}
\numberofmembers{3}

\field{Mathematics}
\campus{Riverside}

\maketitle
\copyrightpage{}

\degreesemester{Winter}

\begin{frontmatter}

{\ssp
\begin{acknowledgements}
First, I would like to thank my advisor John Baez for taking me on to be his student. Without his patience and humor, this thesis would not exist. I would like to thank Mike Shulman, Christina Vasilakopoulou, Daniel Cicala, Blake Pollard and David Weisbart for their help and guidance over the past several years. I would like to thank the other members of our category theory research group during my time here at University of California, Riverside---Brandon Coya, Jason Erbele, Joe Moeller, Jade Master and Christian Williams---for helpful discussions and feedback. I would like to thank my grandparents Valerie and Lup\'e for their love and support and for allowing me to focus on school. I would like to thank my little brother, Christian, my little sister, Catherine, my older sister, Emily, my older brother, Andrew, and my dad, Frank, for all the memories created while growing up, and of course, Buster. I would like to thank my cohort---Adam and Bryansito $><$, Eddie, Josh, Kevin, Mikahl, Ryan, Tim and James---for their companionship while traversing the gauntlet of grad school. I would also like to thank my hometown friends---Jimmy, Richie, Tj, Dombot, Kurt, Daniel and Austin---for all the memories in middle school, high school and early college. I would like to thank the community over at the nLab; whenever I needed to look up a particular concept or idea, the nLab was one of the first places that I would look. And lastly, I would like to thank my mom, Jodi, for her love and support and for bringing me into this world.
\end{acknowledgements}
 }
\begin{dedication}

\vspace*{\fill}

\begin{center}
To my mother, Jodi.
\end{center}

\vspace*{\fill}

\end{dedication}

{\ssp
\abstract
Fong developed `decorated cospans' to model various kinds of open systems: that is, systems with inputs and outputs. In this framework, open systems
are seen as the morphisms of a category and can be composed as such, allowing larger open systems to be built up from smaller ones. Much work has already been done in this direction, but there is a problem: the notion of isomorphism between decorated cospans is often too restrictive. Here we introduce and compare two ways around this problem: structured cospans, and a new version of decorated cospans. Structured cospans are very simple: given a functor $L \maps \A \to \X$, a `structured cospan' is a diagram in $\X$ of the form $L(a) \rightarrow x \leftarrow L(b)$. If $\A$ and $\X$ have finite colimits and $L$ is a left adjoint, there is a symmetric monoidal category whose objects are those of $\A$ and whose morphisms are isomorphism classes of structured cospans.   However, this category arises from a more fundamental structure: a symmetric monoidal double category. Under certain conditions this symmetric monoidal double category is equivalent to one built using our new version of decorated cospans. We apply these ideas to symmetric monoidal double categories of open electrical circuits, open Markov processes and open Petri nets.
\endabstract
}
\tableofcontents

\end{frontmatter}

{\ssp
\chapter{Introduction} 

This is a thesis about compositional frameworks for describing `open networks', which are networks with prescribed `inputs' and `outputs'.   One well-known type of network is a `Petri net'. Petri nets are important in computer science, chemistry and other subjects. For example, the chemical reaction that takes two atoms of hydrogen and one atom of oxygen and produces a molecule of water can be represented by this very simple Petri net:
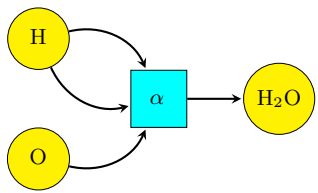
\[
\scalebox{0.8}{

}
\]
Here we have a set of `places' (or in chemistry, `species') drawn in yellow 
and a set of `transitions' (or `reactions') drawn in blue.   The disjoint union of these two sets then forms the vertex set of a directed bipartite graph, which is one description of a Petri net.

Networks can often be seen as pieces of larger networks.  This naturally leads to the idea of an \emph{open} Petri net, meaning that the set of places is equipped with inputs and outputs. We can do this by prescribing two functions into the set of places that pick out these inputs and outputs. For example:
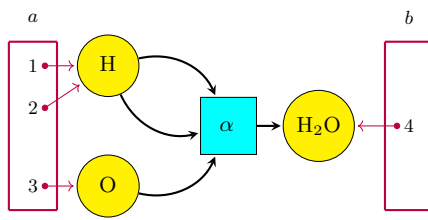
\[
\scalebox{0.8}{
%
}
\]
The inputs and outputs let us compose open Petri nets.  For example, suppose we have another open Petri net that represents the chemical reaction of two molecules of water turning into hydronium and hydroxide:
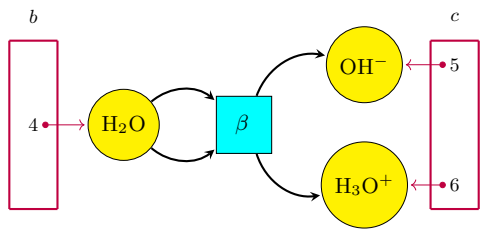
\[
\scalebox{0.8}{
%
}
\]
Since the outputs of the first open Petri net coincide with the inputs of the second, we can compose them by identifying the outputs of the first with the inputs of the second:  
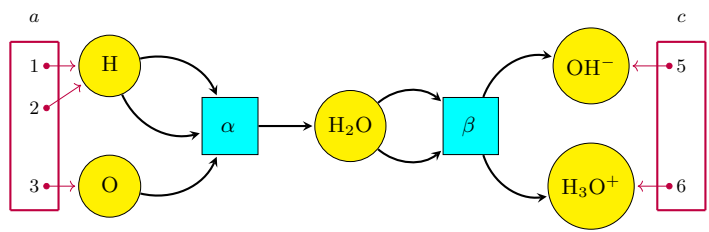
\[
\scalebox{0.8}{
%
}
\]
Similarly we can `tensor' two open Petri nets by placing them side by side:
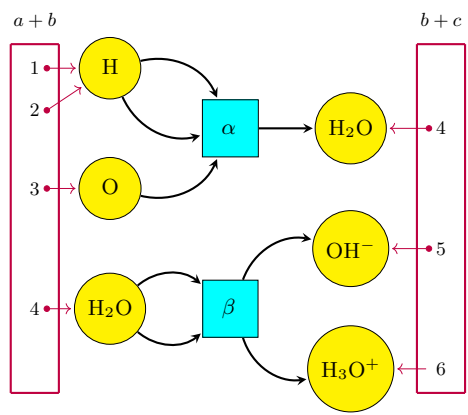
\[
\scalebox{0.8}{
%
}
\]
The compositional nature of these open Petri nets, and of open networks in general, is suggestive of an underlying categorical structure.  Moreover, the ability to tensor these open networks naturally leads to a symmetric monoidal structure on these categories. In this thesis we study two frameworks for constructing and working with symmetric monoidal categories whose morphisms are open networks. The first, `decorated cospans', was introduced by Brendan Fong \cite{BF}. The second, `structured cospans', is new. Here we study both frameworks using symmetric monoidal \emph{double} categories, which have 2-morphisms that describe maps \emph{between} open networks. 

The outline of the thesis is as follows. In Chapter \ref{Chapter2}, we present Fong's decorated cospans and give some examples in which they have been applied: graphs, electrical circuits, Markov processes and Petri nets. In Chapter \ref{Chapter3}, we introduce the framework of structured cospans. In Chapter \ref{Chapter4}, we revisit decorated cospans but at the level of double categories. In Chapter \ref{Chapter5}, we explore some of the similarities between double categories and bicategories, and in Chapter \ref{Chapter6}, we give an application of double categories to Markov processes and `coarse-grainings' and show that coarse-graining is compatible with black-boxing. This last application is constructed using neither structured cospans nor decorated cospans due to the complexity of its 2-morphisms, but is nevertheless a great example of how the rich structure of double categories and their appropriate maps can be used to model complicated open dynamical systems.

The first piece of work that this thesis is built upon, \textsl{A bicategory of decorated cospans} \cite{Cour}, was an initial attempt at categorifying Fong's theory of `decorated cospans', which we introduce in Chapter \ref{Chapter2}.  Following a suggestion of Mike Shulman \cite{Shul}, this attempt made extensive use of double categories, and it was here that the current author's journey into double categories began. Over the course of this journey, John Baez noticed a flaw with the decorated cospans framework, which we explain in Section \ref{Graphs} and also at the beginning of Chapter \ref{Chapter3}. Thus, Baez conceived another framework which simultaneously corrected this flaw and was more convenient to use: `structured cospans'.  This is the main content of Chapter \ref{Chapter3}. This new framework also employs double categories, and several applications which were previously illustrated using decorated cospans were explored using structured cospans in a recent paper with Baez, \textsl{Structured cospans} \cite{BC2}. Then, following along on this double categorical campaign, a more direct fix to decorated cospans was introduced by Baez, Vasilakopoulou and the author in \textsl{Structured versus decorated cospans} \cite{BCV}.  This material constitutes Chapter \ref{Chapter4}: the main result is that the new improved decorated cospans are equivalent to structured cospans under certain mild conditions. Tangential to all of this, Baez and the current author wrote \textsl{Coarse-graining open Markov processes} \cite{BC}.   While this work also makes use of double categories, it uses neither decorated nor structured cospans, due to some more sophisticated structure that is necessary. This material makes up Chapter \ref{Chapter6}.

$\textnormal{ }$
}

{\ssp
\chapter{Decorated cospan categories}\label{Chapter2}
This chapter is devoted to Fong's theory of decorated cospans and a few of its applications. Fong's theory of decorated cospans is well-suited to describing open networks: that is, networks with prescribed inputs and outputs.   We can build larger networks from smaller ones by attaching the inputs of one to the outputs of another. This suggests that we should treat open networks as morphisms in a category. In addition to composing open networks, we can also put them side by side in parallel, giving a monoidal category. Fong's Theorem on decorated cospans  provides a framework that captures all of this structure and more. Fong's decorated cospan categories can then serve as syntax categories for functors that describe the behavior of open networks, such as the `black-box' functors studied by Baez, Fong, Master and Pollard \cite{BC,BF,BFP,BM,BP}.

In Section \ref{FT}, we present Fong's Theorem. For definitions of the terms used in this theorem, see Appendix \ref{Appendix}. In Section \ref{FongApps}, we present some previously studied applications of decorated cospans which will later be revisited in subsequent chapters from the perspective of other compositional frameworks. These examples include open graphs, open electrical circuits, open Markov processes and open Petri nets.

\section{Fong's Theorem}\label{FT}
\begin{definition}
A \define{cospan} in any category $\mathsf{C}$ is a diagram of the form
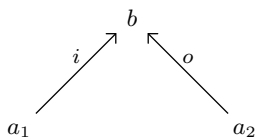
\[
\begin{tikzpicture}[scale=1.5]
\node (A) at (0,0) {$b$};
\node (B) at (-1,-1) {$a_1$};
\node (C) at (1,-1) {$a_2$};
\path[->,font=\scriptsize,>=angle 90]
(B) edge node [above] {$i$}(A)
(C) edge node [above] {$o$}(A);
\end{tikzpicture}
\]
In other words, a cospan is an ordered pair of morphisms $i$ and $o$ in $\mathsf{C}$ whose target coincide.
\end{definition}

A result of Fong \cite{Fong} which has been fundamental in the inspiration of a large portion of this thesis is the following.

\begin{theorem}[Fong]\label{Fong}
Let $\mathsf{C}$ be a category with finite colimts and $F \colon (\mathsf{C},+,0) \to (\mathsf{Set},\times,1)$ a symmetric lax monoidal functor. Then there exists a symmetric monoidal category $F\mathsf{Cospan}$ which has:
\begin{enumerate}
\item{objects as those of $\mathsf{C}$ and}
\item{morphisms as isomorphism classes of $F$-decorated cospans in $\mathsf{C}$, which are pairs:
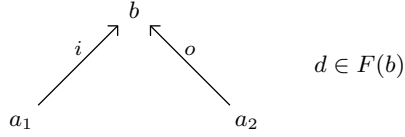
\[

\]
Two $F$-decorated cospans are in the same isomorphism class if the following diagrams commute:
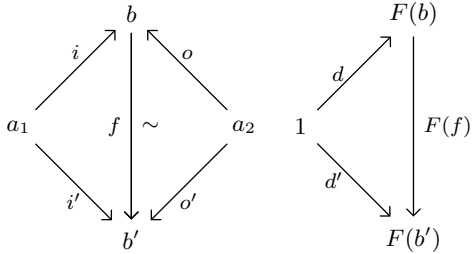
\[
%
\]
for some isomorphism $f$. The composite of two composable $F$-decorated cospans
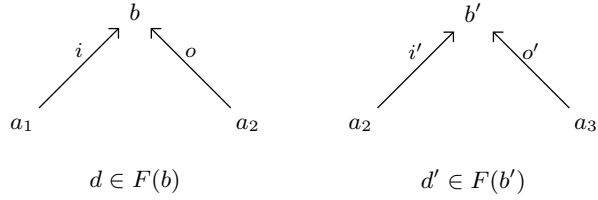
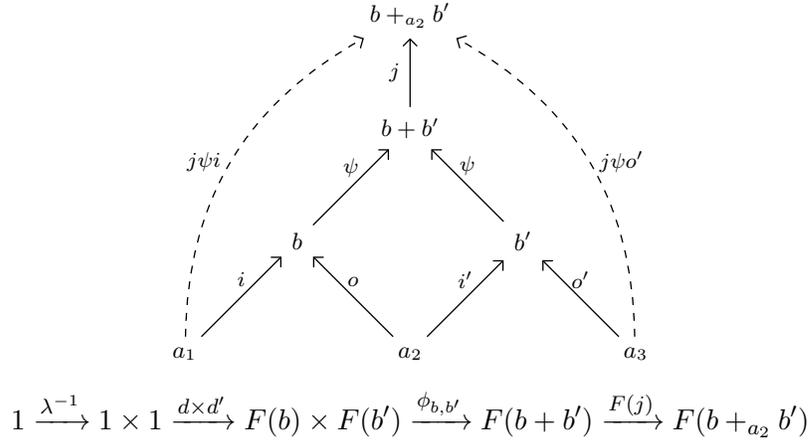
\[
%
\]
is given by
\[
%
\]
$$1 \xrightarrow{\lambda^{-1}} 1 \times 1 \xrightarrow{d \times d'} F(b) \times F(b') \xrightarrow{\phi_{b,b'}} F(b+b') \xrightarrow{F(j)} F(b+_{a_2}b')$$
where $\psi$ is the natural map into a coproduct, $j$ is the natural map from a coproduct into a pushout, and $\phi_{b,b'} \colon F(b) \times F(b') \to F(b+b')$ is the natural transformation coming from the structure of the symmetric lax monoidal functor $F \colon (\mathsf{C},+,0) \to (\mathsf{Set},\times,1)$.

The tensor product of two objects $a_1$ and $a_2$ is given by their binary coproduct $a_1+a_2$ in $\mathsf{C}$.

The tensor product of two $F$-decorated cospans is given pointwise:
\[
\begin{tikzpicture}[scale=1.5]
\node (A) at (0,0) {$b$};
\node (B) at (-1,-1) {$a_1$};
\node (C) at (1,-1) {$a_2$};
\node (E) at (2,-1) {$a_1'$};
\node (F) at (3,0) {$b'$};
\node (G) at (4,-1) {$a_2'$};
\node (H) at (3,-1.5) {$d' \in F(b')$};
\node (D) at (0,-1.5) {$d \in F(b)$};
\node (I) at (1.5,-0.5) {$\otimes$};
\node (J) at (4.5,-0.5) {$=$};
\node (K) at (5,-1) {$a_1+a_1'$};
\node (L) at (6,0) {$b+b'$};
\node (M) at (7,-1) {$a_2+a_2'$};
\node (N) at (6,-1.5) {$d+d' \in F(b+b')$};
\path[->,font=\scriptsize,>=angle 90]
(B) edge node [left] {$i$}(A)
(C) edge node [right] {$o$}(A)
(E) edge node [left] {$i'$}(F)
(G) edge node [right] {$o'$}(F)
(K) edge node [left] {$i+i'$}(L)
(M) edge node [right] {$o+o'$}(L);
\end{tikzpicture}
\]
$$d+d' := 1 \xrightarrow{\lambda^{-1}} 1 \times 1 \xrightarrow{d \times d'} F(b) \times F(b') \xrightarrow{\phi_{b,b'}} F(b+b')$$
}

\end{enumerate}
\end{theorem}


We will also need a variant of Fong's Theorem that gives a merely monoidal category:

\begin{theorem}\label{Fong2}
Let $\mathsf{C}$ be a category with finite colimts and $F \colon (\mathsf{C},+,0) \to (\mathsf{Set},\times,1)$ a lax monoidal functor. Then there exists a monoidal category $F\mathsf{Cospan}$ where the relevant structure is given as in Theorem \ref{Fong}.
\end{theorem}

The necessity of this weaker result was pointed out by an anonymous referee of Moeller and Vasilakopoulou \cite{MV}, which we explain in the introduction of Chapter \ref{Chapter3} on `structured cospans'.

\section{Applications}\label{FongApps}
In this section we present some examples of applications of decorated cospans which have been studied in previous works \cite{BF,BFP,BP,BWY, Fong}.

\subsection{Graphs}\label{Graphs}

Our first example is the category of `open graphs'. This makes clear some difficulties in Fong's approach to decorated cospans---problems that will be solved using our double category approach. Let $(\mathsf{FinSet},+,0)$ denote the category of finite sets and functions made symmetric monoidal using coproducts.   To apply Fong's Theorem, we seek a symmetric lax monoidal functor $F \colon (\mathsf{FinSet},+,0) \to (\mathsf{Set},\times,1)$ that assigns to a finite set $N$ the set of all graphs whose underlying set of vertices is $N$. So, we define a \define{graph structure} on $N$ to be a diagram in $\mathsf{FinSet}$ of the following form.
\[
\begin{tikzpicture}[scale=1.5]
\node (A) at (0,0) {$E$};
\node (B) at (1,0) {$N$};
\path[->,font=\scriptsize,>=angle 90]
(A) edge[bend left] node [above] {$s$}(B)
(A) edge[bend right] node [below] {$t$}(B);
\end{tikzpicture}
\]
Here $E$ is the set of \define{edges} of the graph while $s,t \colon E \to N$ are the \define{source} and \define{target} functions, respectively. 

If we naively try to take $F(N)$ to be the set of graph structures on $N$, we immediately notice a problem: this is not a set, but rather a proper class. Fong \cite{Fong} gets around this by replacing $\mathsf{FinSet}$ with an equivalent small category, which by abuse of notation we shall call $\mathsf{FinSet}$.  Using this small version of $\mathsf{FinSet}$ in the definition of graph structure, we see that there is an actual set $F(N)$ of graph structures on any $N \in \mathsf{FinSet}$. Given a function $f \colon N \to N'$ we define $F(f) \colon F(N) \to F(N')$ as follows. Given a graph structure on $N$, the function $f$ induces a graph structure on $N'$ if we demand that the following diagrams commute:
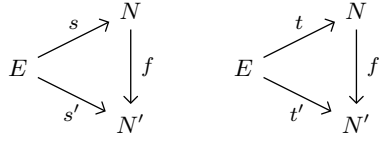
\[
\begin{tikzpicture}[scale=1.5]
\node (A) at (5,-0.5) {$E$};
\node (E) at (6,0) {$N$};
\node (E') at (6,-1) {$N'$};
\node (B) at (7,-0.5) {$E$};
\node (F) at (8,0) {$N$};
\node (F') at (8,-1) {$N'$};
\path[->,font=\scriptsize,>=angle 90]
(E) edge node [right] {$f$} (E')
(F) edge node [right] {$f$} (F')
(A) edge node[above]{$s$}(E)
(A) edge node [below] {$s'$} (E')
(B) edge node[above]{$t$}(F)
(B) edge node [below] {$t'$} (F');
\end{tikzpicture}
\]
This results in a graph structure on $N'$ given by $s',t' \colon E \to N'$ where $s'=fs$ and $t'=ft$. In other words, we are pushing forward the set $E$ of edges along the function $f$ in such a way that sources and targets of edges are preserved. It is clear that this procedure is associative and preserves identities, and thus defines a functor $F \colon \mathsf{FinSet} \to \mathsf{Set}$.

The next question is whether $F$ is lax monoidal.  For this, note that given a graph structure $d_1$ on a finite set $N_1$ and a graph structure $d_2$ on another finite set $N_2$, there is a graph structure $d_1 + d_2$ on $N_1+N_2$, given by taking pointwise coproducts of the respective graph structures on $N_1$ and $N_2$:
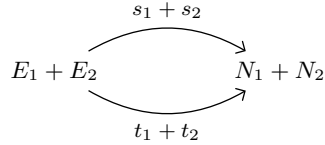
\[
\begin{tikzpicture}[scale=1.5]
\node (A) at (0,0) {$E_1 + E_2$};
\node (B) at (2,0) {$N_1 + N_2$};
\path[->,font=\scriptsize,>=angle 90]
(A) edge[bend left] node [above] {$s_1 + s_2$}(B)
(A) edge[bend right] node [below] {$t_1 + t_2$}(B);
\end{tikzpicture}
\]
One can check that there is a natural transformation $$\mu_{N_1,N_2} \colon F(N_1) \times F(N_2) \to F(N_1+N_2)$$ mapping $(d_1, d_2)$ to $d_1 + d_2$, as one would expect if $F$ were lax monoidal and $\mu$ were its laxator. Note the non-invertibility of the maps $\mu_{N_1,N_2}$. For example, the figure below shows two graphs $d_1 \in F(N_1)$ and $d_2 \in F(N_2)$ in black; taking them together we get $d_1 + d_2 \in F(N_1 + N_2)$. If we also include the red edge we obtain a graph that is not in the image of the laxator $\mu_{N_1, N_2}$, but is a perfectly fine element of $F(N_1 + N_2)$. 
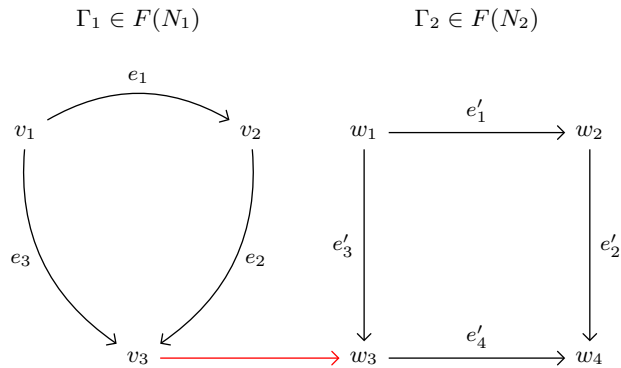
\[
\begin{tikzpicture}[scale=1.5]
\node (A) at (0,0) {$v_1$};
\node (B) at (2,0) {$v_2$};
\node (C) at (1,-2) {$v_3$};
\node (D) at (3,0) {$w_1$};
\node (E) at (5,0) {$w_2$};
\node (F) at (3,-2) {$w_3$};
\node (G) at (5,-2) {$w_4$};
\node (H) at (1,1) {$\Gamma_1 \in F(N_1)$};
\node (I) at (4,1) {$\Gamma_2 \in F(N_2)$};
\path[->,font=\scriptsize,>=angle 90]
(D) edge node [above] {$e_1'$} (E)
(D) edge node [left] {$e_3'$} (F)
(E) edge node [right] {$e_2'$} (G)
(F) edge node [above] {$e_4'$} (G)
(C) edge [red] node {$$} (F)
(A) edge[bend left] node [above] {$e_1$}(B)
(B) edge[bend left] node [right] {$e_2$}(C)
(A) edge[bend right] node [left] {$e_3$}(C);
\end{tikzpicture}
\]
We also have a morphism $\mu \colon 1 \to F(\emptyset)$ which is, in fact, an isomorphism as the empty graph with no edges is the only possible graph structure on $\emptyset$. However, as pointed out by the anonymous referee of Moeller and Vasilakopoulou's paper \cite{MV}, $\mu$ does not obey the hexagon law required of a lax monoidal functor! We explain why at the start of Chapter \ref{Chapter3}. 
To fix this, we can use Mac Lane's Theorem to choose a small strict monoidal category equivalent to $(\mathsf{FinSet},+,0)$---that is, one for which the associator and unitors are identities. (See Theorem \ref{Mac Lane} below.)   Henceforth we use $(\mathsf{FinSet},+,0)$ to denote this small strict monoidal category.
Then we obtain the desired lax monoidal functor $F \colon (\mathsf{FinSet},+,0) \to (\mathsf{Set},\times,1)$, so we can apply Theorem \ref{Fong2} and get a monoidal category of decorated cospans. 
Unfortunately, we cannot use Fong's Theorem (Theorem \ref{Fong}) to make this category symmetric monoidal, as there is no symmetric monoidal category equivalent to $(\mathsf{FinSet},+,0)$ for which the symmetries are identities. By Theorem \ref{Fong2}, we have the following:

\begin{corollary}
Let $F \colon (\mathsf{FinSet},+,0) \to (\mathsf{Set},\times,1)$ be the lax monoidal functor described above which assigns to $N \in \mathsf{FinSet}$ set of all graph structures whose underlying set of vertices is $N$. Then there exists a monoidal category $F\mathsf{Cospan}$ which has:
\begin{enumerate}
\item{objects as those of $(\mathsf{FinSet},+,0)$ and}
\item{morphisms as isomorphism classes of \define{open graphs}, where an open graph is given by a pair of diagrams:
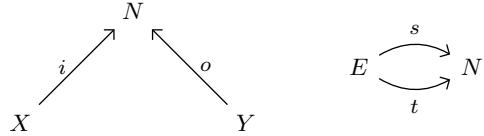
\[
\begin{tikzpicture}[scale=1.5]
\node (A) at (0,0) {$N$};
\node (B) at (-1,-1) {$X$};
\node (C) at (1,-1) {$Y$};
\node (D) at (2,-0.5) {$E$};
\node (E) at (3,-0.5) {$N$};
\path[->,font=\scriptsize,>=angle 90]
(D) edge [bend left] node [above] {$s$} (E)
(D) edge [bend right] node [below] {$t$} (E)
(B) edge node [left] {$i$}(A)
(C) edge node [right] {$o$}(A);
\end{tikzpicture}
\]
Two open graphs are in the same isomorphism class if the following diagrams commute:
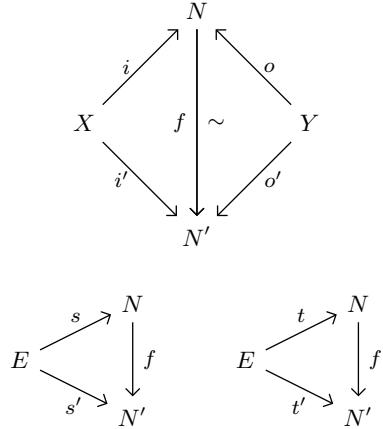
\[
\begin{tikzpicture}[scale=1.5]
\node (A) at (0,0) {$N$};
\node (B) at (-1,-1) {$X$};
\node (C) at (1,-1) {$Y$};
\node (F) at (0,-2) {$N'$};
\path[->,font=\scriptsize,>=angle 90]
(A) edge node [left] {$f$} (F)
(A) edge node [right] {$\sim$} (F)
(B) edge node [left] {$i$}(A)
(C) edge node [right] {$o$}(A)
(B) edge node [left] {$i'$}(F)
(C) edge node [right] {$o'$}(F);
\end{tikzpicture}
\]
\[
\begin{tikzpicture}[scale=1.5]
\node (A) at (5,-0.5) {$E$};
\node (E) at (6,0) {$N$};
\node (E') at (6,-1) {$N'$};
\node (B) at (7,-0.5) {$E$};
\node (F) at (8,0) {$N$};
\node (F') at (8,-1) {$N'$};
\path[->,font=\scriptsize,>=angle 90]
(E) edge node [right] {$f$} (E')
(F) edge node [right] {$f$} (F')
(A) edge node[above]{$s$}(E)
(A) edge node [below] {$s'$} (E')
(B) edge node[above]{$t$}(F)
(B) edge node [below] {$t'$} (F');
\end{tikzpicture}
\]
for some isomorphism $f$. Composition and tensoring of objects and morphisms are given as in Theorem \ref{Fong}.}
\end{enumerate}
\end{corollary}


Again we emphasize that in the above theorem we are using $(\mathsf{FinSet},+, 0)$ to mean some small strict monoidal category equivalent to the usual category of this name. For any object $N$ in this category, $F(N)$ is the set of all graph structures on $N$ defined using this equivalent category. Thus, given graph structures on objects $N_1, N_2$ and $N_3$:
\[
\begin{tikzpicture}[scale=1.5]
\node (D) at (2,-0.5) {$E_1$};
\node (E) at (3,-0.5) {$N_1$};
\node (D') at (4,-0.5) {$E_2$};
\node (E') at (5,-0.5) {$N_2$};
\node (D'') at (6,-0.5) {$E_3$};
\node (E'') at (7,-0.5) {$N_3$};
\path[->,font=\scriptsize,>=angle 90]
(D) edge [bend left] node [above] {$s_1$} (E)
(D) edge [bend right] node [below] {$t_1$} (E)
(D') edge [bend left] node [above] {$s_2$} (E')
(D') edge [bend right] node [below] {$t_2$} (E')
(D'') edge [bend left] node [above] {$s_3$} (E'')
(D'') edge [bend right] node [below] {$t_3$} (E'');
\end{tikzpicture}
\]
the following two graph structures are equal:
\[
\begin{tikzpicture}[scale=1.5]
\node (D) at (0,-0.5) {$E_1+(E_2+E_3)$};
\node (E) at (3,-0.5) {$N_1+(N_2+N_3)$};
\node (D'') at (5,-0.5) {$(E_1+E_2)+E_3$};
\node (E'') at (8,-0.5) {$N_1+(N_2+N_3)$};
\path[->,font=\scriptsize,>=angle 90]
(D) edge [bend left] node [above] {$s_1+(s_2+s_3)$} (E)
(D) edge [bend right] node [below] {$t_1+(t_2+t_3)$} (E)
(D'') edge [bend left] node [above] {$(s_1+s_2)+s_3$} (E'')
(D'') edge [bend right] node [below] {$(t_1+t_2)+t_3$} (E'');
\end{tikzpicture}
\]
This strictification in the graph structures is necessary in order for the functor $F$ of the previous corollary to be lax monoidal. We will also employ this strictification of structures in the following two applications.
\subsection{Electrical circuits}\label{k-graphs}
The remaining two applications, while taking on more of an applied flavor, are structurally very similar.
\begin{definition}
Given a field $k$, a \define{field with positive elements} is a pair $(k,k^+)$ where $k^+ \subset k$ is a subset such that $r^2 \in k^+$ for every nonzero $r \in k$ and such that $k^+$ is closed under addition, multiplication and division.
\end{definition}
\begin{definition}\label{definition:k-graph}
Let $k$ be a field with positive elements. A $k$\define{-graph} is given by a diagram:
\[
\begin{tikzpicture}[scale=1.5]
\node (C) at (1,-0.5) {$k^+$};
\node (D) at (2,-0.5) {$E$};
\node (E) at (3,-0.5) {$N$};
\path[->,font=\scriptsize,>=angle 90]
(D) edge [bend left] node [above] {$s$} (E)
(D) edge [bend right] node [below] {$t$} (E)
(D) edge node [above] {$r$}(C);
\end{tikzpicture}
\]
where $r(e) \in k^+$ is the \define{resistance} along the edge $e \in E$.
\end{definition}
Following the same ideas as in the previous example and using a small strict monoidally equivalent copy of $\mathsf{FinSet}$, we see there is a lax monoidal functor that assigns to any $N \in \mathsf{FinSet}$ the set of all $k$-graph structures on $N$.  Thus, by Theorem \ref{Fong2}, we have the following.
\begin{theorem}
Let $F \colon (\mathsf{FinSet},+,0) \to (\mathsf{Set},\times,1)$ be the lax monoidal functor which assigns to any $N \in \mathsf{FinSet}$ the set of all $k$-graph structures on $N$. Then there exists a monoidal category $F\mathsf{Cospan}$ which has:
\begin{enumerate}
\item{objects as those of $(\mathsf{FinSet},+,0)$ and}
\item{morphisms as isomorphism classes of \define{open $k$-graphs}, where an open $k$-graph is given by a pair of diagrams:
\[
\begin{tikzpicture}[scale=1.5]
\node (A) at (0,0) {$N$};
\node (B) at (-1,-1) {$X$};
\node (C) at (1,-1) {$Y$};
\node (F) at (2,-0.5) {$k^+$};
\node (D) at (3,-0.5) {$E$};
\node (E) at (4,-0.5) {$N$};
\path[->,font=\scriptsize,>=angle 90]
(D) edge node[above] {$r$} (F)
(D) edge [bend left] node [above] {$s$} (E)
(D) edge [bend right] node [below] {$t$} (E)
(B) edge node [left] {$i$}(A)
(C) edge node [right] {$o$}(A);
\end{tikzpicture}
\]
Two open graphs are in the same isomorphism class if the following diagrams commute:
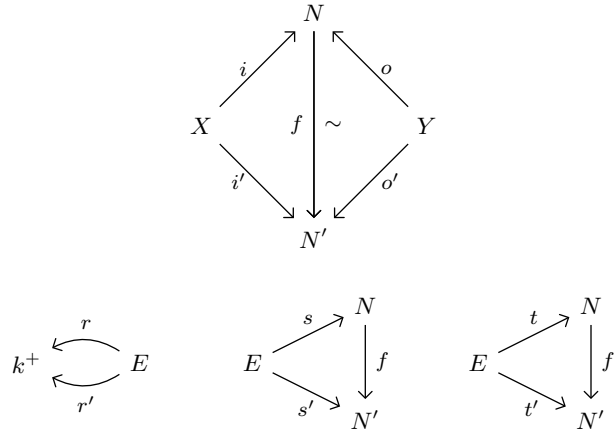
\[
\begin{tikzpicture}[scale=1.5]
\node (A) at (0,0) {$N$};
\node (B) at (-1,-1) {$X$};
\node (C) at (1,-1) {$Y$};
\node (F) at (0,-2) {$N'$};
\path[->,font=\scriptsize,>=angle 90]
(A) edge node [left] {$f$} (F)
(A) edge node [right] {$\sim$} (F)
(B) edge node [left] {$i$}(A)
(C) edge node [right] {$o$}(A)
(B) edge node [left] {$i'$}(F)
(C) edge node [right] {$o'$}(F);
\end{tikzpicture}
\]
\[
\begin{tikzpicture}[scale=1.5]
\node (C) at (3,-0.5) {$k^+$};
\node (D) at (4,-0.5) {$E$};
\node (A) at (5,-0.5) {$E$};
\node (E) at (6,0) {$N$};
\node (E') at (6,-1) {$N'$};
\node (B) at (7,-0.5) {$E$};
\node (F) at (8,0) {$N$};
\node (F') at (8,-1) {$N'$};
\path[->,font=\scriptsize,>=angle 90]
(E) edge node [right] {$f$} (E')
(F) edge node [right] {$f$} (F')
(D) edge[bend right] node[above]{$r$} (C)
(A) edge node[above]{$s$}(E)
(A) edge node [below] {$s'$} (E')
(D) edge[bend left] node[below]{$r'$} (C)
(B) edge node[above]{$t$}(F)
(B) edge node [below] {$t'$} (F');
\end{tikzpicture}
\]
for some isomorphism $f$. Composition and tensoring of objects and morphisms are given as in Theorem \ref{Fong}.}
\end{enumerate}
\end{theorem}
An electrical circuit made of resistors can then be seen as a $k$-graph in which we take the field $k$ to be $\R$ and take $k^+$ to consist of the positive real numbers.  Baez and Fong also consider more general circuits containing resistors, inductors and capacitors, using a larger field with positive elements \cite{BF}.   They study the behavior of these circuits using a `black-boxing' functor from $F\Cospan$ to a category of linear relations.

\subsection{Petri nets}
Our final example involves Petri nets, which have been studied extensively by Baez and Master in a recent work \cite{BM}.
\begin{definition}\label{definition:Petri_net}
A \define{Petri net} is given by the following diagram in $\mathsf{Set}$.
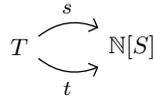
\[
\begin{tikzpicture}[scale=1.5]
\node (D) at (3,-0.5) {$T$};
\node (E) at (4,-0.5) {$\mathbb{N}[S]$};
\path[->,font=\scriptsize,>=angle 90]
(D) edge [bend left] node [above] {$s$} (E)
(D) edge [bend right] node [below] {$t$} (E);
\end{tikzpicture}
\]
We call $S$ the set of \define{species} and $T$ the set of \define{transitions}; $\N[S]$ stands for the free commutative monoid on $S$.
\end{definition}

In this example, we wish to use Fong's Theorem with a functor $F$ that assigns to each set $S$ the set $F(S)$ of all Petri nets having $S$ as their set of species.   Unfortunately, if we do this, $F(S)$ is not a set: it is a proper class.  To avoid this problem, we invoke the axiom of universes and choose a Grothendieck universe $U$.  We call sets in $U$ \define{small} and arbitrary sets \define{large}.  

We let $(\mathsf{Set},+, 0)$ be a strict monoidal category that is monoidally equivalent to the category of small sets with coproduct as its monoidal structure.    The category $(\mathsf{Set},+,0)$ is a large category: more precisely, it is a category with a large set of objects and a large set of morphisms. For any $S \in (\mathsf{Set},+,0)$, there is a large set $F(S)$ of Petri nets having $S$ as its set of species and some $T \in (\mathsf{Set},+,0)$ as its set of transitions. We write $(\mathsf{SET}, \times, 1)$ for the category of large sets with product as its monoidal structure. We can make $F \colon (\mathsf{Set},+,0) \to (\mathsf{SET}, \times, 1)$ into a lax monoidal functor where the natural transformation $$\mu_{S_1,S_2} \colon F(S_1) \times F(S_2) \to F(S_1+S_2)$$ is obtained in the same way as the previous natural transformations in the last three examples, namely by considering two individual Petri nets in parallel as a single Petri net. By Fong's Theorem \ref{Fong2}, we have the following.

\begin{theorem}
Let $F \colon (\mathsf{Set},+,0) \to (\mathsf{SET},\times,1)$ be the lax monoidal functor that assigns to a set $S$ the large set $F(S)$ of all Petri nets whose set of species is given by the set $S$. Then there exists a monoidal category $F\mathsf{Cospan}$ which has:
\begin{enumerate}
\item{objects as those of $(\mathsf{Set},+,0)$ and}
\item{morphisms as isomorphism classes of \define{open Petri nets} which are given by pairs of diagrams:
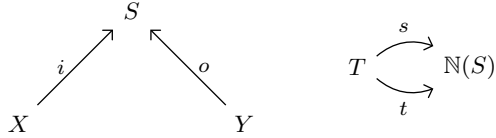
\[
\begin{tikzpicture}[scale=1.5]
\node (A) at (0,0) {$S$};
\node (B) at (-1,-1) {$X$};
\node (C) at (1,-1) {$Y$};
\node (D) at (2,-0.5) {$T$};
\node (E) at (3,-0.5) {$\mathbb{N}(S)$};
\path[->,font=\scriptsize,>=angle 90]
(D) edge [bend left] node [above] {$s$} (E)
(D) edge [bend right] node [below] {$t$} (E)
(B) edge node [left] {$i$}(A)
(C) edge node [right] {$o$}(A);
\end{tikzpicture}
\]
Two open Petri nets are in the same isomorphism class if the following diagrams commute:
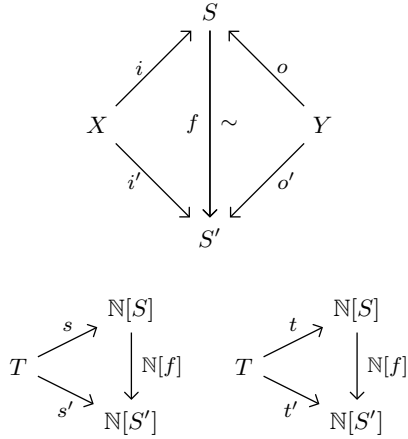
\[
\begin{tikzpicture}[scale=1.5]
\node (A) at (0,0) {$S$};
\node (B) at (-1,-1) {$X$};
\node (C) at (1,-1) {$Y$};
\node (F) at (0,-2) {$S'$};
\path[->,font=\scriptsize,>=angle 90]
(A) edge node [left] {$f$} (F)
(A) edge node [right] {$\sim$} (F)
(B) edge node [left] {$i$}(A)
(C) edge node [right] {$o$}(A)
(B) edge node [left] {$i'$}(F)
(C) edge node [right] {$o'$}(F);
\end{tikzpicture}
\]
\[
\begin{tikzpicture}[scale=1.5]
\node (A) at (5,-0.5) {$T$};
\node (E) at (6,0) {$\mathbb{N}[S]$};
\node (E') at (6,-1) {$\mathbb{N}[S']$};
\node (B) at (7,-0.5) {$T$};
\node (F) at (8,0) {$\mathbb{N}[S]$};
\node (F') at (8,-1) {$\mathbb{N}[S']$};
\path[->,font=\scriptsize,>=angle 90]
(E) edge node [right] {$\mathbb{N}[f]$} (E')
(F) edge node [right] {$\mathbb{N}[f]$} (F')
(A) edge node[above]{$s$}(E)
(A) edge node [below] {$s'$} (E')
(B) edge node[above]{$t$}(F)
(B) edge node [below] {$t'$} (F');
\end{tikzpicture}
\]
for some isomorphism $f$. Composition and tensoring of objects and morphisms is given as in Theorem \ref{Fong}.
}
\end{enumerate}
\end{theorem}
Following ideas similar to those in the last two examples, Baez and Master study the reachability relation of states of open Petri nets via black-boxing \cite{BM}. They in fact go further and construct a `double category' of open Petri nets and a corresponding black box double functor which shows a certain compatibility relation between `maps of open Petri nets' and their black-boxings. Double categories are at the heart of this thesis and we will begin using them in the next chapter.

$\textnormal{ }$
}

{\ssp
\chapter{Structured cospan double categories}\label{Chapter3}
The present chapter is about a particular kind of double categories, namely `foot-replaced double categories'. The first main result of this chapter is the construction of foot-replaced double categories in Theorem \ref{trick1} and the corresponding symmetric monoidal versions of these in Theorem \ref{trick2}. 
The most important kind of foot-replaced double categories are the `structured cospan double categories', which are the content of Theorem \ref{SC}. In Section \ref{SecSCApp} we revisit the applications of Section \ref{FongApps}, but from the perspective of structured cospans. In Section \ref{SCMaps} we define maps of foot-replaced double categories, of which maps between structured cospan double categories are a special case. But first, let us explain the need for some of these concepts. At this point it would be fruitful for readers unfamiliar with double categories to read Appendix \ref{DoubleCatAppendix}.

Recall the first example of Fong's theory of decorated cospans introduced in the previous chapter. Let $F \colon \textsf{FinSet} \to \textsf{Set}$ be the symmetric lax monoidal functor that assigns to a finite set $b$ the (large) set of all possible graph structures on the finite set $b$, where a graph structure on $b$ is given by a diagram in $\textsf{Set}$ of the form:
\[

\]
Let $b=\{ v_1,v_2 \}$ be a two element set. Then one element of the (large) set $F(b)$, which is the collection of all graph structures on the finite set $b$, is given by a single edge $e$ whose source and target are $v_1$ and $v_2$, respectively.
\[
%
Denote this element of $F(b)$ as $d \colon 1 \to F(b)$. Let $a_1=\{ 1 \}$ and $a_2=\{2\}$ and define functions $i \colon a_1 \to b$ and $o \colon a_2 \to b$ by $i(1)=v_1$ and $o(2)=v_2$. Then we have an $F$-decorated cospan: 
\[
%
\]
which is given by this open graph:
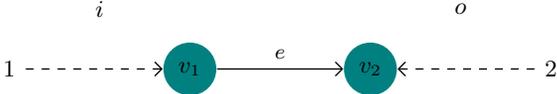
\[
%
\]

There are some subtleties to this framework; consider two decorated cospans with the same inputs and outputs.
\[
%
\]
For these two $F$-decorated cospans to be in the same isomorphism class, the following triangle is to commute:
\[
%
\]
This commutative triangle in $\mathsf{Set}$ in the context of the symmetric lax monoidal functor $F \colon \mathsf{FinSet} \to \mathsf{Set}$ says the following: given a decoration $d \in F(b)$, which is a graph structure with underlying set of vertices $b$, the function $F(f)$ pushes forward the graph structure $d$ to the graph structure $d^\prime \in F(b^\prime)$ with underlying set of vertices $b^\prime$, and \emph{precisely} this graph structure. The graph structure is given by the set of edges of $d$. For example, take $b = \{v_1 ,v_2\}$ as before and let $d \in F(b)$ be given by:
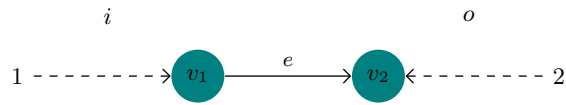
\[

\]
Let $b^\prime = \{w_1,w_2\}$ and define a bijection $f \colon b \to b^\prime$ by $f(v_i)=w_i$ for $i=1,2$. Then the requirement $F(f)(d)=d^\prime$ says that $d^\prime \in F(b^\prime)$ must be given by:
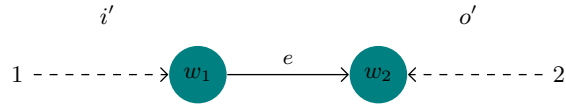
\[
%
\]
The important point is that the single edge of $d^\prime$ must also be $e$. If we were to label it say, $e^\prime$, there is no bijection $f \colon b \to b^\prime$ such that the triangle above commutes, and hence no isomorphism between these two $F$-decorated cospans.
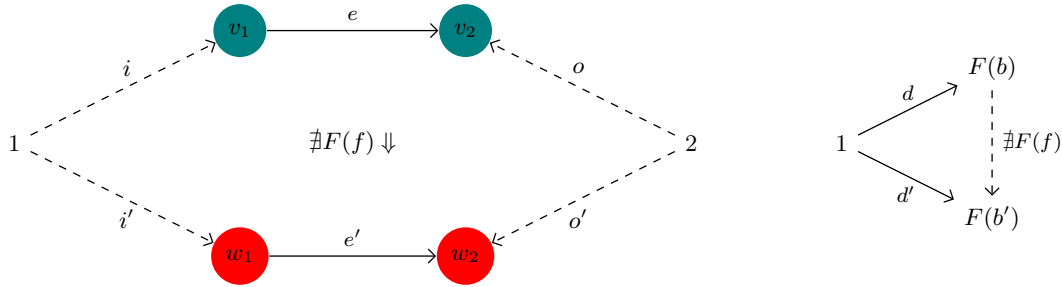
\[
%
\]
Thus, these two $F$-decorated cospans constitute distinct isomorphism classes! This nuisance is amplified when viewed from a higher categorical perspective, as seen in the first attempt at building a bicategory of decorated cospans \cite{Cour}. In the first proposed bicategory $F\mathbf{Cospan}(\mathsf{C})$, there is no 2-morphism from the former single-edged graph to the latter, when clearly there ought to be. The theory of foot-replaced double categories serves to remedy this situation. Again, for an introduction to double categories, see Appendix \ref{DoubleCatAppendix}.

Another obstacle with decorated cospans was pointed out by an anonymous referee of Moeller and Vasilakopoulou \cite{MV}. For the original incarnation of decorated cospans, we start with a symmetric lax monoidal functor $F \colon (\mathsf{C},+,0) \to (\mathsf{Set},\times,1)$ where $\mathsf{C}$ is a finitely cocomplete category made symmetric monoidal with chosen binary coproducts and an initial object. The anonymous referee has pointed out that even in the simplest of examples, namely the example of open graphs in Section \ref{Graphs}, the `laxator hexagon' required to commute in the definition of symmetric lax monoidal functor (Definition \ref{defn:monoidal_functor}) may do so only up to isomorphism. This can be seen explicitly with the following example.

Let $F \colon (\mathsf{FinSet},+,0) \to (\mathsf{Set},\times,1)$ be the functor of Section \ref{Graphs}. In order for $F$ to actually be a lax monoidal functor, the following laxator hexagon must commute:
\[
\begin{tikzpicture}[scale=1.5]
\node (A) at (0,0.5) {$(F(a) \otimes F(b)) \otimes F(c)$};
\node (A') at (4.5,0.5) {$F(a) \otimes (F(b) \otimes F(c))$};
\node (B) at (0,-0.25) {$F(a \otimes b) \otimes F(c)$};
\node (C) at (4.5,-0.25) {$F(a) \otimes F(b \otimes c)$};
\node (C') at (0,-1) {$F((a \otimes b) \otimes c)$};
\node (D) at (4.5,-1) {$F(a \otimes (b \otimes c))$};
\path[->,font=\scriptsize,>=angle 90]
(A) edge node[above]{$\alpha'$} (A')
(A) edge node[left]{$\mu_{a,b} \otimes 1_{F(c)}$} (B)
(A') edge node[right]{$1_{F(a)} \otimes \mu_{b,c}$} (C)
(B) edge node[left]{$\mu_{a \otimes b,c}$} (C')
(C) edge node [right] {$\mu_{a,b \otimes c}$} (D)
(C') edge node [above] {$F(\alpha)$} (D);
\end{tikzpicture}
\]
Let $a=\{a_1,a_2\}$, $b=\{b_1,b_2\}$ and $c=\{c_1,c_2\}$ all be two-element sets, and let $d_a \in F(a)$, $d_b \in F(b)$ and $d_c \in F(c)$ be given by the following graph structures:
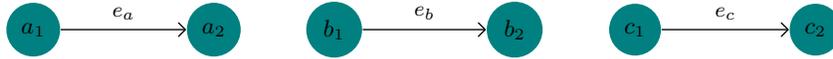
\[
\begin{tikzpicture}
  [scale=.8,auto=left]
  \node [style={circle,fill=teal}] (n1) at (1,10) {$a_1$};
  \node[style={circle,fill=teal}] (n2) at (4,10)  {$a_2$};
  \node [style={circle,fill=teal}] (n1') at (6,10) {$b_1$};
  \node[style={circle,fill=teal}] (n2') at (9,10)  {$b_2$};
  \node [style={circle,fill=teal}] (n1'') at (11,10) {$c_1$};
  \node[style={circle,fill=teal}] (n2'') at (14,10)  {$c_2$};
\path[->,font=\scriptsize,>=angle 90]
(n1) edge node[above]{$e_a$} (n2)
(n1') edge node[above]{$e_b$} (n2')
(n1'') edge node[above]{$e_c$} (n2'');
\end{tikzpicture}
\]
Then the graph $(d_a \times d_b) \times d_c \in (F(a) \otimes F(b)) \otimes F(c) = (F(a) \times F(b)) \times F(c)$ is an object of the category given by the top left corner of the above hexagon. Starting from this top left corner and traversing the object $(d_a \times d_b) \times d_c$ through the hexagon right and then down to $F(a \otimes (b \otimes c))=F(a + (b + c))$ results in a graph with vertex set $a+(b+c)$ and edge set $\{e_a\} + (\{e_b\} + \{e_c\})$, whereas traversing the hexagon down and then right results in a graph with the same vertex set $a+(b+c)$ but now edge set $(\{e_a\}+\{e_b\}) + \{e_c\}$. These two graphs would visually appear to be the same and indeed have the same sets of vertices, but their edge sets would only be (naturally) isomorphic, causing the above hexagon to not commute on-the-nose as required by the definition of lax monoidal functor.

One remedy to this as suggested by John Baez is to replace the finitely cocomplete category $(\mathsf{FinSet},+,0)$ containing our graph structures with an equivalent strictified version courtesy of a theorem of Mac Lane:

\begin{theorem}[Mac Lane \cite{ML}]\label{Mac Lane}
Given a (braided, symmetric) monoidal category $\mathsf{C}$, there exists a strict (braided, symmetric) monoidal category $\mathsf{C}'$ and a (braided) monoidal equivalence $F \colon \mathsf{C} \to \mathsf{C}'$. 
\end{theorem}

A monoidal equivalence $F$ is a functor that is simultaneously a monoidal functor and an equivalence, and a \emph{strict} (braided, symmetric) monoidal category is a (braided, symmetric) monoidal category in which the associator and left and right unitors are identity morphisms. By taking our graph structures from the strict monoidal category $(\mathsf{FinSet},+,0)$, the two graphs each with vertex sets $a+(b+c)$ and edge sets $\{e_a\} + (\{e_b\} + \{e_c\})$ and $(\{e_a\} + \{e_b\}) + \{e_c\}$ are now identified and thus the laxator hexagon commutes. A similar problem arises with two unitality squares which is also resolved by this strictification, and thus we obtain the lax monoidal functor $F \colon (\mathsf{FinSet},+,0) \to (\mathsf{Set},\times,1)$ of 
Section \ref{Graphs} and are able to utilize Theorem \ref{Fong2}. Unfortunately, we are unable to obtain the desired \emph{symmetric} monoidal category of Fong's original Theorem \ref{Fong}. Structured cospans will also serve as a remedy to this problem.

\section{Foot-replaced double categories}
The main content of this chapter are foot-replaced double categories as introduced in a work with Baez \cite{BC2}. A special case of foot-replaced double categories are given by structured cospan double categories. A \define{cospan} in any category is diagram of the form:
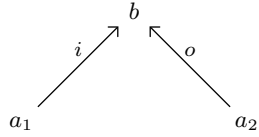
\[
\begin{tikzpicture}[scale=1.5]
\node (A) at (0,0) {$b$};
\node (B) at (-1,-1) {$a_1$};
\node (C) at (1,-1) {$a_2$};
\path[->,font=\scriptsize,>=angle 90]
(B) edge node [above] {$i$}(A)
(C) edge node [above] {$o$}(A);
\end{tikzpicture}
\]
We call $b$ the \define{apex} of the cospan, $i$ and $o$ the \define{legs} of the cospan, and $a_1$ and $a_2$ the \define{feet} of the cospan. In the framework of structured cospan double categories, given a functor $L \colon \mathsf{A} \to \mathsf{X}$ a \define{structured cospan} is a cospan in $\mathsf{X}$ of the form:
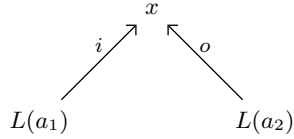
\[
\begin{tikzpicture}[scale=1.5]
\node (A) at (0,0) {$x$};
\node (B) at (-1,-1) {$L(a_1)$};
\node (C) at (1,-1) {$L(a_2)$};
\path[->,font=\scriptsize,>=angle 90]
(B) edge node [above] {$i$}(A)
(C) edge node [above] {$o$}(A);
\end{tikzpicture}
\]
Formally, this is a cospan in $\mathsf{X}$ whose feet are objects of $\mathsf{X}$, but from the perspective of structured cospans, the feet of this cospan are the objects $a_1$ and $a_2$ in $\mathsf{A}$. Here we are replacing the feet of the cospan in $\mathsf{X}$ with objects from another category $\mathsf{A}$, hence the name `foot-replaced double category'.
\begin{theorem} \label{trick1}
Given a double category $\mathbb{X}$ and a functor $L \maps \mathsf{A} \to \mathbb{X}_0$, there is a unique double category $_{L} \mathbb{X}$ for which:
\begin{itemize}
\item an object is an object of $\mathsf{A}$,
\item a vertical 1-morphism is a morphism of $\mathsf{A}$, 
\item a horizontal 1-cell from $a$ to $a'$ is a horizontal 1-cell $L(a) \xrightarrow{M} L(a')$ of $\mathbb{X}$, 
\item a 2-morphism is a 2-morphism in $\mathbb{X}$ of the form:
\[
\begin{tikzpicture}[scale=1.5]
\node (A) at (0,0) {$L(a)$};
\node (C) at (1,0) {$L(b)$};
\node (A') at (0,-1) {$L(a')$};
\node (C') at (1,-1) {$L(b')$,};
\node (B) at (0.5,-0.5) {$\Downarrow \alpha$};
\path[->,font=\scriptsize,>=angle 90]
(A) edge node[above]{$M$} (C)
(A) edge node [left]{$L(f)$} (A')
(C)edge node[right]{$L(g)$}(C')
(A')edge node [above] {$N$}(C');
\end{tikzpicture}
\]
\item composition of vertical 1-morphisms is composition in $\mathsf{A}$,
\item composition of horizontal 1-morphisms are defined as in $\mathbb{X}$,
\item vertical and horizontal composition of 2-morphisms is defined as in $\mathbb{X}$,
\item the associator and unitors are defined as in $\mathbb{X}$.
\end{itemize}
\end{theorem}
The proof is a straightforward verification using the definition of a double category, which is Definition \ref{defn:double_category}. Throughout this thesis we use `double category' to mean `pseudo double category': composition of horizontal 1-cells need not be strictly associative. However, if the double category $\mathbb{X}$ is strict, so is the foot-replaced double category $_L \mathbb{X}$. 

There is also a version of Theorem \ref{trick1} for symmetric monoidal double categories.

\begin{theorem}\label{trick2}
If $\lX$ is a symmetric monoidal double category, $\A$ is a symmetric monoidal
category and $L \maps \A \to \lX_0$ is a (strong) symmetric monoidal functor, then the double category $_L \lX$ becomes symmetric monoidal in a canonical way. 
\end{theorem}

\begin{proof}
As noted in Definition\ \ref{defn:double_category}, every double category $\lD$ has not only a category of objects $\lD_0$, but also a category of arrows $\lD_1$ with horizontal 1-cells of $\lD$ as
objects and 2-morphisms of $\lD$ as morphisms.  The definition of a symmetric monoidal double category, which is Definition \ref{defn:symmetric_monoidal_double_category}, can be expressed in terms of structure involving these categories.

For the double category $_L \lX$, the category of objects $_L \lX_0$ is just $\A$.   The category of arrows $_L \lX_1$ has horizontal 1-cells in $\lX$ of this form:
\[   L(a) \xrightarrow{M} L(b) \]
as objects and diagrams in $\lX$ of this form:
\[
\begin{tikzpicture}[scale=1.5]
\node (A) at (0,0) {$L(a)$};
\node (C) at (1,0) {$L(b)$};
\node (A') at (0,-1) {$L(a')$};
\node (C') at (1,-1) {$L(b')$};
\node (B) at (0.5,-0.5) {$\Downarrow \alpha$};
\path[->,font=\scriptsize,>=angle 90]
(A) edge node[above]{$M$} (C)
(A) edge node [left]{$L(f)$} (A')
(C) edge node[right]{$L(g)$}(C')
(A') edge node [above] {$N$}(C');
\end{tikzpicture}
\]
as morphisms, which are composed vertically.

As explained in Definition\ \ref{defn:monoidal_double_category}, to 
make $_L \lX$ into a monoidal double category we need to do the following:

(1) We must choose a monoidal structure for $_L \lX_0 = \A$ and for $_L \lX_1$.   The category
$\A$ is monoidal by hypothesis; we give $_L \lX_1$ a monoidal structure
using the fact that $\lX_1$ and the functor $L$ are strong monoidal,
as follows.    Given two objects of $_L \lX_1$:
\[ L(a_1) \xrightarrow{M} L(a_2) \qquad  L(b_1) \xrightarrow{N} L(b_2) \]
their tensor product is
\[  L(a_1 \otimes b_1) \xrightarrow{\phi^{-1}_{a_1,b_1}} 
        L(a_1) \otimes L(b_1) \xrightarrow{M \otimes N} L(a_2) \otimes L(b_2) 
        \xrightarrow{\phi_{a_2,b_2}} L(a_2 \otimes b_2), \]
defined using the laxator $\phi_{a,b} \maps L(a) \otimes L(b) \to L(a \otimes b)$ for $L$.  Note that $\phi$ is invertible because $L$ is strong monoidal.  Given two morphisms of $_L \lX_1$:
\[

\]
their tensor product is defined to be
\[
%
\]
The monoidal unit for $_L \lX_1$ is 
\begin{equation}
\label{eq:monoidal_unit}     L(I) \xrightarrow{\hat{U}(L(I))} L(I) 
\end{equation}
where $I$ is the monoidal unit for $\A$ and $\hat{U} \maps \lX_0 \to \lX_1$ is the identity-assigning functor for $\lX$.  The associator and unitors for $_L \lX_1$ are built
from those in $\lX_1$. Explicitly, given three horizontal 1-cells $M,N$ and $P$ in $_L \lX_1$:
\[
%
\]
the associator $\alpha_{M,N,P} \colon (M \otimes N) \otimes P \xrightarrow{\sim} M \otimes (N \otimes P)$ in $_L \lX_1$ is given by:
\[
%
\]

(2) Any double category $\lD$ has an identity-assigning functor $U \maps \lD_0 \to \lD_1$, and
for $\lD$ to be monoidal we need $U$ to preserve the monoidal unit.  This is true for $_L \lX$ because $U \maps \A \to {}_L \lX_1$ maps any object $a \in \A$ to 
\[         L(a) \xrightarrow{\hat{U}(L(a))} L(a), \]
so $U$ maps the monoidal unit $I \in \A$ to the monoidal unit for ${}_L \lX_1$, given in
Equation \eqref{eq:monoidal_unit}.

(3) In a monoidal double category $\lD$ the source and target functors $S, T \maps \lD_1 \to \lD_0$
must be strict monoidal.  For ${}_L \lX$ this is easy to check, given the monoidal structures
defined in item (1), because the source and target of an object
\[   L(a) \xrightarrow{M} L(b) \]
of $_L \lX_1$ are $a \in {}_L \lX_0$ and $b \in {}_L \lX_0$, respectively, and the source and target of a morphism 
\[
\begin{tikzpicture}[scale=1.5]
\node (A) at (0,0) {$L(a)$};
\node (C) at (1,0) {$L(b)$};
\node (A') at (0,-1) {$L(a')$};
\node (C') at (1,-1) {$L(b')$};
\node (B) at (0.5,-0.5) {$\Downarrow \alpha$};
\path[->,font=\scriptsize,>=angle 90]
(A) edge node[above]{$M$} (C)
(A) edge node [left]{$L(f)$} (A')
(C)edge node[right]{$L(g)$}(C')
(A')edge node [above] {$N$}(C');
\end{tikzpicture}
\]
in $_L \lX_1$ are the morphisms $f \colon a \to a'$ and $g \colon b \to b'$ in $_L \lX_0$, respectively. We can choose the images of the source and target functors to ensure that they are strict symmetric monoidal, meaning that for two horizontal 1-cells $M$ and $N$, $$S(M \otimes N) = a \otimes a' = S(M) \otimes S(N)$$ and likewise for the target morphism $T$.  The unit for the tensor product in ${}_L \lX_1$ is given in Equation \eqref{eq:monoidal_unit}, and applying $S$ or $T$ we obtain $I \in {}_L \lX_0$.

(4) A \define{globular 2-morphism} in a double category $\lD$ is a morphism $\alpha$ in $\lD_1$ such that $S\alpha$ and $T\alpha$ are identity morphisms in $\lD_0$.   In a monoidal double category $\lD$ we must have invertible globular 2-morphisms 
\[\chi \maps (M_1\otimes N_1)\odot (M_2\otimes N_2)\xrightarrow{\sim} (M_1\odot M_2)\otimes (N_1\odot N_2)\]
and
\[ \mu\maps U_{A\otimes B} \xrightarrow{\sim} (U_A \otimes U_B)\]
expressing the compatiblity of the composition functor $\odot \colon \lD_1 \times_{\lD_0} \lD_1 \to \lD_1$ and identity-assigning functor $U \maps \lD_0 \to \lD_1$ with the tensor product.   These must make three diagrams commute, as detailed in Definition\ \ref{defn:monoidal_double_category}.  
In the case of $_L \lX$ this follows from the commutativity of the corresponding diagrams in $\lX$ together with the natural isomorphisms given by the invertible laxators of the strong monoidal functor $L \colon \A \to \X$. Explicitly, given composable horizontal 1-cells $M_1,M_2,N_1$ and $N_2$ in $_L \lX_1$:
\[

\]
the globular 2-morphism $\chi$ for $_L \lX$ is given by:
\[
%
\]
where the middle $\chi$ in the above diagram in the corresponding globular 2-morphism for the symmetric monoidal double category $\lX$ and $\phi_{a_i,b_i} \colon L(a_i) \otimes L(b_i) \to L(a_i \otimes b_i)$ is the natural isomorphism of the strong monoidal functor $L \colon \A \to \lX_0$. Similarly, the other globular 2-morphism $\mu$ for $_L \lX$ is given by:
\[
%
\]

(5) In a monoidal double category, the associator and left and right unitors must be transformations of double categories.  This means that six diagrams must commute, as detailed in Definition\ \ref{defn:monoidal_double_category}.   In the case of $_L \lX$ this follows from the commuting of the corresponding diagrams in $\lX$ together with the natural isomorphisms given by the invertible laxators of the strong monoidal functor $L \colon \A \to \X$. For instance, one of the diagrams required to commute is given by:
\[
		%
\]
For the symmetric monoidal double category $_L \lX$, this diagram may be seen as:
\[
		%
	\]
Here we have `unrolled' the diagram to make it fit on the page; the reader should identify the objects at the top of the diagram with those at the bottom.

Similarly, a braided monoidal double category is a monoidal double
category with the following additional structure.

(6) $\lD_{0}$ and $\lD_1$ are braided monoidal categories.

(7) The functors $S$ and $T$ are strict braided monoidal (i.e.\ they
  preserve the braidings).

(8) The following diagrams commute, expressing that the braiding is
  a transformation of double categories.
  \[\xymatrix{(M_1\odot M_2)\otimes (N_1\odot N_2) \ar[r]^\fs\ar[d]_\chi &
    (N_1\odot N_2)\otimes (M_1 \odot M_2)\ar[d]^\chi\\
    (M_1\otimes N_1)\odot (M_2\otimes N_2) \ar[r]_{\fs\odot \fs} &
    (N_1\otimes M_1) \odot (N_2 \otimes M_2)}
  \]
  \[\xymatrix{U_A \otimes U_B \ar[r]^(0.55)\mu \ar[d]_\fs &
    U_{A\otimes B} \ar[d]^{U_\fs}\\
    U_B\otimes U_A \ar[r]_(0.55)\mu &
    U_{B\otimes A}}.
  \]

These follow from the fact that $\lX_0$ and $\lX_1$ are braided monoidal categories and that the corresponding functors $S$ and $T$ of $\lX$ are strict braided monoidal and we can choose the source and target functors of $_L \lX$ to agree with the braidings of $_L \lX_0$ and $_L \lX_1$, meaning that $$\beta' (S(M \otimes N)) = \beta' (S(M) \otimes S(N)) = \beta'(a \otimes a') = a' \otimes a = S(N \otimes M) = S(\beta(M \otimes N))$$ and likewise for the target morphism $T$. The above diagrams commute in $_L \lX$ as the corresponding diagrams commute in $\lX$ and the laxators of the strong monoidal functor $L$ are invertible.

(9) $\lD_0$ and $\lD_1$ are symmetric monoidal categories.

This follows from the fact that $\A$, $\lX_0$ and $\lX_1$ are symmetric monoidal categories. Explicitly, the triangle identity for $_L \lX_1$ is given by:
\[

	\]
Here we have again `unrolled' the diagram to make it fit on the page; the objects at the top of the diagram should be identified with those at the bottom.

Now for notation, let $M,N,P$ and $Q$ be horizontal 1-cells in $_L \lX$ given by:
\[
%
\]
As horizontal 1-cells of the symmetric monoidal double category $\lX$ together with the associator $\hat{\alpha}$ of $\lX$, the following pentagon commutes:
\[
%
\]
Unrolling the pentagon identity for $_L \lX_1$, we obtain the following diagram:
\[
\hspace{-0.9in}
		%
	\]
in which the top and the bottom tensor products of horizontal 1-cells coincide. The \textcolor{red}{red} is to highlight that the pentagon identity of $\lX_1$ is nested within the pentagon identity of $_L \lX_1$, and likewise for the triangle identity on the previous page, although that one we have not colored.
\end{proof}

\section{Structured cospan double categories}

The most important example of a double category in this thesis is given by $\mathbb{C}\mathbf{sp}(\mathsf{X})$ for some category $\mathsf{X}$ with pushouts. This double category has:
\begin{enumerate}
\item{objects as those of $\mathsf{X}$,}
\item{vertical 1-morphisms as morphisms of $\mathsf{X}$,}
\item{horizontal 1-cells as cospans in $\mathsf{X}$, and}
\item{2-morphisms as maps of cospans in $\mathsf{X}$ given by commutative diagrams of the form:
\[
\begin{tikzpicture}[scale=1.5]
\node (E) at (3,0) {$x$};
\node (F) at (5,0) {$y$};
\node (G) at (4,0) {$z$};
\node (E') at (3,-1) {$x'$};
\node (F') at (5,-1) {$y'$};
\node (G') at (4,-1) {$z'$};
\path[->,font=\scriptsize,>=angle 90]
(F) edge node[above]{$o$} (G)
(E) edge node[left]{$f$} (E')
(F) edge node[right]{$h$} (F')
(G) edge node[left]{$g$} (G')
(E) edge node[above]{$i$} (G)
(E') edge node[below]{$i'$} (G')
(F') edge node[below]{$o'$} (G');
\end{tikzpicture}
\]
}
\end{enumerate}

That $\lCsp(\X)$ is indeed a double category when $\X$ is a category with pushouts was shown by Niefield \cite{Nie}; see also \cite{Cour}. This also follows from Theorem \ref{DCord} when the decorations are taken to be trivial---see Corollary \ref{DCordCor}.

\begin{theorem} \label{_L Csp(X)}
Let $L \maps \A \to \X$ be a functor where $\X$ is a category with pushouts. Then there exists a double category $_L \lCsp(\X)$ for which:
\begin{itemize}
\item an object is an object of $\A$,
\item a vertical 1-morphism is a morphism of $\A$,
\item a horizontal 1-cell from $a$ to $b$ is an $L$-\define{structured cospan}, meaning a cospan in $\X$ of the form:
\[

\]
\item a 2-morphism is a \define{map of} $L$-\define{structured cospans}, meaning a commutative diagram in $\X$ of this form:
\[
%
\]
\item composition of vertical 1-morphisms is morphism composition in $\A$,
\item composition of horizontal 1-cells is done using chosen pushouts in $\X$: 
\[
%
\]
where $j_x$ and $j_y$ are the canonical morphisms from $x$ and $y$ into the pushout, 
\item 
the horizontal composite of two 2-morphisms:
\[
%
\]
is given by
\[
%
\]
\item The vertical composite of two 2-morphisms:
\[
%
\]
is given by
\[
%
\]
\item The associator and unitors are defined using the universal property of pushouts.
\end{itemize}
\end{theorem}

\begin{proof}
We apply Theorem \ref{trick1} to the double category $\mathbb{C} \mathbf{sp}(\X)$.
\end{proof}

If the category $\lX$ has not only pushouts but also finite colimits, meaning pushouts and an initial object which will serve as the unit object for tensoring, then the aforementioned double category $\mathbb{C}\mathbf{sp}(\lX)$ is in fact symmetric monoidal.

\begin{corollary}
Given a category $\X$ with finite colimits, the double category $\lCsp(\X)$  is symmetric monoidal with the monoidal structure given by chosen coproducts in $\X$.  Thus:
\begin{itemize}
\item the tensor product of two objects $x_1$ and $x_2$ is $x_1 + x_2$,
\item the tensor product of two vertical 1-morphisms is given by
\[

\]
\item the tensor product of two horizontal 1-cells is given by
\[
%
\]
\item the tensor product of two 2-morphisms is given by
\[
%
\]
\item The unit for the tensor product is a chosen initial object of $\X$,
\item The symmetry for any two objects $x$ and $y$ is defined using the canonical isomorphism $x + y \cong y + x$.
\end{itemize}
\end{corollary}

\begin{proof}
This is just a special case of Theorem \ref{DC} where, as in Corollary \ref{DCordCor}, each $F$-decorated cospan is once again equipped with the trivial decoration.
\end{proof}

We then have the following symmetric monoidal double category of \emph{structured cospans}, the primary result of the aforementioned work \cite{BC2}.

\begin{theorem}\label{SC}
Let $L \maps \mathsf{A} \to \mathsf{X}$ be a functor preserving finite coproducts, where $\mathsf{A}$ has finite coproducts and $\mathsf{X}$ has finite colimits. Then the double category $_L \mathbb{C}\mathbf{sp}(\mathsf{X})$ is symmetric monoidal with the monoidal structure given by chosen coproducts in $\mathsf{A}$ and $\mathsf{X}$. Thus:
\begin{enumerate}
\item{the tensor product of two objects $a_1$ and $a_2$ is $a_1 + a_2$,}
\item{the tensor product of two vertical 1-morphisms is given by
\[

\]
}
\item{the tensor product of two horizontal 1-cells is given by
\[
%
\]
where the feet use the tensor product of $\mathsf{A}$ and the legs and apices use the tensor product of $\mathsf{X}$ and invertible laxators of $L$, and likewise
}

\item{the tensor product of two 2-morphisms is given by:
\[
%
\]}
\end{enumerate}
The unit for the tensor product is the initial object of $\mathsf{X}$ which is isomorphic to the image of the unit object of $\mathsf{A}$ under the functor $L$, and the symmetry for any two objects $a$ and $b$ is defined using the canonical isomorphism $a + b \cong b + a$.
\end{theorem}

Theorem \ref{SC} is one of the main results on structured cospans in a joint work with Baez \cite{BC2}. The method of proof used there however is different from the more direct approach taken here in this thesis. The word `rex' is a standard abbreviation of `right exact', which means finitely cocontinuous, i.e., preserving finite colimits.  Denoting by $\mathbf{Rex}$ the 2-category of finitely cocomplete categories, finitely cocontinuous functors and natural transformations, it is shown that if $\mathsf{A} \in \mathbf{Rex}$, then $\mathbb{C}\mathbf{sp}(\mathsf{A})$ is a `pseudocategory object' in $\mathbf{Rex}$---see Definition \ref{pseudocategory}. A morphism $L \colon \mathsf{A} \to \mathsf{X}$ then yields the above symmetric monoidal double category $_L \mathbb{C}\mathbf{sp}(\mathsf{X})$ being realized as a pseudocategory object in $\mathbf{Rex}$. Denoting by $\mathbf{SymMonCat}$ the 2-category of symmetric monoidal categories, (strong) symmetric monoidal functors and monoidal natural transformations, there exists a 2-functor $\Phi \colon \mathbf{Rex} \to \mathbf{SymMonCat}$ which turns a finitely cocomplete category into a symmetric monoidal category by prescription of chosen binary coproducts for every pair of objects to serve as their tensor product and a chosen initial object to serve as the monoidal unit. The rest of the symmetric monoidal structure is then induced by these choices. This 2-functor $\Phi$ preserves the necessary pullbacks and applying this 2-functor $\Phi$ to $_L \mathbb{C}\mathbf{sp}(\mathsf{X})$ then results in $\Phi({_L \mathbb{C}\mathbf{sp}(\mathbf{X})})$ as a pseudocategory object in $\mathbf{SymMonCat}$. A pseudocategory object in the 2-category $\mathbf{Cat}$ is the same as a double category.  A pseudocategory object in $\mathbf{SymMonCat}$ is almost the same as a symmetric monoidal double category, but not quite, because the source and target functors $S$ and $T$ are not required to be strict symmetric monoidal functors.  Luckily, an easy verification shows that that this is indeed the case for  $\Phi(_L \lCsp(\X))$, so it is a symmetric monoidal double category.

Analogous comments also apply for maps between structured cospan double categories, which are given by weakly commuting squares in $\mathbf{Rex}$:
\[
\begin{tikzpicture}[scale=1.5]
\node (A) at (0,0) {$\A$};
\node (B) at (1,0) {$\X$};
\node (C) at (0,-1) {$\A'$};
\node (D) at (1,-1) {$\X'$};
\node (F) at (0.5,-0.5) {$\alpha \NEarrow$};
\path[->,font=\scriptsize,>=angle 90]
(A) edge node [above] {$L$} (B)
(A) edge node [left]{$F_0$} (C)
(B) edge node [right]{$F_1$} (D)
(C) edge node [below] {$L'$} (D);
\end{tikzpicture}
\]

Assuming $L \colon \mathsf{A} \to \mathsf{X}$ is a morphism in $\mathbf{Rex}$ is stronger than the hypothesis used in Theorem \ref{SC}, but this simplifies many proofs and also produces stronger results: not only can we tensor and compose structured cospans as we can in an ordinary symmetric monoidal double category of structured cospans, but we can even take \emph{finite colimits} of structured cospans, themselves. This is not the case for the symmetric monoidal double category $_L \mathbb{C}\mathbf{sp}(\mathsf{X})$ of Theorem \ref{SC} due to $\mathsf{A}$ only being required to have finite coproducts and only requiring finite coproducts be preserved by $L$.


A well-known result regarding adjoints is the following.
\begin{proposition}
Every left adjoint $L \colon \mathsf{A} \to \mathsf{X}$ preserves all colimits and every right adjoint $R \colon \mathsf{X} \to \mathsf{A}$ preserves all limits.
\end{proposition}

The following is a particularly useful result on structured cospan double categories.

\begin{corollary}\label{AdjointSC}
Let $L \maps \mathsf{A} \to \mathsf{X}$ be a left adjoint between two categories $\mathsf{A}$ and $\X$ with finite colimits. Then the double category $_L \mathbb{C}\mathbf{sp}(\mathsf{X})$ is symmetric monoidal with the monoidal structure given as in Theorem \ref{SC}.
\end{corollary}
The examples we present of structured cospan double categories, which are to be seen as improvements of the corresponding examples of decorated cospans of the previous chapter, will be applications of the above corollary. Another application may be found in the work of Cicala \cite{Cic} who uses structured cospan double categories to study rewrite rules in a topos.

\section{Applications}\label{SecSCApp}

\subsection{Graphs}\label{scgraphs}

\begin{definition}
Let $\mathsf{FinGraph}$ be the category whose objects are finite graphs, which are diagrams in $\mathsf{FinSet}$ of the form:
\[
\begin{tikzpicture}[scale=1.5]
\node (B) at (1,0) {$E$};
\node (C) at (2,0) {$N$};
\path[->,font=\scriptsize,>=angle 90]
(B)edge[bend left] node[above]{$s$}(C)
(B)edge[bend right] node[below]{$t$}(C);
\end{tikzpicture}
\]
and whose morphisms are given by pairs of functions $(f,g)$ such that the following two squares commute:
\[
\begin{tikzpicture}[scale=1.5]
\node (A) at (5,0) {$E$};
\node (A') at (5,-1) {$E'$};
\node (E) at (6,0) {$N$};
\node (E') at (6,-1) {$N'$};
\node (B) at (7,0) {$E$};
\node (B') at (7,-1) {$E'$};
\node (F) at (8,0) {$N$};
\node (F') at (8,-1) {$N'$};
\path[->,font=\scriptsize,>=angle 90]
(E) edge node [right] {$f$} (E')
(F) edge node [right] {$f$} (F')
(A) edge node[above]{$s$}(E)
(A') edge node [below] {$s'$} (E')
(A) edge node [left] {$g$} (A')
(B) edge node [left] {$g$} (B')
(B) edge node[above]{$t$}(F)
(B') edge node [below] {$t'$} (F');
\end{tikzpicture}
\]
\end{definition}
Define a functor $L \maps \mathsf{FinSet} \to \mathsf{FinGraph}$ where given a set $N$, $L(N)$ is the discrete graph on $N$ (with no edges) and given a function $f \maps N \to N'$, $L(f) \maps L(N) \to L(N')$ is the graph morphism that takes vertices of $L(N)$ to $L(N')$ as prescribed by the function $f$. This functor $L$ preserves finite coproducts as it is left adjoint to the forgetful functor $R \maps \mathsf{FinGraph} \to \mathsf{FinSet}$ that takes a graph $(E,N,s,t)$ where $N$ and $E$ are finite to its underlying set of vertices $N$. The categories $\mathsf{FinSet}$ and $\mathsf{FinGraph}$ both have finite colimits. By Corollary \ref{AdjointSC}, we have the following.

\begin{theorem}\label{scgraphs}
Let $L \maps \mathsf{FinSet} \to \mathsf{FinGraph}$ be the left adjoint defined above. Then there exists a symmetric monoidal double category $_L \mathbb{C}\mathbf{sp}(\mathsf{FinGraph})$ which has:
\begin{enumerate}
\item{finite sets as objects,}
\item{functions as vertical 1-morphisms,}
\item{\define{open graphs} which are cospans of graphs of the form
\[

\]
as horizontal 1-cells, and}
\item{\define{maps of open graphs} which are maps of cospans of graphs as 2-morphisms, as in the following commutative diagram:
\[
%
\]
}
\end{enumerate} 
\end{theorem}

\subsection{Electrical circuits}
Recall from Section \ref{k-graphs} that given a field $k$, a field with positive elements is a pair $(k,k^+)$ where $k^+ \subset k$ is a subset such that $r^2 \in k^+$ for every nonzero $r \in k$ and such that $k^+$ is closed under addition, multiplication and division. A recent work of Baez and Fong \cite{BF} studies $k$-graphs where a $k$-graph $\Gamma$ is given by a diagram in $\mathsf{Set}$ of the form:
\[
%
\]
where $E$ and $N$ are finite sets. Here $k$ is a field with positive elements and the finite sets $E$ and $N$ denote the sets of edges and nodes, respectively, of the $k$-graph $\Gamma$. An open $k$-graph is then given by a cospan of finite sets:
\[
%
\]
where the apex $N$ is decorated with a $k$-graph as above. Fong and Baez use the decorated cospan machinery of Fong to construct a monoidal category $F\mathsf{Cospan}$ from a lax monoidal functor $F \maps \mathsf{FinSet} \to \mathsf{Set}$. This functor $F$ is defined on objects by:
\[
%
\]
and on morphisms by
\[
%
\]
To fit the above construction into the framework of structured cospans, first we define a category $\mathsf{FinGraph}_k$ whose objects are given by finite $k$-graphs:
\[
%
\]
and a morphism from this $k$-graph to another:
\[
%
\]
consists of a pair of functions $f \maps N \to N'$ and $g \maps E \to E'$ such that the following diagrams commute:
\[
%
\]
Next, we define a left adjoint $L \maps \mathsf{FinSet} \to \mathsf{FinGraph}_k$ which is defined on sets by:
\[
%
\]
and on morphisms by:
\[
%
\]
\begin{lemma}\label{left_adj}
The functor $L \maps \mathsf{FinSet} \to \mathsf{FinGraph}_k$ defined above is left adjoint to the forgetful functor $R \maps \mathsf{FinGraph}_k \to \mathsf{FinSet}$.
\end{lemma}
\begin{proof}
The functor $L \maps \mathsf{FinSet} \to \mathsf{FinGraph}_k$ has a right adjoint given by the forgetful functor $R \maps \mathsf{FinGraph}_k \to \mathsf{FinSet}$ which maps a finite $k$-graph
\[
%
\]
to its underlying vertex set $N$. We then have a natural isomorphism $\hom_{\mathsf{FinGraph}_k}(L(c),d) \cong \hom_{\mathsf{FinSet}}(c,R(d))$. 
\end{proof}

\begin{lemma}\label{finite_colimits}
The category $\mathsf{FinGraph}_k$ has finite colimits.
\end{lemma}

\begin{proof}
The category $\mathsf{FinGraph}_k$ has an initial object given by the empty $k$-graph as well as pushouts given by taking the pushout of the underlying span of finite graphs which is done pointwise.
\end{proof}

\begin{theorem}\label{left_adj_smdc}
Let $L \maps \mathsf{FinSet} \to \mathsf{FinGraph}_k$ be the left adjoint as described above. Then there exists a symmetric monoidal double category $_{L} \mathbb{C}\mathbf{sp} (\mathsf{FinGraph}_k)$ which has:
\begin{enumerate}
\item{finite sets as objects,}
\item{functions as vertical 1-morphisms,}
\item{\define{open $k$-graphs}: that is, cospans of finite sets where the apex is equipped with a $k$-graph
\[
\begin{tikzpicture}[scale=1.5]
\node (A) at (0,0) {$L(a)$};
\node (B) at (1,0) {$N$};
\node (C) at (2,0) {$L(b)$};
\path[->,font=\scriptsize,>=angle 90]
(A) edge node[above]{$i$} (B)
(C)edge node[above]{$o$}(B);
\end{tikzpicture}
\]
\[
\begin{tikzpicture}[scale=1.5]
\node (A) at (0,0) {$k^+$};
\node (B) at (1,0) {$E$};
\node (C) at (2,0) {$N$};
\path[->,font=\scriptsize,>=angle 90]
(B) edge node[above]{$r$} (A)
(B)edge[bend left] node[above]{$s$}(C)
(B)edge[bend right] node[below]{$t$}(C);
\end{tikzpicture}
\]
as horizontal 1-cells, and}
\item{maps of cospans of finite sets equipped with a map of $k$-graphs
\[
\begin{tikzpicture}[scale=1.5]
\node (A) at (0,0) {$L(a)$};
\node (B) at (1,0) {$N$};
\node (C) at (2,0) {$L(b)$};
\node (A') at (0,-1) {$L(a')$};
\node (B') at (1,-1) {$N'$};
\node (C') at (2,-1) {$L(b')$};
\path[->,font=\scriptsize,>=angle 90]
(A) edge node[above]{$i$} (B)
(C)edge node[above]{$o$}(B)
(A') edge node[above]{$i'$} (B')
(C')edge node[above]{$o'$}(B')
(A) edge node [left] {$L(h_1)$} (A')
(C) edge node [right] {$L(h_2)$} (C')
(B) edge node [left] {$f$} (B');
\end{tikzpicture}
\]
\[
\begin{tikzpicture}[scale=1.5]
\node (C) at (3,-0.5) {$k^+$};
\node (D) at (4,0) {$E$};
\node (D') at (4,-1) {$E'$};
\node (A) at (5,0) {$E$};
\node (A') at (5,-1) {$E'$};
\node (E) at (6,0) {$N$};
\node (E') at (6,-1) {$N'$};
\node (B) at (7,0) {$E$};
\node (B') at (7,-1) {$E'$};
\node (F) at (8,0) {$N$};
\node (F') at (8,-1) {$N'$};
\path[->,font=\scriptsize,>=angle 90]
(E) edge node [right] {$f$} (E')
(F) edge node [right] {$f$} (F')
(D) edge node[above]{$r$} (C)
(A) edge node[above]{$s$}(E)
(A') edge node [below] {$s'$} (E')
(D') edge node[below]{$r'$} (C)
(D) edge node [left] {$g$} (D')
(A) edge node [left] {$g$} (A')
(B) edge node [left] {$g$} (B')
(B) edge node[above]{$t$}(F)
(B') edge node [below] {$t'$} (F');
\end{tikzpicture}
\]
as 2-morphisms.}
\end{enumerate}
\end{theorem}

\begin{proof}
As $\mathsf{FinGraph}_k$ has finite colimits, we get a symmetric monoidal double category $\mathbb{C}\mathbf{sp}(\mathsf{FinGraph}_k)$ and hence a symmetric monoidal structured cospan double category $_L \mathbb{C}\mathbf{sp}(\mathsf{FinGraph}_k)$.
\end{proof}

\subsection{Petri nets}\label{petri_nets_structured_cospans}
For the last example, Baez and Pollard have constructed a black-boxing functor $\blacksquare \maps \Dynam \to \SemiAlg\Rel$ \cite{BP}. Here, $\Dynam$ is a symmetric monoidal category of `open dynamical systems' and $\SemiAlg\Rel$ is a symmetric monoidal category of `semialgebraic relations'. A particular kind of dynamical system is given by a Petri net with rates. Petri nets have also been studied extensively by Baez and Master \cite{BM} in the context of double categories and double functors.

Recall that a Petri net consists of a set $S$ of \emph{species}, a set $T$ of \emph{transitions} and functions $s,t \maps S \times T \to \mathbb{N}$. For a species $\sigma \in S$ and a transition $\tau \in T$, $s(\sigma,\tau)$ is the number of times the species $\sigma$ appears as an input for the transition $\tau$ and $t(\sigma,\tau)$ is the number of times the species $\sigma$ appears as an output for the transition $\tau$. 

\begin{definition}
A \define{Petri net with rates} is a Petri net with finite sets of species and transitions together with a function $r \maps T \to [0,\infty)$ where $r(\tau)$ is the rate of the transition $\tau$.
\end{definition}
We can also say that a Petri net with rates is a diagram of the form:
\[

\]
where $S$ and $T$ are finite sets and $\N[S]$ is the free commutative monoid on $S$. An open Petri net with rates is then given by a cospan of finite sets whose apex is equipped with a Petri net with rates.
\[
%
\]
A \define{map of Petri nets with rates} is given by a pair of functions $f \colon S \to S'$ and $g \colon T \to T'$ which make the following diagrams commute:
\[
%
\]
Two Petri nets with rates are then in the same isomorphism class if the following diagrams commute:
\[
%
\]
for some isomorphism $f$. Define a functor $L \maps \FinSet \to \Petri_{\mathrm{rates}}$ where for a finite set $S$, $L(S)$ is the \emph{discrete} Petri net with rates with $S$ as its set of species and no transitions. In other words,
\[
%
\]
\begin{lemma}\label{left_adj_petri_rates}
The functor $L \maps \FinSet \to \Petri_{\mathrm{rates}}$ defined above is left adjoint to the forgetful functor $R \maps \Petri_{\mathrm{rates}} \to \FinSet$.
\end{lemma}
\begin{proof}
This is similar as to why the functors used in the previous two applications are also left adjoints.
\end{proof}
\begin{lemma}\label{finite_colimits_petri_rates}
The category $\Petri_{\mathrm{rates}}$ has finite colimits.
\end{lemma}
\begin{proof}
This is similar to the proof of Lemma \ref{finite_colimits} --- the category $\Petri_{\textnormal{rates}}$ has pushouts and an initial object.
\end{proof}
\begin{theorem}\label{scpetri}
Let $L \maps \FinSet \to \Petri_{\mathrm{rates}}$ be the left adjoint described above. Then there exists a symmetric monoidal double category $_L \lCsp(\Petri_{\mathrm{rates}})$ which has:
\begin{enumerate}
\item{finite sets as objects,}
\item{functions as vertical 1-morphisms,}
\item{cospans of sets whose apices are equipped with the stuff of a Petri net with rates as horizontal 1-cells, and}
\item{maps of cospans as above as 2-morphisms, as in the following commutative diagrams.
\[
\begin{tikzpicture}[scale=1.5]
\node (E) at (3,0) {$L(a)$};
\node (F) at (5,0) {$L(b)$};
\node (G) at (4,0) {$S$};
\node (E') at (3,-1) {$L(a')$};
\node (F') at (5,-1) {$L(b')$};
\node (G') at (4,-1) {$S'$};
\path[->,font=\scriptsize,>=angle 90]
(F) edge node[above]{$o$} (G)
(E) edge node[left]{$L(h_1)$} (E')
(F) edge node[right]{$L(h_2)$} (F')
(G) edge node[left]{$f$} (G')
(E) edge node[above]{$i$} (G)
(E') edge node[above]{$i'$} (G')
(F') edge node[above]{$o'$} (G');
\end{tikzpicture}
\]
\[
\begin{tikzpicture}[scale=1.5]
\node (C) at (3,-0.5) {$[0,\infty)$};
\node (D) at (4,0) {$T$};
\node (D') at (4,-1) {$T'$};
\node (A) at (5,0) {$T$};
\node (A') at (5,-1) {$T$};
\node (E) at (6,0) {$\mathbb{N}[S]$};
\node (E') at (6,-1) {$\mathbb{N}[S']$};
\node (B) at (7,0) {$T$};
\node (B') at (7,-1) {$T'$};
\node (F) at (8,0) {$\mathbb{N}[S]$};
\node (F') at (8,-1) {$\mathbb{N}[S']$};
\path[->,font=\scriptsize,>=angle 90]
(E) edge node [right] {$\mathbb{N}[f]$} (E')
(F) edge node [right] {$\mathbb{N}[f]$} (F')
(D) edge node[above]{$r$} (C)
(A) edge node[above]{$s$}(E)
(A') edge node [below] {$s'$} (E')
(D') edge node[below]{$r'$} (C)
(D) edge node [left] {$g$} (D')
(A) edge node [left] {$g$} (A')
(B) edge node [left] {$g$} (B')
(B) edge node[above]{$t$}(F)
(B') edge node [below] {$t'$} (F');
\end{tikzpicture}
\]
}
\end{enumerate}
\end{theorem}

\begin{proof}
This follows from Corollary \ref{AdjointSC}, Theorem \ref{left_adj_petri_rates} and Lemma \ref{finite_colimits_petri_rates}.
\end{proof}

\section{Maps of foot-replaced double categories}\label{SCMaps}

In this section we define maps between foot-replaced double categories. In Theorem \ref{trick1} we showed how to construct a foot-replaced double category ${_L \lX}$ starting from a pair $$(\lX,L\maps \A \to \lX_0)$$ where $\lX$ is a double category and $L \maps \A \to \lX_0$ is a functor that maps the category $\A$, which contains the objects and morphisms of the foot-replaced double category $_L \lX$, into the category of objects $\lX_0$ of the double category $\lX$. 
Suppose that we have two foot-replaced double categories: $_L \lX$ obtained from a pair $(\lX,L \maps \A \to \lX_0)$ and $_{L'}\lX'$ obtained from a pair $(\lX',L' \maps \A' \to \lX'_0).$ Then we can construct a map from $_L \lX$ to $_{L'} \lX'$ given a functor $F \colon \mathsf{A} \to \mathsf{A'}$ together with a double functor $\mathbb{F} \maps \lX \to \lX'$ such that the following diagram commutes up to a specified isomorphism $\theta$:
\[
\begin{tikzpicture}[scale=1.5]
\node (A) at (0,0) {$\A$};
\node (B) at (1,0) {$\lX_0$};
\node (C) at (0,-1) {$\A'$};
\node (D) at (1,-1) {$\lX'_0$};
\node (F) at (0.5,-0.5) {$\SWarrow \theta$};
\path[->,font=\scriptsize,>=angle 90]
(A) edge node [above] {$L$} (B)
(A) edge node [left]{$F$} (C)
(B) edge node [right]{$\mathbb{F}_0$} (D)
(C) edge node [above] {$L'$} (D);
\end{tikzpicture}
\]
In the case where $_{L} \lX$ and $_{L'} \lX'$ are symmetric monoidal and we wish to construct a symmetric monoidal double functor between them, we will then require that both the functor $F$ and double functor $\mathbb{F}$ be symmetric monoidal, and that $\theta$ be monoidal as well. (For the definition of `symmetric monoidal double functor', see Definition \ref{defn:monoidal_double_functor}, and for the definition of `monoidal transformation', see Definition \ref{defn:monoidal_transformation}.) 
\begin{theorem}\label{MapTrick1}
Let ${_L \lX}$ and $_{L'} \lX'$ be two foot-replaced double categories. Given a functor $F \maps \A \to \A'$ and a double functor $\mathbb{F} \maps \lX \to \lX'$ such that the following diagram commutes up to isomorphism:
\[

\]
the triple $(F,\mathbb{F},\theta)$ results in a double functor ${_F \mathbb{F}} \colon _{L} \lX \to {_{L'} \lX'}$. This double functor ${ _F \mathbb{F} }$ maps objects, vertical 1-morphisms, horizontal 1-cells and 2-morphisms as follows:
\begin{enumerate}
\item Objects:
$$a \mapsto F(a)$$
\item Vertical 1-morphisms:
\[
%
\]
\end{enumerate}
\end{theorem}

\begin{proof}
We will show that from the triple $(F,\mathbb{F},\theta)$ we can produce a double functor ${ _F \mathbb{F} } \maps _L \lX \to {_{L'} \lX'}$. This means that we must have $${ _F \mathbb{F} }_0 = F \maps _L \lX_0 \to {_{L'} \lX'_0}$$ and $${ _F \mathbb{F} }_1 \maps _L \lX_1 \to {_{L'} \lX'_1}$$
such that the following diagrams commute:
\[
%
\]
where $S,T$ and $S',T'$ are the source and target functors of the double categories $_L \lX$ and $_{L'} \lX'$, respectively, together with natural transformations $${ _F \mathbb{F} }_{\odot} \maps { _F \mathbb{F} }(M) \odot { _F \mathbb{F} }(N) \to { _F \mathbb{F} }(M \odot N)$$ for every pair of composable horizontal 1-cells $M$ and $N$ of $_L \lX$ and a natural transformation $${ _F \mathbb{F} }_U \maps {U'}_{F(a)} \to { _F \mathbb{F} }(U_a)$$ for every object $a \in {_L \lX}$ that satisfy the standard coherence axioms of a monoidal category given by the laxator hexagon and unitality squares.

The functors ${ _F \mathbb{F} }_0 = F$ and ${ _F \mathbb{F} }_1$ are defined as in the statement of the theorem. To see that the above squares commute, if we focus on the left one, starting at the upper left corner, for an object of $_L \lX_1$ which is given by a horizontal 1-cell, we have going right that:
\[

\]
and then going down yields source $F(a)$. If we go down and then right, we get that the source of the top horizontal 1-cell is the object $a$ which then maps to $F(a)$ under the double functor ${ _F \mathbb{F} }$. A morphism in ${_L \lX}_1$ is given by a 2-morphism of the form
\[
%
\]
so, again focusing on the left square, going right gives
\[
%
\]
and then going down yields source $F(f)$. On the other hand, going down we get that the source of the original 2-morphism is $f$ which then maps to $F(f)$ under the double functor ${ _F \mathbb{F} }$, and so the left square commutes. The right square is analogous.

That ${ _F \mathbb{F} }$ is functorial on vertical 1-morphisms is clear, as the pair ${ _F \mathbb{F} }$ acts as the functor $F \maps \A \to \A'$ on objects and vertical 1-morphisms. Given two vertically composable 2-morphisms in $_L \lX$:
\[
%
\]
we wish to show that ${ _F \mathbb{F} }_1$ is functorial. If we first compose the above two 2-morphisms in $_L \lX$, we get:
\[
%
\]
and then the image of this 2-morphism under ${ _F \mathbb{F} }_1$ is given by:
\[
%
\]
On the other hand, if we first map over the two 2-morphisms, we get
\[
%
\]
and then composing these in $_{L'} \lX'$ yields
\[
%
\]
by the functoriality of $\mathbb{F}_0=F, \mathbb{F}_1$ and $L'$.

Now let $M$ and $N$ be two composable horizontal 1-cells in $_L \lX$ given by:
\[
%
\]
We then have a natural transformation $${ _F \mathbb{F} }_{M,N} \maps { _F \mathbb{F} }(M) \odot { _F \mathbb{F} }(N) \to { _F \mathbb{F} }(M \odot N)$$ given by:
\[
%
\]
and for any object $a$, a natural transformation $${ _F \mathbb{F} }_a \colon U'_{{ _F \mathbb{F} }(a)} \to { _F \mathbb{F} }(U_a)$$ given by:  
\[
%
\]
The double functor ${ _F \mathbb{F} }$ is pseudo, lax or oplax depending on whether the double functor $\mathbb{F}$ is pseudo, lax or oplax, respectively.
\end{proof}
If both $F \maps \A \to \A'$ and $\mathbb{F} \maps \lX \to \lX'$ are (strong) symmetric monoidal and $\theta \colon \mathbb{F}_0 L \colon L'F$ a monoidal natural isomorphism, then ${ _F \mathbb{F} } \maps _L \lX \to { _{L'}\lX'}$ is a (strong) symmetric monoidal double functor.
\begin{theorem}
Let $_L \lX$ and $_{L'} \lX'$ be symmetric monoidal foot-replaced double categories obtained from pairs $(\lX,L \maps \A \to \lX_0)$ and $(\lX',L' \maps \A' \to \lX_0')$, respectively, via Theorem \ref{trick2}. If ${ _F \mathbb{F} } \maps _L \lX \to { _{L'} \lX'}$ is a foot-replaced double functor obtained from a square
\[
\begin{tikzpicture}[scale=1.5]
\node (A) at (0,0) {$\A$};
\node (B) at (1,0) {$\lX_0$};
\node (C) at (0,-1) {$\A'$};
\node (D) at (1,-1) {$\lX'_0$};
\node (F) at (0.5,-0.5) {$\SWarrow \theta$};
\path[->,font=\scriptsize,>=angle 90]
(A) edge node [above] {$L$} (B)
(A) edge node [left]{$F$} (C)
(B) edge node [right]{$\mathbb{F}_0$} (D)
(C) edge node [above] {$L'$} (D);
\end{tikzpicture}
\]
as in Theorem \ref{MapTrick1} with $\theta$ monoidal and $F$ and $\mathbb{F}$ (strong) symmetric monoidal, then ${ _F \mathbb{F} }$ is a (strong) symmetric monoidal double functor of foot-replaced double categories.
\end{theorem}

\begin{proof}
Since the functor $F \maps \A \to \A'$ is symmetric monoidal, for every pair of objects $a$ and $b$ of $\A$, we have a natural transformation $$\mu_{a,b} \maps F(a) \otimes F(b) \to F(a \otimes b)$$ together with a morphism $$\epsilon \maps 1_{_{L'} \lX'} \to F(1_{_L \lX})$$ 
where the unit object of $ _{L'} \lX'$ is given by $1_{{_{L'}} \lX'} = 1_{\A'} \cong F(1_{\A})$ and the unit object of $_L \lX$ is given by $1_{_L \lX} = 1_{\A}$. These together make the following diagrams commute for every triple of objects $a,b,c$ of $_L \lX$, which are just objects of $\A$. Note that the object component of the double functor ${ _F \mathbb{F} }$ is just ${ _F \mathbb{F} }_0=F$.
 \[\xymatrix{
    (F(a) \otimes F(b)) \otimes F(c) \ar[r]^{\alpha'}\ar[d]_{\mu_{a,b} \otimes 1}
    & F(a) \otimes (F(b) \otimes F(c)) \ar[d]^{1 \otimes \mu_{b,c}}\\
    F(a \otimes b) \otimes F(c) \ar[d]_{\mu_{a \otimes b,c}} &
    F(a) \otimes F(b \otimes c) \ar[d]^{\mu_{a,b \otimes c}}\\
    F((a \otimes b) \otimes c)\ar[r]^{F\alpha} &
    F(a \otimes (b \otimes c))}\]
\[

\]
Moreover, the following diagram commutes where by an abuse of notation, we denote the braidings in both categories $\A$ and $\A'$ as $\beta$.
\[
%
\]
The double functor $\mathbb{F} \maps \lX \to \lX'$ is also symmetric monoidal, which means that for every pair of horizontal 1-cells $M$ and $N$, we have a natural transformation $$\mathbb{F}_{M,N}\maps \mathbb{F}(M) \otimes \mathbb{F}(N) \to \mathbb{F}(M \otimes N)$$ and a morphism
$$\delta \maps U_{1_{\mathsf{A'}}} \to \mathbb{F}(U_{1_{\mathsf{A}}})$$
which satisfy the usual axioms. From these, we can construct the corresponding transformations for ${ _F \mathbb{F} }$.  Given horizontal 1-cells $M$ and $M'$ in $_L \lX$:
\[
%
\]
their images ${ _F \mathbb{F} }(M)$ and ${ _F \mathbb{F} }(M')$ are given by:

\[
%
\]
and their tensor product ${ _F \mathbb{F} }(M) \otimes { _F \mathbb{F} }(M')$ is given by:
\[
%
\]
where $\sigma_{F(a),F(a')} \colon L'(F(a)) \otimes L'(F(a')) \to L'(F(a) \otimes F(a'))$ is the natural isomorphism coming from the (strong) symmetric monoidal functor $L' \colon \mathsf{A}' \to {\lX_0}'$, $\sigma'_{a,a'} \colon L(a) \otimes L(a') \to L(a \otimes a')$ is the natural isomorphism coming from the (strong) symmetric monoidal functor $L \colon \mathsf{A} \to {\lX_0}$, and $\mu_{x,y}' \maps \mathbb{F}_0(x) \otimes \mathbb{F}_0(y) \to \mathbb{F}_0(x \otimes y)$ is the natural isomorphism coming from the (strong) symmetric monoidal functor $\mathbb{F}_0 \colon \lX_0 \to \lX_0'$. On the other hand, $M \otimes M'$ is given by:
\[

\]
and the image ${ _F \mathbb{F} }(M \otimes M')$ is given by:
\[
%
\]
We then have a natural transformation $$\nu'_{M,M'} \colon { _F \mathbb{F} }(M) \otimes { _F \mathbb{F} }(M') \to { _F \mathbb{F} }(M \otimes M')$$ given by the 2-isomorphism:
\[
%
\]
which we can rewrite as:
\[
%
\]

For the unit constraint, the horizontal 1-cell unit of ${_L \lX}$ is given by $U_{L(1_\mathsf{A})}$:
\[
%
\]
and the image ${ _F \mathbb{F} }(U_{L(1_\mathsf{A})})$ is given by:
\[
%
\]
On the other hand, the horizontal 1-cell unit of ${_{L'} \lX'}$ is given by $U_{L'(1_{\mathsf{A'}})}$:
\[
%
\]
and we then get a natural transformation $\delta' \colon U_{L'(1_{\mathsf{A'}})} \to { _F \mathbb{F} }(U_{L(1_\mathsf{A})})$ given by:
\[
%
\]
where $\tau \colon 1_\mathsf{A'} \to F(1_\mathsf{A}) $ comes from the (strong) symmetric monoidal functor $F \colon \mathsf{A} \to \mathsf{A'}$.

These transformations $\nu'$ and $\delta'$ together make the following diagrams commute for every triple of horizontal 1-cells $M,N,P$ of ${_L \lX}$.
 \[\xymatrix{
    ({ _F \mathbb{F} }(M) \otimes { _F \mathbb{F} }(N)) \otimes { _F \mathbb{F} }(P) \ar[r]^{\alpha'}\ar[d]_{\nu'_{M,N} \otimes 1}
    & { _F \mathbb{F} }(M) \otimes ({ _F \mathbb{F} }(N) \otimes { _F \mathbb{F} }(P)) \ar[d]^{1 \otimes \nu'_{N,P}}\\
    { _F \mathbb{F} }(M \otimes N) \otimes { _F \mathbb{F} }(P) \ar[d]_{\nu'_{M \otimes N,P}} &
    { _F \mathbb{F} }(M) \otimes { _F \mathbb{F} }(N \otimes P) \ar[d]^{\nu'_{M,N \otimes P}}\\
    { _F \mathbb{F} }((M \otimes N) \otimes P)\ar[r]^{{ _F \mathbb{F} }(\alpha)} &
    { _F \mathbb{F} }(M \otimes (N \otimes P))}\]
\[
\begin{tikzpicture}[scale=1.5]
\node (A) at (0,1) {${ _F \mathbb{F} }(M) \otimes U_{1_{_{L'} \lX'}}$};
\node (C) at (3,1) {${ _F \mathbb{F} }(M)$};
\node (A') at (0,0) {${ _F \mathbb{F} }(M) \otimes { _F \mathbb{F} }(U_{1_{_L \lX}})$};
\node (C') at (3,0) {${ _F \mathbb{F} }(M \otimes U_{1_{_L \lX}})$};
\node (B) at (6,1) {$U_{1_{_{L'} \lX'}} \otimes { _F \mathbb{F} }(M)$};
\node (B') at (6,0) {${ _F \mathbb{F} }(U_{1_{_L \lX}}) \otimes { _F \mathbb{F} }(M)$};
\node (D) at (9,1) {${ _F \mathbb{F} }(M)$};
\node (D') at (9,0) {${ _F \mathbb{F} }(U_{1_{_L \lX}} \otimes M)$};
\path[->,font=\scriptsize,>=angle 90]
(A) edge node[left]{$1 \otimes \delta'$} (A')
(C') edge node[right]{${ _F \mathbb{F} }(r_{M})$} (C)
(A) edge node[above]{$r_{{ _F \mathbb{F} }(M)}$} (C)
(A') edge node[above]{$\nu'_{M,U_{1_{_L \lX}}}$} (C')
(B) edge node[left]{$\delta' \otimes 1$} (B')
(B') edge node[above]{$\nu'_{U_{1_{_L \lX}},M}$} (D')
(B) edge node[above]{$\ell_{{ _F \mathbb{F} }(M)}$} (D)
(D') edge node[right]{${ _F \mathbb{F} }(\ell_{M})$} (D);
\end{tikzpicture}
\]
By another abuse of notation, the following diagram commutes where we denote the braiding in both ${_L\lX}_1$ and ${_{L'}\lX'}_1$ by $\beta$.
\[
\begin{tikzpicture}[scale=1.5]
\node (B) at (5,1) {${ _F \mathbb{F} }(M) \otimes { _F \mathbb{F} }(N)$};
\node (B') at (5,0) {${ _F \mathbb{F} }(M \otimes N)$};
\node (D) at (10,1) {${ _F \mathbb{F} }(N) \otimes { _F \mathbb{F} }(M)$};
\node (D') at (10,0) {${ _F \mathbb{F} }(N \otimes M)$};
\path[->,font=\scriptsize,>=angle 90]
(B) edge node[left]{$\nu'_{M,N}$} (B')
(B') edge node[above]{${ _F \mathbb{F} }(\beta_{M,N})$} (D')
(B) edge node[above]{$\beta_{{ _F \mathbb{F} }(M),{ _F \mathbb{F} }(N)}$} (D)
(D) edge node[right]{$\nu'_{N,M}$} (D');
\end{tikzpicture}
\]
Lastly, we have transformations $\Phi_{M,N} \colon \otimes \circ ({ _F \mathbb{F} },{ _F \mathbb{F} }) \Rightarrow { _F \mathbb{F} } \circ \otimes$ and $\Phi_U \colon I_{ {_{L'}\mathbb{X}'} } \Rightarrow { _F \mathbb{F} } \circ I_{{ _L \mathbb{X}}}$ satisfying the axioms of a symmetric monoidal functor with respect to $\otimes$ which come from the corresponding transformations $\Psi_{M,N} \colon \otimes \circ (\mathbb{F},\mathbb{F}) \Rightarrow \mathbb{F} \circ \otimes$, $\Psi_U \colon I_{\mathbb{X}'} \Rightarrow \mathbb{F} \circ I_{\mathbb{X}}$ of the symmetric monoidal double functor $\mathbb{F}$, the natural isomorphisms $\mu_{a,b}$ and $\mu$ of the symmetric (strong) monoidal functor $F \colon \mathsf{A} \to \mathsf{A'}$, and the monoidal natural isomorphism $\theta \colon \mathbb{F}_0 L \Rightarrow L' F$.
\end{proof}

\section{Transformations of foot-replaced double categories}
We can also consider double transformations between these foot-replaced double functors and symmetric monoidal versions of such. By the previous section, we can produce a map between two foot-replaced double categories $_L \lX = (\lX,L \maps \A \to \lX_0)$ and $_{L'} \lX' = (\lX',L' \maps \A' \to \lX'_0)$ from a triple $(F,\mathbb{F},\theta)$ as in the following diagram.
\[

\]
This leads to a double functor $_F \mathbb{F} \colon {_L \mathbb{X} } \to { _{L'} \mathbb{X}' }$ by Theorem \ref{MapTrick1}. Given another double functor $_G \mathbb{G} \colon {_L \lX} \to { _{L'} \lX'}$ coming from a triple $(G,\mathbb{G},\psi)$, we can construct a foot-replaced double transformation from ${ _F \mathbb{F} }$ to ${ _G \mathbb{G} }$ from a pair $(\phi,\Phi)$ where $\phi \maps F \Rightarrow G$ is a natural transformation and $\Phi \maps \mathbb{F} \Rightarrow \mathbb{G}$ is a double transformation such that the following diagram commutes
\[
%
\]
meaning that the following composites are equal.
\[
%
\]
We will denote the double transformation that results from the pair $(\phi,\Phi)$ as $_\phi \Phi \colon _F \mathbb{F} \Rightarrow _G \mathbb{G}$. 

\begin{theorem}
Let ${ _F \mathbb{F} } \maps _L \lX \to{_{L'}\lX'}$ and ${ _G \mathbb{G} }  \maps _L \lX \to {_{L'}\lX'}$ be double functors obtained from triples $(F,\mathbb{F},\theta)$ and $(G,\mathbb{G},\psi)$ via Theorem \ref{MapTrick1}, respectively. Given a double transformation $\Phi \maps \mathbb{F} \Rightarrow \mathbb{G}$ and a transformation $\phi \maps F \Rightarrow G$ such that the diagrams above commute, then from the pair $(\phi,\Phi)$ we can construct a double transformation $\Xi = { _\phi \Phi} \maps _F \mathbb{F} \Rightarrow { _G \mathbb{G} }$ (see Definition \ref{double_transformation}). The object component $\Xi_0$ is given by the composite $$ \Xi_a = \psi_a^{-1} L'(\phi_a) \theta_a \maps { _F \mathbb{F} }(a) \to { _G \mathbb{G} }(a)$$ and the arrow component $\Xi_1$ is given by $\Phi_1$, the arrow component of the double transformation $\Phi$.
\end{theorem}

\begin{proof}
Because $\Phi \maps \mathbb{F} \Rightarrow \mathbb{G}$ is a double transformation and the diagram on the previous page commutes, we have that the following equations hold.
\[

\]
Here we use the isomorphisms $\theta_a \colon \mathbb{F}_0(L(a)) \xrightarrow{\sim} L' (F(a))$ and $\psi_a \colon \mathbb{G}_0(L(a)) \xrightarrow{\sim} L'(G(a))$ together with the natural transformation $\phi \colon F \Rightarrow G$ to cook up the object component of the double natural transformation ${ _\phi \Phi} \colon { _F \mathbb{F} } \Rightarrow { _G \mathbb{G} }$.
In detail, every object of $_L \lX$ is of the form $L(a)$ for some $a$ in $\A$. We thus have for every object $L(a)$ in $_L \lX$ a map $\theta_a \colon \mathbb{F}_0(L(a)) \xrightarrow{\sim} L' (F(a))$. The natural transformation $\phi \maps F \Rightarrow G$ evaluated at $a$ then gives a map $\phi_a \maps F(a) \to G(a)$ and applying the functor $L'$ to the map $\phi_a$ then gives a map $L'(\phi_a) \maps L'(F(a)) \to L'(G(a))$. Then, we use the other natural isomorphism $\psi_a \colon \mathbb{G}_0(L(a)) \to L'(G(a))$ to obtain a map $\psi_{a}^{-1} \colon L'(G(a)) \xrightarrow{\sim} \mathbb{G}_0(L(a))$, and thus $$ \Xi_a = \psi_a^{-1} L'(\phi_a) \theta_a \maps { _F \mathbb{F} }(a) \to { _G \mathbb{G} }(a).$$ Moreover, the map $\Xi_a$ for each object $a$ will make the above equations hold for ${ _\phi \Phi} \colon { _F \mathbb{F} } \Rightarrow { _G \mathbb{G} }$ as $$\Xi_a = \psi_a^{-1} L'(\phi_a) \theta_a = \psi_a^{-1} \psi_a {\Phi_0}_{L(a)}={\Phi_0}_{L(a)}$$ and the corresponding equations utilizing the component ${\Phi_0}_{L(a)}$ hold as $\Phi \maps \lX \Rightarrow \lX'$ is a double transformation.

Finally, because $\Phi \maps \mathbb{F} \Rightarrow \mathbb{G}$ is a double transformation and by the commutativity of the diagram on the previous page, for a horizontal 1-cell $M$ in $_L \lX$ we have that $S({\Phi_1}_{M})=\Xi_{S(M)}$ and $T({\Phi_1}_{M}) = \Xi_{T(M)}$.
\end{proof}

The double transformation ${ _\phi \Phi}$ is a double natural isomorphism if and only if $\phi$ is a natural isomorphism and $\Phi$ is a double natural isomorphism.

As with functors of foot-replaced double categories, if both the transformation $\phi \maps F \Rightarrow G$ and the double transformation $\Phi \maps \mathbb{F} \Rightarrow \mathbb{G}$ are symmetric monoidal, then ${ _\phi \Phi } \maps { _F \mathbb{F} } \Rightarrow { _G \mathbb{G} }$ is a symmetric monoidal double transformation of symmetric monoidal foot-replaced double functors. 

\begin{theorem}
Let ${ _\phi \Phi } \maps { _F \mathbb{F} } \Rightarrow { _G \mathbb{G} }$ be a foot-replaced double transformation between two symmetric monoidal foot-replaced double functors ${ _F \mathbb{F} } \maps _L \lX \to {_{L'}\lX'}$ and ${ _G \mathbb{G} } \maps _L \lX \to {_{L'}\lX'}$, where $_L \lX = (\lX,L \maps \A \to \lX_0)$ and $_{L'} \lX'=(\lX',L' \maps \A' \to \lX'_0)$. If $\phi \maps F \Rightarrow G$ is a monoidal transformation and $\Phi \maps \mathbb{F} \Rightarrow \mathbb{G}$ is a monoidal double transformation, then ${ _\phi \Phi } \maps { _F \mathbb{F} } \Rightarrow { _G \mathbb{G} }$ is a monoidal double transformation (see Definition \ref{monoidal_double_transformation}) of foot-replaced double functors.
\end{theorem}
\begin{proof}
The double transformation ${ _\phi \Phi }$ acts as $\Xi$ (defined above) on objects and vertical 1-morphisms. This means that the following diagrams commute.
\[
\begin{tikzpicture}[scale=1.5]
\node (B) at (5,1) {${ _F \mathbb{F} }(a) \otimes { _F \mathbb{F} }(b)$};
\node (B') at (5,0) {${ _F \mathbb{F} }(a \otimes b)$};
\node (D) at (8,1) {${ _G \mathbb{G} }(a) \otimes { _G \mathbb{G} }(b)$};
\node (D') at (8,0) {${ _G \mathbb{G} }(a \otimes b)$};
\path[->,font=\scriptsize,>=angle 90]
(B) edge node[left]{$\mu_{a,b}$} (B')
(B') edge node[above]{$\Xi_{a \otimes b}$} (D')
(B) edge node[above]{$\Xi_a \otimes \Xi_b$} (D)
(D) edge node[right]{$\mu'_{a,b}$} (D');
\end{tikzpicture}
\]
\[
\begin{tikzpicture}[scale=1.5]
\node (B) at (5,1) {$1_{_{L'}\lX'}$};
\node (B') at (6,0) {${ _F \mathbb{F} }(1_{_L \lX})$};
\node (D) at (7,1) {${ _G \mathbb{G} }(1_{_L\lX})$};
\path[->,font=\scriptsize,>=angle 90]
(B) edge node[left]{$\epsilon$} (B')
(B') edge node[right]{$\phi_{1_{ _L \lX}}$} (D)
(B) edge node[above]{$\epsilon'$} (D);
\end{tikzpicture}
\]
Similarly, the double transformation ${ _\phi \Phi }$ acts as $\Phi$ on horizontal 1-cells and 2-morphisms, which means that the following diagrams commute.
\[
\begin{tikzpicture}[scale=1.5]
\node (A) at (9,1) {${ _F \mathbb{F} }(M) \otimes { _F \mathbb{F} }(N)$};
\node (A') at (13,1) {${ _G \mathbb{G} }(M) \otimes { _G \mathbb{G} }(N)$};
\node (C) at (9,0) {${ _F \mathbb{F} }(M \otimes N)$};
\node (C') at (13,0) {${ _G \mathbb{G} }(M \otimes N)$};
\path[->,font=\scriptsize,>=angle 90]
(A) edge node[above]{${\Phi_1}_M \otimes {\Phi_1}_N$} (A')
(A') edge node[right]{${\mu'}_{M,N}$} (C')
(A) edge node[left]{$\mu_{M,N}$} (C)
(C) edge node[above]{${\Phi_1}_{M \otimes N}$} (C');
\end{tikzpicture}
\]
\[
\begin{tikzpicture}[scale=1.5]
\node (A) at (8,1) {$U_{1_{_{L'}\lX'}}$};
\node (C) at (9,0) {${ _F \mathbb{F} }(U_{1_{_L \lX}})$};
\node (C') at (10,1) {${ _G \mathbb{G} }(U_{1_{_L \lX}})$};
\path[->,font=\scriptsize,>=angle 90]
(A) edge node[left]{$\delta$} (C)
(C) edge node[right]{${\Phi_1}_{U_{1_{_L \lX}}}$} (C')
(A) edge node[above]{$\delta'$} (C');
\end{tikzpicture}
\]
Hence both the object and arrow components are monoidal natural transformations and thus ${ _\phi \Phi } \maps { _F \mathbb{F} } \Rightarrow { _G \mathbb{G} }$ is a symmetric monoidal double transformation.
\end{proof}
}

{\ssp
\chapter{Decorated cospan double categories}\label{Chapter4}

In this chapter we present an improved version of Fong's theory of decorated cospan categories \cite{Fong} from the perspective of double categories. The main difference here is that, given a category $\mathsf{A}$ with finite colimits, we instead start with a \emph{pseudofunctor} $F \colon \mathsf{A} \to \mathsf{Cat}$ rather than functor $F \colon \mathsf{A} \to \mathsf{Set}$. The additional structure of $\mathsf{Cat}$ viewed as a 2-category then allows us more flexibility in defining what the isomorphism class of an $F$-decorated cospan consists of. This ultimately results in a second solution to the problems with the original incarnation of decorated cospans, structured cospans being the first.

Given a finitely cocomplete category $\mathsf{A}$ and a lax monoidal pseudofunctor $F \colon \mathsf{A} \to \mathbf{Cat}$, the first result is the existence of a double category $F\mathbb{C}\mathbf{sp}$ in which $F$-decorated cospans appear as horizontal 1-cells, except now we can exploit the 2-categorical structure of $\mathbf{Cat}$ to define 2-morphisms. This is Theorem \ref{DCord}. In Theorem \ref{DC} we show that when this lax monoidal pseudofunctor $F$ is symmetric, then the resulting double category $F\mathbb{C}\mathbf{sp}$ is in fact symmetric monoidal. We then define maps between decorated cospan double categories in Section \ref{DCMaps}. Finally, as both structured cospan double categories and decorated cospan double categories are solutions to the problems with Fong's original decorated cospans, in Section \ref{SCDCequiv} we show that under certain conditions these approaches lead to equivalent symmetric monoidal double categories, the main result being Theorem \ref{Equiv}.

\section{A double category of decorated cospans}

\begin{theorem}\label{DCord}
Let $\mathsf{A}$ be a category with finite colimits and $F \colon (\mathsf{A},+,0) \to (\mathbf{Cat},\times,1)$ a lax monoidal pseudofunctor. Then there exists a double category $F\mathbb{C}\mathbf{sp}$ for which:
\begin{enumerate}
\item{an object is an object of $\mathsf{A}$,}
\item{a vertical 1-morphism is a morphism of $\mathsf{A}$,}
\item{a horizontal 1-cell is an \define{$F$-decorated cospan} in $\mathsf{A}$, which is a pair:
\[

\]
}
\item{a 2-morphism is a \define{map of $F$-decorated cospans} in $\mathsf{A}$, which is a pair consisting of a commutative diagram:
\[
%
\]
and a morphism $\iota \colon F(h)(x) \to x^\prime$ in $F(m^\prime)$,}
\item{composition of vertical 1-morphisms is composition in $\A$,}
\item{the composite of two horizontal 1-cells:
\[
%
\]
is done using chosen pushouts in $\A$:
\[
%
\]
where the decoration $x \odot y$ on the apex is given by: $$1 \xrightarrow{\lambda^{-1}} 1 \times 1 \xrightarrow{x \times y} F(m) \times F(n) \xrightarrow{\phi_{m,n}} F(m+n) \xrightarrow{F(\psi)} F(m+_{b}n)$$
}
\item{the vertical composite of two 2-morphisms:
\[
%
\]
is given by:
\[
%
\]
where the morphism $\iota_{\alpha' \alpha}$ comes from the pasting of the two diagrams representing the morphisms $\iota_\alpha$ and $\iota_{\alpha'}$:
\[
%
\]
}
\item{the horizontal composite of two 2-morphisms:
\[
%
\]
also uses chosen pushouts in $\A$ and is given by:
\[
%
\]
where the morphism of decorations $\iota_{\alpha \odot \beta}$ is given by the diagram:
\[
%
\]
}
\end{enumerate}
\end{theorem}
\begin{proof}
We begin by defining the functors $$U \colon F\mathbb{C}\mathbf{sp}_0 \to F\mathbb{C}\mathbf{sp}_1$$ $$S,T \colon F\mathbb{C}\mathbf{sp}_1 \to F\mathbb{C}\mathbf{sp}_0$$and$$\odot \colon F\mathbb{C}\mathbf{sp}_1 \times_{F\mathbb{C}\mathbf{sp}_0} F\mathbb{C}\mathbf{sp}_1 \to F\mathbb{C}\mathbf{sp}_1$$necessary to obtain a double category. The functor $U \colon F\mathbb{C}\mathbf{sp}_0 \to F\mathbb{C}\mathbf{sp}_1$ is defined on objects as: 
\[
%
\]
where $!_a \in F(a)$ is the trivial decoration on $a$ given by the composite of the unique map $F(!) \colon F(0) \to F(a)$ and the morphism $\phi \colon 1 \to F(0)$  which comes from the structure of the lax monoidal pseudofunctor $F \colon \mathsf{A} \to \mathbf{Cat}$. For morphisms, the functor $U$ is defined as:
\[
%
\]
together with the morphism $\iota_{f} = F(f) F(!) \phi \colon 1 \to F(a^\prime)$. We also have source and target functors $S, T \colon F\mathbb{C}\mathbf{sp}_1 \to F\mathbb{C}\mathbf{sp}_0$ where the source of the horizontal 1-cell
\[
%
\]
is the object $a$ in $\mathsf{A}$ and the source of the 2-morphism
\[
%
\]
is the source of the underlying map of cospans in $\mathsf{A}$, namely the morphism $f$ in $\mathsf{A}$; the target functor is defined similarly. These functors satisfy the equations $$SU(a)=a=TU(a)$$for all objects and morphisms of $\mathsf{A}$.

Given two composable horizontal 1-cells $M$ and $N$:
\[
%
\]
the composite $N \odot M$ is given by:
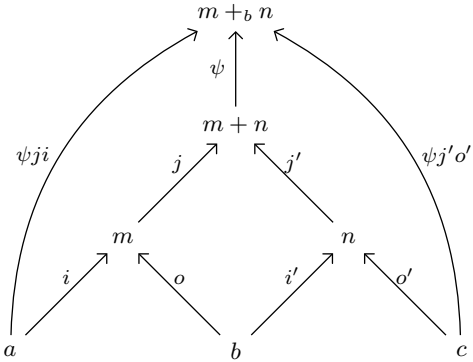
\[
%
\]
with the corresponding decoration of the apex $x \odot y \in F(m+_{b} n)$ being the element determined by:
$$1 \xrightarrow{\lambda^{-1}} 1 \times 1 \xrightarrow{x \times y} F(m) \times F(n) \xrightarrow{\phi_{m,n}} F(m+n) \xrightarrow{F(\psi)} F(m+_{b}n)$$
where $\psi \colon m + n \to m+_{b} n$ is the natural map from the coproduct to the pushout and $\phi_{m,n} \colon F(m) \times F(n) \to F(m+n)$ is the natural transformation coming from the structure of the lax monoidal pseudofunctor $F \colon \mathsf{A} \to \mathbf{Cat}$. The source and target functors satisfy the equations $S(N \odot M)=S(M)$ and $T(N \odot M)=T(N)$.

Given three composable horizontal 1-cells $M_1, M_2$ and $M_3$:
\[
\begin{tikzpicture}[scale=1.5]
\node (A) at (0,0) {$a$};
\node (B) at (1,0) {$m_1$};
\node (C) at (2,0) {$b$};
\node (D) at (1,-0.5) {$x \in F(m_1)$};
\node (E) at (3,0) {$b$};
\node (F) at (4,0) {$m_2$};
\node (G) at (5,0) {$c$};
\node (H) at (4,-0.5) {$y \in F(m_2)$};
\node (I) at (6,0) {$c$};
\node (J) at (7,0) {$m_3$};
\node (K) at (8,0) {$d$};
\node (L) at (7,-0.5) {$z \in F(m_3)$};
\path[->,font=\scriptsize,>=angle 90]
(A) edge node[above]{$i$} (B)
(C) edge node[above]{$o$} (B)
(E) edge node[above]{$i^\prime$} (F)
(G) edge node[above]{$o^\prime$} (F)
(I) edge node[above]{$i''$} (J)
(K) edge node[above]{$o''$} (J);
\end{tikzpicture}
\]
we get a natural isomorphism $a_{M_1,M_2,M_3} \colon (M_1 \odot M_2) \odot M_3 \to M_1 \odot (M_2 \odot M_3)$ which is the globular 2-morphism given by a map of cospans $(1,\sigma,1)$:
\[
\begin{tikzpicture}[scale=1.5]
\node (A) at (0,0.5) {$a$};
\node (A') at (0,-0.5) {$a$};
\node (B) at (1.5,0.5) {$(m_1+_{b} m_2)+_{c} m_3$};
\node (C) at (3,0.5) {$d$};
\node (C') at (3,-0.5) {$d$};
\node (D) at (1.5,-0.5) {$m_1+_{b}(m_2 +_{c} m_3)$};
\node (E) at (5.5,0.5) {$(x \odot y) \odot z \in F((m_1+_{b} m_2)+_{c} m_3)$};
\node (F) at (5.5,-0.5) {$x \odot (y \odot z) \in F(m_1+_{b} (m_2 +_{c} m_3))$};
\path[->,font=\scriptsize,>=angle 90]
(A) edge node[above]{$$} (B)
(C) edge node[above]{$$} (B)
(A) edge node[left]{$1$} (A')
(C) edge node[right]{$1$} (C')
(A') edge node {$$} (D)
(C') edge node {$$} (D)
(B) edge node [left] {$\sigma$} (D);
\end{tikzpicture}
\]
with the decorations on the cospan's apices given by:
$$\hspace{-.5in} (x \odot y) \odot z := 1 \xrightarrow{\zeta_1} F(m_1+_{b} m_2) \times F(m_3) \xrightarrow{\phi_{m_1+_{b} m_2, m_3}} F((m_1+_{b}m_2) +m_3) \xrightarrow{F(j_{m_1+_{b} m_2,m_3})} F((m_1+_{b} m_2)+_{c} m_3)$$ $$\zeta_1 = (1 \times z) \rho^{-1} F(j_{m_1,m_2}) \phi_{m_1,m_2} (x \times y) \lambda^{-1}$$
and
$$\hspace{-.5in} x \odot (y \odot z) := 1 \xrightarrow{\zeta_2} F(m_1) \times F(m_2 +_{c} m_3) \xrightarrow{\phi_{m_1, m_2 +_{c} m_3}} F(m_1+(m_2 +_{c} m_3)) \xrightarrow{F(j_{m_1,m_2 +_{c} m_3})} F(m_1+_{b} (m_2+_{c} m_3))$$ $$\zeta_2 = (x \times 1) \lambda^{-1} F(j_{m_2,m_3}) \phi_{m_2,m_3} (y \times z) \rho^{-1}$$
together with the isomorphism $\iota_\sigma \colon F(\sigma)((x \odot y) \odot z) \to x \odot (y \odot z)$. The map $\sigma \colon (m_1+_{b} m_2)+_{c}m_3 \to m_1+_{b}(m_2+_{c}m_3)$ is the universal map between two colimits of the same diagram. We can also define left and right unitors as follows. Given a horizontal 1-cell $M$:
\[

\]
if we, say, compose with the identity horizontal 1-cell of $b$ on the right:
\[
%
\]
where $!_{b} = F(!)  \phi \colon 1 \to F(b)$ is the trivial decoration on $b$, the result is:
\[
%
\]
where $\psi_m \colon m \to m+b$ is the natural map into the coproduct and likewise for $\psi_{b}$ and $j \colon m+b \to m+_{b} b$ is the natural map from the coproduct to the pushout. The decoration $x \odot !_b \colon 1 \to F(m+_{b} b)$ is given by: $$1 \xrightarrow{\lambda^{-1}} 1 \times 1 \xrightarrow{x \times !_{b}} F(m) \times F(b) \xrightarrow{\phi_{m,b}} F(m+b) \xrightarrow{F(j_{m,b})} F(m+_{b} b).$$ We then have that the right unitor $R \colon M \odot 1_{b} \xrightarrow{\sim} M$ is given by the globular 2-morphism $(1,r,1)$ from the above composite to $M$:
\[
%
\]
where the isomorphism $r \colon m+_{b} b \xrightarrow{\sim} m$ is a universal map together with the isomorphism $\iota_r \colon F(r)(x \odot !_b) \to x$. The left unitor is similar. The source and target functor applied to the left and right unitors and associators yield identities, and the left and right unitors together with the associator satisfy the standard pentagon and triangle identities of a monoidal category or bicategory. Finally, for the interchange law, given four 2-morphisms $\alpha, \beta, \alpha^\prime$ and $\beta^\prime$:
\[
%
\]
if we first compose horizontally we obtain:
\[
%
\]
To obtain the morphism of decorations for a horizontal composite, we have as initial data:
\[
%
\]
These two 2-morphisms $\iota_\alpha$ and $\iota_\beta$ are two 2-morphisms in the monoidal 2-category $(\mathbf{Cat},\times,1)$ and so we can tensor them which results in:
\[
%
\]
where the middle square commutes since $F$ is a lax monoidal pseudofunctor and the right square commutes because we have taken a commutative square and applied the pseudofunctor $F$ to it. The decorations $x \odot y$ and $x' \odot y'$ are given respectively by top and bottom composite of arrows and the morphism of decorations $\iota_{\alpha \odot \beta}$ is given by composing $\iota_\alpha \times \iota_\beta$ with the two commuting squares, which can equivalently be viewed as a morphism in $F(m' +_{b'} n')$. 

Returning to the interchange law, composing the two horizontal compositions above vertically then results in:
\[

\]
The vertical composite of two morphisms of decorations is straightforward. On the other hand, if we first compose vertically we obtain:
\[
%
\]
and then composing horizontally results in:
\[
%
\]
As usual for the interchange law in double categories of this nature, only the `interior' of the two composites appears different, but the two morphisms $(h_1^\prime +_{g^\prime} h_2^\prime)(h_1 +_g h_2) \colon m+_{b} n \to m'' +_{b''} n''$ and $(h_1^\prime h_1) +_{g^\prime g} (h_2^\prime h_2) \colon m+_{b} n \to m'' +_{b''}n''$ are the same universal map realized in two different ways. The two morphisms of decorations $\iota_{(\alpha^\prime \odot \beta^\prime)(\alpha \odot \beta)}$ and $\iota_{(\alpha^\prime \alpha) \odot (\beta^\prime \beta)}$ are obtained as two different compositions of four 2-morphisms in $\mathbf{Cat}$, namely horizontally then vertically and vertically then horizontally. As $\mathbf{Cat}$ is a 2-category, the interchange law for these 2-morphisms already holds, and as a result, the decoration morphisms $$\iota_{(\alpha' \odot \beta')(\alpha \odot \beta)} \colon F((h_1^\prime +_{g^\prime} h_2^\prime)(h_1 +_g h_2))(x \odot y) \to x'' \odot y''$$ and $$\iota_{(\alpha^\prime \alpha)\odot(\beta^\prime \beta)} \colon F((h_1^\prime h_1)+_{g^\prime g} (h_2^\prime h_2))(x \odot y) \to x'' \odot y''$$ are also the same. Thus the interchange law for 2-morphisms holds and $F\mathbb{C}\mathbf{sp}$ is a double category.
\end{proof}

\begin{corollary}\label{DCordCor}
Given a category $\mathsf{A}$ with pushouts, $\lCsp(\A)$ is a double category with the relevant structure given as in Theorem \ref{DCord}.
\end{corollary}

\begin{proof}
This is a special case of Theorem \ref{DCord} where each $F$-decorated cospan is equipped with the trivial decoration. Namely, given a cospan in $\A$:
\[
\begin{tikzpicture}[scale=1.5]
\node (A) at (0,0) {$a$};
\node (B) at (1,0) {$m$};
\node (C) at (2,0) {$b$};
\path[->,font=\scriptsize,>=angle 90]
(A) edge node[above]{$i$} (B)
(C) edge node[above]{$o$} (B);
\end{tikzpicture}
\]
the trivial decoration on the apex $m$ is given by the composite $$!_m = F(!) \phi \colon 1 \to F(m)$$where $\phi \colon 1 \to F(0)$ is the morphism between monoidal units coming from the structure of a lax monoidal pseudofunctor and $! \colon 0 \to m$ is the unique morphism from the initial object $0$ of $\A$ to the object $m$. By equipping each $F$-decorated cospan with the trivial decoration, all of the diagrams involving decorations commute trivially, and the proof of Theorem \ref{DCord} reduces to a proof that $\lCsp(\A)$ is a double category.
\end{proof}

If the lax monoidal pseudofunctor $F \colon (\mathsf{A},+,0) \to (\mathbf{Cat},\times,1)$ is symmetric lax monoidal, then the above double category $F \mathbb{C} \mathbf{sp}$ is also symmetric monoidal.

\begin{theorem}\label{DC}
Let $\mathsf{A}$ be a category with finite colimits and $F \colon (\mathsf{A},+,0) \to (\mathbf{Cat},\times,1)$ a symmetric lax monoidal pseudofunctor. Then the double category $F\mathbb{C}\mathbf{sp}$ of Theorem \ref{DCord} is symmetric monoidal where:
\begin{enumerate}
\item{the tensor product of two objects $a_1$ and $a_2$ is a chosen coproduct $a_1 + a_2$,}
\item{the tensor product of two vertical 1-morphisms is given by:
\[

\]
}
\item{the tensor product of two horizontal 1-cells:
\[
%
\]
is given by:
\[
%
\]
where the decoration on the apex is given by:
$$x_1+x_2 := 1 \xrightarrow{\lambda^{-1}} 1 \times 1 \xrightarrow{x_1 \times x_2} F(m_1) \times F(m_2) \xrightarrow{\phi_{m_1,m_2}} F(m_1+m_2)$$where $\phi_{m_1,m_2} \colon F(m_1) \times F(m_2) \to F(m_1+m_2)$ is the laxator of the lax monoidal pseudofunctor $F$,
}
\item{the tensor product of two 2-morphisms:
\[
%
\]
 is given by:
\[
%
\]
where $\iota_{\alpha_1 + \alpha_2}$ is given by the diagram:
\[
%
\]
}
\end{enumerate}
The unit for the tensor product is a chosen initial object of $\mathsf{A}$ and the symmetry for any two objects $a$ and $b$ is defined using the canonical isomorphism $a + b \cong b + a$.
\end{theorem}
\begin{proof}
First we note that the category of objects $F\mathbb{C}\mathbf{sp}_0=\mathsf{A}$ is symmetric monoidal under binary coproducts and the left and right unitors, associators and braidings are given as natural maps. The category of arrows $F\mathbb{C}\mathbf{sp}_1$ has:
\begin{enumerate}
\item{objects as $F$-decorated cospans which are pairs:
\[
%
\]
and}
\item{morphisms as maps of cospans in $\mathsf{A}$
\[
%
\]
together with a morphism $\iota \colon F(h)(x) \to x^\prime$.
}
\end{enumerate}
Given two objects $M_1$ and $M_2$ of $F\mathbb{C}\mathbf{sp}_1$:
\[
%
\]
their tensor product $M_1 \otimes M_2$ is given by taking the coproducts of the cospans of $\mathsf{A}$
\[
%
\]
and where the decoration on the apex is obtained using the natural transformation of the symmetric lax monoidal pseudofunctor $F:$ $$x_1+x_2 := 1 \xrightarrow{\lambda^{-1}} 1 \times 1 \xrightarrow{x_1 \times x_2} F(m_1) \times F(m_2) \xrightarrow{\phi_{m_1,m_2}} F(m_1+m_2).$$The monoidal unit 0 is given by:
\[
%
\]
where $0$ is the monoidal unit of $\mathsf{A}$ and $!_0 \colon 1 \to F(0)$ is the morphism which is part of the structure of the symmetric lax monoidal pseudofunctor $F \colon \mathsf{A} \to \mathbf{Cat}$. Tensoring an object with the monoidal unit, say, on the left:
\[
%
\]
where $!_0 + x \in F(0+m)$ is given by $$1 \xrightarrow{\lambda^{-1}} 1 \times 1 \xrightarrow{!_0 \times x} F(0) \times F(m) \xrightarrow{\phi_{0,m}} F(0+m).$$The left unitor is then an isomorphism in $F\mathbb{C}\mathbf{sp}_1$ given by:
\[
%
\]
where $\ell$ is the left unitor of $(\mathsf{A},+,0)$, together with the isomorphism $\iota_{\lambda} \colon F(\ell)(!_0 + x) \to x$. The right unitor is similar.

Given three objects $M_1, M_2$ and $M_3$ in $F\mathbb{C}\mathbf{sp}_1$:
\[
%
\]
tensoring the first two and then the third results in $(M_1 \otimes M_2) \otimes M_3$:
\[
%
\]
where $(x_1+x_2)+x_3 \colon 1 \to F((m_1+m_2)+m_3)$ is given by: $$\hspace{-.4in}1 \xrightarrow{(x_1 \times x_2) \times x_3} (F(m_1) \times F(m_2)) \times F(m_3) \xrightarrow{\phi_{m_1,m_2} \times 1} F(m_1+m_2) \times F(m_3) \xrightarrow{\phi_{m_1+m_2,m_3}} F((m_1+m_2)+m_3)$$whereas tensoring the last two and then the first results in $M_1 \otimes (M_2 \otimes M_3)$:
\[
%
\]
where $x_1+(x_2+x_3) \colon 1 \to F(m_1+(m_2+m_3))$ is given by: $$\hspace{-.4in}1 \xrightarrow{x_1 \times (x_2 \times x_3)} F(m_1) \times (F(m_2) \times F(m_3)) \xrightarrow{1 \times \phi_{m_2,m_3}} F(m_1) \times F(m_2+m_3) \xrightarrow{\phi_{m_1,m_2+m_3}} F(m_1+(m_2+m_3)).$$If we let $a$ denote the associator of $(\mathsf{A},+,0)$, the associator of $F\mathbb{C}\mathbf{sp}_1$ is then a map of cospans in $\mathsf{A}$ from $(M_1 \otimes M_2) \otimes M_3$ to $M_1 \otimes (M_2 \otimes M_3)$ given by:
\[
%
\]
together with the isomorphism $\iota_a \colon F(a)((x_1+x_2)+x_3) \to x_1+(x_2+x_3)$. These associators and left and right unitors together satisfy the pentagon and triangle identities of a monoidal category. If we denote the above associator simply as $a$ and the left and right unitors as $\lambda$ and $\rho$, respectively, then given four objects in $F\mathbb{C}\mathbf{sp}_1$, say $M_1, M_2, M_3$ and $M_4$:
\[
%
\]
the following pentagon of underlying cospans and maps of cospans commutes:
\[
%
\]
as well as the following pentagon of corresponding decorations in the category $F(m_1 +(m_2 +(m_3+m_4)))$:
\[
%
\]
since $$F(aa)(((x_1+x_2)+x_3)+x_4)=F((1 \otimes a)a(a \otimes 1))(((x_1+x_2)+x_3)+x_4)$$ as the corresponding pentagon of cospan apices in the symmetric monoidal category $(\mathsf{A},+,0)$ commutes, and then appliying the pseudofunctor $F$ to this commutative pentagon yields a commutative pentagon in $\mathbf{Cat}$.

Similarly, if we denote the left and right unitors as $\lambda$ and $\rho$, respectively, then the following triangle of cospans and underlying maps of cospans commutes:
\[
%
\]
as well as the following triangle of corresponding decorations in the category $F(m_1+m_2)$:
\[
%
\]
since $$F(\rho \otimes 1)((x_1+0)+x_2)=F((1 \otimes \lambda)a)((x_1+0)+x_2)$$ as the corresponding triangle of cospan apices in the symmetric monoidal category $(\mathsf{A},+,0)$ commutes and applying the pseudofunctor $F$ to this commutative triangle results in a commutative triangle in $\mathbf{Cat}$.

For a tensor product of objects $M_1 \otimes M_2$ in $F\mathbb{C}\mathbf{sp}_1$, the source and target functors $S,T \colon F\mathbb{C}\mathbf{sp}_1 \to F\mathbb{C}\mathbf{sp}_0$ satisfy the following equations: $$S(M_1 \otimes M_2)=S(M_1) \otimes S(M_2)$$ $$T(M_1 \otimes M_2)=T(M_1) \otimes T(M_2).$$
For two objects $M_1$ and $M_2$ in $F\mathbb{C}\mathbf{sp}_1$, we have a braiding $\beta_{M_1,M_2} \colon M_1 \otimes M_2 \to M_2 \otimes M_1$ given by:
\[

\]
$$\iota_{\beta_{M_1,M_2}} \colon F(\beta_{m_1,m_2})(x_1+x_2) \xrightarrow{\sim} x_2+x_1$$ where the vertical 1-morphisms are given by braidings in $(\mathsf{A},+,0)$. This braiding makes the following triangle of underlying cospans commute:
\[
%
\]
as well as the following diagram of corresponding decorations in the category $F(m_1+m_2)$:
\[
%
\]
since $F(\beta_{m_2,m_1} \beta_{m_1,m_2}) (x_1+x_2) = x_1+x_2$. Thus $F\mathbb{C}\mathbf{sp}_1$ is also symmetric monoidal.

Next we derive the globular isomorphisms required in the definition of a symmetric monoidal double category relating horizontal composition and the tensor product. Given four horizontal 1-cells $M_1, M_2, N_1$ and $N_2$ respectively by:
\[
%
\]
we have that $(M_1 \otimes N_1) \odot (M_2 \otimes N_2)$ is given by:
\[
%
\]
where the decoration $(x_1+y_1) \odot (x_2+y_2) \in F((m_1+n_1)+_{b+b'}(m_2+n_2))$ is given by:
\[
%
\]
and $(M_1 \odot M_2) \otimes (N_1 \odot N_2)$ is given by:
\[
%
\]
where the decoration $(x_1 \odot x_2) + (y_1 \odot y_2) \in F((m_1+_{b}m_2)+(n_1+_{b'}n_2))$ is given by:
\[
%
\]
and where $\psi$ and $j$ are the natural maps into a coproduct and from a coproduct into a pushout, respectively. We then get a globular 2-morphism $$\chi \colon (M_1 \otimes N_1) \odot (M_2 \otimes N_2) \to (M_1 \odot M_2) \otimes (N_1 \odot N_2)$$ given by:
\[
%
\]
$$\iota_{\hat{\chi}} \colon F(\hat{\chi})((x_1+y_1)\odot(x_2+y_2)) \to (x_1 \odot x_2)+(y_1 \odot y_2)$$
where $\hat{\chi}$ is the universal map between two colimits of the same diagram.
For two objects $a,b \in \mathsf{A}$, $U_{a+b}$ is given by:
\[
%
\]
where $$!_{a+b} \colon 1 \xrightarrow{\phi} F(0) \xrightarrow{F(!_{a+b})}  F(a+b).$$
Similarly, we have $U_a$ and $U_b$ given respectively by:
\[
%
\]
and then $U_a + U_b$ is given by:
\[
%
\]
where
$$!_a + !_b \colon 1 \xrightarrow{\lambda^{-1}} 1 \times 1 \xrightarrow{\phi \times \phi} F(0) \times F(0) \xrightarrow{F(!_a) \times F(!_b)} F(a) \times F(b) \xrightarrow{\phi_{a,b}} F(a+b).$$
We then have the second globular isomorphism $$\mu_{a,b} \colon U_{a+b} \to U_a + U_b$$ given by the identity 2-morphism:
\[
%
\]
$$\iota_{a,b} \colon !_{a+b} \xrightarrow{\sim} !_a + !_b$$
where $!_{a+b}$ and $!_a + !_b$ are both initial objects in $F(a+b)$, hence isomorphic.

There are many coherence laws to be checked, most of which are similar in flavor and make use of the two above globular isomorphisms. We check a few to give a sense of what these are like. For example,  given horizontal 1-cells $M_i,N_i,P_i$ for $i=1,2$, the following commutative diagram expresses the associativity isomorphism as a transformation of double categories.
\[
%
\]
Here, $a$ is the associator of $F\mathbb{C}\mathbf{sp}_1$ and $\chi$ is the first globular isomorphism above. To see that this diagram does indeed commute, we first consider this diagram with respect to only the underlying cospans of each horizontal 1-cell. For notation:
	\[
		%
	\]
The above diagram when written out as cospans then becomes:
\[
		%
	\]
which does indeed commute. Here, all of the vertical 1-morphisms on the left and right are associators or identities, the middle vertical 1-morphisms labeled on the left are the 2-morphisms from the previous commutative diagram, and the cospan legs are natural maps into each colimit, all of which are naturally isomorphic to each other as all the middle objects are colimits of the same diagram, namely the previous collection of cospans, taken in various ways. By identifying the top and bottom edges of the above diagram, it can be visualized as a hexagonal prism. Every face of this prism commutes. As for the morphisms of decorations, which are labeled on the right of the interior vertical 1-morphisms, each isomorphism $\iota_n$ goes from the domain under the image of the functor $F$ applied to the natural isomorphism adjacent to it to the codomain as written, meaning that, for example: $$\iota_1 \colon F(a \odot a)(((x_1+y_1)+z_1) \odot ((x_2+y_2)+z_2)) \to (x_1+(y_1+z_1)) \odot (x_2+(y_2+z_2)).$$  The following diagram commutes in the category $F((m_1+_b m_2) + ((n_1+_{b'} n_2) + (p_1+_{b''} p_2)))$:
\[
\begin{tikzpicture}[scale=1.5]
\node (A) at (0,0.5) {$F( a (\chi \otimes 1)\chi)(((x_1+y_1)+z_1) \odot ((x_2+y_2)+z_2))$};
\node (A') at (6,0.5) {$F((1 \otimes \chi) \chi)((x_1+(y_1+z_1)) \odot (x_2+(y_2+z_2)))$};
\node (B) at (0,-0.5) {$F(a(\chi  \otimes 1))(((x_1+y_1) \odot (x_2+y_2)) + (z_1 \odot z_2))$};
\node (C) at (6,-0.5) {$F(1 \otimes \chi)((x_1 \odot x_2) + ((y_1+z_1) \odot (y_2+z_2)))$};
\node (C') at (0,-1.5) {$F(a)(((x_1 \odot x_2)+(y_1 \odot y_2)) + (z_1 \odot z_2))$};
\node (D) at (6,-1.5) {$(x_1 \odot x_2) + ((y_1 \odot y_2) + (z_1 \odot z_2))$};
\path[->,font=\scriptsize,>=angle 90]
(A) edge node[above]{$F((1 \otimes \chi) \chi)(\iota_1)$} (A')
(A) edge node[left]{$F(a(\chi \otimes 1))(\iota_4)$} (B)
(A') edge node[right]{$F(1 \otimes \chi)(\iota_2)$} (C)
(B) edge node[left]{$F(a)(\iota_5)$} (C')
(C) edge node [right] {$\iota_3$} (D)
(C') edge node [above] {$\iota_6$} (D);
\end{tikzpicture}
\]
since $$F(a(\chi \otimes 1)\chi)(((x_1+y_1)+z_1) \odot ((x_2+y_2)+z_2)) = F((1 \otimes \chi) \chi (a \odot a))(((x_1+y_1)+z_1) \odot ((x_2+y_2)+z_2))$$
as the hexagon formed by the morphisms between the cospan apices of the above underlying diagram of maps of cospans commutes and then applying the pseudofunctor $F$ to this hexagon yields a commutative hexagon in $\mathbf{Cat}$.

Another requirement for a double category to be symmetric monoidal is that the braiding $$\beta_{ (-,-) } \colon F\mathbb{C}\mathbf{sp}_1 \times F\mathbb{C}\mathbf{sp}_1 \to F\mathbb{C}\mathbf{sp}_1 \times F\mathbb{C}\mathbf{sp}_1$$ be a transformation of double categories, and one of the diagrams that is required to commute is the following:
\[

\]
Using the same notation as the previous coherence diagram, the diagram for the underlying maps of cospans becomes:
\[
		%
	\]
All the comments about the previous underlying coherence diagram of maps of cospans apply to this one. As for the decorations, the following diagram commutes in the category $F((n_1+m_1)+_{(b'+b)}(n_2+m_2))$:
\[
%
\]
since $$F(\chi \beta)((x_1 \odot x_2)+(y_1 \odot y_2)) = F((\beta \odot \beta)\chi)((x_1 \odot x_2)+(y_1 \odot y_2))$$
as the square formed by the morphisms between the cospan apices of the above underlying diagram of maps of cospans commutes and then applying the pseudofunctor $F$ to this square yields a commutative square in $\mathbf{Cat}$. The other diagrams are shown to commute similarly.
\end{proof}

\section{Maps of decorated cospan double categories}\label{DCMaps}

Given another symmetric lax monoidal pseudofunctor $F^\prime \colon \mathsf{A^\prime} \to \mathbf{Cat}$, we can obtain another symmetric monoidal double category $F^\prime \mathbb{C}\mathbf{sp}$. A map from $F\mathbb{C}\mathbf{sp}$ to $F^\prime \mathbb{C}\mathbf{sp}$ will then be a double functor $\mathbb{H} \colon F\mathbb{C}\mathbf{sp} \to F' \mathbb{C}\mathbf{sp}$ whose object component is given by a finite colimit preserving functor $\mathbb{H}_0 = H \colon \mathsf{A} \to \mathsf{A^\prime}$ and whose arrow component is given by a functor $\mathbb{H}_1$ defined on horizontal 1-cells by:
\[

\]
where $E \colon \mathbf{Cat} \to \mathbf{Cat}$ is a symmetric lax monoidal pseudofunctor such that the following diagram commutes up to a monoidal natural isomorphism $\theta \colon EF \Rightarrow F'H$:
\[
%
\]
We summarize this in the following theorem:
\begin{theorem}
Given two finitely cocomplete categories $\mathsf{A}$ and $\mathsf{A}'$, two symmetric lax monoidal pseudofunctors $F \colon \mathsf{A} \to \mathbf{Cat}$ and $F' \colon \mathsf{A}' \to \mathbf{Cat}$, a finite colimit preserving functor $H \colon \mathsf{A} \to \mathsf{A}'$, a symmetric lax monoidal pseudofunctor $E \colon \mathbf{Cat} \to \mathbf{Cat}$ and a monoidal natural isomorphism $\theta \colon EF \Rightarrow F^\prime H$ as in the following diagram, the triple $(H,E,\theta)$ induces a symmetric monoidal double functor $\mathbb{H} \colon F\mathbb{C}\mathbf{sp} \to F'\mathbb{C}\mathbf{sp}$ as defined above.
\[

\]
\end{theorem}
\begin{proof}
Recall that we can think of the object $d \in F(c)$ as a morphism $d \colon 1 \to F(c)$ and the morphism $\iota \colon F(h)(d) \to d^\prime$ of $F(c^\prime)$ as a natural transformation  in $\mathbf{Cat}$:
\[
%
\]
Applying the symmetric lax monoidal pseudofunctor $E \colon \mathbf{Cat} \to \mathbf{Cat}$ to this diagram yields:
\[
%
\]
Then because the above square commutes up to the isomorphism $\theta \colon EF \Rightarrow F'H$, we get:
\[
%
\]
which results in a 2-morphism $E(\iota) \colon F^\prime(H(h))(\theta_c E(d)\phi) \to (\theta_{c'} E(d^\prime)\phi)$ in $F^\prime(H(c^\prime))$. To check that the above recipe is functorial, suppose we are given two vertically composable 2-morphisms in $F\mathbb{C}\mathbf{sp}$:
\[
%
\]
If we first compose these, the result is:
\[
%
\]
and then the image of this 2-morphism under the double functor $\mathbb{H}$ is given by:
\[
%
\]
On the other hand, applying the double functor $\mathbb{H}$ first gives:
\[
%
\]
and then composing these gives:
\[
%
\]
Thus $\mathbb{H}_1$ is functorial on 2-morphisms, and it is evident that $\mathbb{H}$ satisfies the equations $S \mathbb{H} = HS$ and $T \mathbb{H}=HT$.

Given two composable horizontal 1-cells $M$ and $N$ in $F\mathbb{C}\mathbf{sp}$:
\[
%
\]
composing first gives $M \odot N$:
\[
%
\]
where $$d \colon 1 \xrightarrow{\lambda^{-1}} 1 \times 1 \xrightarrow{d_1 \times d_2} F(c_1) \times F(c_2) \xrightarrow{\phi_{c_1,c_2}} F(c_1+c_2) \xrightarrow{F(j)}F(c_1 +_b c_2).$$ The image of this horizontal 1-cell is then given by $\mathbb{H}(M \odot N)$:
\[
%
\]
where $$d_{\mathbb{H}(M \odot N)} = \theta_{c_1 +_b c_2} E(d_{M \odot N}) \phi \colon 1 \xrightarrow{\phi} E(1) \xrightarrow{E(d_{M \odot N})} E(F(c_1 +_b c_2)) \xrightarrow{\theta_{c_1 +_b c_2}} F^\prime(H(c_1 +_b c_2)).$$ On the other hand, the image of each horizontal 1-cell under the double functor $\mathbb{H}$ is given respectively by $\mathbb{H}(M)$ and $\mathbb{H}(N)$:
\[
%
\]
Composing these then gives $\mathbb{H}(M) \odot \mathbb{H}(N)$:
\[
%
\]
where $$\hspace{-.4in}\scriptstyle{d_{\mathbb{H}(M) \odot \mathbb{H}(N)} \colon 1 \xrightarrow{(\theta_{c_1} \times \theta_{c_2})(E(d_M) \times E(d_N)) \phi} F^\prime(H(c_1)) \times F^\prime(H(c_2)) \xrightarrow{\Phi_{H(c_1),H(c_2)}} F^\prime(H(c_1)+ H(c_2)) \xrightarrow{F^\prime (J)} F^\prime(H(c_1) +_{H(b)} H(c_2)).}$$
We then have a comparison constraint: $$\mathbb{H}_{M,N} \colon \mathbb{H}(M) \odot \mathbb{H}(N) \xrightarrow{\sim} \mathbb{H}(M \odot N)$$given by the globular 2-isomorphism:
\[
%
\]
where $\kappa$ is the natural isomorphism $$\kappa \colon H(c_1 +_b c_2) \xrightarrow{\sim} H(c_1) +_{H(b)} H(c_2)$$ which comes from the finite colimit preserving functor $H \colon \mathsf{A} \to \mathsf{A^\prime}$. The above diagram commutes by a similar argument to the one used in Theorem \ref{Equiv}. Similarly, for every object $c \in \mathsf{A}$, we have a unit comparison constraint $$\mathbb{H}_U \colon U_{\mathbb{H}(c)} \to \mathbb{H}(U_c)$$ given by the globular 2-isomorphism:
\[
%
\]
where the morphism of decorations is the morphism $\iota \colon !_{H(c)} \to (\theta_c E(!_c) \phi)$ in $F'(H(c))$. These comparison constrains satisfy the coherence axioms of a monoidal category, namely that these diagrams commute:
\[
%
\]
The diagrams involving the morphisms of decorations are similar to those in Theorem \ref{DC}. This shows that $\mathbb{H}=(H,E,\theta)$ is a double functor. Next we show that this double functor is symmetric monoidal. First, that the object component $\mathbb{H}_0=H$ is symmetric monoidal is clear as $H \colon \mathsf{A} \to \mathsf{A}^\prime$ preserves finite colimits. As for the arrow component $\mathbb{H}_1$, given two horizontal 1-cells $M_1$ and $M_2$ in $F\mathbb{C}\mathbf{sp}$:
\[
%
\]
their tensor product $M_1 \otimes M_2$ in $F\mathbb{C}\mathbf{sp}$ is given by:
\[
%
\]
$$d_{M_1 \otimes M_2} \colon 1 \xrightarrow{d_1 \times d_2} F(c_1) \times F(c_2) \xrightarrow{\phi_{c_1,c_2}} F(c_1+c_2)$$
and the image of this horizontal 1-cell under the double functor $\mathbb{H}$ is $\mathbb{H}(M_1 \otimes M_2)$ given by:
\[
%
\]
On the other hand, the image of $M_1$ and $M_2$ is given by $\mathbb{H}(M_1)$ and $\mathbb{H}(M_2)$:
\[
%
\]
and their tensor product $\mathbb{H}(M_1) \otimes \mathbb{H}(M_2)$ is given by:
\[
%
\]
$$\hspace{-.4in} \scriptstyle{d_{\mathbb{H}(M_1) \otimes \mathbb{H}(M_2)} \colon 1 \xrightarrow{(\phi \times \phi)(\lambda^{-1} \times \lambda^{-1})} E(1) \times E(1) \xrightarrow{(\theta_{c_1} \times \theta_{c_2})(E(d_{M_1}) \times E(d_{M_2}))} F^\prime(H(c_1)) \times F^\prime(H(c_2)) \xrightarrow{\Phi_{H(c_1),H(c_2)}} F^\prime (H(c_1)+H(c_2)).}$$  We then have a natural 2-isomorphism $\mu_{M_1,M_2} \colon \mathbb{H}(M_1) \otimes \mathbb{H}(M_2) \to \mathbb{H}(M_1 \otimes M_2)$ in $F'\mathbb{C}\mathbf{sp}$ given by:
\[
%
\]
$$\iota_\kappa \colon F^\prime(\kappa)(d_{\mathbb{H}(M_1) \otimes \mathbb{H}(M_2)}) \to d_{\mathbb{H}(M_1 \otimes M_2)}$$
where $\kappa$ denotes the natural isomorphism arising from $H$ preserving finite colimits. This natural 2-isomorphism together with the associators of $F\mathbb{C}\mathbf{sp}$ and $F'\mathbb{C}\mathbf{sp}$, respectively $\alpha$ and $\alpha^\prime$, make the following diagram commute:
\[
%
\]
with the corresponding diagram of decorations in $F'(H(c_1+(c_2+c_3)))$:
\[
%
\]
where $$F^\prime(\alpha \kappa (\kappa+1))(d_{(\mathbb{H}(M_1) \otimes \mathbb{H}(M_2)) \otimes \mathbb{H}(M_3)}) = F^\prime(\kappa (\kappa+1) \alpha^\prime)(d_{(\mathbb{H}(M_1) \otimes \mathbb{H}(M_2)) \otimes \mathbb{H}(M_3)})$$
as the corresponding hexagon for the finite colimit preserving functor $H \colon \mathsf{A} \to \mathsf{A'}$ commutes. The map $\mu_{M_1,M_2}$ is also compatible with the braidings $\beta$ and $\beta^\prime$ of $F \mathbb{C}\mathbf{sp}_1$ and $F' \mathbb{C}\mathbf{sp}_1$, respectively, and make the necessary square commute as a consequence of the corresponding commutative square involving braidings from the finite colimit preserving functor $H \colon \mathsf{A} \to \mathsf{A'}$.

The monoidal unit of $F\mathbb{C}\mathbf{sp}_1$ is given by:
\[

\]
where $1_\mathsf{A}$ is the monoidal unit of the finitely cocomplete category $\mathsf{A}$. The image of this horizontal 1-cell under $\mathbb{H}$ is given by:
\[
%
\]
as $H$ preserves finite colimits. We then have a 2-isomorphism in $F'\mathbb{C}\mathbf{sp}$ given by: $$\mu \colon 1_{F'\mathbb{C}\mathbf{sp}_1} \to \mathbb{H}(1_{F\mathbb{C}\mathbf{sp}_1})$$ 
\[
%
\]
together with the morphism $\iota_\mu \colon F^\prime(\kappa)(!_{1_\mathsf{A^\prime}}) \to (\theta_{1_\mathsf{A}} E(!_{1_\mathsf{A}})\phi)$ in $F^\prime(H(1_\mathsf{A}))$. The following square then commutes for any horizontal 1-cell $M$ of $F \mathbb{C}\mathbf{sp}$: 
\[
%
\]
where we have abbreviated the monoidal units of $F\mathbb{C}\mathbf{sp}_1$ and $F'\mathbb{C}\mathbf{sp}_1$ as $1_\mathsf{A}$ and $1_\mathsf{A^\prime}$, respectively. The diagram of corresponding decorations is given by:
\[
%
\]
where $$F^\prime(\ell)(d_{!_{1_\mathsf{A^\prime}}} \otimes d_{\mathbb{H}(M)})=F^\prime(H(\ell^\prime)\kappa(\mu \otimes 1))(d_{!_{1_\mathsf{A^\prime}}} \otimes d_{\mathbb{H}(M)})$$ since the corresponding square involving left unitors for the finite colimit preserving functor $H \colon \mathsf{A} \to \mathsf{A'}$ commutes. The other square involving the right unitors $r$ and $r^\prime$ is similar. The comparison and unit constraints $\mathbb{H}_{M,N}$ and $\mathbb{H}_U$ are monoidal transformations and this suffices for a functor of symmetric monoidal double categories which are isofibrant, which $F\mathbb{C}\mathbf{sp}$ and $F'\mathbb{C}\mathbf{sp}$ are by Lemma \ref{DCFibrant}. Note that because the comparison constraints $\mu$ and $\mu_{(- ,-)}$ are both isomorphisms, the symmetric monoidal double functor $\mathbb{H}$ is strong.
\end{proof}

\section{Structured cospans versus decorated cospans}\label{SCDCequiv}
In this section we compare the double categories obtained via structured cospans and decorated cospans. Under conditions discovered by Christina Vasilakopoulou, the two frameworks will be shown to be equivalent as double categories. This is Theorem \ref{Equiv} and the main content of this section. But first, we make precise what it meant by an `equivalence of double categories'.

We define an equivalence of double categories following Shulman \cite{Shul2}. Given a double category $\mathbb{A}$, we write $_f \mathbb{A}_g(M,N)$ for the set of 2-morphisms in $\mathbb{A}$ of the form:
\[
  \xymatrix@-.5pc{
    A \ar[r]|{|}^{M}  \ar[d]_f \ar@{}[dr]|{\Downarrow a}&
    B\ar[d]^g\\
    C \ar[r]|{|}_N & D
  }
\]
We call $M$ and $N$ the \define{horizontal source and target} of the 2-morphism $a$, respectively, and likewise we call $f$ and $g$ the \define{vertical source and target} of the 2-morphism $a$, respectively. Thus $_f \mathbb{A}_g(M,N)$ denotes the set of 2-morphisms in $\mathbb{A}$ with horizontal source and target $M$ and $N$ and vertical source and target $f$ and $g$.
\begin{definition}
A (possibly lax or oplax) double functor $\mathbb{F} \colon \mathbb{A} \to \mathbb{X}$ is \define{full} (respectively, \define{faithful}) if $\mathbb{F}_0 \colon \mathbb{A}_0 \to \mathbb{X}_0$ is full (respectively, faithful) and each map $$\mathbb{F}_1 \colon _f \mathbb{A}_g(M,N) \to _{\mathbb{F}(f)} \mathbb{X}_{\mathbb{F}(g)}(\mathbb{F}(M),\mathbb{F}(N))$$ is surjective (respectively, injective).
\end{definition}
\begin{definition}
A (possibly lax or oplax) double functor $\mathbb{F} \colon \mathbb{A} \to \mathbb{X}$ is \define{essentially surjective} if we can simultaneously make the following choices:
\begin{enumerate}
\item{For each object $x \in \mathbb{X}$, we can find an object $a \in \mathbb{A}$ together with a vertical 1-isomorphism $\alpha_x \colon \mathbb{F}(a) \to x$, and}
\item{For each horizontal 1-cell $N \colon x_1 \to x_2$  of $\mathbb{X}$, we can find a horizontal 1-cell $M \colon a_1 \to a_2$ of $\mathbb{A}$ and a 2-isomorphism $a_{N}$ of $\mathbb{X}$ as in the following diagram:
\[
  \xymatrix@-.5pc{
    \mathbb{F}(a_1) \ar[r]|{|}^{\mathbb{F}(M)}  \ar[d]_{\alpha_{x_1}} \ar@{}[dr]|{\Downarrow a_N}&
    \mathbb{F}(a_1) \ar[d]^{\alpha_{x_2}}\\
    x_1 \ar[r]|{|}_N & x_2
  }
\]
}
\end{enumerate}
\end{definition}
\begin{definition}
A double functor $\mathbb{F} \colon \mathbb{A} \to \mathbb{X}$ is \define{strong} if the comparison and unit constraints are globular isomorphisms, meaning that for each composable pair of horizontal 1-cells $M$ and $N$ we have a natural isomorphism $$\mathbb{F}_{M,N} \colon \mathbb{F}(M) \odot \mathbb{F}(N) \xrightarrow{\sim} \mathbb{F}(M \odot N)$$and for each object $a \in \mathbb{A}$ a natural isomorphism $$\mathbb{F}_a \colon \hat{U}_{\mathbb{F}(a)} \xrightarrow{\sim} \mathbb{F}(U_a).$$
\end{definition}

Shulman \cite[Theorem 7.8]{Shul2} proved that a strong double functor is part of a `double equivalence' if and only if it is full, faithful and essentially surjective in the sense of a double functor as given above. We will take this theorem and use it as the definition of a double equivalence.

\begin{definition}
Given a strong double functor $\mathbb{F} \colon \mathbb{A} \to \mathbb{X}$, $\mathbb{F}$ is part of a \define{double equivalence} if and only if $\mathbb{F}$ is full, faithful and essentially surjective. We say that $\mathbb{F} \colon \mathbb{A} \to \mathbb{X}$ is a double equivalence and that $\mathbb{A}$ and $\mathbb{X}$ are equivalent as double categories.
\end{definition}

\begin{definition}
Given a double equivalence $\mathbb{F} \colon \mathbb{A} \to \mathbb{X}$, if $\mathbb{F}$, $\mathbb{A}$ and $\mathbb{X}$ are all symmetric monoidal, then $\mathbb{F}$ is a \define{symmetric monoidal double equivalence}, and $\mathbb{A}$ and $\mathbb{X}$ are equivalent as symmetric monoidal double categories.
\end{definition}

Given a symmetric lax monoidal pseudofunctor $F \colon (\mathsf{A},+,0) \to (\mathbf{Cat},\times,1)$, one can obtain a functor $R \colon \int F \to \mathsf{A}$ by the Grothendieck construction, as explained in Definition \ref{Groth}. Moreover, if the pseudofunctor $F \colon A \to \mathbf{Cat}$ factors through $\mathbf{Rex} \to \mathbf{Cat}$ as an ordinary pseudofunctor, the category $\int F$ will have finite colimits and this functor $R$ will preserve finite colimits and be right adjoint to a fully faithful left adjoint $L \colon \mathsf{A} \to \int F$ between two categories with finite colimits which then allows for the construction of a structured cospan double category. The bridge which allows us to obtain a left adjoint $L \colon (\A,+,0) \to (\int F,+,0)$ from a lax monoidal pseudofunctor $F \colon (\A,+,0) \to (\mathbf{Cat},\times,1)$ is established in Lemma \ref{lem:fiberwiselimits}, Corollary \ref{Rex} and Proposition \ref{prop:opfibtolari}. In this case, the resulting decorated cospan double category $F \mathbb{C} \mathbf{sp}$ and structured cospan double category $_L \mathbb{C} \mathbf{sp}(\int{F})$ are equivalent as symmetric monoidal double categories.

First we find conditions under which an opfibration has a left adjoint. This bridge between the notions of opfibration and left adjoint is due to Christina Vasilakopoulou, who together with Baez and the author have investigated this situation and its consequences in more detail \cite{BCV}.

The definitions of 2-category and pseudofunctor are given in Definitions \ref{2-cat_definition} and \ref{pseudofunctor_definition}, respectively, of the Appendix.

\begin{definition}
Let \define{Rex} denote the 2-category of categories with finite colimits and finite colimit preserving functors.
\end{definition}

\begin{definition}
A functor $R \colon \mathsf{X} \to \mathsf{A}$ is a \define{Grothendieck opfibration} if for any object $a \in \mathsf{A}$ and every object $x \in \mathsf{X}$ such that $R(x)=a$, for any morphism $f \colon a \to b$ there exists a \define{cocartesian lifting} of $f$. This means that there exists a morphism  $\beta$ in $\mathsf{X}$ whose domain is $x$ which satisfies the following universal property: for any morphism $g \colon b \to b'$ in $\mathsf{A}$ and morphism $\gamma \colon x \to y'$ in $\mathsf{X}$ such that $R(\gamma) = g \circ f$, there exists a unique morphism $\delta \colon y \to y'$ such that $\gamma = \delta \circ \beta$ and  $R(\delta)=g$.
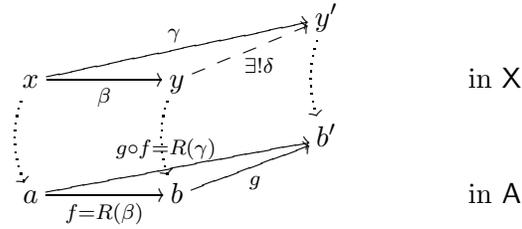
\begin{displaymath}
\xymatrix @R=.1in @C=.6in
{&& y'\ar @{.>}@/_/[dd] &&\\
x\ar[r]_-{\beta} \ar @{.>}@/_/[dd]
\ar[urr]^-{\gamma} & 
y \ar @{.>}@/_/[dd] \ar @{-->}[ur]_-{\exists! \delta}
&& \textrm{in }\X\\
&& b' &&\\
a\ar[r]_-{f=R(\beta)} \ar[urr]^-{g\circ f=R(\gamma)}
 & b \ar[ur]_-g && \textrm{in }\A}
\end{displaymath}
We call $\mathsf{X}$ the \define{total category} and $\mathsf{A}$ the \define{base category} of the opfibration $R \colon \mathsf{X} \to \mathsf{A}$.
\end{definition}
For any object $a \in \mathsf{A}$, the \define{fiber category} $\mathsf{X}_a$ consists of all objects $x \in \mathsf{X}$ such that $R(x)=a$ and all morphisms $\gamma \colon x \to x'$ such that $R(f)=1_a$. The axiom of choice allows us to select a cocartesian lifting for any $f \colon a \to b$ which we denote by $$\textnormal{Cocart}(f,x) \colon x \to f_!(x).$$
This choice also induces \define{reindexing functors} $$f_! \colon \mathsf{X}_a \to \mathsf{X}_b$$ between any two fiber categories $\mathsf{X}_a$ and $\mathsf{X}_b$. Note that by the universal property of a cocartesian lifting, we have natural isomorphisms $(1_a)_! \cong 1_{\mathsf{X}_a}$ and for any composable morphisms $f$ and $g$ in $\mathsf{A}$, $(f \circ g)_! \cong f_! \circ g_!$. If these natural isomorphisms are equalities, we say that $R$ is a \define{split opfibration}.

\begin{definition}
Let $\mathbf{OpFib}(\mathsf{A})$ be the 2-subcategory of the slice 2-category of $\mathbf{Cat}/\mathsf{A}$ of opfibrations over $\mathsf{A}$, cocartesian lifting preserving functors and natural transformations with vertical components.
\end{definition}

There is a 2-equivalence between opfibrations and pseudofunctors which is given by the well known `Grothendieck construction'.

\begin{definition}\label{Groth}
Given a pseudofunctor $F \colon \mathsf{A} \to \mathbf{Cat}$ where $\mathsf{A}$ is a category with trivial 2-morphisms, the \define{Grothendieck category} $\int F$ has:
\begin{enumerate}
\item{objects as pairs $(a, x \in F(a))$ and}
\item{a morphism from a pair $(a, x\in F(a))$ to another pair $(b, y \in F(b))$ is given by a pair $(f \colon a \to b, \iota \colon F(f)(x) \to y)$ in $\mathsf{A} \times F(b)$.}
Note that a morphism can be viewed as a morphism together with a 2-morphism:
\[
\begin{tikzpicture}[scale=1]
\node (A) at (0,0) {$a$};
\node (B) at (0,-1.5) {$b$};
\node (C) at (1,-0.75) {$1$};
\node (D) at (2.5,0) {$F(a)$};
\node (E) at (2.5,-1.5) {$F(b)$};
\node (F) at (2,-.75) {$\iota \SWarrow$};
\path[->,font=\scriptsize,>=angle 90]
(A) edge node [left]{$f$}(B)
(C) edge node [above]{$x$}(D)
(C) edge node [below]{$y$}(E)
(D) edge node [right]{$F(f)$} (E);
\end{tikzpicture}
\]
\end{enumerate}
There is an opfibration $R \colon \int F \to \mathsf{A}$ where the fiber categories are given by $(\int F)_a = F(a)$ and the associated reindexing functors are given by $f_! = F(f)$. We call the entirety of this the \define{Grothendieck construction} of the pseudofunctor $F$.
\end{definition}
The Grothendieck construction provides one direction of a well known equivalence.

\begin{theorem}\label{corresp}
\begin{enumerate}
\item{Every opfibration $R \colon \mathsf{X} \to \mathsf{A}$ gives rise to a pseudofunctor $F_R \colon \mathsf{A} \to \mathbf{Cat}$.}
\item{Every pseudofunctor $F \colon \mathsf{A} \to \mathbf{Cat}$ gives rise to an opfibration $R_F \colon \int F \to \mathsf{A}$.}\
\item{The above two correspondences give rise to an equivalence of 2-categories $$[ \mathsf{A}, \mathbf{Cat} ]_{\textnormal{ps}} \simeq \mathbf{OpFib}(\mathsf{A})$$ such that $F_{R_F} \cong F$ and $R_{F_R} \cong R$.}
\end{enumerate}
\end{theorem}

Moeller and Vasilakopoulou \cite{MV} have generalized the Grothendieck construction to the monoidal situation, meaning that \emph{lax monoidal} pseudofunctors $F \colon \mathsf{A} \to \mathbf{Cat}$ correspond bijectively to monoidal structures on the total category $\int F$ such that the corresponding opfibration $R_F \colon \int F \to \mathsf{A}$ is a strict monoidal functor and the tensor product $\otimes_{\int F}$ preserves cocartesian liftings. If $\mathsf{A}$ is cocartesian monoidal, there is a further correspondence given by:

\begin{gather}
\textrm{lax monoidal pseudofunctors }F\colon(\A,+,0)\to(\mathbf{Cat},\times,1) \notag\\
\vertsimeq \notag\\
\textrm{monoidal opfibrations }R\colon(\X,\otimes,I)\to(\A,+,0) \label{monGroth}\\
\vertsimeq \notag\\
\textrm{pseudofunctors }F\colon\A\to\mathbf{MonCat} \notag
\end{gather}

The second equivalence is due to Shulman \cite{Shul2}. In detail, given a lax monoidal structure $(\phi, \phi_0)$ on a pseudofunctor $F$, each fiber category inherits a monoidal structure via: 

\begin{gather}\label{eq:explicitstructure1}
\otimes_a\colon F(a)\times F(a)\xrightarrow{\phi_{a,a}}F(a+a)\xrightarrow{F(\nabla)}F(a)\\
I_x\colon\mathbf{1}\xrightarrow{\phi_0}F(0)\xrightarrow{F(!)}F(a).\nonumber
\end{gather}

These correspondences further restrict when the Grothendieck category $\int F$ is cocartesian monoidal itself. In this case, the monoidal opfibration clauses for $R \colon (\mathsf{X},+,0) \to (\mathsf{A},+,0)$ results in a functor (strictly) preserving coproducts and the initial object, and these bijectively correspond to pseudofunctors $F \colon \mathsf{A} \to \mathbf{cocartCat}$ where $\mathbf{cocartCat}$ is the 2-category of cocartesian categories, coproduct preserving functors and natural transformations. The following statement, which relates the existence of any class of colimits in the total category of an opfibration to their existence in the fibers, then brings pushouts into the picture by addressing when opfibrations preserve all finite colimits.  For more details, see the work of Hermida \cite{Herm}.

\begin{lemma}\label{lem:fiberwiselimits}
Let $\mathsf{J}$ be a small category and $R \colon \mathsf{X} \to \mathsf{A}$ an opfibration. If the base category $\mathsf{A}$ has $\mathsf{J}$-colimits, then the following are equivalent:
\begin{enumerate}
\item{All the fiber categories have $\mathsf{J}$-colimits and all reindexing functors preserve them.}
\item{The total category $\mathsf{X}$ has $\mathsf{J}$-colimits and $R$ preserves them.}
\end{enumerate}
\end{lemma}

The first part regards the existence of colimits locally in each fiber which can equivalently be expressed as the image of the associated pseudofunctor $F \colon \mathsf{A} \to \mathbf{Cat}$ landing in the sub-2-category $\mathbf{Rex}$ of finitely cocomplete categories and finite colimit preserving functors. The second part regards the existence of colimits globally in the total category $\int F$. These two combine to result in:

\begin{corollary}\label{Rex}
Let $\mathsf{A}$ be a category with finite colimts and $F \colon (\mathsf{A},+,0) \to (\mathbf{Cat},\times,1)$ a lax monoidal pseudofunctor. If the pseudofunctor $\mathsf{A} \to \mathbf{MonCat}$ via the correspondence in Equation \ref{monGroth} factors through $\mathbf{Rex}$, meaning that each $F(a)$ is finitely cocomplete and that the associated reindexing functors are finitely cocontinuous, then the Grothendieck category $\int F$ has all finite colimits and the corresponding opfibration $R_F \colon \int F \to \mathsf{A}$ preserves them.
\end{corollary}

For applications to structured cospans, we want a left adjoint $L_F$ to the induced monoidal opfibration $R_F \colon \int F \to \mathsf{A}$ of the Grothendieck construction of $F$. Gray found sufficient conditions for the existence of such a left adjoint.

\begin{proposition}[{{\cite[Prop.\ 4.4]{Gray}}}]\label{prop:opfibtolari}
Let $R \colon \mathsf{X} \to \mathsf{A}$ be an opfibration. Then $R$ is a right adjoint \define{left inverse}, meaning that the unit $ \eta \colon 1_\mathsf{A} \to RL$ is an identity, if and only if its fibers have initial objects which are preserved by the reindexing functors.
\end{proposition}

\begin{proof}
The left adjoint $L \colon \mathsf{A} \to \mathsf{X}$ maps an object $a$ to the initial object in its fiber which we denote by $\bot_a$ or $!_a$ in other sections of this thesis. By construction, we have that $R(L(a)) = R(\bot_a) = a$. For a morphism $f \colon a \to a'$, $L(f)$ is given by: $$\bot_a \xrightarrow{ \textnormal{Cocart}(f,{\bot_a}) } f_!(\bot_a) \xrightarrow{} \bot_{a'}$$ where the second arrow is the unique isomorphism between initial objects in the fiber above $a'$ as $f_!$ preserves them. For more details, see Gray \cite[Proposition\ 4.4]{Gray}.
\end{proof}

Notice that under Lemma \ref{lem:fiberwiselimits}, if $\mathsf{A}$ has an initial object $0_\mathsf{A}$, then the above conditions are equivalent to $\mathsf{X}$ having an initial object $0_\mathsf{X}$ above $0_\mathsf{A}$. Then $\bot_a$ is the cocartesian lifting of the unique map $!_a \colon 0_\mathsf{A} \to a$ in the base category $\mathsf{A}$:

\begin{displaymath}
\xymatrix @C=.4in @R=.2in
{0_\X \ar @{.>}[d]\ar[rr]^-{\textnormal{Cocart}({{!_a},{0_\mathsf{X}}})} && (!_a)_!(0_\mathsf{X})=:\bot_a\ar @{.>}[d] & \textrm{in }\X  \\
0_\A\ar[rr]^-{!_a} && a & \textrm{in }\A}
\end{displaymath}
\vspace{.075in}

Furthermore, if $R=R_F$ for a pseudofunctor $F \colon \mathsf{A} \to \mathbf{Cat}$ as in Theorem \ref{corresp}, the reindexing functors $(!_a)_!$ of the opfibration are given by $F(!_a)$ and therefore $\bot_a = (a,F(!_a)(0_\mathsf{X}))$. Lastly, if the pseudofunctor $(F,\phi,\phi_0) \colon (\mathsf{A},+,0) \to (\mathbf{Cat},\times,1)$ is lax monoidal to begin with, the Grothendieck construction in the cocartesian case expresses $\bot_a$ as the image of the composite $$1 \xrightarrow{\phi_0} F(0_\mathsf{A}) \xrightarrow{F(!_a)} F(a).$$

Regarding the opposite direction, which is not needed in the proof of the main result of this chapter below, we have the following result. For a discussion on the `strict cocontinuity' condition, we refer to the work of Cicala and Vasilakopoulou \cite{CV}.

\begin{proposition}
Suppose that $R \colon \mathsf{X} \to \mathsf{A}$ is a right adjoint and left inverse. If $\mathsf{X}$ and $\mathsf{A}$ both have chosen pushouts and initial objects and $R$ strictly preserves them, then $R$ is an opfibration.
\end{proposition}

Before presenting the main proof, we outline a sketch. Given a lax monoidal pseudofunctor $F \colon (\mathsf{A},+,0_\mathsf{A}) \to (\mathbf{Cat},\times,1)$, the double category of decorated cospans $F\mathbb{C}\mathbf{sp}$ has $\mathsf{A}$ as its category of objects, horizontal 1-cells as $F$-decorated cospans given by pairs $(a \rightarrow m \leftarrow b,x \in F(m))$ and 2-morphisms as maps of cospans $k \colon m \to m'$ together with a morphism $F(k)(x) \to x'$ as in Theorem \ref{DC}.

When the pseudofunctor $F$ factors through $\mathbf{Rex}$, by Corollary \ref{Rex}, the Grothendieck construction yields a finitely cocomplete Grothendieck category $\int F$ such that the corresponding opfibration $R_F \colon (\int F,+,0) \to (\mathsf{A},+,0)$ preserves all finite colimits. In particular, the initial object is preserved and so Lemma \ref{lem:fiberwiselimits} and Corollary \ref{prop:opfibtolari} apply to obtain a left adjoint $L_F \colon \mathsf{A} \to \int F$ which is right inverse to $R_F$. This left adjoint is explicitly defined on objects by $L(a) = (a, \bot_a)$ where $\bot_a$ is initial in the finitely cocomplete category $F(a)$. We can also express $\bot_a$ as $\bot_a = F(!_a) \phi_0$. Diagrammatically, this process can be expressed as:

\begin{displaymath}
 F\colon\A\to\mathbf{Cat}\quad\mapsto\quad\begin{tikzcd}[baseline=.3]\inta F\ar[d,"R_F"'] \\ \A \end{tikzcd}\quad\mapsto\quad\begin{tikzcd}\A\ar[r,bend left,pos=.55,"L_F"]\ar[r,phantom,"\bot"description] & \inta F\ar[l,bend left,pos=.45,"R_F"]\end{tikzcd}
\end{displaymath}

From this left adjoint $L_F \colon \mathsf{A} \to \int F$ which goes between finitely cocomplete categories and preserves finite colimits, we can obtain a double category of structured cospans ${ _{L_F} \mathbb{C} \mathbf{sp}(\int F) }$. This double category will also have $\mathsf{A}$ as its category of objects, but now horizontal 1-cells are given by cospans of the form $L_F(a) \rightarrow x \leftarrow L_F(b)$ in the Grothendieck category $\int F$. Explicitly, horizontal 1-cells are given by:

\begin{equation}\label{eq:scsphor1cell}
 (a,\bot_a)\xrightarrow{\scalebox{0.7}{$\begin{cases}i\colon a\to m &\textrm{in }\A \\!\colon F(i)(\bot_a)\to x &\textrm{in }F(m)\end{cases}$}}(m,x)\xleftarrow{\scalebox{0.7}{$\begin{cases}o\colon b\to m &\textrm{in }\A \\!\colon F(o)(\bot_b)\to x &\textrm{in }F(m)\end{cases}$}}(b,\bot_b)
\end{equation}

where $x \in F(m)$, as in Definition \ref{Groth}. A 2-morphism is given explicitly by:

\begin{displaymath}\label{SC2morph}
 \begin{tikzcd}[sep=1.5in,ampersand replacement=\&]
 (a,\bot_a)\ar[r,"{\scalebox{0.8}{${\begin{cases}i\colon a\to m &\textrm{in }\A \\!\colon F(i)(\bot_a)\to x &\textrm{in }F(m)\end{cases}}$}}"]\ar[d,"{\scalebox{0.8}{${\begin{cases}f\colon a\to a' &\textrm{in }\A \\\chi_a\colon F(f)(\bot_a)\cong \bot_{a'} &\textrm{in }F(a')\end{cases}}$}}"description] \& (m,x) \ar[d,"{\scalebox{0.8}{${\begin{cases}k\colon m\to m' &\textrm{in }\A \\ \iota\colon F(k)(x)\to x' &\textrm{in }F(m')\end{cases}}$}}"description] \& (b,\bot_b)\ar[l,"{\scalebox{0.8}{${\begin{cases}o\colon b\to m &\textrm{in }\A \\!\colon F(o)(\bot_b)\to x &\textrm{in }F(m)\end{cases}}$}}"']\ar[d,"{\scalebox{0.8}{${\begin{cases}g\colon b\to b' &\textrm{in }\A \\\chi_b\colon F(g)(\bot_b)\cong \bot_{b'} &\textrm{in }F(b')\end{cases}}$}}"description] \\
 (a',\bot_{a'})\ar[r,"{\scalebox{0.8}{${\begin{cases}i'\colon a'\to m' &\textrm{in }\A \\!\colon F(i')(\bot_{a'})\to x' &\textrm{in }F(m')\end{cases}}$}}"'] \& (m',x') \& (b',\bot_{b'})\ar[l,"{\scalebox{0.8}{${\begin{cases}o'\colon b'\to m' &\textrm{in }\A \\!\colon F(o')(\bot_{b'})\to x' &\textrm{in }F(m')\end{cases}}$}}"]
 \end{tikzcd}
\end{displaymath}
where the three vertical 1-morphisms in the middle come from $L_F$ applied to vertical 1-morphisms in $F\mathbb{C}\mathbf{sp}$, which are just morphisms of $\mathsf{A}$. Each of the above squares commutes which says that $ki=i'f$ and $ko=o'g$ in $\mathsf{A}$. Then in the Grothendieck category, we have:
\begin{gather}\label{eq:Grothcommutativity}
 F(k\circ i)(\bot_a)\xrightarrow{\cong} Fk(Fi(\bot_a))\xrightarrow{Fk(!)}Fk(x)\xrightarrow{\iota}x' \hspace{.1in}= \\
 F(i'\circ f)(\bot_a)\xrightarrow{\cong} Fi'(Ff(\bot_a))\xrightarrow{Fi'(\chi_a)}Fi'(\bot_{a'})\xrightarrow{!}x'\nonumber
\end{gather}
in $F(m')$. Note that all the maps in the above equality are unique and originate from initial objects, which are preserved by reindexing functors. Thus no extra conditions are imposed on these morphisms, and likewise for the square involving $o$ and $o'$.

We define a double functor $\mathbb{E} \colon { _{L_F} \mathbb{C} \mathbf{sp}} (\int F) \to F \mathbb{C}\mathbf{sp}$ whose object component is the identity on the category $\mathsf{A}$. Given a horizontal 1-cell:
\begin{equation}\label{eq:scsphor1cell}
 (a,\bot_a)\xrightarrow{\scalebox{0.7}{$\begin{cases}i\colon a\to m &\textrm{in }\A \\!\colon F(i)(\bot_a)\to x &\textrm{in }F(m)\end{cases}$}}(m,x)\xleftarrow{\scalebox{0.7}{$\begin{cases}o\colon b\to m &\textrm{in }\A \\!\colon F(o)(\bot_b)\to x &\textrm{in }F(m)\end{cases}$}}(b,\bot_b)
\end{equation}
the image is given by $$a\xrightarrow{i} m \xleftarrow{o} b \textnormal{ together with the decoration } x \in F(m).$$ Note that this is actually a bijective correspondence as the unique maps from the initial objects in the fibers provides no extra information. Given a 2-morphism of $L_F$-structured cospans as in Equation (\ref{SC2morph}), the image is given by the following map of cospans in $\mathsf{A}$:
\begin{displaymath}
 \begin{tikzcd}
a\ar[r,"i"]\ar[d,"f"'] & m\ar[d,"k"] & b\ar[l,"o"']\ar[d,"g"] \\
a'\ar[r,"i'"']& m' & b'\ar[l,"o'"]
 \end{tikzcd}
\end{displaymath}
together with the morphism $\iota \colon F(k)(x) \to x'$ as in Equation (\ref{SC2morph}). This is again a bijective correspondence and commutativity of Equation (\ref{eq:Grothcommutativity}) holds by initiality of the domain.

The double functor $\mathbb{E}=(\mathbb{E}_0,\mathbb{E}_1)$ is in fact strong. We have natural isomorphisms:
\begin{gather*}
 \mathbb{E}(M) \odot \mathbb{E}(N) \xrightarrow{\sim} \mathbb{E}(M \odot N) \\
\hat{U}_{\mathbb{E}(m)} \xrightarrow{\sim} \mathbb{E}(U_m)
\end{gather*}
for any composable horizontal 1-cells: $$M = (a,\bot_a) \xrightarrow{i} (m,x) \xleftarrow{o} (b,\bot_b)$$ and $$N = (b,\bot_b) \xrightarrow{i'} (n,y) \xleftarrow{o'} (c,\bot_c)$$ and any object $m \in {_{L_F} \mathbb{C}\mathbf{sp}}(\int F)$. The horizontal composite $\mathbb{E}(M) \odot \mathbb{E}(N)$ is given as in Theorem \ref{DCord} via a pushout and decoration:
\begin{displaymath}
 \begin{tikzcd}
  & m+_bn & \\
  a\ar[ur,"j_m\circ i"] && c,\ar[ul,"j_n\circ o'"']
 \end{tikzcd}
 \begin{tikzcd}
  1\ar[r,"x\times y"]\ar[drr,dashed] & F(m)\times F(n)\ar[r,"\phi_{m,n}"] & F(m+ n)\ar[d,"F(j)"] \\ 
  && F(m+_b n)
 \end{tikzcd}
\end{displaymath}
where $j_m \colon m \to m+_b n$ and $j_n \colon n \to m+_b n$ are the canonical maps into a pushout. If we first compose $M$ and $N$ in the structured cospan double category $_{L_F} \mathbb{C}\mathbf{sp}(\int F)$ by using fiberwise pushouts constructed using Lemma \ref{lem:fiberwiselimits}, we obtain:
\begin{displaymath}
 \begin{tikzcd}
 && (m+_b n,F(j_m)x+_{\bot_{m+_bn}}F(j_n)y) && \\
 & (m,x) \ar[ur] && (n,y)\ar[ul] \\
 (a,\bot_a)\ar[ur] && (b,\bot_b)\ar[ur]\ar[ul] && (c,\bot_c)\ar[ul]
 \end{tikzcd}
\end{displaymath}
and the image of this composite is given by the cospan $a \xrightarrow{} m+_b n \xleftarrow{} c$ together with the same decoration as the following diagram commutes:
\begin{displaymath}
 \begin{tikzcd}
F(m)\times F(n)\ar[rr,"\phi"]\ar[d,"F(j_m)\times F(j_n)"'] && F(m+n)\ar[d,"F(j)"] \\
F(m+_bn)\times F(m+_bn)\ar[r,"\phi"] & F((m+_bn)+(m+_bn))\ar[r,"F(\nabla)"] & F(m+_bn)
 \end{tikzcd}
\end{displaymath}
as the pushout is over an initial object and hence really a coproduct. The fiberwise coproduct in $F(m+_b n)$ is given as in Equation (\ref{eq:explicitstructure1}).

Lastly, for the identity morphisms, we have that $U_m$ is given by: $$(m,\bot_m) \xrightarrow{} (m,\bot_m) \xleftarrow{} (m,\bot_m)$$ with $1_m$ as the $\mathsf{A}$-component of the cospan legs together with isomorphisms between initial objects in the fibers. Hence $\mathbb{E}(U_m)$ is the identity cospan on $m$ in $\mathsf{A}$ together with the `initial decoration' or `trivial decoration' $\bot_m \in F(m)$. On the other hand, $U_{\mathbb{E}(m)}$ is the same cospan and decoration. This concludes the outline that $\mathbb{E}$ is a strong double functor.

Finally, here is the main result relating structured and decorated cospans \cite{BCV}. 

\begin{theorem}\label{Equiv}
Let $\mathsf{A}$ be a category with finite colimits and $F \colon \mathsf{A} \to \mathbf{Cat}$ a symmetric lax monoidal pseudofunctor such that $F$ factors through $\mathbf{Rex}$ as above. Then the symmetric monoidal double category $_L \mathbb{C}\mathbf{sp}(\int{F})$ built using structured cospans and the symmetric monoidal double category $F\mathbb{C}\mathbf{sp}$ built using decorated cospans are equivalent as symmetric monoidal double categories.
\end{theorem}

We will sometimes denote a decoration $x \in F(m)$ as $d_{\mathbb{E}(M)} \in F(R(x))$ where $M$ is a horizontal 1-cell of $$_L \mathbb{C} \mathbf{sp}(\mathsf{X}) = _{L_F} \mathbb{C} \mathbf{sp}(\int F),$$ and given an object $a \in { _L \mathbb{C}\mathbf{sp}(\mathsf{X})}$, the initial decoration or trivial decoration will be denoted as $\bot_a \in F(a)$ or $!_a \in F(a)$. Note, that as mentioned above, $\bot_a$ is determined by the unique map $!_a \colon 0_\mathsf{A} \to a$. The object $d_{\mathbb{E}(M)}$ is not to be mistaken for an object of $\mathsf{A}$ which we will denote by $a,b$ and $c$, or $m$ and $n$ with various primes and subscripts.

\begin{proof}[Proof of Theorem \ref{Equiv}]
As each $F \colon \mathsf{A} \to \mathbf{Cat}$ factors through $\mathbf{Rex}$, there exists a fully faithful left adjoint $L \colon \mathsf{A} \to \int{F}$ of the Grothendieck construction $R \colon \int{F} \to \mathsf{A}$ of $F$, $\int{F}$ is finitely cocomplete and $R$ preserves finite colimits. 

Next we define a double functor $\mathbb{E}$, prove it is a double equivalence, and show it is symmetric monoidal. For notation, let $\int F = \mathsf{X}$. We define a double functor $\mathbb{E} \colon _L \mathbb{C}\mathbf{sp}(\mathsf{X}) \to F\mathbb{C}\mathbf{sp}$ as follows: the object component of the double functor $\mathbb{E}$ is given by $\mathbb{E}_0 = 1_{\mathsf{A}}$ as both double categories $_L \mathbb{C}\mathbf{sp}(\mathrm{X})$ and $F\mathbb{C}\mathbf{sp}$ have objects and morphisms of $\mathsf{A}$ as objects and vertical 1-morphisms, respectively. The functor $\mathbb{E}_0$ is trivially an equivalence of categories.

Given a horizontal 1-cell $M$ of $_L \mathbb{C}\mathbf{sp}(\mathsf{X})$, which is a cospan in $\mathsf{X}$ of the form:
\[

\]
the image $\mathbb{E}_1(M)$ is given by the pair:
\[
%
\]
where $R\colon \mathsf{X} \to \mathsf{A}$ is the right adjoint to the functor $L \colon \mathsf{A} \to \mathsf{X}$ and $\eta \colon 1_{\mathsf{A}} \to RL$ is the unit of the adjunction $L \dashv R$ which is an isomorphism since $L$ is fully faithful. Similarly, the image of a 2-morphism $\alpha \colon M \to N$ in $_L \mathbb{C}\mathbf{sp}(\mathsf{X})$:
\[
%
\]
is given by the 2-morphism $\mathbb{E}_1(\alpha) \colon \mathbb{E}_1(M) \to \mathbb{E}_1(N)$ in $F\mathbb{C}\mathbf{sp}$ given by:
\[
%
\]
together with a morphism $\iota \colon F(R(\alpha))(x) \to x^\prime$ in $F(R(x^\prime))$ which comes from the Grothendieck construction of the pseudofunctor $F \colon \mathsf{A} \to \mathbf{Cat}$. That $\mathbb{E}_0$ is a functor is clear. For $\mathbb{E}_1$, given two vertically composable 2-morphisms $M$ and $M'$ in $_L \mathbb{C}\mathbf{sp}(\mathsf{X})$,
\[
%
\]
their vertical composite $M'M$ is given by:
\[
%
\]
and the image of this 2-morphism $\mathbb{E}_1(M'M)$ is given by:
\[
%
\]
together with a morphism $\iota_{M' M} \colon F(R(\alpha' \alpha))(x) \to x''$ in $F(R(x''))$. On the other hand, the individual images $\mathbb{E}_1(M)$ and $\mathbb{E}_1(M')$ are given by:
\[
%
\]
together with morphisms $\iota_M \colon F(R(\alpha))(x) \to x'$ in $F(R(x'))$ and $\iota_{M'} \colon F(R(\alpha')(x') \to x''$ in $F(R(x''))$, respectively. The vertical composite $\mathbb{E}_1(M')\mathbb{E}_1(M)$ of the above two 2-morphisms is given by $\mathbb{E}_1(M'M)$ as $R$ is a functor and $\iota_{M'M}=\iota_{M'} \iota_M$. The functors $\mathbb{E}_0$ and $\mathbb{E}_1$ satisfy the equations $\mathbb{E}_0 S = S \mathbb{E}_1$ and $\mathbb{E}_0 T = T \mathbb{E}_1$.

To show that $\mathbb{E}$ is part of a double equivalence, we need to show it is essentially surjective, full, faithful and strong.   To show it is essentially surjective, given a horizontal 1-cell in $F\mathbb{C}\mathbf{sp}$:
\[
\begin{tikzpicture}[scale=1.5]
\node (A) at (0,0) {$c_1$};
\node (B) at (1,0) {$c$};
\node (C) at (2,0) {$c_2$};
\node (D) at (3,0) {$x \in F(c)$};
\path[->,font=\scriptsize,>=angle 90]
(A) edge node[above]{$i$} (B)
(C) edge node[above]{$o$} (B);
\end{tikzpicture}
\]
we can find a 2-isomorphism in $F\mathbb{C}\mathbf{sp}$ whose codomain is the above horizontal 1-cell and whose domain is the image of the following horizontal 1-cell in $_L \mathbb{C}\mathbf{sp}(\mathsf{X})$:
\[
\begin{tikzpicture}[scale=1.5]
\node (A) at (0,0) {$L(c_1)$};
\node (B) at (1,0) {$x$};
\node (C) at (2,0) {$L(c_2)$};
\path[->,font=\scriptsize,>=angle 90]
(A) edge node[above]{$i^\prime$} (B)
(C) edge node[above]{$o^\prime$} (B);
\end{tikzpicture}
\]
with the 2-isomorphism in $F\mathbb{C}\mathbf{sp}$ given by:
\[
\begin{tikzpicture}[scale=1.5]
\node (A) at (-1,0) {$c_1$};
\node (B) at (1,0) {$R(x)$};
\node (C) at (3,0) {$c_2$};
\node (A') at (-1,-1) {$c_1$};
\node (B') at (1,-1) {$c$};
\node (C') at (3,-1) {$c_2$};
\node (D) at (4,0) {$x \in F(R(x))$};
\node (D') at (4,-1) {$x \in F(c)$};
\path[->,font=\scriptsize,>=angle 90]
(A) edge node[above]{$R(i^\prime) \eta_{c_1}$} (B)
(C) edge node[above]{$R(o^\prime) \eta_{c_2}$} (B)
(A') edge node[above]{$i$} (B')
(C') edge node[above]{$o$} (B')
(A) edge node [left]{$1$} (A')
(B) edge node [left]{${(R(e) \eta_c)}^{-1}$} (B')
(C) edge node [right]{$1$} (C');
\end{tikzpicture}
\]
$$\iota \colon F({(R(e)\eta_c)}^{-1})(x) \to x$$
where $e \colon L(c) \to x$ is given by the unique map from the trivial decoration on $c$ to $x \in F(c)$. The object and arrow components $\mathbb{E}_0$ and $\mathbb{E}_1$ satisfy the equations $S \mathbb{E}_1 = \mathbb{E}_0 S$ and $T \mathbb{E}_1 = \mathbb{E}_0 T$.

To show that the double functor $\mathbb{E}$ is full and faithful, we need to show that the map  $$\mathbb{E}_1 \colon _f { _L \mathbb{C}\mathbf{sp}(\mathsf{X})}_g(M,N) \to {_{\mathbb{E}(f)} {F\mathbb{C}\mathbf{sp}}_{\mathbb{E}(g)}}(\mathbb{E}(M),\mathbb{E}(N))$$ is bijective for arbitrary vertical 1-morphisms $f$ and $g$ and horizontal 1-cells $M$ and $N$ of $_L \mathbb{C}\mathbf{sp}(\mathsf{X})$. Consider a 2-morphism in $_L \mathbb{C}\mathbf{sp}(\mathsf{X})$ with horizontal source and target $M$ and $N$, respectively and vertical source and target $f$ and $g$, respectively:
\[
\begin{tikzpicture}[scale=1.5]
\node (A) at (0,0) {$L(c_1)$};
\node (B) at (1,0) {$x$};
\node (C) at (2,0) {$L(c_2)$};
\node (A') at (0,-1) {$L(c_1^\prime)$};
\node (B') at (1,-1) {$x^\prime$};
\node (C') at (2,-1) {$L(c_2^\prime)$};
\node (D) at (1,0.5) {$M$};
\node (E) at (-1,-0.5) {$f$};
\node (F) at (1,-1.5) {$N$};
\node (G) at (3,-0.5) {$g$};
\path[->,font=\scriptsize,>=angle 90]
(A) edge node[above]{$i$} (B)
(C) edge node[above]{$o$} (B)
(A') edge node[above]{$i^\prime$} (B')
(C') edge node[above]{$o^\prime$} (B')
(A) edge node [left]{$L(f)$} (A')
(B) edge node [left]{$\alpha$} (B')
(C) edge node [right]{$L(g)$} (C');
\end{tikzpicture}
\]
Thus the set $$_f { _L \mathbb{C}\textnormal{sp}(\mathsf{X})}_g(M,N)$$ consists of triples $$(f,\alpha,g)$$ rendering the above diagram commutative where $f$ and $g$ are morphisms of $\mathsf{A}$ and $\alpha$ is a morphism of $\mathsf{X}$. The image of the above 2-morphism under the double functor $\mathbb{E}$ has horizontal source and target given by $\mathbb{E}(M)$ and $\mathbb{E}(N)$, respectively, and vertical source and target given by $\mathbb{E}(f)$ and $\mathbb{E}(g)$, respectively:
\[
\begin{tikzpicture}[scale=1.5]
\node (A) at (0,0) {$c_1$};
\node (B) at (1,0) {$R(x)$};
\node (C) at (2,0) {$c_2$};
\node (A') at (0,-1) {$c_1^\prime$};
\node (B') at (1,-1) {$R(x^\prime)$};
\node (C') at (2,-1) {$c_2^\prime$};
\node (D) at (1,0.5) {$x \in F(R(x))$};
\node (D') at (1,-1.5) {$x^\prime \in F(R(x^\prime))$};
\node (E) at (1,1) {$\mathbb{E}(M)$};
\node (F) at (-1,-0.5) {$\mathbb{E}(f)$};
\node (G) at (1,-2) {$\mathbb{E}(N)$};
\node (H) at (3,-0.5) {$\mathbb{E}(g)$};
\path[->,font=\scriptsize,>=angle 90]
(A) edge node[above]{$R(i)\eta_{c_1}$} (B)
(C) edge node[above]{$R(o)\eta_{c_2}$} (B)
(A') edge node[above]{$R(i^\prime)\eta_{c_1^\prime}$} (B')
(C') edge node[above]{$R(o^\prime)\eta_{c_2^\prime}$} (B')
(A) edge node [left]{$f$} (A')
(B) edge node [left]{$R(\alpha)$} (B')
(C) edge node [right]{$g$} (C');
\end{tikzpicture}
\]
together with a morphism $\iota \colon F(R(\alpha))(x) \to x^\prime$ of $F(R(x^\prime))$.
Thus the set $$_{\mathbb{E}(f)} {F\mathbb{C}\mathbf{sp}}_{\mathbb{E}(g)}(\mathbb{E}(M),\mathbb{E}(N))$$ consists of 4-tuples $$(f,R(\alpha),g,\iota)$$ rendering the above diagram commutative and where $f,g$ and $R(\alpha)$ are morphisms of $\mathsf{A}$ and $\iota$ is a morphism in $F(R(x'))$. The morphisms $R(\alpha) \colon R(x) \to R(x^\prime)$ and $\iota \colon F(R(\alpha))(x) \to x^\prime$ together determine the morphism $\alpha \colon x \to x^\prime$ in $\mathrm{X}$ and conversely: given two objects $x=(c,x \in F(c))$ and $x^\prime=(c^\prime,x^\prime \in F(c^\prime))$ of $\mathsf{X}=\int{F}$, a morphism from $\alpha \colon x \to x^\prime$ is a pair $$(h \colon c \to c^\prime, \iota \colon F(h)(x) \to x^\prime)$$ where $h \colon c \to c^\prime$ is given by $R(\alpha) \colon R(x) \to R(x^\prime)$. This shows that $\mathbb{E}$ is fully faithful.

Next we show that the double functor $\mathbb{E}$ is strong by exhibiting a natural isomorphism $$\mathbb{E}_{M,N} \colon \mathbb{E}(M) \odot \mathbb{E}(N) \xrightarrow{\sim} \mathbb{E}(M \odot N)$$ for every pair of composable horizontal 1-cells $M$ and $N$ of $_L \mathbb{C}\mathbf{sp}(\mathsf{X})$ and for every object $c \in { _L \mathbb{C}\mathbf{sp}(\mathsf{X})}$ a natural isomorphism $$\mathbb{E}_c \colon \hat{U}_{\mathbb{E}(c)} \xrightarrow{\sim} \mathbb{E}(U_c)$$ where $U$ and $\hat{U}$ are the unit functors of $_L \mathbb{C} \mathbf{sp}(\mathsf{X})$ and $F\mathbb{C}\mathbf{sp}$, respectively. For any object $c$, the horizontal 1-cell $\hat{U}_{\mathbb{E}(c)}$ is given by $\hat{U}_c$ which is given by the pair:
\[

\]
The horizontal 1-cell $U_c$ is given by
\[
%
\]
and so $\mathbb{E}(U_c)$ is given by the pair:
\[
%
\]
We can then obtain the natural isomorphism $\mathbb{E}_c \colon \hat{U}_{\mathbb{E}(c)} \xrightarrow{\sim} \mathbb{E}(U_c)$ as the 2-morphism
\[
%
\]
$$\iota \colon F(\eta_c)(!_c) \xrightarrow{!} !_{R(L(c))}$$
in $F\mathbb{C}\mathbf{sp}$.

Next, given composable horizontal 1-cells $M$ and $N$ in $_L \mathbb{C}\mathbf{sp}(\mathsf{X})$:
\[
%
\]
their images $\mathbb{E}(M)$ and $\mathbb{E}(N)$ are given by:
\[
%
\]
and so $\mathbb{E}(M) \odot \mathbb{E}(N)$ is given by:
\[
%
\]
$$\scriptstyle{d_{\mathbb{E}(M) \odot \mathbb{E}(N)} \colon 1 \xrightarrow{\lambda^{-1}} 1 \times 1 \xrightarrow{x \times x^\prime} F(R(x)) \times F(R(x^\prime)) \xrightarrow{\phi_{R(x),R(x^\prime)}} F(R(x)+R(x^\prime)) \xrightarrow{F(j_{R(x),R(x^\prime)})} F(R(x)+_{c_2}R(x^\prime))}$$where $\psi$ denotes each natural map into the coproduct and $j$ denotes the natural map from the coproduct to the pushout. On the other hand, $M \odot N$ is given by
\[
%
\]
where $\zeta$ is a natural map into a coproduct and $J$ is the natural map from the coproduct to the pushout. Then $E(M \odot N)$ is given by
\[
%
\]
and so $\mathbb{E}_{M,N} \colon \mathbb{E}(M) \odot \mathbb{E}(N) \xrightarrow{\sim} \mathbb{E}(M \odot N)$ is given by the 2-morphism:
\[
%
\]
First, the right adjoint $R$ also preserves finite colimits and so we have a natural isomorphism $$\kappa \colon R(x) +_{R(L(c_2))} R(x^\prime) \to R(x+_{L(c_2)}x^\prime).$$ Also, since the left adjoint $L \colon \mathsf{A} \to \mathsf{X}$ is fully faithful, the unit of the adjunction $L \dashv R$ at the object $c_2$ gives a natural isomorphism $\eta_{c_2} \colon c_2 \to R(L(c_2))$ which results in a natural isomorphism $$j_{\eta_{c_2}} \colon R(x) +_{c_2} R(x^\prime) \to R(x) +_{R(L(c_2))} R(x^\prime).$$ Composing these two results in a natural isomorphism $$\sigma \colon = \kappa j_{\eta_{c_2}} \colon R(x) +_{c_2} R(x^\prime) \to R(x+_{L(c_2)}x^\prime).$$
Next, to see that the above diagram commutes, it suffices to show that for the object $c_1 \in \mathsf{A}$, $$R(J)R(\zeta)R(i)\eta_{c_1}(c_1) = R(J \zeta i)\eta_{c_1}(c_1)  \stackrel{!}{=} \sigma j \psi R(i)\eta_{c_1}(c_1) = \kappa j_{\eta_{c_2}} \psi R(i) \eta_{c_1}(c_1).$$ This follows as $R(i) \eta_{c_1} \colon c_1 \to R(x)$ and the following diagram commutes:
\[
\begin{tikzpicture}[scale=1.5]
\node (B) at (0,0) {$R(x)$};
\node (C) at (2,0) {$R(x)+R(x^\prime)$};
\node (A') at (4,0) {$R(x)+_{c_2}R(x^\prime)$};
\node (B') at (4,-2) {$R(x+_{L(c_2)}x^\prime)$};
\node (D) at (0,-2) {$R(x+x^\prime)$};
\node (D') at (4,-1) {$R(x)+_{R(L(c_2))} R(x^\prime)$};
\path[->,font=\scriptsize,>=angle 90]
(C) edge node[above]{$j$} (A')
(B) edge node[above]{$\psi$} (C)
(D) edge node[above]{$R(J)$} (B')
(B) edge node [left]{$R(\zeta)$} (D)
(A') edge node [right]{$j_{\eta_{c_2}}$} (D')
(A') edge [out=345,in=15] node [right]{$\sigma$} (B')
(D') edge node [right]{$\kappa$} (B');
\end{tikzpicture}
\]
Lastly, this map of cospans comes with an isomorphism $\iota \colon F(\sigma)(d_{\mathbb{E}(M) \odot \mathbb{E}(N)}) \to d_{\mathbb{E}(M \odot N)}$ in $F(R(x+_{L(c_2)}x^\prime))$. This shows that $\mathbb{E}$ is strong, and so $\mathbb{E} \colon _L \mathbb{C}\mathbf{sp}(\mathsf{X}) \xrightarrow{\sim} F\mathbb{C}\mathbf{sp}$ is part of a double equivalence by  a Theorem of Shulman \cite[Theorem 7.8]{Shul2}.

Next we will show that this equivalence of double categories $\mathbb{E} \colon _L \mathbb{C}\mathbf{sp}(\mathsf{X}) \to F\mathbb{C}\mathbf{sp}$ is symmetric monoidal. First, note that we have a natural isomorphism $\epsilon \colon 1_{F\mathbb{C}\mathbf{sp}} \to \mathbb{E}(1_{_L \mathbb{C}\mathbf{sp}(\mathsf{X})})$ and natural isomorphisms $\mu_{c_1,c_2} \colon \mathbb{E}(c_1) \otimes \mathbb{E}(c_2) \to \mathbb{E}(c_1 \otimes c_2)$ for every pair of objects $c_1,c_2 \in {_L \mathbb{C} \mathbf{sp}(\mathsf{X})}$ both of which are given by identities since both double categories $_L \mathbb{C}\mathbf{sp}(\mathsf{X})$ and $F\mathbb{C}\mathbf{sp}$ have $\mathsf{A}$ as their category of objects and $\mathbb{E}_0=1_{\mathsf{A}}$. The diagrams utilizing these maps that are required to commute do so trivially.

For the arrow component $\mathbb{E}_1$, we have a natural isomorphism $\delta \colon U_{1_{F\mathbb{C}\mathbf{sp}}} \to \mathbb{E}(U_{1_{_L \mathbb{C}\mathbf{sp}(\mathsf{X})}})$ where the horizontal 1-cell $U_{1_{F\mathbb{C}\mathbf{sp}}}$ is given by:
\[

\]
where $!_{1_{\mathsf{A}}} = \phi \colon 1 \to F(1_\mathsf{A})$ is the trivial decoration which comes from the structure of the symmetric lax monoidal pseudofunctor $F \colon \mathsf{A} \to \mathbf{Cat}$. The horizontal 1-cell $U_{1_{_L \mathbb{C}\mathbf{sp}(\mathsf{X})}}$ is given by:
\[
%
\]
where here we make use of the fact that the left adjoint $L \colon (\mathsf{A},+,1_\mathsf{A}) \to (\mathsf{X},+,1_\mathsf{X})$ preserves all colimits and thus $L(1_\mathsf{A}) \cong 1_\mathsf{X}$. The horizontal 1-cell $\mathbb{E}(U_{1_{_L \mathbb{C} \mathbf{sp}(\mathsf{X})}})$ is then given by the pair:
\[
%
\]
The natural isomorphism $\delta$ is then given by the 2-morphism:
\[
%
\]
$$\iota_{{\eta_{1_\mathsf{A}}}} \colon F({\eta_{1_\mathsf{A}}})(!_{1_\mathsf{A}}) \to !_{R(L(1_\mathsf{A}))}$$
of $F\mathbb{C}\mathbf{sp}$. This is just the natural isomorphism $\mathbb{E}_{1_\mathsf{A}}$ from earlier.

Given two horizontal 1-cells $M$ and $N$ of $_L \mathbb{C}\mathbf{sp}(\mathsf{X})$:
\[
%
\]
their images $\mathbb{E}(M)$ and $\mathbb{E}(N)$ are given by:
\[
%
\]
and so $\mathbb{E}(M) \otimes \mathbb{E}(N)$ is given by:
\[
%
\]
where $$d_{\mathbb{E}(M) \otimes \mathbb{E}(N)} \colon 1 \xrightarrow{\lambda^{-1}} 1 \times 1 \xrightarrow{x \times x^\prime} F(R(x)) \times F(R(x^\prime)) \xrightarrow{\phi_{R(x),R(x^\prime)}} F(R(x)+R(x^\prime)).$$ On the other hand, $M \otimes N$ is given by
\[
%
\]
and $\mathbb{E}(M \otimes N)$ is given by:
\[
%
\]
We then have a natural 2-isomorphism $\mu_{M,N} \colon E(M) \otimes E(N) \xrightarrow{\sim} E(M \otimes N)$ in $F\mathbb{C}\mathbf{sp}$ given by:
\[
%
\]
where $\kappa$ is the isomorphism which comes from $R \colon \mathsf{X} \to \mathsf{A}$ preserving finite colimits.

The natural isomorphisms $\delta$ and $\mu$ satisfy the left and right unitality squares, associativity hexagon and braiding square. To see this, let $M_1,M_2$ and $M_3$ be horizontal 1-cells in $_L \mathbb{C}\mathbf{sp}(\mathsf{X})$ given by:
\[
%
\]
The left unitality square:
\[
%
\]
has an underlying diagram of maps of cospans given by:
\[
		%
	\]
with the corresponding maps of decorations amounting to the following commutative diagram in $F(R(x_1))$:
\[
%
\]
where $x_{!+1}$ is the decoration $x_1$ on the object $R(L(1_\mathsf{A})+x_1) \in \mathsf{A}$. The above square commutes because $$F(\lambda_\mathsf{A})(!_{1_\mathsf{A}}+x_1) = F(R(\lambda_\mathsf{X})(\mu_{L(1_\mathsf{A}),x_1})(\eta_{1_\mathsf{A}}+1))(!_{1_\mathsf{A}}+x_1)$$ as the corresponding left unitality square for the finite colimit preserving functor $R \colon (\mathsf{X},1_\mathsf{X},+) \to (\mathsf{A},1_\mathsf{A},+)$ commutes. The right unitality square is similar. The associator hexagon:
\[
%
\]
has underlying maps of cospans given by:
\[
		%
	\]
Here, due to limited space, we have omitted the natural isomorphisms $\phi_{c_i,c_j} \colon L(c_i) + L(c_j) \to L(c_i + c_j)$ on the inward pointing morphisms which make up the legs of each cospan. The corresponding maps of decorations amount to the following commutative diagram in $F(R(x_1+(x_2+x_3)))$:
\[
%
\]
The above hexagon commutes because $$F((\kappa)(1+\kappa)(a_\mathsf{A}))((x_1+x_2)+x_3) = F((R(a_\mathsf{X}))(\kappa)(\kappa+1))((x_1+x_2)+x_3)$$ as the corresponding associator hexagon for the finite colimit preserving functor $R \colon (\mathsf{X},1_\mathsf{X},+) \to (\mathsf{A},1_\mathsf{A},+)$ commutes. Lastly, the braiding square:
\[
%
\]
has underlying map of cospans given by:
\[
		%
	\]
Again, we have omitted the natural isomorphisms $\phi_{c_i,c_j}$ on the inward pointing morphisms on each cospan leg due to space restrictions. The corresponding maps of decorations amount to the following commutative diagram in $F(R(x_2+x_1))$:
\[
%
\]
The above square commutes because $$F((\kappa)(\beta_\mathsf{A}))(x_1+x_2) = F((R(\beta_\mathsf{X}))(\kappa))(x_1+x_2)$$ as the corresponding braiding square for the finite colimit preserving functor $R \colon (\mathsf{X},1_\mathsf{X},+) \to (\mathsf{A},1_\mathsf{A},+)$ commutes. The comparison and unit constraints $\mathbb{E}_{M,N}$ and $\mathbb{E}_c$ are monoidal natural transformations, and as both $_L \mathbb{C} \mathbf{sp}(\mathsf{X})$ and $F\mathbb{C}\mathbf{sp}$ are isofibrant by Lemmas \ref{trick3} and \ref{DCFibrant}, respectively, the double functor $\mathbb{E} \colon _L\mathbb{C}\mathbf{sp}(\mathsf{X}) \to F\mathbb{C}\mathbf{sp}$ is symmetric monoidal.
\end{proof}

\section{Applications}

In this section we present the three examples that were illustrated with the original decorated cospans as well as structured cospans. The first example regarding graphs was mentioned in the introduction and is the easiest example to keep in mind. The next two examples, taking on more of an applied flavor, consist of electrical circuits and Petri nets. Each of these has been studied extensively in work on `black-boxing' \cite{BCR,BF,BFP,BM,BP}. Black-boxing is a way of describing the behavior of an open system, that is, a system with prescribed inputs and outputs such as the terminals of an electrical circuit, by observing the activity at the inputs and the outputs, typically while the system is in a `steady state'. The relation between the activity at inputs and outputs can be seen as a morphism in some category of relations.  A black-boxing functor sending open electrical circuits to Lagrangian linear relations was first constructed using Fong's theory of decorated cospans \cite{BF}, and later via the theory of props \cite{BCR}. A black-boxing double functor sending open Petri nets to relations was constructed using structured cospans \cite{BM}. A black-boxing functor for a special class of Markov processes was constructed using decorated cospans \cite{BFP}; later it was generalized and enhanced to a double functor \cite{BC}, as explained in Chapter \ref{Chapter6}.



\subsection{Graphs}

As a first example, let $L \colon \mathsf{FinSet} \to \mathsf{FinGraph}$ be the functor that assigns to a set $N$ the \define{discrete graph} $L(N)$ which is the edgeless graph with $N$ as its set of vertices. Both $\mathsf{FinSet}$ and $\mathsf{FinGraph}$ have finite colimits and the functor $L \colon \mathsf{FinSet} \to \mathsf{FinGraph}$ is left adjoint to the forgetful functor $R \colon \mathsf{FinGraph} \to \mathsf{FinSet}$ which assigns to a finite graph $G$ its underlying finite set of vertices, $R(G)$. Using structured cospans and appealing to Theorem \ref{SC}, we get a symmetric monoidal double category $_L \mathbb{C}\mathbf{sp}(\mathsf{FinGraph})$ which has:
\begin{enumerate}
\item{finite sets as objects,}
\item{functions as vertical 1-morphisms,}
\item{\define{open graphs}, or cospans of graphs of the form
\[
\begin{tikzpicture}[scale=1.5]
\node (A) at (0,0) {$L(N)$};
\node (B) at (1,0) {$G$};
\node (C) at (2,0) {$L(M)$};
\path[->,font=\scriptsize,>=angle 90]
(A) edge node[above]{$I$} (B)
(C) edge node[above]{$O$} (B);
\end{tikzpicture}
\]
as horizontal 1-cells, where $L(N)$ and $L(M)$ are discrete graphs on the sets $N$ and $M$, respectively, $G$ is a graph and $I$ and $O$ are graph morphisms, and}
\item{maps of cospans of graphs of the form
\[
\begin{tikzpicture}[scale=1.5]
\node (A) at (0,0) {$L(N_1)$};
\node (B) at (1,0) {$G_1$};
\node (C) at (2,0) {$L(M_1)$};
\node (A') at (0,-1) {$L(N_2)$};
\node (B') at (1,-1) {$G_2$};
\node (C') at (2,-1) {$L(M_2)$};
\path[->,font=\scriptsize,>=angle 90]
(A) edge node[above]{$I_1$} (B)
(C) edge node[above]{$O_1$} (B)
(A') edge node[above]{$I_2$} (B')
(C') edge node[above]{$O_2$} (B')
(A) edge node [left]{$L(f)$} (A')
(B) edge node [left]{$\alpha$} (B')
(C) edge node [right]{$L(g)$} (C');
\end{tikzpicture}
\]
as 2-morphisms, where $L(f)$ and $L(g)$ are maps of discrete graphs induced by the underlying functions $f$ and $g$, respectively, and $\alpha \colon G_1 \to G_2$ is a graph morphism.
}
\end{enumerate}

This is precisely Theorem \ref{scgraphs}. We can obtain a similar symmetric monoidal double category using decorated cospans. Let $F \colon \mathsf{FinSet} \to \mathbf{Cat}$ be the symmetric lax monoidal pseudofunctor that assigns to a finite set $N$ the \emph{category} of all graph structures whose underlying set of vertices is $N$. Thus, $F(N)$ is the category where:
\begin{enumerate}
\item{objects are given by graphs each having $N$ as their set of vertices
\[
\begin{tikzpicture}[scale=1.5]
\node (B) at (1,0) {$E$};
\node (C) at (2,0) {$N$};
\path[->,font=\scriptsize,>=angle 90]
(B)edge[bend left] node[above]{$s$}(C)
(B)edge[bend right] node[below]{$t$}(C);
\end{tikzpicture}
\]
 and}
\item{morphisms are given by maps of edges $f \colon E \to E'$ making the following two triangles commute:
\[
\begin{tikzpicture}[scale=1.5]
\node (A) at (5,0) {$E$};
\node (A') at (5,-1) {$E'$};
\node (E) at (6,-0.5) {$N$};
\node (B) at (7,0) {$E$};
\node (B') at (7,-1) {$E'$};
\node (F) at (8,-0.5) {$N$};
\path[->,font=\scriptsize,>=angle 90]
(A) edge node[above]{$s$}(E)
(A') edge node [below] {$s'$} (E)
(A) edge node [left] {$f$} (A')
(B) edge node [left] {$f$} (B')
(B) edge node[above]{$t$}(F)
(B') edge node [below] {$t'$} (F);
\end{tikzpicture}
\]
}
\end{enumerate}

The laxator $$\mu_{N_1,N_2} \colon F(N_1) \times F(N_2) \to F(N_1+N_2)$$for this symmetric lax monoidal pseudofunctor $F$ is analogous to the laxator for the monoidal functor $F$ of Section \ref{Graphs}. By Theorem \ref{DC}, we have the following:

\begin{theorem}\label{dcgraphs}
Let $F \colon \mathsf{FinSet} \to \mathbf{Cat}$ be the symmetric lax monoidal pseudofunctor which assigns to a finite set $N$ the category of all graph structures whose underlying set of vertices is $N$. Then there exists a symmetric monoidal double category $F\mathbb{C}\mathbf{sp}$ which has:
\begin{enumerate}
\item{finite sets as objects,}
\item{functions as vertical 1-morphisms,}
\item{horizontal 1-cells as pairs:
\[
\begin{tikzpicture}[scale=1.5]
\node (A) at (0,0) {$N$};
\node (B) at (1,0) {$P$};
\node (C) at (2,0) {$M$};
\node (D) at (3.25,0) {$G \in F(P)$};
\path[->,font=\scriptsize,>=angle 90]
(A) edge node[above]{$i$} (B)
(C) edge node[above]{$o$} (B);
\end{tikzpicture}
\]
which can also be thought of as open graphs, and}
\item{2-morphisms as maps of cospans of finite sets
\[
\begin{tikzpicture}[scale=1.5]
\node (A) at (0,0) {$N_1$};
\node (A') at (0,-1) {$N_2$};
\node (C') at (2,-1) {$M_2$};
\node (B) at (1,0) {$P_1$};
\node (C) at (2,0) {$M_1$};
\node (D) at (1,-1) {$P_2$};
\node (E) at (3,0) {$G_1 \in F(P_1)$};
\node (F) at (3,-1) {$G_2 \in F(P_2)$};
\path[->,font=\scriptsize,>=angle 90]
(A) edge node[above]{$i_1$} (B)
(C) edge node[above]{$o_1$} (B)
(A) edge node[left]{$f$} (A')
(C) edge node[right]{$g$} (C')
(C') edge node [above] {$o_2$} (D)
(A') edge node [above] {$i_2$} (D)
(B) edge node [left] {$h$} (D);
\end{tikzpicture}
\]
together with a graph morphism $\iota \colon F(h)(G_1) \to G_2$ in $F(P_2)$.}
\end{enumerate}
\end{theorem}
\begin{proof}
This follows immediately from Theorem \ref{DC}.
\end{proof}
We thus have two symmetric monoidal double categories: $_L \mathbb{C}\mathbf{sp}(\mathsf{FinGraph})$ obtained from structured cospans and $F\mathbb{C}\mathbf{sp}$ obtained from decorated cospans. Both of these double categories have $\mathsf{FinSet}$ as their categories of objects, open graphs as horizontal 1-cells and maps of open graphs as 2-morphisms, and by Theorem \ref{Equiv}, we have an equivalence of symmetric monoidal double categories $_L \mathbb{C}\mathbf{sp}(\mathsf{FinGraph}) \simeq F\mathbb{C}\mathbf{sp}$.

\begin{corollary}
The symmetric monoidal double category $_L \lC\mathbf{sp}(\mathsf{FinGraph})$ of Theorem \ref{scgraphs} and the symmetric monoidal double category $F\mathbb{C}\mathbf{sp}$ of Theorem \ref{dcgraphs} are equivalent.
\end{corollary}

\begin{proof}
This follows immediately from Theorem \ref{Equiv}.
\end{proof}

\subsection{Electrical circuits}
In a previous work \cite{BP}, Baez and Fong attempted to use decorated cospans to construct a symmetric monoidal category of open $k$-graphs. Now we can fix the problems in this construction. Recall from Definition \ref{definition:k-graph} that given a field $k$ with positive elements, a $k$-graph is given by a diagram in $\mathsf{Set}$ of the form:
\[
\begin{tikzpicture}[scale=1.5]
\node (B) at (1,0) {$E$};
\node (C) at (2,0) {$V$};
\node (A) at (0,0) {$k^+$};
\path[->,font=\scriptsize,>=angle 90]
(B)edge node[above] {$r$}(A)
(B)edge[bend left] node[above]{$s$}(C)
(B)edge[bend right] node[below]{$t$}(C);
\end{tikzpicture}
\]
Here the finite sets $E$ and $V$ are the sets of edges and vertices, respectively, and if we take the field $k = \mathbb{R}$, the function $r \colon E \to \mathbb{R}^+$ assigns to each edge $e \in E$ a positive real number $r(e) \in \mathbb{R}^+$ which can be interpreted as the resistance at the edge $e$. We restrict to finite sets to avoid convergence issues with certain summations. An open $k$-graph is then given by a cospan of finite sets 
\[
\begin{tikzpicture}[scale=1.5]
\node (B) at (1,0) {$V$};
\node (C) at (2,0) {$Y$};
\node (A) at (0,0) {$X$};
\path[->,font=\scriptsize,>=angle 90]
(A)edge node[above] {$i$}(B)
(C)edge node[above]{$o$}(B);
\end{tikzpicture}
\]
where the apex $V$ is equipped with the structure of a $k$-graph. See the original paper for more details \cite{BP}.

Let $\mathsf{FinGraph}_k$ be the category whose objects are given by $k$-graphs and morphisms by morphisms of $k$-graphs, where a morphism of $k$-graphs is given by a pair of functions $f \colon E \to E^\prime$ and $g \colon V \to V^\prime$ between the edge sets and vertex sets, respectively, of two $k$-graphs that respect the source and target functions of each, and such that the resistances of each edge are preserved. In the original work introducing structured cospans, it is shown that the category $\mathsf{FinGraph}_k$ has finite colimits \cite{BC2}. We can then obtain a double category of open $k$-graphs by defining a left adjoint $L \colon \mathsf{FinSet} \to \mathsf{FinGraph}_k$ that assigns to a finite set $V$ the discrete $k$-graph $L(V)$ given by the $k$-graph with $V$ as its set of vertices and no edges. The resulting symmetric monoidal double category $_L \mathbb{C}\mathbf{sp}(\mathsf{FinGraph}_k)$ has:
\begin{enumerate}
\item{finite sets as objects,}
\item{functions as vertical 1-morphisms,}
\item{open $k$-graphs as horizontal 1-cells
\[
\begin{tikzpicture}[scale=1.5]
\node (B) at (1,0) {$V$};
\node (C) at (2,0) {$Y$};
\node (A) at (0,0) {$X$};
\node (D) at (3,0) {$k^+$};
\node (E) at (4,0) {$E$};
\node (F) at (5,0) {$V$};
\path[->,font=\scriptsize,>=angle 90]
(E) edge node [above] {$r$} (D)
(E) edge [bend left] node [above] {$s$} (F)
(E) edge [bend right] node [below] {$t$} (F)
(A)edge node[above] {$i$}(B)
(C)edge node[above]{$o$}(B);
\end{tikzpicture}
\]
and}
\item{maps of cospans as 2-morphisms together with a map of $k$-graphs between the apices.
\[
\begin{tikzpicture}[scale=1.5]
\node (A) at (0,0) {$X_1$};
\node (A') at (0,-1) {$X_2$};
\node (C') at (2,-1) {$Y_2$};
\node (B) at (1,0) {$V_1$};
\node (C) at (2,0) {$Y_1$};
\node (D) at (1,-1) {$V_2$};
\path[->,font=\scriptsize,>=angle 90]
(A) edge node[above]{$i_1$} (B)
(C) edge node[above]{$o_1$} (B)
(A) edge node[left]{$h$} (A')
(C) edge node[right]{$h^\prime$} (C')
(C') edge node [above] {$o_2$} (D)
(A') edge node [above] {$i_2$} (D)
(B) edge node [left] {$g$} (D);
\end{tikzpicture}
\]
\[
\begin{tikzpicture}[scale=1.5]
\node (C) at (3,-0.5) {$k^+$};
\node (D) at (4,0) {$E_1$};
\node (D') at (4,-1) {$E_2$};
\node (A) at (5,0) {$E_1$};
\node (A') at (5,-1) {$E_2$};
\node (E) at (6,0) {$V_1$};
\node (E') at (6,-1) {$V_2$};
\node (B) at (7,0) {$E_1$};
\node (B') at (7,-1) {$E_2$};
\node (F) at (8,0) {$V_1$};
\node (F') at (8,-1) {$V_2$};
\path[->,font=\scriptsize,>=angle 90]
(E) edge node [right] {$g$} (E')
(F) edge node [right] {$g$} (F')
(D) edge node[above]{$r_1$} (C)
(A) edge node[above]{$s_1$}(E)
(A') edge node [below] {$s_2$} (E')
(D') edge node[below]{$r_2$} (C)
(D) edge node [left] {$f$} (D')
(A) edge node [left] {$f$} (A')
(B) edge node [left] {$f$} (B')
(B) edge node[above]{$t_1$}(F)
(B') edge node [below] {$t_2$} (F');
\end{tikzpicture}
\]
}
\end{enumerate}
We can also obtain a similar double category using decorated cospans: define a pseudofunctor $F \colon \mathsf{FinSet} \to \mathbf{Cat}$ that assigns to a finite set $V$ the category of all $k$-graph structures on the set $V$ and to a function $f \colon V \to V^\prime$ the corresponding functor $F(f) \colon F(V) \to F(V^\prime)$ between decoration categories. Both categories $\mathsf{FinSet}$ and $\mathbf{Cat}$ are symmetric monoidal and the pseudofunctor $F \colon \mathsf{FinSet} \to \mathbf{Cat}$ is symmetric lax monoidal, as given a $k$-graph structure on a finite set $V_1$ denoted by an element $K_1 \in F(V_1)$ and a $k$-graph structure on a finite set $V_2$ denoted by an element $K_2 \in F(V_2)$, there is a natural $k$-graph structure $\phi_{V_1,V_2}(K_1,K_2)$ on $V_1+V_2$. Thus we get a natural transformation $$\phi_{V_1,V_2} \colon F(V_1) \times F(V_2) \to F(V_1+V_2)$$ as well as a morphism $\phi \colon 1 \to F(\emptyset)$ which together satisfy the coherence conditions of a monoidal functor. The braiding is also clear as the following diagram commutes:
\[
\begin{tikzpicture}[scale=1.5]
\node (E) at (3,0) {$F(V_1) \times F(V_2)$};
\node (G) at (5,0) {$F(V_2) \times F(V_1)$};
\node (E') at (3,-1) {$F(V_1+V_2)$};
\node (G') at (5,-1) {$F(V_2+V_1)$};
\path[->,font=\scriptsize,>=angle 90]
(E) edge node[left]{$\phi_{V_1,V_2}$} (E')
(G) edge node[right]{$\phi_{V_2,V_1}$} (G')
(E) edge node[above]{$\beta'_{V_1,V_2}$} (G)
(E') edge node[above]{$F(\beta_{V_1,V_2})$} (G');
\end{tikzpicture}
\]
Thus the pseudofunctor $F$ is symmetric lax monoidal and so by Theorem \ref{DC} we have the following:

\begin{theorem}\label{dccircs}
Let $F \colon \mathsf{FinSet} \to \mathbf{Cat}$ be the symmetric lax monoidal pseudofunctor which assigns to a finite set $N$ the category of all $k$-graph structures whose underlying set of vertices is $N$. Then there exists a symmetric monoidal double category $F\mathbb{C}\mathbf{sp}$ which has:
\begin{enumerate}
\item{objects as finite sets,}
\item{vertical 1-morphisms as functions,}
\item{horizontal 1-cells as cospans of sets together with the structure of a $k$-graph given by an element of the image of the apex under the pseudofunctor $F$:
\[
\begin{tikzpicture}[scale=1.5]
\node (D) at (-3,0) {$U$};
\node (E) at (-2,0) {$V$};
\node (F) at (-1,0) {$W$};
\node (A) at (0,0) {$K \in F(V)$};
\path[->,font=\scriptsize,>=angle 90]
(D) edge node [above] {$i$} (E)
(F) edge node [above] {$o$} (E);
\end{tikzpicture}
\]
which are open $k$-graphs, and}
\item{2-morphisms as maps of cospans of finite sets 
\[
\begin{tikzpicture}[scale=1.5]
\node (E) at (3,0) {$U_1$};
\node (F) at (5,0) {$W_1$};
\node (G) at (4,0) {$V_1$};
\node (E') at (3,-1) {$U_2$};
\node (F') at (5,-1) {$W_2$};
\node (G') at (4,-1) {$V_2$};
\node (A) at (6,0) {$K_1 \in F(V_1)$};
\node (B) at (6,-1) {$K_2 \in F(V_2)$};
\path[->,font=\scriptsize,>=angle 90]
(F) edge node[above]{$o_1$} (G)
(E) edge node[left]{$f$} (E')
(F) edge node[right]{$g$} (F')
(G) edge node[left]{$h$} (G')
(E) edge node[above]{$i_1$} (G)
(E') edge node[above]{$i_2$} (G')
(F') edge node[above]{$o_2$} (G');
\end{tikzpicture}
\]
together with a morphism of $k$-graphs $\iota \colon F(h)(K_1) \to K_2$ in $F(V_2)$.}
\end{enumerate}

\end{theorem}

\begin{corollary}
The symmetric monoidal double category $_L \lC\mathbf{sp}(\mathsf{FinGraph_k})$ of Theorem \ref{left_adj_smdc} and the symmetric monoidal double category $F\mathbb{C}\mathbf{sp}$ of Theorem \ref{dccircs} are equivalent.
\end{corollary}

\begin{proof}
This follows immediately from Theorem \ref{Equiv}.
\end{proof}

\subsection{Petri nets}

In a previous work, Baez and Master used the framework of structured cospans to obtain a symmetric monoidal double category of `open Petri nets' \cite{BM}. Recall from Definition \ref{definition:Petri_net} that a Petri net is given by a diagram in $\mathsf{Set}$ of the form:
\[
\begin{tikzpicture}[scale=1.5]
\node (B) at (1,0) {$T$};
\node (C) at (2,0) {$\mathbb{N}[S].$};
\path[->,font=\scriptsize,>=angle 90]
(B)edge[bend left] node[above]{$s$}(C)
(B)edge[bend right] node[below]{$t$}(C);
\end{tikzpicture}
\]
Here, $T$ is the finite set of \define{transitions} and $S$ is the finite set of \define{species}, and $\mathbb{N}[S]$ is the free commutative monoid on the set $S$. Each transition then has a formal linear combination of species given by an element of $\mathbb{N}[S]$ as its source and target as prescribed by the functions $s$ and $t$, respectively. An example of a Petri net is given by:
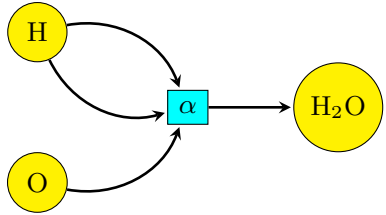
\[
\begin{tikzpicture}
	\begin{pgfonlayer}{nodelayer}
		\node [style=species] (I) at (0,1) {H};
		\node [style=species] (T) at (0,-1) {O};
		\node [style=transition] (W) at (2,0) {$\alpha$};
		\node [style=species] (Water) at (4,0) {$\textnormal{H}_2$O};
	\end{pgfonlayer}
	\begin{pgfonlayer}{edgelayer}
		\draw [style=inarrow, bend right=40, looseness=1.00] (I) to (W);
		\draw [style=inarrow, bend left=40, looseness=1.00] (I) to (W);
		\draw [style=inarrow, bend right=40, looseness=1.00] (T) to (W);
		\draw [style=inarrow] (W) to (Water);
	\end{pgfonlayer}
\end{tikzpicture}
\]
This Petri net has a single transition $\alpha$ with $2\textnormal{H}+\textnormal{O}$ as its source and $\textnormal{H}_2 \textnormal{O}$ as its target. See the original paper for more details on Petri nets \cite{BM}.

Each set of species $S$ gives rise to a discrete Petri net $L(S)$ with $S$ as its set of species and no transitions. Baez and Master note the existence of a left adjoint $L \colon \mathsf{Set} \to \mathsf{Petri}$ where $\mathsf{Petri}$ is the category whose objects are Petri nets and whose morphisms are `morphisms of Petri nets'. They also show that $\mathsf{Petri}$ has finite colimits and thus using Theorem \ref{SC} obtain a symmetric monoidal double category $\mathbb{O}\mathbf{pen}(\mathsf{Petri})$ of open Petri nets which has:
\begin{enumerate}
\item{objects given by sets,}
\item{vertical 1-morphisms given by functions,}
\item{horizontal 1-cells as open Petri nets which are given by cospans in $\mathsf{Petri}$ of the form:
\[
\begin{tikzpicture}[scale=1.5]
\node (D) at (-3,0) {$L(X)$};
\node (E) at (-2,0) {$P$};
\node (F) at (-1,0) {$L(Y)$};
\path[->,font=\scriptsize,>=angle 90]
(D) edge node [above] {$I$} (E)
(F) edge node [above] {$O$} (E);
\end{tikzpicture}
\]
and}
\item{2-morphisms as maps of cospans in $\mathsf{Petri}$ of the form:
\[
\begin{tikzpicture}[scale=1.5]
\node (E) at (3,0) {$L(X_1)$};
\node (F) at (5,0) {$L(Y_1)$};
\node (G) at (4,0) {$P_1$};
\node (E') at (3,-1) {$L(X_2)$};
\node (F') at (5,-1) {$L(Y_2)$};
\node (G') at (4,-1) {$P_2$};
\path[->,font=\scriptsize,>=angle 90]
(F) edge node[above]{$O_1$} (G)
(E) edge node[left]{$L(f)$} (E')
(F) edge node[right]{$L(g)$} (F')
(G) edge node[left]{$\alpha$} (G')
(E) edge node[above]{$I_1$} (G)
(E') edge node[above]{$I_2$} (G')
(F') edge node[above]{$O_2$} (G');
\end{tikzpicture}
\]
}
\end{enumerate}
We can also obtain a similar double category using decorated cospans: define a pseudofunctor $F \colon \mathsf{Set} \to \mathbf{Cat}$ where given a set $s$, $F(s)$ is the category of all Petri net structures with $s$ as its set of species. This pseudofunctor $F$ is symmetric lax monoidal as both $(\mathsf{Set},+,\emptyset)$ and $(\mathbf{Cat},\times,1)$ are symmetric monoidal and given Petri nets $P \in F(s)$ and $P^\prime \in F(s^\prime)$, we can place them side by side and consider them together as a single Petri net $P+P^\prime \in F(s+s^\prime)$ with set of species $s+s^\prime$, and thus we have natural transformations $\phi_{s,s^\prime} \colon F(s) \times F(s^\prime) \to F(s+s^\prime)$ for any two sets $s$ and $s^\prime$. The other structure morphism between monoidal units $\phi \colon 1 \to F(\emptyset)$ is defined by the unique morphism from the terminal category to the empty Petri net with the empty set for its set of species, which is the only possible Petri net on the empty set. All of the diagrams that are required to commute are straightforward. Appealing to Theorem \ref{DC}, we have the following:

\begin{theorem}\label{dcpetri}
Let $F \colon \mathsf{Set} \to \mathbf{Cat}$ be the symmetric lax monoidal pseudofunctor which assigns to a set $S$ the category of all Petri nets whose set of species is $S$. Then there exists a symmetric monoidal double category $F\mathbb{C}\mathbf{sp}$ which has:
\begin{enumerate}
\item{objects given by sets,}
\item{vertical 1-morphisms given by functions,}
\item{horizontal 1-cells given by open Petri nets presented as pairs:
\[
\begin{tikzpicture}[scale=1.5]
\node (D) at (-3,0) {$X$};
\node (E) at (-2,0) {$Z$};
\node (F) at (-1,0) {$Y$};
\node (A) at (0,0) {$P \in F(Z)$};
\path[->,font=\scriptsize,>=angle 90]
(D) edge node [above] {$i$} (E)
(F) edge node [above] {$o$} (E);
\end{tikzpicture}
\]
and}
\item{2-morphisms as maps of cospans in $\mathsf{Set}$:
\[
\begin{tikzpicture}[scale=1.5]
\node (E) at (3,0) {$X_1$};
\node (F) at (5,0) {$Y_1$};
\node (G) at (4,0) {$Z_1$};
\node (E') at (3,-1) {$X_2$};
\node (F') at (5,-1) {$Y_2$};
\node (G') at (4,-1) {$Z_2$};
\node (A) at (6,0) {$P_1 \in F(Z_1)$};
\node (B) at (6,-1) {$P_2 \in F(Z_2)$};
\path[->,font=\scriptsize,>=angle 90]
(F) edge node[above]{$o_1$} (G)
(E) edge node[left]{$f$} (E')
(F) edge node[right]{$g$} (F')
(G) edge node[left]{$h$} (G')
(E) edge node[above]{$i_1$} (G)
(E') edge node[above]{$i_2$} (G')
(F') edge node[above]{$o_2$} (G');
\end{tikzpicture}
\]
together with a morphism of Petri nets $\iota \colon F(h)(P_1) \to P_2$ in $F(Z_2)$.}
\end{enumerate}
\end{theorem}

Thus we have a symmetric monoidal double category $\mathbb{O}\mathbf{pen}(\mathsf{Petri})$ of open Petri nets obtained from structured cospans and a symmetric monoidal double category $F\mathbb{C}\mathbf{sp}$ of open Petri nets obtain from decorated cospans, and these two symmetric monoidal double categories are equivalent.

\begin{corollary}
The symmetric monoidal double category $\mathbb{O}\mathbf{pen}(\mathsf{Petri})$ constructed by Baez and Master \cite{BM} utilizing structures cospans and the symmetric monoidal double category $F\mathbb{C}\mathbf{sp}$ of Theorem \ref{dcpetri} are equivalent.
\end{corollary}

\begin{proof}
This follows immediately from Theorem \ref{Equiv}.
\end{proof}

We may also construct a symmetric monoidal double category of open Petri nets with rates using decorated cospans, and this is equivalent to the symmetric monoidal double category $_L \mathbb{C}\mathbf{sp}(\mathsf{Petri}_\textnormal{rates})$ of Theorem \ref{scpetri}.

$\textnormal{ }$
}

{\ssp
\chapter{A brief digression to bicategories}\label{Chapter5}
If one prefers bicategories to double categories, one will be happy to learn that all of the main results in this thesis on double categories have bicategorical analogues thanks to a result of Mike Shulman \cite{Shul}. Bicategories are defined in Section \ref{bicat_definitions} of the \ref{Appendix}ppendix. First we discuss the relationship between 2-categories and double categories. As we are mainly interested in symmetric monoidal double categories, we are similarly primarily interested in `symmetric monoidal bicategories'. We will not define monoidal, braided monoidal, `sylleptic' monoidal or symmetric monoidal bicategories here. These definitions can be found in a work of Mike Stay \cite{Stay}.

The first thing we point out is that 2-categories are just a special case of strict double categories and that every strict double category has at least two canonical underlying 2-categories. Given a strict double category $\mathbb{C}$, there exists:
\begin{enumerate}
\item{a 2-category $\mathbf{H}(\mathbb{C})$ called the \define{horizontal 2-category} of $\mathbb{C}$ which has:
\begin{enumerate}
\item objects as objects of $\mathbb{C}$,
\item morphisms as horizontal 1-cells of $\mathbb{C}$, and
\item 2-morphisms as 2-morphisms of $\mathbb{C}$ with identity vertical 1-morphisms, also known as \define{globular 2-morphisms} of $\mathbb{C}$.
\end{enumerate}}
\item{a 2-category $\mathbf{V}(\mathbb{C})$ called the vertical 2-category of $\mathbb{C}$ which has:
\begin{enumerate}
\item objects as objects of $\mathbb{C}$,
\item morphisms as vertical 1-morphisms of $\mathbb{C}$, and
\item 2-morphisms as 2-morphisms of $\mathbb{C}$ with identity horizontal 1-cells, where now composition of 2-morphisms is given by horizontal composition of 2-morphisms in $\mathbb{C}$.
\end{enumerate}}
\end{enumerate}

Every \emph{pseudo} double category $\mathbb{C}$ has an underlying bicategory $\mathbf{H}(\mathbb{C})$ given by as above. Using our conventions, there is no underlying vertical bicategory $\mathbf{V}(\mathbb{C})$ as restricting the horizontal source and target of 2-morphisms, namely the horizontal 1-cells, to be identities does not force the horizontal source and target of the composite 2-morphisms in $\mathbb{C}$ to also be identities, due to the composition of horizontal 1-cells in a pseudo double category being neither strictly unital nor associative. 

Sometimes when the pseudo double category $\mathbb{C}$ is symmetric monoidal, the symmetric monoidal structure can be lifted to the horizontal bicategory $\mathbf{H}(\mathbb{C})$. This is due to the following result of Shulman \cite{Shul}. The definitions of `isofibrant' and `symmetric monoidal double category' are given in Definitions \ref{def:isofibrant} and \ref{defn:monoidal_double_category}, respectively. 

\begin{theorem}[{{\cite[Thm.\ 1.2]{Shul}}}]\label{Shulman1}
Let $\mathbb{X}$ be an isofibrant symmetric monoidal pseudo double category. Then the horizontal bicategory $\mathbf{H}(\mathbb{X})$ of $\mathbb{X}$ is a symmetric monoidal bicategory which has:
\begin{enumerate}
\item{objects as those of $\mathbb{X}$,}
\item{morphisms as horizontal 1-cells of $\mathbb{X}$, and}
\item{2-morphisms as globular 2-morphisms of $\mathbb{X}$.}
\end{enumerate}
\end{theorem}

The property of being isofibrant, meaning fibrant on vertical 1-isomorphisms, is precisely what allows the horizontal bicategory $\mathbf{H}(\lX)$ to inherit the portion of the symmetric monoidal structure that resides in the category of objects of $\lX$, namely, the associators, left and right unitors and braidings.

In the previous chapters we constructed various symmetric monoidal double categories which are in fact isofibrant, and thus have underlying symmetric monoidal bicategories.

\section{Foot-replaced bicategories}

Every foot-replaced double category $_L \lX$ has an underlying foot-replaced bicategory $\mathbf{H}( _L \lX)$ given by taking the 2-morphisms of $\mathbf{H}( _L \lX)$ to be globular 2-morphisms of $_L \lX$.

\begin{lemma}
Given a double category $\lX$, a category $\A$ and a functor $L \maps \A \to \lX_0$, there is a bicategory $\mathbf{H}(_{L} \lX)$ for which:
\begin{itemize}
\item objects are objects of $\A$,
\item morphisms from $a \in \A$ to $a' \in \A$ are horizontal 1-cells
$M \maps L(a) \to L(a')$ of $_L \lX$,
\item 2-morphisms are  globular 2-morphisms of $_L \lX$,
\item composition of morphisms is horizontal composition of horizontal 1-cells in $_L \lX$,
\item vertical and horizontal composition of 2-morphisms is vertical and horizontal
composition of 2-cells in $_L \lX$.
\end{itemize}
\end{lemma}

If the double category $\lX$ is isofibrant symmetric monoidal and we have a strong symmetric monoidal functor $L \colon \A \to \lX_0$, then Shulman's Theorem \ref{Shulman1} allows us to lift the monoidal structure of the foot-replaced double category $_L \lX$ to obtain a symmetric monoidal foot-replaced bicategory $\mathbf{H}(_L \lX)$.

\begin{lemma} \label{trick3}
If $\lX$ is an isofibrant symmetric monoidal double category, $\A$ is a symmetric monoidal
category and $L \maps \A \to \lX_0$ is a (strong) symmetric monoidal functor, then the bicategory $\mathbf{H}(_L \lX)$ becomes symmetric monoidal in a canonical way.
\end{lemma}

\begin{lemma}\label{lemma:isofibrant}
If $\mathsf{X}$ is a category with finite colimits, then the symmetric monoidal double category $\mathbb{C}\mathbf{sp}(\mathsf{X})$ is isofibrant.
\end{lemma}
\begin{proof}
A vertical 1-isomorphism in $\lCsp(\X)$ is a isomorphism $f \maps x \to y$ in $\X$.  We take its companion $\hat{f}$ to be the cospan  
\[   

\]
The unit horizontal 1-cells $U_x$ and $U_y$ are given respectively by
\[
%
\]
and the accompanying 2-morphisms are given by
\[
%
\]
respectively.  An easy calculation verifies Equation\ (\ref{eq:CompanionEq}).
\end{proof}

\begin{theorem}\label{Cspbicat}
Let $L \maps \A \to \X$ be a functor where $\X$ is a category with pushouts.  Then there is
a bicategory $\mathbf{H}(_L \lCsp(\X))$ for which:
\begin{enumerate}
\item an object is an object of $\A$,
\item a morphism from $a$ to $b$ is given by a cospan in $\X$ of the form:
\[
%
\]
with composition the same as composition of horizontal 1-cells in Theorem \ref{_L Csp(X)} and
\item{2-morphisms are given by maps of cospans which are commutative diagrams of the form:
\[
%
\]
with horizontal and vertical composition of 2-morphisms given by horizontal and vertical composition of globular 2-morphisms in Theorem \ref{_L Csp(X)}.}
\end{enumerate}
\end{theorem}

\begin{theorem}\label{scbicat}
Let $L \maps \A \to \X$ be a functor preserving finite coproducts, where $\A$ has finite coproducts and $\X$ has finite colimits.  Then the bicategory of Theorem \ref{Cspbicat} is symmetric monoidal with the monoidal structure given by:
\begin{enumerate}
\item{the tensor product of two objects $a_1$ and $a_2$ is $a_1 + a_2$,}
\item{the tensor product of two morphisms is given by the tensor product of two horizontal 1-cells in Theorem \ref{SC} and}
\item{the tensor product of two 2-morphisms is given by:
\[
\begin{tikzpicture}[scale=1.2]
\node (A) at (0,0) {$L(a_1)$};
\node (B) at (1,1) {$x_1$};
\node (B') at (1,-1) {$x_1'$};
\node (F') at (4,-1) {$x_2'$};
\node (J') at (8.5,-1) {$x_1' + x_2'$};
\node (C) at (2,0) {$L(a_1')$};
\node (D) at (2.5,0) {$\otimes$};
\node (E) at (3,0) {$L(a_2)$};
\node (F) at (4,1) {$x_2$};
\node (G) at (5,0) {$L(a_2')$};
\node (H) at (5.625,0) {$=$};
\node (I) at (6.5,0) {$L(a_1 + a_2)$};
\node (J) at (8.5,1) {$x_1 + x_2$};
\node (K) at (10.5,0) {$L(a_1' + a_2')$};
\path[->,font=\scriptsize,>=angle 90]
(A) edge node[above,left]{$i_1$} (B)
(C)edge node[above,right]{$o_1$}(B)
(A) edge node[below,left]{$i_1'$} (B')
(C)edge node[below,right]{$o_1'$}(B')
(E) edge node[above,left]{$i_2$} (F)
(G)edge node[above,right]{$o_2$}(F)
(E) edge node[below,left]{$i_2'$} (F')
(G)edge node[below,right]{$o_2'$}(F')
(I) edge node [above,left] {$(i_1 + i_2)\phi^{-1}$} (J)
(K) edge node [above,right] {$(o_1 + o_2)\phi^{-1}$} (J)
(I) edge node [below,left] {$(i_1' + i_2')\phi^{-1}$} (J')
(K) edge node [below,right] {$(o_1' + o_2')\phi^{-1}$} (J')
(B) edge node [left] {$\alpha_1$} (B')
(F) edge node [left] {$\alpha_2$} (F')
(J) edge node [left] {$\alpha_1 + \alpha_2$} (J');
\end{tikzpicture}
\]}
\end{enumerate}
where $\phi$ is the natural isomorphism $\phi_{a_1,a_2} \colon L(a_1) \otimes L(a_2) \to L(a_1+a_2)$ of the strong symmetric monoidal functor $L$. The unit for the tensor product is the initial object of $\A$, and the symmetry for any two objects $a$ and $b$ is defined using the canonical isomorphism $a + b \cong b + a$.
\end{theorem}

\subsection{Graphs}

In Section \ref{scgraphs}, we constructed a symmetric monoidal double category $ _L \lCsp(\FinGraph)$ of open graphs. This double category is isofibrant by Lemma \ref{lemma:isofibrant}, and so we may extract from it a symmetric monoidal bicategory in which open graphs appear as morphisms.

\begin{theorem}
There exists a symmetric monoidal bicategory $\mathbf{OpenFinGraph}=\mathbf{H}( _L \lCsp(\FinGraph))$ which has:
\begin{enumerate}
\item{finite sets as objects,}
\item{open graphs: that is, cospans of graphs of the form
\[

\]
as morphisms, and}
\item{maps of cospans of graphs as 2-morphisms, as in the following commutative diagram:
\[
%
\]
}
\end{enumerate} 
\end{theorem}
We can then decategorify this symmetric monoidal bicategory $\mathbf{OpenFinGraph}$ to obtain a symmetric monoidal category $D(\mathbf{OpenFinGraph})$ which has:
\begin{enumerate}
\item{finite sets as objects, and}
\item{isomorphism classes of open graphs
\[
%
\]
as morphisms, where two open graphs are isomorphic if the following diagram commutes:
\[
%
\]
}
\end{enumerate}
Here, the graph isomorphism $h \maps x \to y$ is really a pair of bijections $f \maps N \to N'$ and $g \maps E \to E'$ between the vertex and edge sets of the graphs $x$ and $y$ that make the following diagram commute:
\[
%
\]

\subsection{Electrical circuits}
In Section \ref{SecSCApp}, we constructed a symmetric monoidal double category of open $k$-graphs. This symmetric monoidal double category is in fact isofibrant by Lemma \ref{lemma:isofibrant}, so we can apply Theorem \ref{Shulman1} to obtain a symmetric monoidal bicategory:
\begin{theorem}\label{thm:openfingraph_k}
There exists a symmetric monoidal bicategory $\mathbf{OpenFinGraph}_k= \mathbf{H}( {_L \lCsp(\mathsf{FinGraph}_k)})$ where:
\begin{enumerate}
\item{objects are finite sets,}
\item{morphisms are \define{open $k$-graphs}:
\[
\begin{tikzpicture}[scale=1.5]
\node (A) at (0,0) {$L(a)$};
\node (B) at (1,0) {$N$};
\node (C) at (2,0) {$L(b)$};
\path[->,font=\scriptsize,>=angle 90]
(A) edge node[above]{$i$} (B)
(C)edge node[above]{$o$}(B);
\end{tikzpicture}
\]
\[
\begin{tikzpicture}[scale=1.5]
\node (A) at (0,0) {$k^+$};
\node (B) at (1,0) {$E$};
\node (C) at (2,0) {$N$};
\path[->,font=\scriptsize,>=angle 90]
(B) edge node[above]{$r$} (A)
(B)edge[bend left] node[above]{$s$}(C)
(B)edge[bend right] node[below]{$t$}(C);
\end{tikzpicture}
\]
which are open graphs where the apex of the cospan representing the open graph is equipped with the structure of a $k$-graph, and}
\item{2-morphisms are \define{maps of open $k$-graphs}, which are maps of cospans such that the following diagrams commute
\[
\begin{tikzpicture}[scale=1.5]
\node (A) at (0,-0.5) {$L(a)$};
\node (B) at (1,0.5) {$N$};
\node (C) at (2,-0.5) {$L(b)$};
\node (B') at (1,-1.5) {$N'$};
\path[->,font=\scriptsize,>=angle 90]
(A) edge node[above]{$i$} (B)
(C)edge node[above]{$o$}(B)
(A) edge node[below]{$i'$} (B')
(C)edge node[below]{$o'$}(B')
(B) edge node [left] {$f$} (B');
\end{tikzpicture}
\]
\[
\begin{tikzpicture}[scale=1.5]
\node (C) at (3,-0.5) {$k^+$};
\node (D) at (4,0) {$E$};
\node (D') at (4,-1) {$E'$};
\node (A) at (5,0) {$E$};
\node (A') at (5,-1) {$E'$};
\node (E) at (6,0) {$N$};
\node (E') at (6,-1) {$N'$};
\node (B) at (7,0) {$E$};
\node (B') at (7,-1) {$E'$};
\node (F) at (8,0) {$N$};
\node (F') at (8,-1) {$N'$};
\path[->,font=\scriptsize,>=angle 90]
(E) edge node [right] {$f$} (E')
(F) edge node [right] {$f$} (F')
(D) edge node[above]{$r$} (C)
(A) edge node[above]{$s$}(E)
(A') edge node [below] {$s'$} (E')
(D') edge node[below]{$r'$} (C)
(D) edge node [left] {$g$} (D')
(A) edge node [left] {$g$} (A')
(B) edge node [left] {$g$} (B')
(B) edge node[above]{$t$}(F)
(B') edge node [below] {$t'$} (F');
\end{tikzpicture}
\]
for some morphisms $f$ and $g$.
}
\end{enumerate}
\end{theorem}
\begin{proof}
We obtain $\mathbf{OpenFinGraph}_k= \mathbf{H}( _L \lC \mathbf{sp}(\mathsf{FinGraph}_k))$ by applying Theorems \ref{Cspbicat} and \ref{scbicat} to the functor $L \colon \mathsf{FinSet} \to \mathsf{FinGraph}_k$ of Theorem \ref{left_adj_smdc}.
\end{proof}

We can then decategorify this symmetric monoidal bicategory $\mathbf{OpenFinGraph}_k$ to obtain a symmetric monoidal category $D(\mathbf{OpenFinGraph}_k)$ where:
\begin{enumerate}
\item{objects are finite sets, and}
\item{morphisms are  isomorphism classes of open $k$-graphs, where two open $k$-graphs are in the same isomorphism class if the following diagrams commute:
\[

\]
for some isomorphisms $f$ and $g$.
}
\end{enumerate}
To make contact with Baez and Fong's original work on black-boxing electrical circuits \cite{BF}, recall that we have a monoidal category $F\Cospan$ obtained from the original incarnation of decorated cospans. The monoidal category $D(\mathbf{OpenFinGraph}_k)$ constructed in this section is not only symmetric, but also contains more isomorphisms. For example, consider the following two open $k$-graphs:
\[
%
\]
where $E \neq E'$ but there exists a bijection $g \colon E \xrightarrow{\sim} E'$ such that $s = s' \circ g$ and $t = t' \circ g$; this just says that the two networks look the same but have different edge labels. Then these two open $k$-graphs give different morphisms in $F\Cospan$, but the same morphism in $D(\mathbf{OpenFinGraph}_k)$.

We can define a functor $G \maps F\Cospan \to D(\mathbf{OpenFinGraph}_k)$ that is the identity on objects and that identifies open graphs that are isomorphic in the sense of (2) above. Then we can consider the following diagram:
\[
%
\]
Here the top functor $\blacksquare \maps F\Cospan \to {\LagRel}_k$ is the original black-boxing functor constructed by Baez and Fong \cite{BF}.
While we shall not prove it here, one can extend this functor to a new one, also called $\blacksquare$, defined on $D(\mathbf{OpenFinGraph}_k)$. This also promotes the original black-box functor from a mere monoidal functor to a symmetric monoidal functor $\blacksquare \colon D(\mathbf{OpenFinGraph}_k) \to \LagRel_k$.

\subsection{Petri nets}

In Section \ref{petri_nets_structured_cospans}, we constructed a symmetric monoidal double category of open Petri nets with rates. This symmetric monoidal double category is also isofibrant by Lemma \ref{lemma:isofibrant}, so we can apply Theorem \ref{Shulman1} to obtain a symmetric monoidal bicategory:


\begin{theorem}\label{thm:petri_rates}
There exists a symmetric monoidal bicategory $\mathbf{Petri}_\mathrm{rates} = \mathbf{H}({ _L \lCsp(\Petri_{\mathrm{rates}})})$ where:
\begin{enumerate}
\item{objects are finite sets,}
\item{morphisms are \define{open Petri nets with rates}:
\[

\]
which are cospans of Petri nets whose apices are equipped with a function $r \colon T \to [0,\infty)$ assigning a rate $r(t)$ to every transition $t \in T$, and}
\item{2-morphisms are \define{maps of open Petri nets with rates}, which are maps of open Petri nets such that the following diagrams commute:
\[
%
\]
for some morphisms $f$ and $g$.
}
\end{enumerate}
\end{theorem}
\begin{proof}
We obtain $\mathbf{Petri}_{\mathrm{rates}} = \mathbf{H}( _L \lCsp(\Petri_{\mathrm{rates}}))$ by applying Theorems \ref{Cspbicat} and \ref{scbicat} to the functor $L \colon \mathsf{FinSet} \to \mathsf{Petri}_\mathrm{rates}$ of Theorem \ref{scpetri}.
\end{proof}
Once again, we can then decategorify this bicategory $\mathbf{Petri}_\mathrm{rates}$ to obtain a symmetric monoidal category $D(\mathbf{Petri_\mathrm{rates}})$ where:
\begin{enumerate}
\item{objects are finite sets, and}
\item{morphisms are isomorphism classes of open Petri nets with rates, where two open Petri nets with rates are in the same isomorphism class if the following diagrams commute:
\[
\begin{tikzpicture}[scale=1.5]
\node (G) at (-3,-0.5) {$L(a)$};
\node (H) at (-2,0.5) {$S$};
\node (I) at (-2,-1.5) {$S'$};
\node (J) at (-1,-0.5) {$L(b)$};
\path[->,font=\scriptsize,>=angle 90]
(G) edge node [above] {$i$} (H)
(G) edge node [below] {$i'$} (I)
(H) edge node [left] {$f$} (I)
(J) edge node [above] {$o$} (H)
(J) edge node [below] {$o'$} (I);
\end{tikzpicture}
\]
\[
\begin{tikzpicture}[scale=1.5]
\node (C) at (3,-0.5) {$[0,\infty)$};
\node (D) at (4,0) {$T$};
\node (D') at (4,-1) {$T'$};
\node (A) at (5,0) {$T$};
\node (A') at (5,-1) {$T'$};
\node (E) at (6,0) {$\mathbb{N}[S]$};
\node (E') at (6,-1) {$\mathbb{N}[S]$};
\node (B) at (7,0) {$T$};
\node (B') at (7,-1) {$T'$};
\node (F) at (8,0) {$\mathbb{N}[S]$};
\node (F') at (8,-1) {$\mathbb{N}[S]$};
\path[->,font=\scriptsize,>=angle 90]
(E) edge node [right] {$\mathbb{N}[f]$} (E')
(F) edge node [right] {$\mathbb{N}[f]$} (F')
(D) edge node[above]{$r$} (C)
(A) edge node[above]{$s$}(E)
(A') edge node [below] {$s'$} (E')
(D') edge node[below]{$r'$} (C)
(D) edge node [left] {$g$} (D')
(A) edge node [left] {$g$} (A')
(B) edge node [left] {$g$} (B')
(B) edge node[above]{$t$}(F)
(B') edge node [below] {$t'$} (F');
\end{tikzpicture}
\]
for some isomorphisms $f$ and $g$.}
\end{enumerate}
We can define a functor $G \maps \mathsf{Petri}_{\mathrm{rates}} \to D(\mathbf{Petri}_\mathrm{rates})$ that is the identity on objects and identifies morphisms in $\mathsf{Petri}_{\mathrm{rates}}$, if they are in the same isomorphism class in the sense of (2) above.  We can then consider the following diagram:
\[
\begin{tikzpicture}[scale=1.5]
\node (A) at (0,0) {$\Petri_{\mathrm{rates}}$};
\node (B) at (0,-1) {$D(\mathbf{Petri}_\mathrm{rates})$};
\node (C) at (2,0) {$\SemiAlg\Rel$};
\path[->,font=\scriptsize,>=angle 90]
(A) edge node[left]{$G$} (B)
(A)edge node[above]{$\blacksquare$}(C)
(B) edge [dashed] node [below] {$\blacksquare$}(C);
\end{tikzpicture}
\]
Here the top functor $\blacksquare \maps \mathsf{Petri}_{\mathrm{rates}} \to \SemiAlg\Rel$ was constructed by Baez and Pollard \cite{BP}.   While we shall not prove it here, one can extend this functor to a new one, also called $\blacksquare$, defined on $D(\mathbf{Petri}_\mathrm{rates})$.

\subsection{Maps of foot-replaced bicategories}

A result of Hansen and Shulman \cite{Shul3} not only allows us to lift symmetric monoidal double categories to their underlying symmetric monoidal horizontal-edge bicategories, but also maps between such.

\begin{corollary}
Given two symmetric monoidal foot-replaced double categories $_L \lX$ and $_{L'} \lX'$ and a symmetric monoidal double functor $_F \mathbb{F} \colon _L \lX \to _{L'} \lX'$ between the two, the symmetric monoidal double functor $_F \mathbb{F}$ induces a functor of symmetric monoidal bicategories between the underlying horizontal-edge bicategories of the foot-replaced double categories $_L \lX$ and $_{L'} \lX'$. $$\mathbf{H}({}_F \mathbb{F}) \colon \mathbf{H}( {}_L \lX) \to \mathbf{H}( {}_{L'} \lX')$$
\end{corollary}

\begin{proof}
This follows immediately from the work of Hansen and Shulman \cite{Shul3}.
\end{proof}

\section{Decorated cospan bicategories}
\begin{lemma}\label{DCFibrant}
The double category $F\mathbb{C}\mathbf{sp}$ constructed in Theorem \ref{DCord} is fibrant.
\end{lemma}

\begin{proof}
Let $f \colon c \to c^\prime$ be a vertical 1-morphism in $F\mathbb{C}\mathbf{sp}$. We can lift $f$ to the companion horizontal 1-cell $\hat{f}$:
\[

\]
and then obtain the following two 2-morphisms:
\[
%
\]
which satisfy the equations:
\[
%
\]
The right hand sides of the above two equations are given respectively by the 2-morphisms $U_f$ and $1_{\hat{f}}$. The conjoint of $f$ is given by the $F$-decorated cospan $\check{f}$ which is just the opposite of the companion above:
\[
%
\]
\end{proof}

\begin{corollary}
Let $(\mathsf{C},+,0)$ be a category with finite colimts and $F \colon \mathsf{C} \to \mathbf{Cat}$ a symmetric lax monoidal pseudofunctor. Then there exists a symmetric monoidal bicategory $F \mathbf{Csp} \colon = \mathbf{H}(F\mathbb{C}\mathbf{sp})$ which has:
\begin{enumerate}
\item{objects as those of $\mathsf{A}$,}
\item{morphisms as $F$-decorated cospans:
\[
%
\]
and}
\item{2-morphisms as maps of cospans in $\mathsf{A}$ of the form:
\[
%
\]
together with a morphism $\iota \colon F(h)(d) \to d^\prime$ in $F(c^\prime)$.}
\end{enumerate}
\end{corollary}

\begin{proof}
This follows immediately from Shulman's Theorem \ref{Shulman1} above applied to the fibrant symmetric monoidal double category $F\mathbb{C}\mathbf{sp}$.
\end{proof}

This symmetric monoidal bicategory $F\mathbf{Csp}$ is a superior version of the symmetric monoidal bicategory $F\mathbf{Cospan}(\mathsf{A})$ constructed earlier in a previous work \cite{Cour}, in that there is greater flexibility in what 2-morphisms are allowed. 

\subsection{Maps of decorated cospan bicategories}

Just as a result of Hansen and Shulman \cite{Shul3} allows us to lift maps of symmetric monoidal foot-replace double categories to maps between their underlying horizontal-edge bicategories, we can also lift maps between symmetric monoidal decorated cospan double categories to maps between their underlying horizontal-edge bicategories.

\begin{corollary}
Given two symmetric monoidal decorated cospan double categories $F \mathbb{C}\mathbf{sp}$ and $F' \mathbb{C}\mathbf{sp}$ and a symmetric monoidal double functor $\mathbb{H} \colon F \mathbb{C}\mathbf{sp} \to F' \mathbb{C}\mathbf{sp}$ between the two, the symmetric monoidal double functor $\mathbb{H}$ induces a functor of symmetric monoidal bicategories between the underlying horizontal-edge bicategories of the decorated cospan double categories $F \mathbb{C}\mathbf{sp}$ and $F' \mathbb{C}\mathbf{sp}$. $$\mathbf{H}(\mathbb{H}) \colon \mathbf{H}(F\mathbb{C}\mathbf{sp}) \to \mathbf{H}(F'\mathbb{C}\mathbf{sp})$$
\end{corollary}

\begin{proof}
This follows immediately from the work of Hansen and Shulman \cite{Shul3}.
\end{proof}

\subsection{Decorated cospans revisited}
We can then decategorify the symmetric monoidal bicategory $F\mathbf{Csp}$ to obtain a symmetric monoidal category similar to the monoidal one obtained using Fong's result, but symmetric and with larger isomorphism classes of morphisms:

\begin{corollary}
Given a symmetric lax monoidal pseudofunctor $F \colon \mathsf{A} \to \mathbf{Cat}$ where $\mathsf{A}$ is a category with finite colimits whose monoidal structure is given by binary coproducts, there exists a symmetric monoidal category $F\textnormal{Csp} \colon = D(F\mathbf{Csp})$ where:
\begin{enumerate}
\item{objects are those of $\mathsf{A}$ and}
\item{morphisms are isomorphism classes of $F$-decorated cospans of $\mathsf{A}$, where an $F$-decorated cospan is given by a pair:
\[
\begin{tikzpicture}[scale=1.5]
\node (A) at (0,0) {$a$};
\node (B) at (1,0) {$c$};
\node (C) at (2,0) {$b$};
\node (D) at (4,0) {$d \in F(c)$};
\path[->,font=\scriptsize,>=angle 90]
(A) edge node[above]{$i$} (B)
(C) edge node[above]{$o$} (B);
\end{tikzpicture}
\]
and given another $F$-decorated cospan:
\[
\begin{tikzpicture}[scale=1.5]
\node (A) at (0,0) {$a$};
\node (B) at (1,0) {$c^\prime$};
\node (C) at (2,0) {$b$};
\node (D) at (4,0) {$d^\prime \in F(c^\prime)$};
\path[->,font=\scriptsize,>=angle 90]
(A) edge node[above]{$i^\prime$} (B)
(C) edge node[above]{$o^\prime$} (B);
\end{tikzpicture}
\]
these two $F$-decorated cospans are in the same isomorphism class if there exists an isomorphism $f \colon c \to c^\prime$ such that following diagram commutes:
\[
\begin{tikzpicture}[scale=1.5]
\node (A) at (0,0) {$a$};
\node (B') at (1,0.5) {$c$};
\node (B) at (1,-0.5) {$c^\prime$};
\node (C) at (2,0) {$b$};
\path[->,font=\scriptsize,>=angle 90]
(A) edge node[below]{$i^\prime$} (B)
(C) edge node[below]{$o^\prime$} (B)
(A) edge node[above]{$i$} (B')
(C) edge node[above]{$o$} (B')
(B') edge node[left]{$f$} (B);
\end{tikzpicture}
\]
and there exists an isomorphism $\iota \colon F(f)(d) \to d^\prime$ in $F(c^\prime)$.}
\end{enumerate}
\end{corollary}
In this symmetric monoidal category, isomorphism classes are as they should morally be, and the instance of two graphs having different but isomorphic edge sets does not prevent them from being in the same isomorphism class.

\section{A biequivalence of compositional frameworks}
In Chapter \ref{Chapter4}, it is mentioned that given a symmetric monoidal pseudofunctor $F \colon (\mathsf{A},+,0) \to (\mathbf{Cat},\times,1)$ such that $F$ factors as an ordinary pseudofunctor $F \to \mathbf{Rex} \hookrightarrow \mathbf{Cat}$, where $\mathbf{Rex}$ is the 2-category of finitely cocomplete categories, finite coproduct preserving functors and natural transformations, we can obtain a fully faithful left adjoint $L \colon (\mathsf{A},+,0) \to (\mathsf{X},+,0)$ where $(\mathsf{X},+,0) := (\int{F},+,0)$. Furthermore, the right adjoint $R \colon \mathsf{X} \to \mathsf{A}$ preserves finite colimits. From the pseudofunctor $F \colon \mathsf{A} \to \mathbf{Cat}$, we can obtain a symmetric monoidal double category of decorated cospans by Theorem \ref{DC}. From the left adjoint $L \colon \mathsf{A} \to \mathsf{X}$, we can obtain a symmetric monoidal double category of structured cospans by Theorem \ref{SC}. By Theorem \ref{Equiv}, we have an equivalence of symmetric monoidal double categories $F \mathbb{C}\mathbf{sp} \simeq { _L \mathbb{C}\mathbf{sp}(\mathsf{X})}$. In the previous sections of the present chapter, we proved that each of these symmetric monoidal double categories are fibrant and give rise to underlying symmetric monoidal bicategories $F\mathbf{Csp}$ and $\mathbf{H}( {_L \mathbb{C}\mathbf{sp}(\mathsf{X})})$, respectively, by Theorem \ref{Shulman1} due to Shulman. We can use another result due to Shulman \cite{Shul2} to lift the double equivalence of double categories to a biequivalence of bicategories.

\begin{proposition}[{{\cite[Prop.\ B.3]{Shul2}}}]
An equivalence of fibrant double categories induces a biequivalence of horizontal bicategories.
\end{proposition}
\begin{corollary}
The bicategories $F\mathbf{Csp}$ and $\mathbf{H}( {_L \mathbb{C}\mathbf{sp}(\mathsf{X})})$ are biequivalent.
\end{corollary}
Both the double equivalence and biequivalence are in fact isomorphisms. See \cite{BCV} for more details.

$\textnormal{ }$
}

{\ssp
\chapter{Coarse-graining open Markov processes}\label{Chapter6}

\section{Introduction}
A `Markov process' is a stochastic model describing a sequence of transitions between states in which the probability of a transition depends only on the current state.   The only Markov processes we consider here are continuous-time Markov chains with a finite set of states.   Such a Markov process can be drawn as a labeled graph:
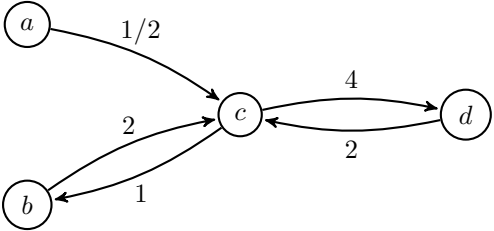
\[
\begin{tikzpicture}[->,>=stealth',shorten >=1pt,thick,scale=1]
  \node[main node] (1) at (0,2.2) {$\scriptstyle{a}$};
  \node[main node](2) at (0,-.2) {$\scriptstyle{b}$};
  \node[main node](3) at (2.83,1)  {$\scriptstyle{c}$};
  \node[main node](4) at (5.83,1) {$\scriptstyle{d}$};
  \path[every node/.style={font=\sffamily\small}, shorten >=1pt]
    (3) edge [bend left=12] node[above] {$4$} (4)
    (4) edge [bend left=12] node[below] {$2$} (3)
    (2) edge [bend left=12] node[above] {$2$} (3) 
    (3) edge [bend left=12] node[below] {$1$} (2)
    (1) edge [bend left=12] node[above] {$1/2$}(3);
\end{tikzpicture}
\]
In this example the set of states is $X = \{a,b,c,d\}$.   The numbers labeling edges
are transition rates, so the probability $\pi_i(t)$ of being in state $i \in X$ at time 
$t \in \R$ evolves according to a linear differential equation
\[    \frac{d}{dt}\, \pi_i(t) = \sum_{j \in X} H_{ij}\, \pi_j(t) \]
called the `master equation', where the matrix $H$ can be read off from the diagram:
\[ H=
\left[\begin{array}{rrrr}
    -1/2    & 0    & 0    & 0  \\
      0                   & -2   & 1   & 0 \\
     1/2   & 2     & -5  & 2 \\
      0                  & 0      & 4   & -2 \\
\end{array}\right].
\]
If there is an edge from a state $j$ to a distinct state $i$, the matrix entry $H_{ij}$ is 
the number labeling that edge, while if there is no such edge, $H_{ij} = 0$.  The diagonal
entries $H_{ii}$ are determined by the requirement that the sum of each column is zero.  This requirement says that the rate at which probability leaves a state equals the rate at which it
goes to other states.   As a consequence, the total probability is conserved:
\[    \frac{d}{dt} \sum_{i \in X} \pi_i(t) = 0 \]
and is typically set equal to $1$.

However, while this sum over all states is conserved, the same need not be true
for the sum of $\pi_i(t)$ over $i$ in a subset $Y \subset X$.   This poses 
a challenge to studying a Markov process as built from smaller parts: the parts are
not themselves Markov processes.  The solution is to describe them as `open' Markov 
processes.   These are a generalization in which probability 
can enter or leave from certain states designated as inputs and outputs:
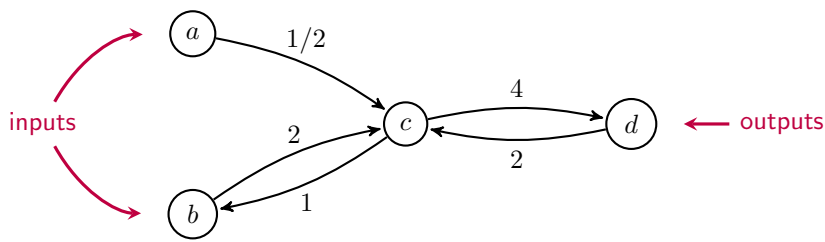
\[
\begin{tikzpicture}[->,>=stealth',shorten >=1pt,thick,scale=1]
  \node[main node] (1) at (0,2.2) {$\scriptstyle{a}$};
  \node[main node](2) at (0,-.2) {$\scriptstyle{b}$};
  \node[main node](3) at (2.83,1)  {$\scriptstyle{c}$};
  \node[main node](4) at (5.83,1) {$\scriptstyle{d}$};
\node(input) [color=purple] at (-2,1) {\small{\textsf{inputs}}};
\node(output) [color=purple] at (7.83,1) {\small{\textsf{outputs}}};
  \path[every node/.style={font=\sffamily\small}, shorten >=1pt]
    (3) edge [bend left=12] node[above] {$4$} (4)
    (4) edge [bend left=12] node[below] {$2$} (3)
    (2) edge [bend left=12] node[above] {$2$} (3) 
    (3) edge [bend left=12] node[below] {$1$} (2)
    (1) edge [bend left=12] node[above] {$1/2$}(3);
    
\path[color=purple, very thick, shorten >=10pt, ->, >=stealth] (output) edge (4);
\path[color=purple, very thick, shorten >=10pt, ->, >=stealth, bend left] (input) edge (1);
\path[color=purple, very thick, shorten >=10pt, ->, >=stealth, bend right]
(input) edge (2);
\end{tikzpicture}
\]
In an open Markov process, probabilities change with time according to the `open master equation', a generalization of the master equation that includes inflows and outflows.  
In the above example, the open master equation is 
\[    \frac{d}{dt}\left[\begin{array}{r} \pi_a(t) \\ \pi_b(t) \\ \pi_c(t) \\ \pi_d(t) \end{array}\right]  \;\; = \; \; \left[\begin{array}{rrrr}
    -1/2    & 0    & 0    & 0  \\
      0                   & -2   & 1   & 0 \\
      1/2   & 2     & -5  & 2 \\
      0                  & 0      & 4   & -2 \\
\end{array}\right] 
\left[\begin{array}{r} \pi_a(t) \\ \pi_b(t) \\ \pi_c(t) \\ \pi_d(t) \end{array}\right]  \; + \;
\left[\begin{array}{c} I_a(t) \\ I_b(t) \\0 \\ 0 \end{array}\right] \; - \; 
\left[\begin{array}{c} 0 \\ 0 \\0 \\ O_d(t) \end{array}\right] .
\]
To the master equation we have added a term describing inflows at the states $a$ and $b$ and subtracted a term describing outflows at the state $d$.   The functions $I_a, I_b$ and $O_d$ are
not part of the data of the open Markov process.  Rather, they are arbitrary smooth real-valued functions of time.  We think of these as provided from outside---for example, though not necessarily, from the rest of a larger Markov process of which the given open Markov process is part.

Open Markov processes can be seen as morphisms in a category, since we can compose two open Markov processes by identifying the outputs of the first with the inputs of the second.   Composition lets us build a Markov process from smaller open parts---or conversely, analyze the behavior of a Markov process in terms of its parts. Categories of this sort have been studied in a number of papers \cite{BF,BFP,FongThesis,PollardThesis}, but here we go further and construct a double category to describe coarse-graining.

`Coarse-graining' is a widely used method of simplifying a Markov process by mapping its set of states $X$ onto some smaller set $X'$ in a manner that respects, or at least approximately respects, the dynamics \cite{And,Buchholz}.  Here we introduce coarse-graining for open Markov processes.   We show how to extend this notion to the case of maps $p \maps X \to X'$ that are not surjective, obtaining a general concept of morphism between open Markov processes.

Since open Markov processes are already morphisms in a category, it is natural to treat morphisms between them as 2-morphisms. To this end, we construct a double category $\MMark$ with:
\begin{enumerate}
\item finite sets as objects, 
\item functions as vertical 1-morphisms,
\item open Markov processes as horizontal 1-cells, 
\item morphisms between open Markov processes as 2-morphisms.
\end{enumerate}
Composition of open Markov processes is only weakly associative, so this double category is not strict. In fact, $\MMark$ is symmetric monoidal: this captures the fact that we can not only compose open Markov processes but also `tensor' them by setting them side by side.  For example, if we compose this open Markov process:
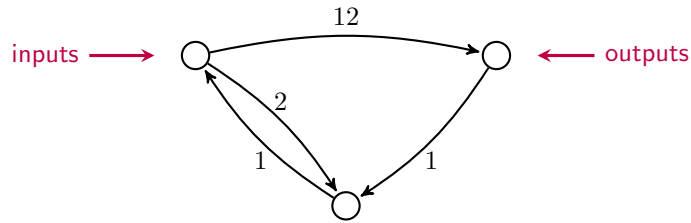
\[

\]
with the one shown before:
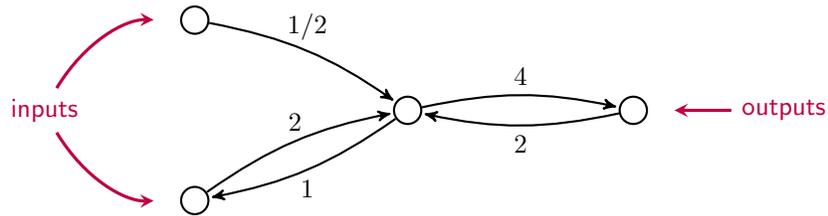
\[
%
\]
we obtain this open Markov process:
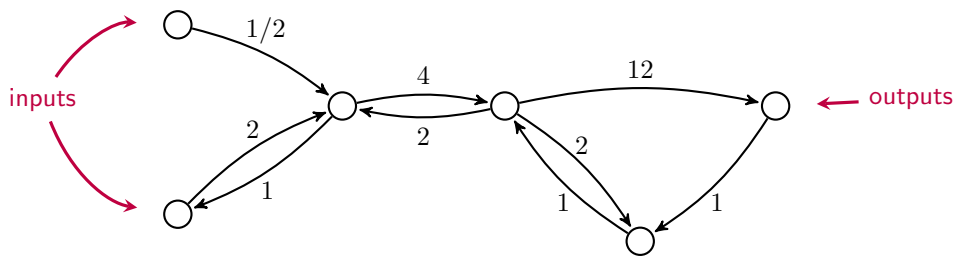
\[
%
\]
but if we tensor them, we obtain this:
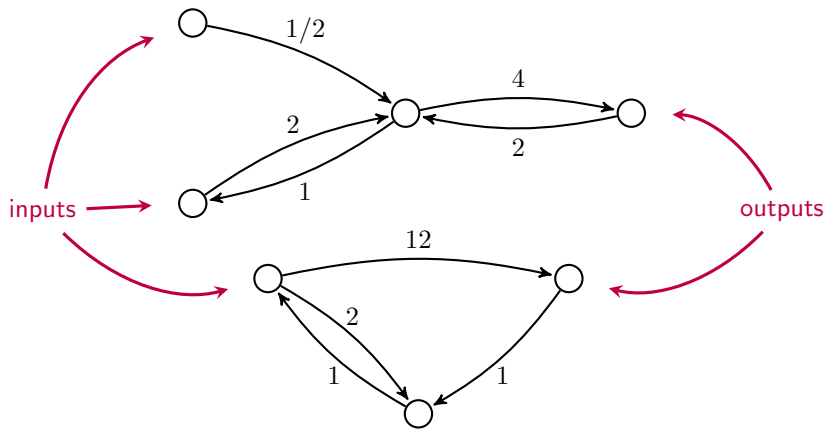
\[
%
\]

If we fix constant probabilities at the inputs and outputs, there typically exist solutions of the open master equation with these boundary conditions that are constant as a function of time.  These are called `steady states'.  Often these are \emph{nonequilibrium} steady states, meaning that there is a nonzero net flow of probabilities at the inputs and outputs.   For example, probability can flow through an open Markov process at a constant rate in a nonequilibrium steady state.

In previous work, Baez, Fong and Pollard studied the relation between probabilities and flows at the inputs and outputs that holds in steady state \cite{BFP,BP}.  They called the process of extracting this relation from an open Markov process `black-boxing', since it gives a way to forget the internal workings of an open system and remember only its externally observable behavior.   They proved that black-boxing is compatible with composition and tensoring.  This result can be summarized by saying that black-boxing is a symmetric monoidal functor. 

For the main result \cite{BC}, we show that black-boxing is compatible with morphisms between open Markov processes.  To make this idea precise, we prove that black-boxing gives a map from the double category $\MMark$ to another double category, called $\LLinRel$, which has:
\begin{enumerate}
\item finite-dimensional real vector spaces $U,V,W,\dots$ as objects,
\item linear maps $f \maps V \to W$ as vertical 1-morphisms from $V$ to $W$,
\item linear relations $R \subseteq V \oplus W$  as horizontal 1-cells from $V$ to $W$,
\item squares 
\[
\begin{tikzpicture}[scale=1.5]
\node (D) at (-4,0.5) {$V_1$};
\node (E) at (-2,0.5) {$V_2$};
\node (F) at (-4,-1) {$W_1$};
\node (A) at (-2,-1) {$W_2$};
\node (B) at (-3,-0.25) {};
\path[->,font=\scriptsize,>=angle 90]
(D) edge node [above]{$R \subseteq V_1 \oplus V_2$}(E)
(E) edge node [right]{$g$}(A)
(D) edge node [left]{$f$}(F)
(F) edge node [above]{$S \subseteq W_1 \oplus W_2$} (A);
\end{tikzpicture}
\]
obeying $(f \oplus g)R \subseteq S$ as 2-morphisms. 
\end{enumerate}
Here a `linear relation' from a vector space $V$ to a vector space $W$ is a linear subspace
$R \subseteq V \oplus W$.   Linear relations can be composed in the same way as relations \cite{BE}.  The double category $\LLinRel$ becomes symmetric monoidal using direct sum as the tensor product, but unlike $\MMark$ it is strict: that is, composition of linear relations is associative.

The main result, Theorem \ref{thm:main}, 
says that black-boxing gives a symmetric monoidal double functor
\[   \blacksquare \maps \MMark \to \LLinRel .\]
The hardest part is to show that black-boxing preserves composition of horizontal
1-cells: that is, black-boxing a composite of open Markov processes gives the composite 
of their black-boxings.  Luckily, for this we can adapt a previous argument \cite{BP} due to Baez and Pollard.  
Thus, the new content of this result concerns the vertical 1-morphisms and especially 
the 2-morphisms, which describe coarse-grainings.

An alternative approach to studying morphisms between open Markov processes would use
bicategories rather than double categories.  The symmetric monoidal double categories $\mathbb{M}\mathbf{ark}$ and $\mathbb{L}\mathbf{inrel}$ can be converted into symmetric monoidal bicategories using Shulman's technique \cite{Shul}. In \cite{BC}, Baez and the author conjectured that the black-boxing double functor would determine a functor between these symmetric monoidal bicategories, and Hansen and Shulman \cite{Shul3} consequently proved this conjecture: see Theorem \ref{blackbox_functor_bicategories}. However, double categories seem to be a simpler framework for coarse-graining open Markov processes.

It is worth comparing some related work.  Baez,  Fong and Pollard constructed a symmetric monoidal category where the morphisms are open Markov processes \cite{BFP,BP}. As in this chapter, they only consider Markov processes where time is continuous and the set of states is finite.  
However, they formalized such Markov processes in a slightly different way than is done here: they defined a Markov process to be a directed multigraph where each edge is assigned a positive number called its `rate constant'.   In other words, they defined it to be a diagram
\[ \xymatrix{ (0,\infty) & E \ar[l]_-r \ar[r]<-.5ex>_t  \ar[r] <.5ex>^s & X  }  \]
where $X$ is a finite set of vertices or `states', $E$ is a finite set of edges or `transitions' 
between states, the functions $s,t \maps E \to X$ give the source and target of each edge, and 
$r \maps E \to (0,\infty)$ gives the rate constant of each edge.   They explained how from this data
one can extract a matrix of real numbers $(H_{ij})_{i,j \in X}$ called the `Hamiltonian' of the Markov process, with two familiar properties:
\begin{enumerate}
\item $H_{ij} \geq 0$ if $i \neq j$, 
\item $\sum_{i \in X} H_{ij} = 0$ for all $j \in X$.
\end{enumerate}
A matrix with these properties is called `infinitesimal stochastic', since these conditions are equivalent to $\exp(tH)$ being stochastic for all $t \ge 0$.    

Here we skip the directed multigraphs and work directly with the Hamiltonians.  Thus, we define a Markov process to be a finite set $X$ together with an infinitesimal stochastic matrix $(H_{ij})_{i,j \in X}$.  This allows us to work more directly with the Hamiltonian and the all-important master equation
\[         \frac{d}{dt} \pi(t) = H \pi(t)  \]
which describes the evolution of a time-dependent probability distribution $\pi(t) \maps X \to \R$.

Clerc, Humphrey and Panangaden have constructed a bicategory \cite{Panan} with finite sets as objects, `open discrete labeled Markov processes' as morphisms, and `simulations' as 2-morphisms. In their framework, `open' has a similar meaning as it does in the works listed above.  These open discrete labeled Markov processes are also equipped with a set of `actions' which represent interactions between the Markov process and the environment, such as an outside entity acting on a stochastic system.   A `simulation' is then a function between the state spaces that map the inputs, outputs and set of actions of one open discrete labeled Markov process to the inputs, outputs and set of actions of another.   

Another compositional framework for Markov processes is given by de Francesco Albasini, Sabadini and Walters \cite{Francesco} in which they construct an algebra of `Markov automata'.  A Markov automaton is a family of matrices with nonnegative real coefficients that is indexed by elements of a binary product of sets, where one set represents a set of `signals on the left interface' of the Markov automata and the other set analogously for the right interface.

\section{Open Markov processes}
Before explaining open Markov processes we should recall a bit about Markov processes.  As mentioned in the Introduction, we use `Markov process' as a short term for `continuous-time Markov chain with a finite set of states', and we identify any such Markov process with the
infinitesimal stochastic matrix appearing in its master equation.  We make this precise with a bit 
of terminology that is useful throughout the chapter. 

Given a finite set $X$, we call a function $v \maps X \to \R$ a `vector' and call its values at points $x \in X$ its `components' $v_x$.   We define a `probability distribution' on $X$ to be a vector $\pi \maps X \to \R$ whose components are nonnegative and sum to $1$.  As usual, we use $\R^X$ to denote the vector space of functions $v \maps X \to \R$.    Given a linear operator $T \maps \R^X \to \R^Y$ we have $(T v)_i = \sum_{j \in X} T_{ij} v_j$ for some `matrix' $T \maps Y \times X \to \R$ with entries $T_{ij}$.   

\begin{definition}
Given a finite set $X$, a linear operator $H \maps \R^X \to \R^X$ is \define{infinitesimal stochastic} if
\begin{enumerate}
\item $H_{ij} \geq 0$ for $i \neq j$ and 
\item $\sum_{i \in X} H_{ij}=0$ for each $j \in X$.
\end{enumerate}
\end{definition}

The reason for being interested in such operators is that when exponentiated
they give stochastic operators.

\begin{definition} Given finite sets $X$ and $Y$, a linear operator $T \maps \R^X \to \R^Y$ is \define{stochastic} if for any probability distribution $\pi$ on $X$, $T\pi$ is a probability
distribution on $Y$.
\end{definition}

Equivalently, $T$ is stochastic if and only if 
\begin{enumerate}
\item $T_{ij} \ge 0$ for all $i \in Y$, $j \in X$ and
\item $\sum_{i \in Y} T_{ij} = 1$ for each $j \in X$.
\end{enumerate}
If we think of $T_{ij}$ as the probability for $j \in X$ to be mapped to $i \in Y$, these conditions
make intuitive sense.   Since stochastic operators are those that preserve probability distributions, the composite of stochastic operators is stochastic.  

In Lemma \ref{lem:exponentiation} we recall that a linear operator $H \maps \R^X \to \R^X$ 
is infinitesimal stochastic if and only if its exponential
\[     \exp(tH) = \sum_{n = 0}^\infty \frac{(tH)^n}{n!}  \]
is stochastic for all $t \ge 0$.  
Thus, given an infinitesimal stochastic operator $H$, for any time $t \ge 0$
we can apply the operator $\exp(tH) \maps \R^X \to \R^X$ to any probability 
distribution $\pi \in \R^X$ and get a probability distribution
\[      \pi(t) = \exp(tH) \pi. \]
These probability distributions $\pi(t)$ obey the \define{master equation}
\[       \frac{d}{dt}\pi(t) = H \pi(t) .\]
Moreover, any solution of the master equation arises this way.

All the material so far is standard \cite[Sec.\ 2.1]{Norris}.  We now turn to open Markov processes.  

\begin{definition}
We define a \define{Markov process} to be a pair $(X,H)$ where $X$ is a finite set and 
$H \maps \R^X \to \R^X$ is an infinitesimal stochastic operator.    We also call $H$ a 
Markov process \define{on} $X$. 
\end{definition}

\begin{definition}
\label{defn:open_markov_process}
We define an \define{open Markov process} to consist of finite sets $X$, $S$ and $T$ and
injections
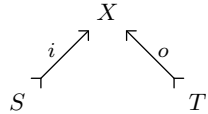
\[
\begin{tikzpicture}[scale=1.2]
\node (D) at (-4,-0.5) {$S$};
\node (E) at (-3,.5) {$X$};
\node (F) at (-2,-0.5) {$T$};
\path[->,font=\scriptsize,>=angle 90]
(D) edge[>->] node[above,pos=0.3] {$i$}(E)
(F) edge[>->] node[above,pos=0.3] {$o$}(E);
\end{tikzpicture}
\]
together with a Markov process $(X,H)$.   We call $S$ the set of \define{inputs} and
$T$ the set of \define{outputs}.
\end{definition}

Thus, an open Markov process is a cospan in $\FinSet$ with injections as legs and a Markov process on its apex.  We do not require that the injections have disjoint images.  We often abbreviate an open Markov process as
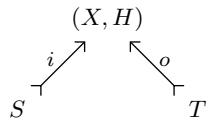
\[
\begin{tikzpicture}[scale=1.2]
\node (D) at (-4,-0.5) {$S$};
\node (E) at (-3,.5) {$(X,H)$};
\node (F) at (-2,-0.5) {$T$};
\path[->,font=\scriptsize,>=angle 90]
(D) edge[>->] node[above,pos=0.3] {$i$}(E)
(F) edge[>->] node[above,pos=0.3] {$o$}(E);
\end{tikzpicture}
\]
or simply $S \stackrel{i}{\rightarrowtail} (X,H) \stackrel{o}{\leftarrowtail} T$.

Given an open Markov process we can write down an `open' version of the master equation, where probability can also flow in or out of the inputs and outputs.   To work with the open
master equation we need two well-known concepts:

\begin{definition}
\label{defn:push_and_pull}
Let $f \maps A \to B$ be a map between finite sets.  The linear map $f^* \maps \R^B \to
\R^A$ sends any vector $v \in \R^B$ to its \define{pullback} along $f$, given by
\[                       f^*(v) = v \circ f  . \]
The linear map $f_* \maps \R^A \to \R^B$ sends any vector $v \in \R^A$ to its \define{pushforward} along $f$, given by
\[                       (f_*(v))(b) = \sum_{ \{a : \; f(a) = b\} } v(a)  .\]
\end{definition}
\noindent If we write $f^*$ and $f_*$ as matrices with respect to the standard bases of
$\R^A$ and $\R^B$, they are simply transposes of one another.

Now, suppose we are given an open Markov process
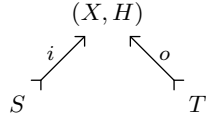
\[
\begin{tikzpicture}[scale=1.2]
\node (D) at (-4,-0.5) {$S$};
\node (E) at (-3,.5) {$(X,H)$};
\node (F) at (-2,-0.5) {$T$};
\path[->,font=\scriptsize,>=angle 90]
(D) edge[>->] node[above,pos=0.3] {$i$}(E)
(F) edge[>->] node[above,pos=0.3] {$o$}(E);
\end{tikzpicture}
\]
together with \define{inflows} $I \maps \R \to \R^S$ and \define{outflows} $O \maps \R \to \R^T$, arbitrary smooth functions of time.  We write the value of the inflow at $s \in S$ at time $t$ as $I_s(t)$, and similarly for the outflows and other functions of time.  We say that a function $v \maps \R \to \R^X$ obeys the \define{open master equation} if
\[ \frac{dv(t)}{dt} = H v(t) + i_*(I(t)) - o_*(O(t)). \]
This says that for any state $j \in X$ the time derivative of $v_j(t)$ 
takes into account not only the usual term from the master equation, but also those of the inflows and outflows.

If the inflows and outflows are constant in time, a solution $v$ of the open master equation
that is also constant in time is called a \define{steady state}. More formally:

\begin{definition}
Given an open Markov process $S \stackrel{i}{\rightarrowtail} (X,H) \stackrel{o}{\leftarrowtail} T$ together with $I \in \R^S$ and $O \in \R^T$, a \define{steady state} with inflows $I$ and outflows $O$ is an element $v \in \R^X$ such that 
\[     H v + i_*(I) - o_*(O) = 0 .\]   
Given $v \in \R^X$, we call $i^*(v) \in \R^S$ and $o^*(v) \in \R^T$ the 
\define{input probabilities} and \define{output probabilities}, respectively.
\end{definition}

\begin{definition}
\label{defn:black-boxing}
Given an open Markov process $S \stackrel{i}{\rightarrowtail} (X,H) \stackrel{o}{\leftarrowtail} T$, we define its \define{black-boxing} to be the set
\[   \blacksquare\big(S \stackrel{i}{\rightarrowtail} (X,H) \stackrel{o}{\leftarrowtail} T\big) 
\subseteq \R^S \oplus \R^S \oplus \R^T \oplus \R^T \]
consisting of all 4-tuples $(i^*(v),I,o^*(v),O)$ where $v \in \R^X$ is some steady state
with inflows $I \in \R^S$ and outflows $O \in \R^T$.
\end{definition}

Thus, black-boxing records the relation between input probabilities, inflows, output probabilities and outflows that holds in steady state.  This is the `externally observable steady state behavior' of the open Markov process.   It has already been shown \cite{BFP,BP} that black-boxing can be seen as a functor between categories.    Here we go further and describe it as a double functor between double categories, in order to study the effect of black-boxing on morphisms between open Markov processes.

\section{Morphisms of open Markov processes}

There are various ways to approximate a Markov process by another Markov process on a smaller set, all of which can be considered forms of coarse-graining \cite{Buchholz}.  A common approach is to take a Markov process $H$ on a finite set $X$ and a surjection $p \maps X \to X'$ and create a Markov process on $X'$.   In general this requires a choice of `stochastic section' for $p$, defined as follows:

\begin{definition}
Given a function $p \maps X \to X'$ between finite sets, a \define{stochastic section} for
$p$ is a stochastic operator $s \maps \R^{X'} \to \R^X$ such that $p_* s = 1_{X'}$. 
\end{definition}

It is easy to check that a stochastic section for $p$ exists if and only if $p$ is a surjection.  In Lemma \ref{lem:new_markov_process} we show that given a Markov process $H$ on $X$ and a surjection $p \maps X \to X'$, any stochastic section $s \maps \R^{X'} \to \R^X$ gives a Markov process on $X'$, namely
\[            H' =  p_* H s. \]
Experts call the matrix corresponding to $p_*$ the \define{collector matrix}, and they call $s$ the \define{distributor matrix} \cite{Buchholz}.  The names help clarify what is going on.  The collector matrix, coming from the surjection $p \maps X \to X'$, typically maps many states of $X$ to each state of $X'$.  The distributor matrix, the stochastic section $s \maps \R^{X'} \to \R^X$, typically maps each state in $X'$ to a linear combination of many states in $X$.  Thus, $H' = p_* H s$ distributes each state of $X'$, applies $H$, and then collects the results.

In general $H'$ depends on the choice of $s$, but sometimes it does not:

\begin{definition} 
We say a Markov process $H$ on $X$ is \define{lumpable} with respect to 
a surjection $p \maps X \to X'$ if the operator $p_* H s$ is independent of the choice of stochastic section $s \maps \R^{X'} \to \R^X$.
\end{definition}

This concept is not new \cite{Buchholz}.   In Theorem\ \ref{thm:lumpability} we show that it is equivalent to another traditional formulation, and also to an even simpler one: $H$ is lumpable with respect to $p$ if and only if $p_* H = H' p_*$.    This equation has the advantage of making sense even when $p$ is not a surjection.   Thus, we can use it to define a more general concept of morphism between Markov processes:

\begin{definition}  
Given Markov processes $(X,H)$ and $(X',H')$, a \define{morphism of Markov
processes} $p \maps (X,H) \to (X',H')$ is a map $p \maps X \to X'$ such that
$p_* H = H' p_*$.  
\end{definition}

There is a category $\Mark$ with Markov processes as objects and the morphisms as defined above, where composition is the usual composition of functions.  But what is the meaning of such a morphism?   Using Lemma \ref{lem:exponentiation} one can check that for any Markov processes $(X,H)$ and $(X',H')$, and any map
$p \maps X \to X'$, we have
\[    p_* H = H' p_* \; \iff \;  p_* \exp(tH) = \exp(tH') p_* \textrm{ for all } t \ge 0.\]
Thus, $p$ is a morphism of Markov processes if evolving a probability distribution
on $X$ via $\exp(tH)$ and then pushing it forward along $p$ is the same as
pushing it forward and then evolving it via $\exp(tH')$.  

We can also define morphisms between open Markov processes:

\begin{definition}\label{defn:coarse-graining} 
A \define{morphism of open Markov processes} from the open Markov process $S \stackrel{i}{\rightarrowtail} (X,H) \stackrel{o}{\leftarrowtail} T$ to the open Markov process $S' \stackrel{i'}{\rightarrowtail} (X',H') \stackrel{o'}{\leftarrowtail} T'$ is a triple of functions $f \maps S \to S'$, $p \maps X \to X'$, $g \maps T \to T'$ such that the squares in this diagram are pullbacks:
\[

\]
and $p_* H = H' p_*$.
\end{definition}
We need the squares to be pullbacks so that in Lemma \ref{lem:black-boxing_2-morphisms} we can black-box morphisms of open Markov processes.  In Lemma \ref{lem:horizontal_composition} we show that horizontally composing these morphisms preserves this pullback property.  But to do this, we need the horizontal arrows in these squares to be injections.  This explains the conditions in Definitions\ \ref{defn:open_markov_process} and \ref{defn:coarse-graining}.

As an example, consider the following diagram:
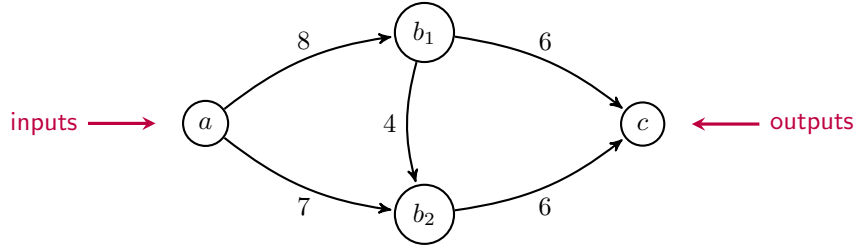
\[
%
\]
This is a way of drawing an open Markov process $S \stackrel{i}{\rightarrowtail} (X,H) \stackrel{o}{\leftarrowtail} T$ where $X = \{a,b_1,b_2,c \}$, $S$ and $T$ are one-element sets, $i$ maps the one element of $S$ to $a$, and $o$ maps the one element of $T$ to $c$.   We can read off the infinitesimal stochastic operator $H \maps \R^X \to \R^X$ from this diagram and obtain
\[ H=
\left[\begin{array}{rrrr}
  \!  -15    & 0    & 0    & 0  \\
      8     & -10 & 0    & 0 \\
      7     & 4    & -6   & 0 \\
      0     & 6    & 6    & 0 \\
\end{array}\right].
\]
The resulting open master equation is
\[    \frac{d}{dt}\left[\begin{array}{r} v_{a}(t) \\ v_{b_1}(t) \\ v_{b_2}(t) \\ v_{c}(t) \end{array}\right]  \;\; = \; \; 
\left[\begin{array}{rrrr}
  \!  -15    & 0    & 0    & 0  \\
      8     & -10 & 0    & 0 \\
      7     & 4    & -6   & 0 \\
      0     & 6    & 6    & 0 \\
\end{array}\right]
\left[\begin{array}{r} v_{a}(t) \\ v_{b_1}(t) \\ v_{b_2}(t) \\ v_{c}(t) \end{array}\right]  \; + \;
\left[\begin{array}{c} I(t) \\ 0 \\0 \\ 0 \end{array}\right] \; - \; 
\left[\begin{array}{c} 0 \\ 0 \\0 \\ O(t) \end{array}\right] .  \]
Here $I$ is an arbitrary smooth 
function of time describing the inflow at the one point of $S$, and $O$ is a similar function
describing the outflow at the one point of $T$.

Suppose we want to simplify this open Markov process by identifying the states $b_1$ and
$b_2$.    To do this we take $X' = \{a,b,c \}$ and define $p \maps X \to X'$ by
\[    p(a) = a, \quad p(b_1) = p(b_2) = b, \quad p(c) = c .\]
To construct the infinitesimal stochastic operator $H' \maps \R^{X'} \to \R^{X'}$ for the
simplified open Markov process we need to choose a stochastic section $s \maps \R^{X'} \to \R^X$ for $p$, for example
\[ s=
\left[
\begin{array}{rrr}
1 & 0 & 0  \\
0 & 1/3 & 0 \\
0 & 2/3 & 0 \\
0 & 0 & 1  \\
\end{array}\right].
\]
This says that if our simplified Markov process is in the state $b$, we assume the original Markov
process has a $1/3$ chance of being in state $b_1$ and a $2/3$ chance of being in state $b_2$. 
The operator $H' = p_* H s$ is then
\[ H' =
\left[\begin{array}{rrr}
    -15    & 0   & 0 \\
    15     & -6& 0\\
    0       & 6 & 0\\
    \end{array}\right] .
\]
It may be difficult to justify the assumptions behind our choice of stochastic section, but 
the example at hand has a nice feature: $H'$ is actually independent of this choice.   In other words, $H$ is lumpable with respect to $p$.  The reason is explained in Theorem \ref{thm:lumpability}.  Suppose we partition $X$ into blocks, each the inverse image of some point of $X'$.   Then $H$ is lumpable with respect to $p$ if and only if when we sum the rows in each block of $H$, all the columns within any given block of the resulting matrix are identical.  This matrix is $p_* H$:
\[ H=
\left[\begin{array}{r;{2pt/2pt}rr;{2pt/2pt}r}
  \!  -15    & 0    & 0    & 0  \\ \hdashline[2pt/2pt]
      8     & -10 & 0    & 0 \\ 
      7     & 4    & -6   & 0 \\  \hdashline[2pt/2pt]
      0     & 6    & 6    & 0 \\ 
\end{array}\right]  \implies 
p_* H = 
\left[\begin{array}{r;{2pt/2pt}rr;{2pt/2pt}r}
  \!  -15    & 0    & 0    & 0  \\ \hdashline[2pt/2pt]
      15     & -6 & -6    & 0 \\  \hdashline[2pt/2pt]
      0     & 6    & 6    & 0 \\ 
\end{array}\right]. 
\] 
While coarse-graining is of practical importance even in the absence of lumpability, the lumpable case is better behaved, so we focus on this case. 

So far we have described a morphism of Markov processes $p \maps (X,H) \to (X',H')$, but 
together with identity functions on the inputs $S$ and outputs $T$ this defines a morphism
of open Markov processes, going from the above open Markov process to this one:
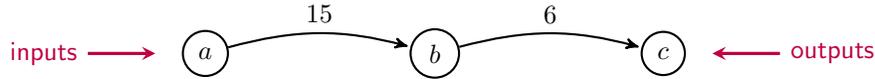
\[
%
\]
The open master equation for this new coarse-grained open Markov process is
\[    \frac{d}{dt}\left[\begin{array}{r} v_a(t) \\ v_b(t) \\ v_c(t) \end{array}\right]  \;\; = \; \; 
\left[\begin{array}{rrr}
    -15    & 0   & 0 \\
    15     & -6& 0\\
    0       & 6 & 0\\
    \end{array}\right]
\left[\begin{array}{r} v_a(t) \\ v_b(t) \\ v_c(t) \end{array}\right]  \; + \;
\left[\begin{array}{c} I(t) \\ 0 \\ 0 \end{array}\right] \; - \; 
\left[\begin{array}{c} 0 \\ 0 \\ O(t) \end{array}\right] .  \]

In Section \ref{sec:MMark} we construct a double category $\MMark$ with open Markov processes as horizontal 1-cells and morphisms between these as 2-morphisms.  This double category is our main object of study.   First, however, we should prove the results mentioned above.   For this it is helpful to recall a few standard concepts:

\begin{definition}
A \define{1-parameter semigroup of operators} is a collection of linear operators $U(t) \maps V \to V$ on a vector space $V$, one for each $t \in [0,\infty)$, such that
\begin{enumerate}
\item $U(0)=1$ and 
\item $U(s+t)=U(s)U(t)$ for all $s,t \in [0,\infty)$.
If $V$ is finite-dimensional we say the collection $U(t)$ is \define{continuous} if $t \mapsto U(t) v$ is continuous for each $v \in V$.
\end{enumerate}
\end{definition}

\begin{definition}
Let $X$ be a finite set. A \define{Markov semigroup} is a continuous 1-parameter semigroup
$U(t) \maps \R^X \to \R^X$ such that $U(t)$ is stochastic for each $t \in [0,\infty)$.
\end{definition}

\begin{lemma}
\label{lem:exponentiation}
Let $X$ be a finite set and $U(t) \maps \R^X \to \R^X$ a Markov semigroup. Then $U(t)= \exp(tH)$ for a unique infinitesimal stochastic operator $H \maps \R^X \to \R^X$, which is given by
\[  Hv = \left. \frac{d}{dt} U(t)v \right|_{t=0} \]
for all $v \in \R^X$. Conversely, given an infinitesimal stochastic operator $H$, then $\exp(tH)=U(t)$ is a Markov semigroup. 
\end{lemma}

\begin{proof}  This is well known.   For a proof that every continuous one-parameter semigroup of 
operators $U(t)$ on a finite-dimensional vector space $V$ is in fact differentiable and of the form $\exp(tH)$ where $Hv = \left. \frac{d}{dt} U(t)v \right|_{t=0}$, see Engel and Nagel \cite[Sec.\ I.2]{EngelNagel}.  For a proof that $U(t)$ is then a Markov semigroup if and only if $H$ is infinitesimal stochastic, see Norris \cite[Theorem\ 2.1.2]{Norris}. \end{proof}

\begin{lemma}
\label{lem:differentiation}
Let $U(t) \maps \R^X \to \R^X$ be a differentiable family of stochastic operators defined for $t \in [0,\infty)$ and having $U(0)=1$. Then $\left. \frac{d}{dt} U(t) \right|_{t=0}$ is infinitesimal stochastic.
\end{lemma}

\begin{proof}
Let $H=\left. \frac{d}{dt} U(t) \right|_{t=0}  = \lim_{t \to 0^+} (U(t)-1)/t$. As $U(t)$ is stochastic, its entries are nonnegative and the column sum of any particular column is 1. Then the column sum of any particular column of $U(t)-1$ will be 0 with the off-diagonal entries being nonnegative. Thus $U(t)-1$ is infinitesimal stochastic for all $t \geq 0$, as is $(U(t)-1)/t$, from which it follows that $\lim_{t\to 0^+} (U(t)-U(0))/t = H$ is infinitesimal stochastic.
\end{proof}

\begin{lemma}
\label{lem:new_markov_process}
Let $p \maps X \to X'$ be a function between finite sets with a stochastic section $s \maps \R^{X'} \to \R^X$, and let $H \maps \R^X \to \R^X$ be an infinitesimal stochastic operator. Then $H' = p_* H s \maps \R^{X'} \to \R^{X'}$ is also infinitesimal stochastic.
\end{lemma}

\begin{proof}
Lemma \ref{lem:exponentiation} implies that $\exp(tH)$ is stochastic for all $t \ge 0$.    For any map $p \maps X \to X'$ the operator $p_* \maps \R^X \to \R^{X'}$ is easily seen to be stochastic, and $s$ is stochastic by assumption.  Thus, $U(t) = p_* \exp(tH) s$ is stochastic for all $t \ge 0$.    Differentiating, we conclude that 
\[   \left. \frac{d}{dt} U(t) \right|_{t=0}= \left. \frac{d}{dt} p_* \exp(tH) s \right|_{t=0} =
\left. p_* \exp(tH)H s \right|_{t=0} = p_* H s \] 
is infinitesimal stochastic by Lemma \ref{lem:differentiation}.
\end{proof}

We can now give some conditions equivalent to lumpability.   The third is widely found in the literature \cite{Buchholz} and the easiest to check in examples.   It makes use of  the standard basis vectors $e_j \in \R^X$ associated to the elements $j$ of any finite set $X$.    The surjection $p \maps X \to X'$ defines a partition on $X$ where two states $j,j' \in X$ lie in the same block of the partition if and only if $p(j) = p(j')$.   The elements of $X'$ correspond to these blocks.   The third condition for lumpability says that $p_* H$ has the same effect on two basis vectors $e_j$ and $e_{j'}$ when $j$ and $j'$ are in the same block.  As mentioned in the example above, this condition
says that if we sum the rows in each block of $H$, all the columns in any given block of the resulting matrix $p_* H$ are identical.

\begin{theorem}
\label{thm:lumpability}
Let $p \maps X \to X'$ be a surjection of finite sets and let $H$ be a Markov process on $X$.  Then the following conditions are equivalent:
\begin{enumerate}
\item $H$ is lumpable with respect to $p$.
\item There exists a linear operator $H' \maps \R^{X'} \to \R^{X'}$ such that $p_* H = H' p_*$.
\item $p_* H e_j = p_* H e_{j'}$ for all $j,j' \in X$ such that $p(j)=p(j')$.   
\end{enumerate}
When these conditions hold there is a unique operator $H' \maps \R^{X'} \to \R^{X'}$ 
such that $p_* H = H' p_*$, it is given by $H' = p_* H s$ for any stochastic section $s$ of $p$, and it is infinitesimal stochastic.
\end{theorem}

\begin{proof}
$(i) \implies (iii)$.    Suppose that $H$ is lumpable with respect to $p$.  Thus, $p_* H s \maps \R^{X'} \to \R^{X'}$ is independent of the choice of stochastic section $s \maps \R^{X'} \to \R^X$.    Such a stochastic section is simply an arbitrary linear operator that maps each basis vector $e_i \in \R^{X'}$ to a probability distribution on $X$ supported on the set $\{j \in X: p(j) = i\}$.   Thus, for any $j,j' \in X$ with $p(j)=p(j')=i$, we can find stochastic sections $s,s' \maps \R^{X'} \to \R^X$ such that $s(e_i) = e_j$ and $s'(e_i)=e_{j'}$.   Since $p_* Hs = p_* Hs'$, we have
\[  p_* H e_j = p_* Hs(e_i)=p_* Hs'(e_i) = p_* H e_{j'}. \]

$(iii) \implies (ii)$.  Define $H' \maps \R^{X'} \to \R^{X'}$ on basis vectors
$e_i \in \R^{X'}$ by setting 
\[      H' e_i = p_* H e_j \]
for any $j$ with $p(j)=i$.   Note that $H'$ is well-defined: since $p$ is a surjection such $j$ exists, and since $H$ is lumpable, $H'$ is independent of the choice of such $j$.    Next, note that for any $j \in X$, if we let $p(j) = i$ we have $p_* H e_j = H' e_i = H' p_* e_j$.  Since the vectors $e_j$ form a basis for $\R^X$, it follows that $p_* H = H' p_*$.

$(ii) \implies (i)$.    Suppose there exists an operator $H' \maps \R^{X'} \to \R^{X'}$ such that $p_* H = H' p_*$.  Choose such an operator; then for any stochastic section $s$ for $p$ we have
\[    p_* Hs = H'p_* s = H' .\]
It follows that $p_* Hs$ is independent of the stochastic section $s$, so $H$ is lumpable with respect to $p$.

Suppose that any, hence all, of conditions $(i), (ii), (iii)$ hold.   Suppose that $H' \maps \R^{X'} \to \R^{X'}$ is an operator with $p_* H = H' p_*$.  Then the argument in the previous paragraph shows that $H' = p_* H s$ for any stochastic section $s$ of $p$.  Thus $H'$ is unique, and by Lemma \ref{lem:new_markov_process} it is infinitesimal stochastic.
\end{proof}

\section{A double category of open Markov processes} \label{sec:MMark}
One of the main results of a joint work with Baez \cite{BC} is the construction of a double category $\mathbb{M}\mathbf{ark}$ of open Markov processes,
The pieces of the double category $\mathbb{M}\mathbf{ark}$ work as follows: 
\begin{enumerate}
\item An object is a finite set. 
\item A vertical 1-morphism $f \maps S \to S'$ is a function.
\item A horizontal 1-cell is an open Markov process 
\[

\]
In other words, it is a pair of injections $S \stackrel{i}{\rightarrowtail} X \stackrel{o}{\leftarrowtail} T$
together with a Markov process $H$ on $X$.
\item A 2-morphism is a morphism of open Markov processes
\[
%
\]
In other words, it is a triple of maps $f,p,g$ such that these squares are pullbacks:
\[
%
\]
and $H' p_* = p_* H$.
\end{enumerate}

Composition of vertical 1-morphisms in $\mathbb{M}\mathbf{ark}$ is straightforward.  So is vertical 
composition of 2-morphisms, since we can paste two pullback squares and get a new
pullback square.   Composition of horizontal 1-cells is a bit more subtle.   Given 
open Markov processes
\begin{equation}
\label{eq:composable}
%
\end{equation}
we first compose their underlying cospans using a pushout:
\[
    \xymatrix{
      && X +_T Y \\
      & X \ar[ur]^{j} && Y \ar[ul]_{k} \\
      \quad S\quad \ar[ur]^{i_1} && T \ar[ul]_{o_1} \ar[ur]^{i_2} &&\quad U \quad \ar[ul]_{o_2}
    }
\]
Since monomorphisms are stable under pushout in a topos, the legs of this new cospan are
again injections, as required.   We then define the composite open Markov process to be
\[
%
\]
where 
\begin{equation}
\label{eq:odot} 
H \odot G = j_*  H  j^* + k_* G  k^* .
\end{equation}
Here we use both pullbacks and pushforwards along the maps $j$ and $k$, as defined in Definition \ref{defn:push_and_pull}.  To check that $H \odot G$ is a Markov process on $X +_T Y$ we need to check that $j_* H j^*$ and $k_* G k^*$, and thus their sum, are infinitesimal stochastic:

\begin{lemma} 
\label{lem:push-pull}
Suppose that $f \maps X \to Y$ is any map between finite sets.   If $H \maps \R^X \to \R^X$ is infinitesimal stochastic, then $f_* H f^* \maps \R^Y \to \R^Y$ is infinitesimal stochastic.
\end{lemma}

\begin{proof}
Using Definition\ \ref{defn:push_and_pull}, we see that the matrix elements of $f^*$ and $f_*$ are
given by
\[               (f^*)_{ji} = (f_*)_{ij} = \left\{ \begin{array}{cc} 1 & f(j) = i \\  0 & \textrm{otherwise} 
                                   \end{array} \right.
\]
for all $i \in Y$, $j \in X$.  Thus, $f_* H f^*$ has matrix entries
\[               (f_* H f^*)_{ii'} = \sum_{ \{j,j': \;  f(j) = i, f(j') = i' \} } H_{jj'}  .\]
To show that $f_* H f^*$ is infinitesimal stochastic we need to show that its off-diagonal entries are nonnegative and its columns sum to zero.  By the above formula, these follow from the same facts for $H$.
\end{proof}

Another formula for horizontal composition is also useful.  Given the composable open Markov processes in Equation (\ref{eq:composable}) we can take the copairing of the maps $j \maps X \to X +_T Y$ and $k \maps Y \to X +_T Y$ and get a map $\ell \maps X + Y \to X +_T Y$.   Then 
\begin{equation}
\label{eq:odot2}
H \odot G = \ell_* (H \oplus G) \ell^*  
\end{equation}
where $H \oplus G \maps \R^{X + Y} \to \R^{X + Y}$ is the direct sum of the operators $H$ and $G$.   This is easy to check from the definitions.

Horizontal composition of 2-morphisms is even subtler:

\begin{lemma}
\label{lem:horizontal_composition}
Suppose that we have horizontally composable 2-morphisms as follows:
\[

\]
Then there is a 2-morphism 
\[
%
\]
whose underlying diagram of finite sets is
\[
%
\]
where $j,k,j',k'$ are the canonical maps from $X,Y,X',Y'$, respectively,
to the pushouts $X +_T Y$ and $X' +_{T'} Y'$.
\end{lemma}

\begin{proof}
To show that we have defined a 2-morphism, we first check that the squares in the above diagram of finite sets are pullbacks.  Then we show that $(p +_g q)_* (H \odot G) = (H' \odot G') (p +_g q)_*$.

For the first part, it suffices by the symmetry of the situation to consider the left square.   We can write it as a pasting of two smaller squares:
\[
%
\]
By assumption the left-hand smaller square is a pullback, so it suffices to prove this for the right-hand one.   For this we use that fact that $\mathsf{FinSet}$ is a topos and thus an adhesive category \cite{LackSobocinski1,LackSobocinski2}, and consider this commutative cube:
\[
%
\]
By assumption the top and bottom faces are pushouts, the two left-hand vertical faces are
pullbacks, and the arrows $o'_1$ and $i_2'$ are monic.   In an adhesive category, this implies that the two right-hand vertical faces are pullbacks as well.   One of these is the square in question.

To show that $(p +_g q)_* (H \odot G) = (H' \odot G') (p +_g q)_*$, we again use the 
above cube.   Because its two right-hand vertical faces commute, we have
\[    (p +_g q)_* j_* = j'_* p_*  \quad \textrm{and} \quad (p +_g q)_* k_* = k'_* q_* \]
so using the definition of $H \odot G$ we obtain
\[ 
\begin{array}{ccl}
(p +_g q)_* (H \odot G) &=& (p +_g q)_* ( j_* H j^* + k_* G  k^*) \\
&=& (p +_g q)_*  j_* H  j^* \; + \; (p +_g q)_*  k_* G  k^* \\
&=& j'_* p_* H j^* \;\; + \; \; k'_* q_* G k^* . 
\end{array}
\]
By assumption we have
\[     p_* H = H' p_*  \quad \textrm{and} \quad q_* G = G' q_*  \]
so we can go a step further, obtaining
\[  (p +_g q)_* (H \odot G) =  j'_* H' p_* j^* \; + \; k'_* G' q_* k^* .\]
Because the two right-hand vertical faces of the cube are pullbacks, Lemma \ref{lem:beck-chevalley} below implies that
\[         p_* j^* = j'^* (p +_g q)_* \quad \textrm{and} \quad 
            q_* k^* = k'^* (p +_g q)_* .
\]
Using these, we obtain
\[ \begin{array}{ccl}
 (p +_g q)_* (H \odot G) 
 &=&  j'_* H'   j'^* (p +_g q)_* \; + \; k'_* G'  k'^* (p +_g q)_* \\
 &=& (j'_* H' j'^* + k'_* G' k'^*) (p +_g q)_* \\
 &=& (H' \odot G') (p +_g q)_* 
\end{array}
\]
completing the proof.
\end{proof}

The following lemma is reminiscent of the Beck--Chevalley condition for adjoint functors:

\begin{lemma}
\label{lem:beck-chevalley} 
Given a pullback square in $\mathsf{FinSet}$:
\[
%
\]
the following square of linear operators commutes:
\[

\]
\end{lemma}

\begin{proof}
Choose $v \in \R^B$ and $c \in C$.  Then
\[ 
\begin{array}{ccl}
(g_* f^* (v))(c) &=& \displaystyle{ \sum_{ \{a: g(a) = c \} } v(f(a)), }\\ \\
(k^* h_* (v))(c) &=& \displaystyle{ \sum_{ \{b:  h(b) = k(c) \} }  v(b) ,  }
\end{array}
\]
so to show $g_* f^* = k^* h_*$ it suffices to show that $f$ restricts to a bijection
\[      f \maps \{ a \in A :  g(a) = c \} \stackrel{\sim}{\longrightarrow} \{ b \in B: h(b) = k(c) \} .\] 
On the one hand, if $a \in A$ has $g(a) = c$ then $b = f(a)$ has
$h(b) = h(f(a)) = k(g(a)) = k(c)$, so the above map is well-defined.  On the other hand, 
if $b \in B$ has $h(b) = k(c)$, then by the definition of pullback there exists a unique $a \in A$
such that $f(a) = b$ and $g(a) = c$, so the above map is a bijection.
\end{proof}

\begin{theorem}
There exists a double category $\MMark$ as defined above.
\end{theorem}

\begin{proof}
Let $\MMark_0$, the category of objects, consist of finite sets and functions.  Let 
$\MMark_1$, the category of arrows, consist of open Markov processes and morphisms between these:
\[

\]
To make $\MMark$ into a double category we need to specify the identity-assigning functor 
\[   u \maps \MMark_0 \to \MMark_1, \]
the source and target functors
\[   s,t \maps \MMark_1 \to \MMark_0, \]
and the composition functor
\[ \odot \maps \MMark_1 \times_{\MMark_0} \MMark_1 \to \MMark_1 .\] 
These are given as follows.

For a finite set $S$, $u(S)$ is given by
\[
%
\]
where $0_S$ is the zero operator from $\R^S$ to $\R^S$.   For a map 
$f \maps S \to S'$ between finite sets, $u(f)$ is given by
\[
%
\]
The source and target functors $s$ and $t$ map a Markov process $S \stackrel{i}{\rightarrowtail} (X,H) \stackrel{o}{\leftarrowtail} T$ to $S$ and $T$, respectively, and they map a morphism of open Markov processes 
\[
%
\]
to $f \maps S \to S'$ and $g \maps T \to T'$, respectively.   The composition functor $\odot$ maps the pair of open Markov processes 
\[
%
\]
to their composite
\[
%
\]
defined as in Equation (\ref{eq:odot}), and it maps the pair of morphisms of open
Markov processes
\[
%
\]
to their horizontal composite as defined as in Lemma \ref{lem:horizontal_composition}.

It is easy to check that $u,s$ and $t$ are functors.  To prove that $\odot$ is a functor, the main thing we need to check is the interchange law.    Suppose we have four morphisms of open Markov processes as follows:
\[
%
\]
Composing horizontally gives
\[
%
\]
and then composing vertically gives
\[
%
\]
Composing vertically gives
\[
%
\]
and then composing horizontally gives
\[
%
\]
The only apparent difference between the two results is the map in the middle: one has $(p' +_{g'}q') \circ (p+_g q)$ while the other has $(p' \circ p) +_{(g' \circ g)} (q' \circ q)$.  But these are in fact the same map, so the interchange law holds. 

The functors $u, s, t$ and $\circ$ obey the necessary relations
\[      s u = 1 = t  u  \]
and the relations saying that the source and target of a composite behave as they should.
Lastly, we have three natural isomorphisms: the associator, left unitor, and right unitor,
which arise from the corresponding natural isomorphisms for the double category of finite sets, functions, cospans of finite sets, and maps of cospans.   The triangle and pentagon equations 
hold in $\MMark$ because they do in this simpler double category \cite{Cour}.
\end{proof}

Next we give $\MMark$ a symmetric monoidal structure.  We call the tensor product `addition'.   Given objects $S,S' \in \MMark_0$ we define their sum $S+S'$ using a chosen coproduct in $\FinSet$.   The unit for this tensor product in $\MMark_0$ is the empty set.  We can similarly define the sum of morphisms in $\MMark_0$, since given maps $f \maps S \to T$ and $f' \maps S' \to T'$ there is a natural map $f + f' \maps S + S' \to T + T'$.   Given two objects in $\MMark_1$:
\[

\]
we define their sum to be
\[
%
\]
where $H_1 \oplus H_2 \maps \R^{X_1 + X_2} \to \R^{X_1 + X_2}$ is the direct sum of the operators $H_1$ and $H_2$.  The unit for this tensor product in $\MMark_1$ is $\emptyset \rightarrowtail (\emptyset, 0_\emptyset) \leftarrowtail \emptyset$ where $0_\emptyset \maps \R^\emptyset \to \R^\emptyset$ is the zero operator.  Finally, given two morphisms in $\MMark_1$:
\[
%
\]
we define their sum to be
\[
%
\]
We complete the description of $\MMark$ as a symmetric monoidal double category
in the proof of this theorem:

\begin{theorem}
\label{thm:MMark_symmetric_monoidal}
The double category $\MMark$ can be given a symmetric monoidal structure with the above properties.
\end{theorem}

\begin{proof}
First we complete the description of $\MMark_0$ and $\MMark_1$ as symmetric monoidal categories. The symmetric monoidal category $\MMark_0$ is just the category of finite sets with a chosen  coproduct of each pair of finite sets providing the symmetric monoidal structure.   We have described the tensor product in $\MMark_1$, which we call `addition', so now we need to introduce the associator, unitors, and braiding, and check that they make $\MMark_1$ into a symmetric monoidal category. 

Given three objects in $\MMark_1$
\[

\]
tensoring the first two and then the third results in
\[
%
\]
whereas tensoring the last two and then the first results in
\[
%
\]
The associator for $\MMark_1$ is then given as follows:
\[
%
\]
where $a$ is the associator in $(\FinSet, +)$.  If we abbreviate an object $S \rightarrowtail (X,H) \leftarrowtail T$ of $\MMark_1$ as $(X,H)$, and denote the associator for $\MMark_1$ as $\alpha$, the pentagon identity says that this diagram commutes:
\[
%
\]
which is clearly true.  Recall that the monoidal unit for $\MMark_1$ is given by  $\emptyset \rightarrowtail (\emptyset, 0_\emptyset) \leftarrowtail \emptyset$. The left and right unitors for $\MMark_1$, denoted $\lambda$ and $\rho$, are given respectively by the following 2-morphisms:
\[
%
\]
where $\ell$ and $r$ are the left and right unitors in $\FinSet$.
The left and right unitors and associator for $\MMark_1$ satisfy the triangle identity:
\[
%
\]
The braiding in $\MMark_1$ is given as follows:
\[
%
\] 
where $b$ is the braiding in $(\FinSet,+)$.   It is easy to check that the braiding in $\MMark_1$ is its own inverse and obeys the hexagon identity, making $\MMark_1$ into a symmetric monoidal category.

The source and target functors $s,t \maps \MMark_1 \to \MMark_0$ are strict symmetric monoidal functors, as required.  To make $\MMark$ into a symmetric monoidal double category we must also give it two other pieces of structure.  One, called $\chi$, says how the composition of horizontal 1-cells interacts with the tensor product in the category of arrows.  The other, called $\mu$, says how the identity-assigning functor $u$ relates the tensor product in the category of objects to the tensor product in the category of arrows. We now define these two isomorphisms.

Given horizontal 1-cells 
\[

\]
the horizontal composites of the top two and the bottom two are given, respectively, by
\[
%
\]
`Adding' the left two and right two, respectively, we obtain
\[
%
\]
Thus the sum of the horizontal composites is 
\[
%
\]
while the horizontal composite of the sums is
\[
%
\]
The required globular 2-isomorphism $\chi$ between these is
\[
%
\] 
where $\hat{\chi}$ is the bijection 
\[ \hat{\chi} \maps (X_1 +_{T_1} Y_1) + (X_2 +_{T_2} Y_2) \to 
(X_1 + X_2) +_{T_1 + T_2} (Y_1 + Y_2)\]
obtained from taking the colimit of the diagram
\begin{center}
%
\end{center}
in two different ways.  We call $\chi$ `globular' because its source and target 1-morphisms
are identities.   We need to check that $\chi$ indeed defines a 2-isomorphism in 
$\MMark$.

To do this, we need to show that
\begin{equation}
\label{eq:chi}
((H_1 \oplus H_2) \odot (G_1 \oplus G_2))\, \hat{\chi}_*  =  \hat{\chi}_* \, ((H_1 \odot G_1) \oplus (H_2 \odot G_2)) .
\end{equation}
To simplify notation, let $K =  (X_1 +_{T_1} Y_1) + (X_2 +_{T_2} Y_2)$ and 
$K'=(X_1 + X_2) +_{T_1 + T_2} (Y_1 + Y_2)$ so that $\hat{\chi} \colon K \stackrel{\sim}{\to} K'$.
Let 
\[   q \maps X_1 + X_2 + Y_1 + Y_2 \to K , \quad 
      q' \maps X_1 + X_2 + Y_1 + Y_2 \to K'  \]
be the canonical maps coming from the definitions of $K$ and $K'$ as colimits, and note that
\[  q' = \hat{\chi} q \]
by the universal property of the colimit.   A calculation using Equation (\ref{eq:odot2}) implies that
\[    (H_1 \odot G_1) \oplus (H_2 \odot G_2) =
q_* \, ((H_1 \oplus H_2) \oplus (G_1 \oplus G_2)) \, q^* \]
and similarly 
\[ (H_1 \oplus H_2) \odot (G_1 \oplus G_2) 
= q'_* ((H_1 \oplus H_2) \oplus (G_1 \oplus G_2)) q'^*. \]
Together these facts give
\[  \begin{array}{ccl}  (H_1 \oplus H_2) \odot (G_1 \oplus G_2) 
&=& \hat{\chi}_* q_* \, ((H_1 \oplus H_2) \oplus (G_1 \oplus G_2)) \, q^* \hat{\chi}^* \\
&=&  \hat{\chi}_* \, ((H_1 \odot G_1) \oplus (H_2 \odot G_2)\, \hat{\chi}^*  .
\end{array}  \]
and since $\hat{\chi}$ is a bijection, $\hat{\chi}^*$ is the inverse of $\hat{\chi}_*$, 
so Equation (\ref{eq:chi}) follows.

For the other globular 2-isomorphism, if $S$ and $T$ are finite sets, then $u(S+T)$ is given by
\[
\begin{tikzpicture}[scale=1.5]
\node (D) at (-5,-0.5) {$S+T$};
\node (E) at (-3,-0.5) {$(S+T,0_{S+T})$};
\node (F) at (-1,-0.5) {$S+T$};
\path[->,font=\scriptsize,>=angle 90]
(D) edge[>->] node[above] {$1_{S+T}$}(E)
(F) edge[>->] node[above] {$1_{S+T}$}(E);
\end{tikzpicture}
\]
while $u(S) \oplus u(T)$ is given by
\[
\begin{tikzpicture}[scale=1.5]
\node (D) at (-5,-0.5) {$S+T$};
\node (E) at (-3,-0.5) {$(S+T,0_S \oplus 0_T)$};
\node (F) at (-1,-0.5) {$S+T$};
\path[->,font=\scriptsize,>=angle 90]
(D) edge[>->] node[above] {$1_S + 1_T$}(E)
(F) edge[>->] node[above] {$1_S + 1_T$}(E);
\end{tikzpicture}
\]
so there is a globular 2-isomorphism $\mu$ between these, namely the identity 2-morphism.
All the commutative diagrams in the definition of symmetric monoidal double category \cite{Shul} can be checked in a straightforward way.  
\end{proof}

\subsection{A bicategory of open Markov processes}
If one prefers to work with bicategories as opposed to double categories, then one can lift the above symmetric monoidal double category $\MMark$ to a symmetric monoidal bicategory $\mathbf{Mark}$ using a result of Shulman. This bicategory $\mathbf{Mark}$ will have:
\begin{enumerate}
\item finite sets as objects,
\item open Markov processes as morphisms,
\item morphisms of open Markov processes as 2-morphisms.
\end{enumerate}
To do this, we need to check that the symmetric monoidal double category $\MMark$ is isofibrant---meaning fibrant on vertical 1-morphisms which happen to be isomorphisms. See the \ref{Appendix}ppendix for details.

\begin{definition}
Let $\lD$ be a double category. Then the $\define{horizontal} \textnormal{ } \define{bicategory}$ of $\lD$, which we denote as $\mathbf{H}(\lD)$, is the bicategory with
\begin{enumerate}
\item objects of $\lD$ as objects,
\item  horizontal 1-cells of $\lD$ as 1-morphisms, 
\item globular 2-morphisms of $\lD$ (i.e., 2-morphisms with identities as their source and target)  as 2-morphisms,
\end{enumerate}
and vertical and horizontal composition, identities, associators and unitors arising from those in 
$\lD$.
\end{definition}

\begin{lemma}
\label{lem:isofibrant}
The symmetric monoidal double category $\MMark$ is isofibrant.
\end{lemma}

\begin{proof}
In what follows, all unlabeled arrows are identities. To show that $\MMark$ is isofibrant, we need to show that every vertical 1-isomorphism has both a companion and a conjoint \cite{Shul}. Given a vertical 1-isomorphism $f \maps S \to S'$, meaning a bijection between finite sets, then a companion of $f$ is given by the horizontal 1-cell:
\[

\]
together with two 2-morphisms
\[
%
\]
such that vertical composition gives
\[
%
\]
and horizontal composition gives
\[
%
\]
A conjoint of $f \maps S \to S'$ is given by the horizontal 1-cell
\[
%
\]
together with two 2-morphisms
\[
%
\]
that satisfy equations analogous to the two above.
\end{proof}

\begin{theorem} \label{markbicat}
The bicategory $\bold{Mark}$ is a symmetric monoidal bicategory.
\end{theorem}

\begin{proof}
This follows immediately from Theorem \ref{Shulman1} of Shulman: $\MMark$ is an isofibrant symmetric monoidal double category, so we obtain the symmetric monoidal bicategory $\bold{Mark}$ as the horizontal bicategory of $\MMark$.
\end{proof}

\section{A double category of linear relations}
The general idea of `black-boxing', as mentioned in Chapter \ref{Chapter2}, is to take a system and forget everything except the relation
between its inputs and outputs, as if we had placed it in a black box and were unable to see its inner workings.  Previous work of Baez and Pollard \cite{BP} constructed a black-boxing functor $\blacksquare \maps \Dynam \to \SemiAlgRel$ where $\Dynam$ is a category of finite sets and `open dynamical systems' and $\SemiAlgRel$ is a category of finite-dimensional real vector spaces and relations defined by polynomials and inequalities.   When we black-box such an open dynamical system, we obtain the relation between inputs and outputs that holds in steady state. 

A special case of an open dynamical system is an open Markov process as defined in this chapter.  Thus, we could restrict the black-boxing functor $\blacksquare \maps \Dynam \to \SemiAlgRel$ to a category $\mathsf{Mark}$ with finite sets as objects and open Markov processes as morphisms.  Since the steady state behavior of a Markov process is \emph{linear}, we would get a functor $\blacksquare \maps \Mark \to \LinRel$ where $\LinRel$ is the category of finite-dimensional real vector spaces and \emph{linear} relations \cite{BE}.   However, we will go further and define black-boxing on the double category $\MMark$.   This will exhibit the relation between black-boxing and morphisms between open Markov processes.

The symmetric monoidal double category $\mathbb{L}\mathbf{inRel}$ of linear relations introduced in this section will serve as the codomain of a symmetric monoidal black-box double functor in Section \ref{Black}. This double category $\mathbb{L}\mathbf{inRel}$ will have:
\begin{enumerate}
\item finite-dimensional real vector spaces $U,V,W,\dots$ as objects,
\item linear maps $f \maps V \to W$ as vertical 1-morphisms from $V$ to $W$,
\item linear relations $R \subseteq V \oplus W$  as horizontal 1-cells from $V$ to $W$,
\item squares 
\[
\begin{tikzpicture}[scale=1.5]
\node (D) at (-4,0.5) {$V_1$};
\node (E) at (-2,0.5) {$V_2$};
\node (F) at (-4,-1) {$W_1$};
\node (A) at (-2,-1) {$W_2$};
\node (B) at (-3,-0.25) {};
\path[->,font=\scriptsize,>=angle 90]
(D) edge node [above]{$R \subseteq V_1 \oplus V_2$}(E)
(E) edge node [right]{$g$}(A)
(D) edge node [left]{$f$}(F)
(F) edge node [above]{$S \subseteq W_1 \oplus W_2$} (A);
\end{tikzpicture}
\]
obeying $(f \oplus g)R \subseteq S$ as 2-morphisms. 
\end{enumerate}
The last item deserves some explanation.  A preorder is a category such that for any 
pair of objects $x,y$ there exists at most one morphism $\alpha \maps x \to y$.  When such
a morphism exists we usually write $x \le y$.  Similarly there is a kind of double category for which  given any `frame'
\[
\begin{tikzpicture}[scale=1]
\node (D) at (-4,0.5) {$A$};
\node (E) at (-2,0.5) {$B$};
\node (F) at (-4,-1) {$C$};
\node (A) at (-2,-1) {$D$};
\node (B) at (-3,-0.25) {};
\path[->,font=\scriptsize,>=angle 90]
(D) edge node [above]{$M$}(E)
(E) edge node [right]{$g$}(A)
(D) edge node [left]{$f$}(F)
(F) edge node [above]{$N$} (A);
\end{tikzpicture}
\]
there exists at most one 2-morphism 
\[
\begin{tikzpicture}[scale=1]
\node (D) at (-4,0.5) {$A$};
\node (E) at (-2,0.5) {$B$};
\node (F) at (-4,-1) {$C$};
\node (A) at (-2,-1) {$D$};
\node (B) at (-3,-0.25) {$\Downarrow \alpha$};
\path[->,font=\scriptsize,>=angle 90]
(D) edge node [above]{$M$}(E)
(E) edge node [right]{$g$}(A)
(D) edge node [left]{$f$}(F)
(F) edge node [above]{$N$} (A);
\end{tikzpicture}
\]
filling this frame. For lack of a better term let us call this a \define{degenerate} double category.  
Item (4) implies that $\LLinRel$ will be degenerate in this sense.

In $\LLinRel$, composition of vertical 1-morphisms is the usual composition of linear maps, while  
composition of horizontal 1-cells is the usual composition of linear relations.  Since composition of
linear relations obeys the associative and unit laws strictly, $\LLinRel$ will be a \emph{strict} double category.   Since $\LLinRel$ is degenerate, there is at most one way to define the vertical composite of 2-morphisms 
\[

\]
so we need merely check that a 2-morphism $\beta\alpha$ filling the frame at right exists.  This amounts to noting that
\[     (f \oplus g)R \subseteq S, \; (f' \oplus g')S \subseteq T \; \implies \;
(f' \oplus g')(f \oplus g)R \subseteq T .\]
Similarly, there is at most one way to define the horizontal composite of 2-morphisms
\[
%
\]
so we need merely check that a filler $\alpha' \circ \alpha$ exists, which amounts to noting that
\[   (f \oplus g)R \subseteq S , \; (g \oplus h)R' \subseteq S' \; \implies \;
(f \oplus h)(R'R) \subseteq S'S . \]

\begin{theorem}
\label{thm:LLinRel}
There exists a strict double category $\LLinRel$ with the above properties.
\end{theorem}

\begin{proof}
The category of objects $\LLinRel_0$ has finite-dimensional real vector spaces as objects and linear maps as morphisms.  The category of arrows $\LLinRel_1$ has linear relations as objects 
and squares 
\[
%
\]
with $(f \oplus g)R \subseteq S$ as morphisms.  The source and target functors $s,t \maps \LLinRel_1 \to \LLinRel_0$ are clear. The identity-assigning functor $u  \maps \LLinRel_0 \to \LLinRel_1$ sends a finite-dimensional real vector space $V$ to the identity map $1_V$ and a linear map $f \maps V \to W$ to the unique 2-morphism
\[
%
\]
The composition functor $\odot \maps \LLinRel_1 \times_{\LLinRel_0} \LLinRel_1 \to \LLinRel_1$
acts on objects by the usual composition of linear relations, and it acts on 2-morphisms by horizontal composition as described above.  These functors can be shown to obey all the axioms of a double category.   In particular, because $\LLinRel$ is degenerate, all the required equations between 2-morphisms, such as the interchange law, hold automatically.
\end{proof}

Next we make $\LLinRel$ into a symmetric monoidal double category.  To do this, we first give $\LLinRel_0$ the structure of a symmetric monoidal category.  We do this using a specific choice of direct sum for each pair of finite-dimensional real vector spaces as the tensor product, and a specific 0-dimensional vector space as the unit object.   Then we give $\LLinRel_1$ a symmetric monoidal structure as follows.  Given linear relations $R_1 \subseteq V_1 \oplus W_1$ and $R_2 \subseteq V_2 \oplus W_2$, we define their direct sum by
\[  R_1 \oplus R_2 = \{(v_1,v_2,w_1,w_2) : \; (v_1,w_1) \in R_1, (v_2,w_2) \in R_2 \} \subseteq V_1 \oplus V_2 \oplus W_1 \oplus W_2. \]
Given two 2-morphisms in $\LLinRel_1$:
\[
\begin{tikzpicture}[scale=1.5]
\node (D) at (-4,0.5) {$V_1$};
\node (E) at (-2,0.5) {$V_2$};
\node (F) at (-4,-1) {$W_1$};
\node (A) at (-2,-1) {$W_2$};
\node (D') at (-0.5,0.5) {$V'_1$};
\node (E') at (1.5,0.5) {$V'_2$};
\node (F') at (-0.5,-1) {$W'_1$};
\node (A') at (1.5,-1) {$W'_2$};
\node (B') at (0.5,-0.25) {$\Downarrow \alpha'$};
\node (B) at (-3,-0.25) {$\Downarrow \alpha$};
\path[->,font=\scriptsize,>=angle 90]
(D) edge node [above]{$R \subseteq V_1 \oplus V_2$}(E)
(E) edge node [right]{$g$}(A)
(D) edge node [left]{$f$}(F)
(F) edge node [above]{$S \subseteq W_1 \oplus W_2$} (A)
(D') edge node [above]{$R' \subseteq V'_1 \oplus V'_2$}(E')
(E') edge node [right]{$g'$}(A')
(D') edge node [left]{$f'$}(F')
(F') edge node [above]{$S' \subseteq W'_1 \oplus W'_2$} (A');
\end{tikzpicture}
\]
there is at most one way to define their direct sum
\[
\begin{tikzpicture}[scale=1.5]
\node (D) at (-3.5,0.5) {$V_1 \oplus V'_1$};
\node (E) at (.5,0.5) {$V_2 \oplus V'_2$};
\node (F) at (-3.5,-1) {$W_1 \oplus W'_1$};
\node (A) at (.5,-1) {$W_2 \oplus W'_2$};
\node (B') at (-1.5,-0.25) {$\Downarrow \alpha \oplus \alpha'$};
\path[->,font=\scriptsize,>=angle 90]
(D) edge node [above]{$R \oplus R' \subseteq V_1 \oplus V'_1 \oplus V_2 \oplus V'_2$}(E)
(E) edge node [right]{$g \oplus g'$}(A)
(D) edge node [left]{$f \oplus f'$}(F)
(F) edge node [above]{$S \oplus S' \subseteq W_1 \oplus W'_1 \oplus W_2 \oplus W'_2$} (A);
\end{tikzpicture}
\]
because $\LLinRel$ is degenerate.   To show that $\alpha \oplus \alpha'$ exists, we need
merely note that
\[    (f \oplus g) R \subseteq S, \; (f' \oplus g')R' \subseteq S' \; \implies \; 
(f \oplus f' \oplus g \oplus g') (R \oplus R') \subseteq S \oplus S'.  \]

\begin{theorem}
\label{thm:LLinRel2}
The double category $\LLinRel$ can be given the structure of a symmetric monoidal double category with the above properties.
\end{theorem}

\begin{proof}
We have described $\LLinRel_0$ and $\LLinRel_1$ as symmetric monoidal categories.
The source and target functors $s,t \maps \LLinRel_1 \to \LLinRel_0$ are strict symmetric
monoidal functors.  The required globular 2-isomorphisms $\chi$ and $\mu$ are defined as
follows.   Given four horizontal 1-cells 
\[   R_1 \subseteq U_1 \oplus V_1, \quad R_2 \subseteq V_1 \oplus W_1, \]
\[ S_1 \subseteq U_2 \oplus V_2, \quad S_2 \subseteq V_2 \oplus W_2, \]
the globular 2-isomorphism $\chi \maps (R_2 \oplus S_2)(R_1 \oplus S_1) \Rightarrow (R_2 R_1) \oplus (S_2 S_1)$ is the identity 2-morphism
\[
\begin{tikzpicture}[scale=1.5]
\node (D) at (-4,0.5) {$U_1 \oplus U_2$};
\node (E) at (-1,0.5) {$W_1 \oplus W_2$};
\node (F) at (-4,-1) {$U_1 \oplus U_2$};
\node (A) at (-1,-1) {$W_1 \oplus W_2$.};
\node (B') at (-2.5,-0.25) {$$};
\path[->,font=\scriptsize,>=angle 90]
(D) edge node [above]{$(R_2 \oplus S_2)(R_1 \oplus S_1)$}(E)
(E) edge node [right]{$1$}(A)
(D) edge node [left]{$1$}(F)
(F) edge node [above]{$(R_2 R_1) \oplus (S_2 S_1)$} (A);
\end{tikzpicture}
\]
The globular 2-isomorphism $\mu \maps u(V \oplus W) \Rightarrow u(V) \oplus u(W)$ is
the identity 2-morphism
\[
\begin{tikzpicture}[scale=1.5]
\node (D) at (-4,0.5) {$V \oplus W$};
\node (E) at (-2,0.5) {$V \oplus W$};
\node (F) at (-4,-1) {$V \oplus W$};
\node (A) at (-2,-1) {$V \oplus W$.};
\node (B') at (-3,-0.25) {$$};
\path[->,font=\scriptsize,>=angle 90]
(D) edge node [above]{$1_{V \oplus W} $}(E)
(E) edge node [right]{$1$}(A)
(D) edge node [left]{$1$}(F)
(F) edge node [above]{$1_V \oplus 1_W$} (A);
\end{tikzpicture}
\]
All the commutative diagrams in the definition of symmetric monoidal double category \cite{Shul} can be checked straightforwardly.   In particular, all diagrams of
2-morphisms commute automatically because $\LLinRel$ is degenerate.  
\end{proof}

\subsection{A bicategory of linear relations}
We can also promote the symmetric monoidal double category $\mathbb{L}\mathbf{inRel}$ of linear relations from the previous section to a symmetric monoidal bicategory $\mathbf{LinRel}$ of linear relations due to Shulman's Theorem \ref{Shulman1} by showing $\mathbb{L}\mathbf{inRel}$ is isofibrant.

\begin{lemma}
The symmetric monoidal double category $\LLinRel$ is isofibrant.
\end{lemma}

\begin{proof}
Let $f \maps X \to Y$ be a linear isomorphism between finite-dimensional real vector spaces. Define $\hat{f}$ to be the linear relation given by the linear isomorphism $f$ and define 2-morphisms in $\LLinRel$
\[
\begin{tikzpicture}[scale=1.5]
\node (D) at (-4,0.5) {$X$};
\node (F) at (-2,0.5) {$Y$};
\node (A) at (-4,-1) {$Y$};
\node (B) at (-2,-1) {$Y$};
\node (D') at (0,0.5) {$X$};
\node (F') at (2,0.5) {$X$};
\node (A') at (0,-1) {$X$};
\node (B') at (2,-1) {$Y$};
\node (C) at (-3,-0.25) {$\alpha_{f} \Downarrow$};
\node (C') at (1,-0.25) {${}_{f}\alpha \Downarrow$};
\path[->,font=\scriptsize,>=angle 90]
(D) edge node [above]{$\hat{f}$}(F)
(D) edge node [left]{$f$}(A)
(F) edge node [right]{$1$}(B)
(A) edge node [above]{$1$} (B)
(D') edge node [above]{$1$}(F')
(D') edge node [left]{$1$}(A')
(F') edge node [right]{$f$}(B')
(A') edge node [above]{$\hat{f}$} (B');
\end{tikzpicture}
\]
where $\alpha_{f}$ and ${}_{f} \alpha$, the unique fillers of their frames, are identities. These two 2-morphisms and $\hat{f}$ satisfy the required equations, and the conjoint of $f$ is given by reversing the direction of $\hat{f}$, which is just $f^{-1} \maps Y \to X$. It follows that $\LLinRel$ is isofibrant.
\end{proof}

\begin{theorem}
There exists a symmetric monoidal bicategory $\bold{LinRel}$ with
\begin{enumerate}
\item finite-dimensional real vector spaces as objects,
\item linear relations $R \subseteq V \oplus W$ as morphisms from $V$ to $W$,
\item inclusions $R \subseteq S$ between linear relations $R,S \subseteq V \oplus W$ as 2-morphisms.
\end{enumerate}
\end{theorem}

\begin{proof}
Apply Shulman's result, Theorem\ \ref{Shulman1}, to the isofibrant symmetric monoidal double category $\LLinRel$ to obtain the symmetric monoidal bicategory $\bold{LinRel}$ as the horizontal edge bicategory of $\LLinRel$.
\end{proof}

\section{Black-boxing for open Markov processes}\label{Black}

In this section we present the main result of the chapter which is a symmetric monoidal double functor $\blacksquare \maps \MMark \to \LLinRel$. We proceed as follows:
\begin{enumerate}
\item On objects: for a finite set $S$, we define $\blacksquare(S)$ to be the vector space $\R^S \oplus \R^S$.
\item On horizontal 1-cells: for an open Markov process $S \stackrel{i}{\rightarrowtail} (X,H) \stackrel{o}{\leftarrowtail} T$, we define its black-boxing as in Definition \ref{defn:black-boxing}:
\[ \blacksquare(S \stackrel{i}{\rightarrowtail} (X,H) \stackrel{o}{\leftarrowtail} T) = \]
\[ \{ (i^*(v),I,o^*(v),O) : \; v \in \R^X,  I \in \R^S, O \in \R^T \, 
\textrm{ and } \, H(v) + i_*(I) - o_*(O) = 0\}. \]
\item On vertical 1-morphisms: for a map $f \maps S\to S'$, we define $\blacksquare(f) \maps \R^S \oplus \R^S \to \R^{S'} \oplus \R^{S'}$ to be the linear map $f_* \oplus f_*$.
\end{enumerate}

What remains to be done is define how $\blacksquare$ acts on 2-morphisms of $\MMark$.   This describes the relation between steady state input and output concentrations and flows of a coarse-grained open Markov process in terms of the corresponding relation for the original process:

\begin{lemma}
\label{lem:black-boxing_2-morphisms}
Given a 2-morphism 
\[
\begin{tikzpicture}[scale=1.5]
\node (D) at (-4.5,0.5) {$S$};
\node (E) at (-3,0.5) {$(X,H)$};
\node (F) at (-1.5,0.5) {$T$};
\node (G) at (-3,-1) {$(X',H')$};
\node (A) at (-4.5,-1) {$S'$};
\node (B) at (-1.5,-1) {$T'$,};
\path[->,font=\scriptsize,>=angle 90]
(D) edge node [left]{$f$}(A)
(F) edge node [right]{$g$}(B)
(D) edge[>->] node [above] {$i$}(E)
(F) edge[>->] node [above] {$o$}(E)
(A) edge[>->] node[left] [above] {$i'$} (G)
(B) edge[>->] node[left] [above] {$o'$} (G)
(E) edge node[left] {$p$}(G);
\end{tikzpicture}
\]
in $\MMark$, there exists a (unique) 2-morphism
\[
\begin{tikzpicture}[scale=1.5]
\node (D) at (-4.5,0.5) {$\blacksquare(S)$};
\node (E) at (-1.5,0.5) {$\blacksquare(T)$};
\node (F) at (-4.5,-1) {$\blacksquare(S')$};
\node (A) at (-1.5,-1) {$\blacksquare(T')$};
\path[->,font=\scriptsize,>=angle 90]
(D) edge node [above]{$\blacksquare(S \stackrel{i}{\rightarrowtail} (X,H) \stackrel{o}{\leftarrowtail} T)$}(E)
(E) edge node [right]{$\blacksquare(g)$}(A)
(D) edge node [left]{$\blacksquare(f)$}(F)
(F) edge node [above]{$ \blacksquare(S' \stackrel{i'}{\rightarrowtail} (X',H') \stackrel{o'}{\leftarrowtail} T')$} (A);
\end{tikzpicture}
\]
\vskip -0.5em \noindent 
in $\LLinRel$.
\end{lemma}

\begin{proof}  
Since $\LLinRel$ is degenerate, if there exists a 2-morphism of the claimed kind it is
automatically unique.  To prove that such a 2-morphism exists, it suffices to prove
\[    (i^*(v),I,o^*(v),O) \in V \; \implies \;  (f_* i^*(v),f_*(I), g_* o^*(v),g_*(O)) \in W \]
where
\[  V = \blacksquare(S \stackrel{i}{\rightarrowtail} (X,H) \stackrel{o}{\leftarrowtail} T) = \]
\[  \{ (i^*(v),I,o^*(v),O) : \; v \in \R^X,  I \in \R^S, O \in \R^T \, 
\textrm{ and } \, H(v) + i_*(I) - o_*(O) = 0\} \]
and
\[ W = \blacksquare(S' \stackrel{i'}{\rightarrowtail} (X',H') \stackrel{o'}{\leftarrowtail} T') = \]
\[ \{ (i'^*(v'),I',o'^*(v'),O') : \; v' \in \R^{X'},  I' \in \R^{S'}, O' \in \R^{T'} \, 
\textrm{ and } \, H'(v') + i'_*(I') - o'_*(O') = 0\} .\]

To do this, assume $  (i^*(v),I,o^*(v),O) \in V$, which implies that
\begin{equation}
  H(v) + i_*(I) - o_*(O) = 0 .    \label{eq:master_1}
\end{equation}
Since the commuting squares in $\alpha$ are pullbacks,  Lemma \ref{lem:beck-chevalley} implies that
\[
     f_* i^* = i'^* p_*    , \qquad
     g_* o^* = o'^* p_*  . 
\]
Thus
\[     (f_* i^*(v),f_*(I), g_* o^*(v),g_*(O)) = (i'^* p_*(v), f_*(I), o'^* p_*(v),g_*(O) )  \]
and this is an element of $W$ as desired if
\begin{equation}   
    H' p_*(v) + i'_* f_*(I) - o'_* g_*(O) = 0 .   \label{eq:master_2} 
\end{equation}
To prove Equation \eqref{eq:master_2}, note that
\[
\begin{array}{ccl}
   H' p_*(v) + i'_* f_*(I) - o'_* g_*(O) &=& p_*H (v) + p_* i_*(I) - p_* o_* (O)  \\ 
                                                          &=& p_*(H(v) + i_*(I) - o_*(O)) 
\end{array}
\]
where in the first step we use the fact that the squares in $\alpha$ commute, together with the fact that $H' p_* = p_* H$.   Thus, Equation \eqref{eq:master_1} implies Equation \eqref{eq:master_2}. 
\end{proof}

The following result is a special case of a result by Pollard and Baez on black-boxing open dynamical systems \cite{BP}.   To make this chapter self-contained we adapt the proof
to the case at hand:

\begin{lemma}
\label{lem:black-boxing_symmoncat}
The black-boxing of a composite of two open Markov processes equals the composite of
their black-boxings.
\end{lemma}

\begin{proof}
 Consider composable open Markov processes  
\[         S \stackrel{i}\longrightarrow (X,H) \stackrel{o}\longleftarrow T, \qquad
           T \stackrel{i'}\longrightarrow (Y,G) \stackrel{o'}\longleftarrow U .\]
To compose these, we first form the pushout
\[
    \xymatrix{
      && X +_T Y \\
      & X \ar[ur]^{j} && Y \ar[ul]_{k} \\
      \quad S\quad \ar[ur]^{i} && T \ar[ul]_{o} \ar[ur]^{i'} &&\quad U \quad \ar[ul]_{o'}
    }
\]
Then their composite is
\[ S \stackrel{j i}{\longrightarrow} (X +_T Y , H \odot G) \stackrel{k o'}{\longleftarrow} U \]
where 
\[    H \odot G = j_* H j^* + k_* G  k^*  .\]

To prove that $\blacksquare$ preserves composition, we first show that 
\[ \blacksquare(Y,G) \; \blacksquare(X,H) \subseteq \blacksquare(X+_T Y,H \odot G) .\]  
Thus, given
\[     (i^*(v),I,o^*(v),O) \in \blacksquare(X,H), \qquad  ({i'}^*(v'),I',{o'}^*(v'),O') \in \blacksquare(Y,G) \]
with 
\[   o^*(v) = {i'}^*(v'), \qquad O = I' \]
we need to prove that 
\[     (i^*(v),I,{o'}^*(v'),O') \in \blacksquare(X +_T Y, H \odot G). \]
To do this, it suffices to find $w \in \R^{X +_T Y}$ such that 
\[   (i^*(v),I,{o'}^*(v'),O') = ((ji)^*(w), I, {(ko')}^*(w), O') \]
and $w$ is a steady state of $(X+_T Y,H \odot G)$ with inflows $I$ and outflows $O'$.

Since  $o^*(v) = {i'}^*(v'),$ this diagram commutes:
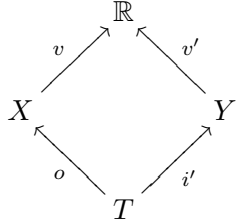
\[
    \xymatrix{
      && \R \\
      & X \ar[ur]^{v} && Y \ar[ul]_{v'} \\
       && T \ar[ul]^{o} \ar[ur]_{i'} &&
    }
\]
so by the universal property of the pushout there is a unique map $w \maps X +_T Y \to \R$ such that this commutes:
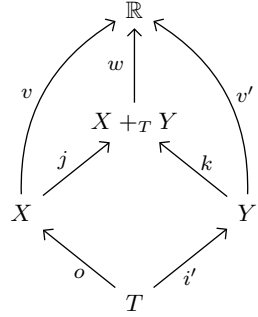
\begin{equation}\label{eq:pushout}
  \begin{tikzpicture}[scale=1.2]
\node (D) at (-4.25,-0.5) {$X$};
\node (E) at (-3,.5) {$X+_T Y$};
\node (F) at (-1.75,-0.5) {$Y$};
\node (A) at (-3,-1.5) {$T$};
\node (B) at (-3,1.75) {$\R$};
\path[->,font=\scriptsize,>=angle 90]
(E) edge node [left] {$w$} (B)
(A) edge node [below] {$i'$} (F)
(A) edge node [below] {$o$} (D)
(F) edge [bend right] node[right] {$v'$} (B)
(D) edge [bend left] node[left]{$v$} (B)
(D) edge node[above,pos=0.3] {$j$}(E)
(F) edge node[above,pos=0.3] {$k$}(E);
\end{tikzpicture}
\end{equation}
This simply says that because the functions $v$ and $v'$ agree on the `overlap' of our two
open Markov processes, we can find a function $w$ that restricts
to $v$ on $X$ and $v'$ on $Y$.

We now prove that $w$ is a steady state of the composite open Markov process with
inflows $I$ and outflows $O'$:
\begin{equation}
   \label{eq:steady_state_3}
   (H \odot G)(w) + (ji)_*(I) - (ko')_*(O') = 0.
\end{equation}
To do this we use the fact that $v$ is a steady state of $S \stackrel{i}\rightarrowtail (X,H) \stackrel{o}\leftarrowtail T$ with inflows $I$ and outflows $O$:
\begin{equation}\label{eq:steady_state_1}
   H(v) + i_*(I) - o_*(O) = 0
\end{equation}
and $v'$ is a steady state of $T \stackrel{i'}\rightarrowtail (Y,G) \stackrel{o'}\leftarrowtail U$ with inflows $I'$ and outflows $O'$:
\begin{equation}\label{eq:steady_state_2}
   G(v') + i'_*(I') - o'_*(O') = 0.
\end{equation}
We push Equation \eqref{eq:steady_state_1} forward along $j$, push Equation \eqref{eq:steady_state_2} forward along $k$, and sum them:
\[   j_*(H(v))  + (ji)_*(I) - (jo)_*(O) + k_*(G(v')) + (ki')_*(I') - (ko')_*(O') = 0. \]
Since $O = I'$ and $jo = ki'$, two terms cancel, leaving us with
\[     j_*(H(v))  + (ji)_*(I) + k_*(G(v')) - (ko')_*(O') = 0. \]
Next we combine the terms involving the infinitesimal stochastic operators $H$ and $G$, with the help of Equation \eqref{eq:pushout} and the definition of $H \odot G$:
\begin{equation}
\label{eq:u}
   \begin{array}{ccl}
  j_*(H(v)) + k_*(G(v')) &=& (j_* H j^* + k_* G k^*)(w) \\
                                    &=& (H \odot G)(w)  .
\end{array}
\end{equation}
This leaves us with
\[         (H \odot G)(w) +  (ji)_*(I) - (ko')_*(O') = 0 \]
which is Equation \eqref{eq:steady_state_3}, precisely what we needed to show.

To finish showing that $\blacksquare$ is a functor, we need to show that 
\[   \blacksquare(X+_T Y,H \odot G) \subseteq \blacksquare(Y,G) \; \blacksquare(X,H)  .\] 
So, suppose we have 
\[    ((ji)^*(w), I, {(ko')}^*(w), O') \in \blacksquare(X+_T Y,H \odot G) .\]
We need to show
\begin{equation}
\label{eq:composite}
  ((ji)^*(w), I, {(ko')}^*(w), O') = (i^*(v),I,o'^*(v'),O')
\end{equation}
where 
\[     (i^*(v),I,o^*(v),O) \in \blacksquare(X,H), \qquad  (i'^*(v'),I',o'^*(v'),O') \in \blacksquare(Y,G) \]
and
\[   o^*(v) = i'^*(v'), \qquad O = I' .\]

To do this, we begin by choosing
\[   v = j^*(w), \qquad v' = k^*(w) .\]
This ensures that Equation \eqref{eq:composite} holds, and since $jo = ki'$, it also ensures that 
\[  o^*(v) = (jo)^*(w) = (ki')^*(w) = {i'}^*(v')  .\]
To finish the job, we need to find an element $O = I' \in \R^T$ such that $v$ is a steady state of $(X,H)$ with inflows $I$ and outflows $O$ and $v'$ is a steady state of $(Y,G)$ with inflows $I'$ and outflows $O'$.  Of course, we are given the fact that $w$ is a steady state of $(X+_T Y,H \odot G)$ with inflows $I$ and outflows $O'$.   

In short, we are given Equation \eqref{eq:steady_state_3}, and we seek $O = I'$ such that Equations \eqref{eq:steady_state_1} and \eqref{eq:steady_state_2} hold.  Thanks to our choices of $v$ and $v'$,  we can use Equation \eqref{eq:u} and rewrite Equation \eqref{eq:steady_state_3} as
\begin{equation}
\label{eq:steady_state_3'}
  j_*(H(v) + i_*(I)) \; + \; k_*(G(v') - o'_*(O')) = 0 .  
\end{equation}
Equations \eqref{eq:steady_state_1} and \eqref{eq:steady_state_2} say that
\begin{equation}
\label{eq:steady_state_1'2'}
\begin{array}{lcl}
   H(v) + i_*(I) - o_*(O) &=& 0 \\ 
   G(v') + i'_*(I') - o'_*(O') &=& 0.
\end{array}
\end{equation}

Now we use the fact that 
\[
    \xymatrix{
      & X +_T Y \\
       X \ar[ur]^{j} && Y \ar[ul]_{k} \\
       & T \ar[ul]^{o} \ar[ur]_{i'} &
    }
\]
is a pushout.  Applying the `free vector space on a finite set' functor, which preserves colimits, this implies that
\[
    \xymatrix{
      & \R^{X +_T Y} \\
       \R^X \ar[ur]^{j_*} && \R^{Y} \ar[ul]_{k_*} \\
       & \R^T \ar[ul]^{o_*} \ar[ur]_{i'_*} &
    }
\]
is a pushout in the category of vector spaces.   Since a pushout is formed by taking first a coproduct and then a coequalizer, this implies that 
\[
     \xymatrix{
      \R^T \ar@<-.5ex>[rr]_-{(0,i'_*)} \ar@<.5ex>[rr]^-{(o_*,0)} && \R^X \oplus \R^{Y} \ar[rr]^{[j_*,k_*]}
   && \R^{X +_T Y}
}
\]
is a coequalizer.  Thus, the kernel of $[j_*,k_*]$ is the image of $(o_*,0) - (0,i'_*)$.   Equation \eqref{eq:steady_state_3'} says precisely that 
\[    (H(v) + i_*(I), G(v') - o'_*(O')) \in \ker([j_*,k_*])  .\]
Thus, it is in the image of $(o_*,0) - (0, i'_*)$.  In other words, there exists some element $O = I' \in \R^T$
such that 
\[   (H(v) + i_*(I), G(v') - o'_*(O')) = (o_*(O), -i'_*(I')).\]
This says that Equations \eqref{eq:steady_state_1} and \eqref{eq:steady_state_2} hold, as desired.
\end{proof}

This is the main result of the paper on coarse-graining open Markov processes \cite{BC}:

\begin{theorem}
\label{thm:main}
There exists a symmetric monoidal double functor $\blacksquare \maps \MMark \to \LLinRel$ with the following behavior:

\begin{enumerate}
\item Objects: $\blacksquare$ sends any finite set $S$ to the vector space $\R^{S} \oplus \R^{S}$.
\item Vertical 1-morphisms: $\blacksquare$ sends any map $f \maps S \to S'$ to the linear
map \hfill \break  $f_* \oplus f_* \maps \R^{S} \oplus \R^S \to \R^{S'} \oplus \R^{S'}$.
\item Horizontal 1-cells: $\blacksquare$ sends any open Markov process $S \stackrel{i}{\rightarrowtail} (X,H) \stackrel{o}{\leftarrowtail} T$ to the linear relation given in Definition \ref{defn:black-boxing}:
\[ \blacksquare(S \stackrel{i}{\rightarrowtail} (X,H) \stackrel{o}{\leftarrowtail} T) = \]
\[ \{ (i^*(v),I,o^*(v),O) : \; H(v) + i_*(I) - o_*(O) = 0 \textrm{ for some } I \in \R^S, v \in \R^X, O \in \R^T \}. \]
\item 2-Morphisms: $\blacksquare$ sends any morphism of open Markov processes
\[
\begin{tikzpicture}[scale=1.5]
\node (D) at (-4.5,0.5) {$S$};
\node (E) at (-3,0.5) {$(X,H)$};
\node (F) at (-1.5,0.5) {$T$};
\node (G) at (-3,-1) {$(X',H')$};
\node (A) at (-4.5,-1) {$S'$};
\node (B) at (-1.5,-1) {$T'$};
\path[->,font=\scriptsize,>=angle 90]
(D) edge node [left]{$f$}(A)
(F) edge node [right]{$g$}(B)
(D) edge[>->] node [above] {$i$}(E)
(F) edge[>->] node [above] {$o$}(E)
(A) edge[>->] node[above] {$i'$} (G)
(B) edge[>->] node[above] {$o'$} (G)
(E) edge node[left] {$p$}(G);
\end{tikzpicture}
\]
to the 2-morphism in $\LLinRel$ given in Lemma \ref{lem:black-boxing_2-morphisms}:
\[
\begin{tikzpicture}[scale=1.5]
\node (D) at (-4.5,0.5) {$\blacksquare(S)$};
\node (E) at (-1.5,0.5) {$\blacksquare(T)$};
\node (F) at (-4.5,-1) {$\blacksquare(S')$};
\node (A) at (-1.5,-1) {$\blacksquare(T')$.};
\path[->,font=\scriptsize,>=angle 90]
(D) edge node [above]{$\blacksquare(S \stackrel{i}{\rightarrowtail} (X,H) \stackrel{o}{\leftarrowtail} T)$}(E)
(E) edge node [right]{$\blacksquare(g)$}(A)
(D) edge node [left]{$\blacksquare(f)$}(F)
(F) edge node [above]{$ \blacksquare(S' \stackrel{i'}{\rightarrowtail} (X',H') \stackrel{o'}{\leftarrowtail} T')$} (A);
\end{tikzpicture}
\]
\end{enumerate}
\end{theorem}

\begin{proof}
First we must define functors $\blacksquare_{0} \maps \MMark_0 \to \LLinRel_0$ and $\blacksquare_1 \maps \MMark_{1} \to \LLinRel_{1}$. The functor $\blacksquare_0$ is defined on finite sets and maps between these as described in (i) and (ii) of the theorem statement, while $\blacksquare_1$ is defined on open Markov processes and morphisms between these as described in (iii) and (iv).  Lemma \ref{lem:black-boxing_2-morphisms} shows that $\blacksquare_1$ is well-defined on morphisms between open Markov processes; given this is it easy to check that $\blacksquare_1$ is a functor.  One can verify that $\blacksquare_0$ and $\blacksquare_1$ combine to define a double functor $\blacksquare \maps \MMark \to \LLinRel$: the hard part is checking that horizontal composition of open Markov processes is preserved, but this was shown in Lemma \ref{lem:black-boxing_symmoncat}.  Horizontal composition of 2-morphisms is automatically preserved because $\LLinRel$ is degenerate.

To make $\blacksquare$ into a symmetric monoidal double functor we need to make
$\blacksquare_0$ and $\blacksquare_1$ into symmetric monoidal functors, which we do using these extra structures:
\begin{itemize}
\item an isomorphism in $\LLinRel_0$ between $\{0\}$ and $\blacksquare(\emptyset)$,
\item a natural isomorphism between $\blacksquare(S)  \oplus \blacksquare(S')$ and 
 $\blacksquare(S + S')$ for any two objects $S,S' \in \MMark_0$,
\item an isomorphism in $\LLinRel_1$ between the unique linear relation $\{0\} \to \{0\}$ and
$\blacksquare(\emptyset \rightarrowtail (\emptyset, 0_\emptyset) \leftarrowtail \emptyset)$, and
\item a natural isomorphism between 
\[
\blacksquare((S \rightarrowtail (X,H) \leftarrowtail T) \; \oplus \; \blacksquare(S' \rightarrowtail (X',H') \leftarrowtail T')  \]
and
\[  \blacksquare(S + S' \rightarrowtail (X+X',H \oplus H') \leftarrowtail T + T') \]
for any two objects $S \rightarrowtail (X,H) \leftarrowtail T$, $S' \rightarrowtail (X',H') \leftarrowtail T'$ of $\MMark_1$.
\end{itemize}
There is an evident choice for each of these extra structures, and it is straightforward to check
that they not only  make $\blacksquare_0$ and $\blacksquare_1$ into symmetric monoidal functors but also meet the extra requirements for a symmetric monoidal double functor listed in Hansen and Shulman's
paper \cite{Shul3}, which may also be found in Definition \ref{defn:monoidal_double_functor}.  In particular, all diagrams of 2-morphisms commute automatically because $\LLinRel$ is degenerate.
\end{proof}

\subsection{A corresponding functor of bicategories}
We have symmetric monoidal bicategories $\bold{Mark}$ and $\bold{LinRel}$, both of which come from discarding the vertical 1-morphisms of the symmetric monoidal double categories $\MMark$ and $\LLinRel$, respectively. Morally, we should be able to do something similar to the symmetric monoidal double functor $\blacksquare \maps \MMark \to \LLinRel$ to obtain a symmetric monoidal functor of bicategories $\blacksquare \maps \bold{Mark} \to \bold{LinRel}$, and indeed we can by a result of Hansen and Shulman \cite{Shul3}.

\begin{theorem}[{{\cite[Thm.\ 6.17]{Shul3}}}]\label{blackbox_functor_bicategories}
There exists a symmetric monoidal functor $\blacksquare \maps \bold{Mark} \to \bold{LinRel}$ that maps:
\begin{enumerate}
\item any finite set $S$ to the finite-dimensional real vector space $\blacksquare(S)=\R^S \oplus \R^S$,
\item any open Markov process $S \stackrel{i}{\rightarrowtail} (X,H) \stackrel{o}{\leftarrowtail} T$ to the linear relation from $\blacksquare(S)$ to $\blacksquare(T)$ given by the linear subspace 
\[ \blacksquare(S \stackrel{i}{\rightarrowtail} (X,H) \stackrel{o}{\leftarrowtail} T) = \]
\[  \{ (i^*(v),I,o^*(v),O) : \; H(v) + i_*(I) - o_*(O) = 0 \} \subseteq \R^S \oplus \R^S \oplus \R^T \oplus \R^T ,\]
\item any morphism of open Markov processes 
\[
\begin{tikzpicture}[scale=1.5]
\node (D) at (-4.2,-0.5) {$S$};
\node (A) at (-4.2,-2) {$S$};
\node (B) at (-1.8,-2) {$T$};
\node (E) at (-3,-0.5) {$(X,H)$};
\node (F) at (-1.8,-0.5) {$T$};
\node (G) at (-3,-2) {$(X',H')$};
\path[->,font=\scriptsize,>=angle 90]
(D) edge[>->] node[above] {$i_1$}(E)
(A) edge[>->] node[above] {$i'_1$} (G)
(B) edge[>->] node [above]{$o'_1$} (G)
(F) edge[>->] node [above]{$o_1$}(E)
(D) edge node [left]{$1_S$}(A)
(F) edge node [right]{$1_T$}(B)
(E) edge node[left] {$p$}(G);
\end{tikzpicture}
\]
to the inclusion
\[   \blacksquare(X,H) \subseteq \blacksquare(X',H') .\]
\end{enumerate}
\end{theorem}

\begin{proof}
This was proved by Hansen and Shulman \cite[Theorem 6.17]{Shul3}, by applying a more general result \cite[Theorem 5.11]{Shul3} to the strong symmetric monoidal double functor $\blacksquare \colon \mathbb{M}\mathbf{ark} \to \mathbb{L}\mathbf{inRel}$ of Theorem \ref{thm:main}.
\end{proof}
}

{\ssp
\chapter{Possible future work}

In this final chapter before the \ref{Appendix}ppendix, I will touch on a few possible avenues in which the work in this thesis can be improved. The three main results are the contents of Chapter \ref{Chapter3}, Chapter \ref{Chapter4} and Chapter \ref{Chapter6}.

Chapter \ref{Chapter3} presents the results regarding the foot-replaced double categories formalism. 
We showed how to build a symmetric monoidal double category $_L \mathbb{C}\define{sp}(X)$ from an
adjoint functor $L: A \to X$ between categories with finite colimits. One possible generalization would be to let
$L$ be a `2-adjoint' between two 2-categories $\mathbf{A}$ and $\mathbf{X}$ with finite `2-colimits'. In the conjectured symmetric monoidal double category $_L \mathbb{C} \define{sp}(\mathbf{X})$ obtained from this 2-adjoint $L$, composing two horizontal 1-cells---two cospans in $\mathbf{X}$---would involve taking `2-pushouts', which involve the typical pushout square commuting not on the nose but only up to isomorphism.

We can also generalize foot-replaced double categories. The idea of replacing the category of objects of a double category $\mathbb{X}$ with some other category $\mathsf{A}$ is easily transferable to even higher level categorifications. For example, if $\mathbb{X}$ is a `triple category', which would involve a category $\mathbb{X}_0$ of objects, a category $\mathbb{X}_1$ of arrows and a category $\mathbb{X}_2$ of `faces', we could replace the category of objects $\mathbb{X}_0$ with some other category $\mathsf{A}$, or even both the category of objects $\mathbb{X}_0$ and category of arrows $\mathbb{X}_1$ with some \emph{double category} $\mathbb{A}$ in the event that the pair $(\mathbb{X}_0,\mathbb{X}_1)$ form a double category. One version of a triple category due to Grandis and Par\`e \cite{GP3} is an 'intercategory'  which is, roughly speaking, a pair of double categories sharing a common `side category'. 

Chapter \ref{Chapter4} explores improvements to Fong's original conception of decorated cospans \cite{Fong}. Here, the main insight was to not consider a \emph{set} of decorations but a \emph{category} of decorations. Even further generalizations could be made here by replacing the finitely cocartesian category $\mathsf{A}$ with a finitely 2-cocartesian 2-category $\mathbf{A}$ and viewing $\mathbf{Cat}$ as a 3-category and defining an appropriate functor $F \colon \mathbf{A} \to \mathbf{Cat}$. In this framework, we could then decorate objects with `higher level \emph{stuff}' \cite{BS}, such as a decoration that makes a 2-category $\mathbf{C}$ into a monoidal 2-category $(\mathbf{C},\otimes,1)$.

Above are only some possible improvements to the \emph{frameworks} themselves, but each framework is suitable to applications not mentioned in this thesis. Biological sciences, economics and even social sciences are bound to have situations which can be modeled by either of the above frameworks. Anytime a concept or an idea can be thought of as a set equipped with some extra structure, decorated cospans are lurking in the background, and very often a trivial form of this structure is captured by a left adjoint.

Chapter \ref{Chapter6} applies double categories to coarse-graining open Markov processes. Here, the Markov processes we consider are really finite state Markov \emph{chains}, but more general Markov processes can be considered. Moreover, more general forms of coarse-graining outside of lumpability can also be considered, but would require a different definition of 2-morphism in the resulting double category. In a `triple category' of coarse-grainings, 3-morphisms would then represent maps between two different ways of applying a coarse-graining to a Markov process. This idea would not be well suited for the double category of coarse-grainings presented here, as the category of arrows $\MMark_1$ is locally posetal, meaning that there is at most one coarse-graining as we have defined it \cite{BC} between two open Markov processes.

$\textnormal{ }$
}

{\ssp
\appendix 
\chapter{Definitions}\label{Appendix}

\section{Everyday categories}\label{AppendixA}
This is a thesis largely about applications of double categories in network theory. The most obvious place to start is with the following question: What is a category?

\begin{definition}
A \define{category} $\mathsf{C}$ consists of a collection of \define{objects} denoted Ob$(\mathsf{C})$ and a collection of \define{morphisms} denoted Mor$(\mathsf{C})$ such that: 
\begin{enumerate}
\item{every morphism $f \in \textnormal{Mor}(\mathsf{C})$ has a source object $s(f) \in \textnormal{Ob}(\mathsf{C})$ and a target object $t(f) \in \textnormal{Ob}(\mathsf{C})$. A morphism $f$ with source $x$ and target $y$ we denote as $f \colon x \to y$, and we denote the collection of all morphisms with source $x$ and target $y$ by $\hom(x,y)$ or $\hom_\mathsf{C}(x,y)$.}
\item{Given a morphism $f \colon x \to y$ and a morphism $g \colon y \to z$, there exists a composite morphism $gf \colon x \to z$. In other words, for any triple of objects $x,y,z \in \textnormal{Ob}(\mathsf{C})$, there is a well-defined map
\[
\circ \colon \hom(x,y) \times \hom(y,z) \to \hom(x,z)
\]
called \define{composition}.}
\item{Composition of morphisms is associative, meaning that given three composable morphisms $f,g,h \in \textnormal{Mor}(\mathsf{C})$ we have $h(gf)=(hg)f$.}
\item{Every object $x \in \textnormal{Ob}(\mathsf{C})$ has an identity morphism $1_x \colon x \to x$ such that for any morphism $f \colon x \to y$, we have
\[
f 1_x = f = 1_y f.
\]}
\end{enumerate}
\end{definition}

If both $\textnormal{Ob}(\mathsf{C})$ and $\textnormal{Mor}(\mathsf{C})$ are sets, we say that $\mathsf{C}$ is a \define{small category}. If for every pair of objects $x,y \in \textnormal{Ob}(\mathsf{C})$ we have that $\hom(x,y)$ is a set, we say that $\mathsf{C}$ is a \define{locally small category}. Here are some examples:
\begin{enumerate}
\item{The primordial example of a category is $\mathsf{Set}$ of sets and functions.}
\item{The category $\mathsf{Grp}$ of groups and group homomorphisms.}
\item{The category $\mathsf{Top}$ of topological spaces and continuous maps.}
\item{The category $\mathsf{Mat}(k)$ of natural numbers and $n \times m$ matrices with entries in a field (or more generally, a ring or rig) $k$ with composition given by matrix multiplication.}
\item{Every monoid is a locally small category with a single object whose morphisms are given by the elements of the monoid.}
\item{The category $\mathsf{Cat}$ of categories and functors.}
\item{The category $\mathsf{Vect}$ of vector spaces and linear maps.}
\item{The category $\mathsf{Diff}$ of smooth manifolds and smooth maps.}
\item{The category $\mathsf{Rel}$ of sets and relations.}
\item{The category $\mathsf{PreOrd}$ of preordered sets and monotone functions.}
\item{The category $\mathsf{Graph}$ of (directed) graphs and graph morphisms, which are pairs of functions preserving the source and target of each edge.}
\item{Any set $S$ gives rise to a category $\mathsf{S}$ whose objects are the elements of the set $S$ containing only identity morphisms.}
\item{There is a category $1$ with only one object $\star$ and only an identity morphism $1_\star$.}
\end{enumerate}

Even though a category is usually named after its objects, it is the morphisms of a category that are the real stars of the show. In fact, we can `do away' with all the objects as the collection of all identity morphisms tell us precisely what the objects of a category are.

Any sort of mathematical gizmo is boring and pointless to study unless that mathematical gizmo can `talk' to other similar mathematical gizmos via maps between the two. So, how do categories talk to each other?

\begin{definition}
Given categories $\mathsf{C}$ and $\mathsf{D}$, a \define{functor} $F \colon \mathsf{C} \to \mathsf{D}$ consists of a map $\textnormal{Ob}(F) \colon \textnormal{Ob}(\mathsf{C}) \to \textnormal{Ob}(\mathsf{D})$ and a map $\textnormal{Mor}(F) \colon \textnormal{Mor}(\mathsf{C}) \to \textnormal{Mor}(\mathsf{D})$ respecting source and target, meaning that $s(F(f))=F(s(f))$ and $t(F(f))=F(t(f))$, such that:
\begin{enumerate}
\item{For any two composable morphisms $f \colon x \to y$ and $g \colon y \to z$ in $\mathsf{C}$, we have $F(f)F(g)=F(fg)$, and}
\item{For any object $x \in \mathsf{C}$, we have $F(1_x) = 1_{F(x)}$.}
\end{enumerate}
We usually denote the maps $\textnormal{Ob}(F)$ and $\textnormal{Mor}(F)$ simply as $F$.
\end{definition}
Here are some examples:
\begin{enumerate}
\item{For any category $\mathsf{C}$, there is an identity functor $1_\mathsf{C} \colon \mathsf{C} \to \mathsf{C}$ that maps every object and morphism of $\mathsf{C}$ to itself.}
\item{There is a forgetful functor $R \colon \mathsf{Grp} \to \mathsf{Set}$, which we call $R$ as it is a right adjoint, that maps any group $G$ to its underlying set $U(G)$ and any group homomorphism $f \colon G \to G^\prime$ to its underlying function $U(f) \colon U(G) \to U(G^\prime)$.}
\item{For any category $\mathsf{C}$, there is a functor $! \colon \mathsf{C} \to 1$ which maps every object of $\mathsf{C}$ to the one object $\star$ of $1$ and any morphism in $\mathsf{C}$ to the only morphism $1_\star$ of $1$.}
\item{There is a functor $F \colon \mathsf{Set} \to \mathsf{Cat}$ which maps any set $S$ to the discrete category on $S$ whose objects are given by elements of $S$ and whose only morphisms are identity morphisms.}
\item{Given categories $\mathsf{C}$ and $\mathsf{D}$ and an object $d \in \mathsf{D}$, there is a functor $F_d \colon \mathsf{C} \to \mathsf{D}$ called the \emph{constant functor at} $d$ which maps every object $\mathsf{C}$ to the object $d$ in $\mathsf{D}$ and every morphism of $\mathsf{C}$ to the morphism $1_d$.}
\end{enumerate}

Functors may look a little similar to functions in that they are maps between objects that we are interested in. However, in the same way that the morphisms are the real stars of the show in a category, one could make the same argument that it is \emph{functors} that are the real stars of category theory: after all, a category $\mathsf{C}$ is ultimately determined by the identity functor $1_\mathsf{C}$ on that category. But we will not go down that road. The real fun of category theory starts when we start to consider \emph{maps between maps}. Our first examples of a map between maps, which are also one of the main reasons that Eilenberg and Mac Lane invented category theory in the 1940's, are natural transformations.

\begin{definition}
Let $F \colon \mathsf{C} \to \mathsf{D}$ and $G \colon \mathsf{C} \to \mathsf{D}$ be functors. Then a \define{natural transformation} $\alpha \colon F \Rightarrow G$ consists of a family of morphisms $\alpha_x \colon F(x) \to G(x)$ indexed by the objects of $\mathsf{C}$ such that for any morphism $f \colon x \to y$ in $\mathsf{C}$, the following \define{naturality square} commutes.
\[
\begin{tikzpicture}[scale=1.5]
\node (A) at (0,0) {$F(x)$};
\node (B) at (1.5,0) {$F(y)$};
\node (C) at (0,-1.5) {$G(x)$};
\node (D) at (1.5,-1.5) {$G(y)$};
\path[->,font=\scriptsize,>=angle 90]
(A) edge node [above] {$F(f)$}(B)
(C) edge node [above] {$G(f)$}(D)
(A) edge node [left] {$\alpha_x$}(C)
(B) edge node [right] {$\alpha_y$}(D);
\end{tikzpicture}
\]
We call $\alpha_x$ the component of $\alpha$ at $x$. If each map $\alpha_x$ is an isomorphism, then we say that $\alpha \colon F \Rightarrow G$ is a \define{natural isomorphism}.
\end{definition}
Here are some examples of natural transformations:
\begin{enumerate}
\item{For any functor $F \colon \mathsf{C} \to \mathsf{D}$, there is an identity natural transformation $1 \colon F \Rightarrow F$ in which the component at each object $x$ is the identity $1_{F(x)}$. This is a natural isomorphism.}

\item{Given a functor $F_d \colon \mathsf{C} \to \mathsf{D}$ which is constant at some object $d \in \mathsf{D}$ and another functor $G \colon \mathsf{C} \to \mathsf{D}$, a natural transformation $\alpha \colon F_d \Rightarrow G$ is a \define{cone over} $\mathsf{D}$, which consists of a family of morphisms $\alpha_x \colon d \to G(x)$ which make a cone-like commutative diagram in which all the top triangular faces commute.
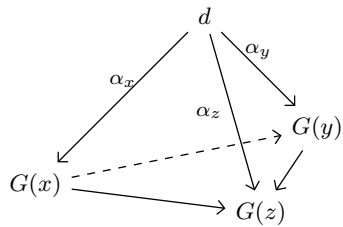
\[
\begin{tikzpicture}[scale=1.5]
\node (A) at (0,0) {$d$};
\node (B) at (-1.5,-1.5) {$G(x)$};
\node (C) at (1,-1) {$G(y)$};
\node (D) at (0.5,-1.75) {$G(z)$};
\path[->,font=\scriptsize,>=angle 90]
(A) edge node [above] {$\alpha_x$}(B)
(A) edge node [above] {$\alpha_y$}(C)
(A) edge node [left] {$\alpha_z$}(D)
(B) edge[dashed] node [right] {$$}(C)
(B) edge node [right] {$$}(D)
(C) edge node [right] {$$}(D);
\end{tikzpicture}
\]
}

\item{Let $\mathsf{Grp}$ denote the category of groups and group homomorphisms, $\mathsf{AbGrp}$ the category of abelian groups and group homomorphisms and $\mathsf{Ab} \colon \mathsf{Grp} \to \mathsf{AbGrp}$ the functor sending each group to its abelianization, namely $\mathsf{Ab}(G) := G/[G,G]$ where $[G,G]$ is the commutator subgroup of $G$. Then there is a natural transformation $\pi \colon 1_\mathsf{Grp} \Rightarrow \mathsf{Ab}$ where the component at each group is given by $\pi_G \colon G \to \mathsf{Ab}(G)$. For any group homomorphism $f \colon G \to H$, the following square commutes.
\[
\begin{tikzpicture}[scale=1.5]
\node (A) at (0,0) {$G$};
\node (B) at (1.5,0) {$H$};
\node (C) at (0,-1.5) {$\mathsf{Ab}(G)$};
\node (D) at (1.5,-1.5) {$\mathsf{Ab}(H)$};
\path[->,font=\scriptsize,>=angle 90]
(A) edge node [above] {$f$}(B)
(C) edge node [above] {$\mathsf{Ab}(f)$}(D)
(A) edge node [left] {$\pi_G$}(C)
(B) edge node [right] {$\pi_H$}(D);
\end{tikzpicture}
\]
This is not a natural isomorphism.}

\item{Given a field $k$ and a finite dimensional vector space $V$ over $k$, there is a canonical isomorphism $\alpha_V \colon V \to V^{**}$ from the vector space $V$ to its double dual. This gives a natural transformation $\alpha \colon 1_{\mathsf{FinVect}_k} \Rightarrow {}^{**}$ where ${}^{**} \colon \mathsf{FinVect}_k \to \mathsf{FinVect}_k$ is the functor sending each finite dimensional vector space $V$ to its double dual $V^{**}$. The following square then commutes for every linear map $L \colon V \to W$ of finite dimensional $k$-vector spaces.
\[
\begin{tikzpicture}[scale=1.5]
\node (A) at (0,0) {$V$};
\node (B) at (1.5,0) {$W$};
\node (C) at (0,-1.5) {$V^{**}$};
\node (D) at (1.5,-1.5) {$W^{**}$};
\path[->,font=\scriptsize,>=angle 90]
(A) edge node [above] {$L$}(B)
(C) edge node [above] {$L^{**}$}(D)
(A) edge node [left] {$\alpha_V$}(C)
(B) edge node [right] {$\alpha_W$}(D);
\end{tikzpicture}
\]
This is a natural isomorphism if all the vector spaces are finite dimensional. If we allow for infinite dimensional vector spaces, we still have a natural transformation, but each map $\alpha_V \colon V \to V^{**}$ is no longer an isomorphism.}

\item{Given commutative rings $R$ and $S$ and a ring homomorphism $f \colon R \to S$, the ring homomorphism $f \colon R \to S$ restricts to a group homomorphism $f^\times \colon R^\times \to  S^\times$ where $R^\times$ denotes the group of units of the commutative ring $R$. This defines a functor ${}^\times \colon \mathsf{CommRing} \to \mathsf{AbGrp}$. There are also well-known groups of linear transformations $GL_n(R)$ and $GL_n(S)$, and every ring homomorphism $f \colon R \to S$ induces a map $GL_n(f) \colon GL_n(R) \to GL_n(S)$ given by application of $f$ to every entry of $H \in GL_n(R)$. This defines another functor $GL_n \colon \mathsf{CommRing} \to \mathsf{AbGrp}$.  There is then a natural transformation $\textnormal{det} \colon GL_n \Rightarrow {}^\times$ where given $H \in GL_n(R)$, $\textnormal{det}_R(H)$ is the determinant of $H$. The following square commutes for every ring homomorphism $f \colon R \to S$.
\[
\begin{tikzpicture}[scale=1.5]
\node (A) at (0,0) {$GL_n(R)$};
\node (B) at (1.5,0) {$GL_n(S)$};
\node (C) at (0,-1.5) {$R^\times$};
\node (D) at (1.5,-1.5) {$S^\times$};
\path[->,font=\scriptsize,>=angle 90]
(A) edge node [above] {$GL_n(f)$}(B)
(C) edge node [above] {$f^\times$}(D)
(A) edge node [left] {$\det_R$}(C)
(B) edge node [right] {$\det_S$}(D);
\end{tikzpicture}
\]
}
\end{enumerate}

\begin{definition}
Given a two categories $\mathsf{A}$ and $\mathsf{X}$ and two functors going in opposite directions between the two:
\[
\begin{tikzpicture}[scale=1.5]
\node (D) at (3,-0.5) {$\mathsf{A}$};
\node (E) at (4,-0.5) {$\mathsf{X}$};
\path[->,font=\scriptsize,>=angle 90]
(D) edge [bend left] node [above] {$L$} (E)
(E) edge [bend left] node [below] {$R$} (D);
\end{tikzpicture}
\]
we say that $L$ and $R$ are \define{adjoint}, with $L$ the \define{left adjoint} and $R$ the \define{right adjoint}, if for every $a \in \mathsf{A}$ and $x \in \mathsf{X}$ there is a natural isomorphism 
\[
\hom_\mathsf{X} (L(a),x) \cong \hom_\mathsf{A} (a,R(x)).
\]
\end{definition}

\subsection{Monoidal categories and monoidal functors}

Next we introduce `monoidal' categories, which are largely the kinds of categories that this thesis is about. Roughly speaking, a monoidal category is a category with a binary operation in which we can multiply or `tensor' two objects in the category much like we can multiply two objects in a monoid.

\begin{definition}
A \define{monoidal category} is a category $\mathsf{C}$ equipped with the extra structure of:
\begin{enumerate}
\item{a functor $\otimes \colon \mathsf{C} \times \mathsf{C} \to \mathsf{C}$ called the \define{tensor product} of $\mathsf{C}$,}
\item{an object $I \in \mathsf{C}$ called the \define{(monoidal) unit} of $\mathsf{C}$,}
\item{for any three objects $a,b,c \in \mathsf{C}$, a natural isomorphism called the \define{associator} $$\alpha \colon ((-) \otimes (-)) \otimes (-) \xrightarrow{\sim} (-) \otimes ((-) \otimes (-))$$ whose components are of the form $$\alpha_{a,b,c} \colon (a \otimes b) \otimes c \xrightarrow{\sim} a \otimes (b \otimes c)$$}
\item{for any object $c$, a natural isomorphism called the \define{left unitor} $$\lambda \colon (I \otimes (-)) \xrightarrow{\sim} (-)$$ whose components are of the form $$\lambda_c \colon I \otimes c \xrightarrow{\sim} c$$}
\item{for any object $c$, a natural isomorphism called the \define{right unitor} $$\rho \colon ((-) \otimes I) \xrightarrow{\sim} (-)$$ whose components are of the form $$\rho_c \colon c \otimes I \xrightarrow{\sim} c$$}
\end{enumerate}
such that the following two diagrams commute, giving equations called the \define{pentagon identity}:
\[
\begin{tikzpicture}[scale=1.5]
\node (A) at (-1.25,0) {$((a \otimes b) \otimes c) \otimes d$};
\node (B) at (3,1.55) {$(a \otimes b) \otimes (c \otimes d)$};
\node (C) at (7.25,0) {$a \otimes (b \otimes (c \otimes d))$};
\node (D) at (1.5,-1.55) {$(a \otimes (b \otimes c)) \otimes d$};
\node (E) at (4.5,-1.55) {$a \otimes ((b \otimes c) \otimes d)$};
\path[->,font=\scriptsize,>=angle 90]
(A) edge node [above] {$\alpha_{a \otimes b,c,d}$}(B)
(B) edge node [above] {$\alpha_{a,b,c \otimes d}$}(C)
(A) edge node [above,right] {$\alpha_{a,b,c} \otimes 1_d$}(D)
(D) edge node [above] {$\alpha_{a,b \otimes c,d}$}(E)
(E) edge node [above,left] {$1_a \otimes \alpha_{b,c,d}$}(C);
\end{tikzpicture}
\]
and the \define{triangle identity}:
\[
\begin{tikzpicture}[scale=1.5]
\node (A) at (0.75,0) {$(a \otimes I) \otimes b$};
\node (B) at (3,1.55) {$a \otimes b$};
\node (C) at (5.25,0) {$a \otimes (I \otimes b)$};
\path[->,font=\scriptsize,>=angle 90]
(A) edge node [left] {$\rho_a \otimes 1_b$}(B)
(C) edge node [right] {$1_a \otimes \lambda_b$}(B)
(A) edge node [above] {$\alpha_{a,1_\mathsf{C},b}$}(C);
\end{tikzpicture}
\]
\end{definition}
Sometimes we abbreviate a monoidal category $\mathsf{C}$ with tensor product $\otimes$ and monoidal unit $1_\mathsf{C}$ as $(\mathsf{C},\otimes,1_\mathsf{C})$. Some examples of monoidal categories which are relevant in this thesis are the following:
\begin{enumerate}
\item{The category $\mathsf{Set}$ together with the tensor product given by cartesian product and monoidal unit given by a singleton $\{ \star \}$.}
\item{If $\mathsf{C}$ is a category with finite colimits, then $\mathsf{C}$ is monoidal with the tensor product given by binary coproducts and monoidal unit given by an initial object $0$.}
\item{The large category $\mathsf{Cat}$ together with the tensor product given by the product of two categories and monoidal unit given by a terminal category $1$.}
\end{enumerate}

Sometimes there is a relationship between the two tensor products $a \otimes b$ and $b \otimes a$ for two objects $a$ and $b$ in a monoidal category $(\mathsf{C},\otimes,I)$.

\begin{definition}
A \define{braided monoidal category} is a monoidal category $(\mathsf{C},\otimes,I)$ equipped with a natural isomorphism $$\beta_{a,b} \colon a \otimes b \xrightarrow{\sim} b \otimes a$$ called the \define{braiding} such that the following hexagons commute.
\[
\begin{tikzpicture}[scale=1.5]
\node (A) at (0,0.5) {$(a \otimes b) \otimes c$};
\node (A') at (4.5,0.5) {$a \otimes (b \otimes c)$};
\node (B) at (0,-0.25) {$(b \otimes a) \otimes c$};
\node (C) at (4.5,-0.25) {$(b \otimes c) \otimes a$};
\node (C') at (0,-1) {$b \otimes (a \otimes c)$};
\node (D) at (4.5,-1) {$b \otimes (c \otimes a)$};
\path[->,font=\scriptsize,>=angle 90]
(A) edge node[above]{$\alpha_{a,b,c}$} (A')
(A) edge node[left]{$\beta_{a,b} \otimes 1_c$} (B)
(A') edge node[right]{$\beta_{a,b \otimes c}$} (C)
(B) edge node[left]{$\alpha_{b,a,c}$} (C')
(C) edge node [right] {$\alpha_{b,c,a}$} (D)
(C') edge node [above] {$1_b \otimes \beta_{a,c}$} (D);
\end{tikzpicture}
\]
\[
\begin{tikzpicture}[scale=1.5]
\node (A) at (0,0.5) {$a \otimes (b \otimes c)$};
\node (A') at (4.5,0.5) {$(a \otimes b) \otimes c$};
\node (B) at (0,-0.25) {$a \otimes (c \otimes b)$};
\node (C) at (4.5,-0.25) {$c \otimes (a \otimes b)$};
\node (C') at (0,-1) {$(a \otimes c) \otimes b$};
\node (D) at (4.5,-1) {$(c \otimes a) \otimes b$};
\path[->,font=\scriptsize,>=angle 90]
(A) edge node[above]{$\alpha_{a,b,c}^{-1}$} (A')
(A) edge node[left]{$1_a \otimes \beta_{b,c}$} (B)
(A') edge node[right]{$\beta_{a \otimes b,c}$} (C)
(B) edge node[left]{$\alpha_{a,c,b}^{-1}$} (C')
(C) edge node [right] {$\alpha_{c,a,b}^{-1}$} (D)
(C') edge node [above] {$\beta_{a,c} \otimes 1_b$} (D);
\end{tikzpicture}
\]
\end{definition}
All of the above examples of monoidal categories are in fact braided monoidal categories. Sometimes the braiding $\beta$ is its own inverse, which finally brings us to:
\begin{definition}
A \define{symmetric monoidal category} is a braided monoidal category $(\mathsf{C},\otimes,I)$ such that for any two objects $a$ and $b$ of $\mathsf{C}$, the braiding $\beta$ is its own inverse, meaning that $$\beta_{b,a} \beta_{a,b} = 1_{a \otimes b}.$$
\end{definition}
All of the above examples are in fact symmetric monoidal categories. What about maps between various such categories?
\begin{definition}\label{defn:monoidal_functor}
Let $(\mathsf{C},\otimes,I_\mathsf{C})$ and $(\mathsf{D},\otimes,I_\mathsf{D})$ be monoidal categories. A \define{(lax) monoidal functor} is a functor $F \colon \mathsf{C} \to \mathsf{D}$ such that:
\begin{enumerate}
\item{there exists an morphism $\mu \colon I_\mathsf{D} \to F(I_\mathsf{C})$ and}
\item{for every pair of objects $a$ and $b$ of $\mathsf{C}$, there exists a natural transformation $$\mu_{a,b} \colon F(a) \otimes F(b) \to F(a \otimes b)$$}
\end{enumerate}
which make the following diagrams commute:
\[
\begin{tikzpicture}[scale=1.5]
\node (A) at (0,0.5) {$(F(a) \otimes F(b)) \otimes F(c)$};
\node (A') at (4.5,0.5) {$F(a) \otimes (F(b) \otimes F(c))$};
\node (B) at (0,-0.25) {$F(a \otimes b) \otimes F(c)$};
\node (C) at (4.5,-0.25) {$F(a) \otimes F(b \otimes c)$};
\node (C') at (0,-1) {$F((a \otimes b) \otimes c)$};
\node (D) at (4.5,-1) {$F(a \otimes (b \otimes c))$};
\path[->,font=\scriptsize,>=angle 90]
(A) edge node[above]{$a'$} (A')
(A) edge node[left]{$\mu_{a,b} \otimes 1_{F(c)}$} (B)
(A') edge node[right]{$1_{F(a)} \otimes \mu_{b,c}$} (C)
(B) edge node[left]{$\mu_{a \otimes b,c}$} (C')
(C) edge node [right] {$\mu_{a,b \otimes c}$} (D)
(C') edge node [above] {$F(a)$} (D);
\end{tikzpicture}
\]
\[
\begin{tikzpicture}[scale=1.5]
\node (A) at (1,1) {$F(a) \otimes I_{\mathsf{D}}$};
\node (C) at (3,1) {$F(a)$};
\node (A') at (1,-1) {$F(a) \otimes F(I_{\mathsf{C}})$};
\node (C') at (3,-1) {$F(a \otimes I_{\mathsf{C}})$};
\node (B) at (5,1) {$I_{\mathsf{D}} \otimes F(a)$};
\node (B') at (5,-1) {$F(I_{\mathsf{C}}) \otimes F(a)$};
\node (D) at (7,1) {$F(a)$};
\node (D') at (7,-1) {$F(I_{\mathsf{C}} \otimes a)$};
\path[->,font=\scriptsize,>=angle 90]
(A) edge node[left]{$1_{F(a)} \otimes \mu$} (A')
(C') edge node[right]{$F(\rho_a)$} (C)
(A) edge node[above]{$\rho_{F(a)}'$} (C)
(A') edge node[above]{$\mu_{a,I_{\mathsf{C}}}$} (C')
(B) edge node[left]{$\mu \otimes 1_{F(a)}$} (B')
(B') edge node[above]{$\mu_{I_{\mathsf{C},a}}$} (D')
(B) edge node[above]{$\lambda_{F(a)}'$} (D)
(D') edge node[right]{$F(\lambda_a)$} (D);
\end{tikzpicture}
\]
The monoidal functor $F$ is called \define{strong} if the morphism $\mu$ and natural transformation $\mu_{-,-}$ are an isomorphism and natural isomorphism, respectively, and the monoidal functor $F$ is called \define{oplax} or \define{colax} if $F \colon \mathsf{C}^\textnormal{op} \to \mathsf{D}^\textnormal{op}$ is a lax monoidal functor.
\end{definition}

\begin{definition}
A (possibly lax or oplax) monoidal functor $F \colon \mathsf{C} \to \mathsf{D}$ is a \define{braided monoidal functor} if $\mathsf{C}$ and $\mathsf{D}$ are braided monoidal categories and the following diagram commutes.
\[
\begin{tikzpicture}[scale=1.5]
\node (A) at (1,1) {$F(a) \otimes F(b)$};
\node (C) at (3,1) {$F(b) \otimes F(a)$};
\node (A') at (1,-1) {$F(a \otimes b)$};
\node (C') at (3,-1) {$F(b \otimes a)$};
\path[->,font=\scriptsize,>=angle 90]
(A) edge node[left]{$\mu_{a,b}$} (A')
(C) edge node[right]{$\mu_{b,a}$} (C')
(A) edge node[above]{$\beta^\prime$} (C)
(A') edge node[above]{$F(\beta)$} (C');
\end{tikzpicture}
\]
\end{definition}

\begin{definition}
A (possibly lax or oplax) \define{symmetric monoidal functor} is a braided monoidal functor $F \colon \mathsf{C} \to \mathsf{D}$ between symmetric monoidal categories.
\end{definition}
\begin{definition}\label{defn:monoidal_transformation}
Given monoidal functors $F \colon (\mathsf{C},\otimes,1_\mathsf{C}) \to (\mathsf{D},\otimes,1_\mathsf{D})$ and $G \colon (\mathsf{C},\otimes,1_\mathsf{C}) \to (\mathsf{D},\otimes,1_\mathsf{D})$, a \define{monoidal natural transformation} $\alpha \colon F \Rightarrow G$ is a transformation $\alpha \colon F \Rightarrow G$ such that the following diagrams commute.
\[
\begin{tikzpicture}[scale=1.5]
\node (A) at (0,0) {$F(x) \otimes F(y)$};
\node (B) at (3,0) {$G(x) \otimes G(y)$};
\node (C) at (0,-1.5) {$F(x \otimes y)$};
\node (D) at (3,-1.5) {$G(x \otimes y)$};
\node (E) at (5,0) {$I_\mathsf{D}$};
\node (F) at (4,-1.5) {$F(I_\mathsf{C})$};
\node (G) at (6,-1.5) {$G(I_\mathsf{C})$};
\path[->,font=\scriptsize,>=angle 90]
(E) edge node [left] {$\mu$} (F)
(F) edge node [above] {$\alpha_{I_\mathsf{C}}$} (G)
(E) edge node [right] {$\mu^\prime$} (G)
(A) edge node [above] {$\alpha_x \otimes \alpha_y$}(B)
(C) edge node [above] {$\alpha_{x \otimes y}$}(D)
(A) edge node [left] {$\mu_{x,y}$}(C)
(B) edge node [right] {$\mu_{x,y}'$}(D);
\end{tikzpicture}
\]
A monoidal transformation $\alpha$ is \define{braided monoidal} or \define{symmetric monoidal} if the functors $F$ and $G$ are braided monoidal or symmetric monoidal, respectively. 
\end{definition}

\subsection{Colimits}

\begin{definition}
Given an arbitrary category $\mathsf{C}$, a \define{diagram} in the category $\mathsf{C}$ is given by a functor $F \colon \mathsf{D} \to \mathsf{C}$.
\end{definition}
Here, the category $\mathsf{D}$ serves as the `shape' of the diagram in the category $\mathsf{C}$. 

\begin{definition}
Given a diagram $F \colon \mathsf{D} \to \mathsf{C}$ in $\mathsf{C}$, a \define{limit} of the diagram $F$, denoted $\lim F$, is given by an object which we also denote by $\lim F$, together with with a family of morphisms $\phi_i \colon \lim F \to F(d_i)$ for every $i \in \mathsf{D}$ such that for any morphism $f \colon d_i \to d_j$ in $\mathsf{D}$, we have that $F(f) \phi_i = \phi_j$. Moreover, the object $\lim F$ together with the family of morphisms $\{ \phi_i \colon i \in \mathsf{D} \}$ are \define{universal} among such, meaning that given another object $c$ together with a family of morphisms $\psi_i \colon c \to F(d_i)$ such that $F(f) \psi_i = \psi_j$, there exists a unique morphism $\theta \colon c \to \lim F$ such that $\psi_i =  \phi_i \theta$ for every $i \in \mathsf{D}$. A limit is \define{finite} if the category $\mathsf{D}$ is finite. Then, a \define{colimit} is just a limit in the opposite category, meaning that given a functor $F \colon \mathsf{D} \to \mathsf{C}$, a colimit of $F$, denoted by $\colim F$, is given by a limit of $F^\textnormal{op} \colon \mathsf{D}^\textnormal{op} \to \mathsf{C}^\textnormal{op}$.
\end{definition}
Limits and colimits are only unique up to a unique isomorphism, hence the usage of the indefinite articles `a' and `an' rather than the definite article `the'.

We largely work with finite colimits in this thesis, and so the examples presented next will be of such. The most famous examples of finite colimits are easily the following:
\begin{enumerate}
\item{initial objects}
\item{binary coproducts}
\item{coequalizers}
\item{pushouts}
\end{enumerate}
In fact, a category $\mathsf{C}$ has finite colimits iff $\mathsf{C}$ has an initial object and pushouts iff $\mathsf{C}$ has binary coproducts and coequalizers. We discuss pushouts in the next section, but let us briefly introduce the other three famous finite colimits.

\begin{definition}
An \define{initial object} $0$ is a colimit of the empty functor $F \colon \emptyset \to \mathsf{C}$.
\end{definition} 
Unraveling what this means, it means that an initial object is an object $0$ in $\mathsf{C}$ together with an empty family of morphisms satisfying no properties such that for any other object $c$ together with an empty family of morphisms satisfying no properties, there exists a unique  morphism $!_c \colon 0 \to c$ which satisfies no properties. In other words, it is just an object $0$ of $\mathsf{C}$ with a unique morphism to any other object of $\mathsf{C}$. If $\mathsf{C}=\mathsf{Set}$, then $0=\emptyset$, and surely there is a unique function $!_S \colon \emptyset \to S$ for any set $S$.

\begin{definition}
A \define{binary coproduct} is a colimit of a functor $F \colon \{ \star \textnormal{ } \star \} \to \mathsf{C}$ where $\{ \star \textnormal{ } \star \}$ denotes the category with two objects and only identity morphisms. 
\end{definition}
Unraveling what this means, given two objects $c_1$ and $c_2$ in $\mathsf{C}$, a binary coproduct of $c_1$ and $c_2$ is an object which we denote as $c_1+c_2$ together with two morphisms $\phi_{c_1} \colon c_1 \to c_1+c_2$ and $\phi_{c_2} \colon c_2 \to c_1+c_2$ such that for any other object $c$ also with morphisms $\psi_1 \colon c_1 \to c$ and $\psi_2 \colon c_2 \to c$, there exists a unique morphism $\theta \colon c_1 + c_2 \to c$ such that $\psi_i = \theta \phi_i$ for $i=1,2$.
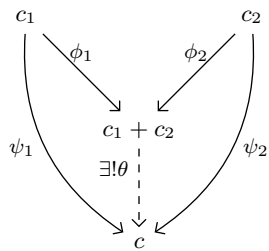
\[
\begin{tikzpicture}[scale=1.5]
\node (B) at (-1,-1) {$c_1$};
\node (C) at (1,-1) {$c_2$};
\node (D) at (0,-2) {$c_1 + c_2$};
\node (E) at (0,-3) {$c$};
\path[->,font=\scriptsize,>=angle 90]
(B) edge node [above] {$\phi_1$}(D)
(C) edge node [above] {$\phi_2$}(D)
(D) edge[dashed] node [near start,left] {$\exists ! \theta$} (E)
(B) edge[bend right] node [left] {$\psi_1$}(E)
(C) edge[bend left] node [right] {$\psi_2$}(E);
\end{tikzpicture}
\]
In other words, such an object $c_1+c_2$ and morphisms $(\phi_1,\phi_2)$ are initial among such. A typical example of a binary coproduct is the disjoint union of two sets together with the natural injection maps of each set into the disjoint union, or the direct sum of two vector spaces $V_1$ and $V_2$ together with the maps $((1_{V_1},0),(0,1_{V_2}))$ into the direct sum.

\begin{definition}
A \define{coequalizer} is a colimit of a functor $F \colon \{ \star \rightrightarrows \star \} \to \mathsf{C}$ where $\{ \star \rightrightarrows \star \}$ denotes the category with two objects, two morphisms from one object to the other, and two identity morphisms. 
\end{definition}
Unraveling what this means, given two morphisms $f,g \colon c \to c^\prime$ in $\mathsf{C}$, a coequalizer of $f$ and $g$ is an object $\textnormal{coeq}(f,g)$ together with a morphism $\phi \colon c^\prime \to \textnormal{coeq}(f,g)$ such that $\phi f= \phi g.$ Such an object and morphism are universal among such, meaning that given another object $\hat{c}$ and morphism $\psi \colon c^\prime \to \hat{c}$ such that $\psi f = \psi g$, there exists a unique morphism $\theta \colon \textnormal{coeq}(f,g) \to \hat{c}$ such that $\theta \phi = \psi$.
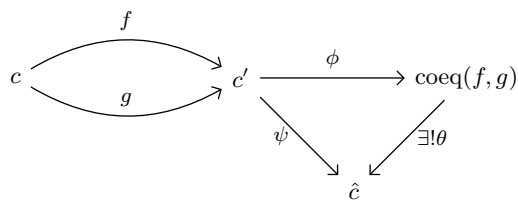
\[
\begin{tikzpicture}[scale=1.5]
\node (A) at (0,0) {$c$};
\node (B) at (2,0) {$c'$};
\node (C) at (4,0) {$\textnormal{coeq}(f,g)$};
\node (D) at (3,-1) {$\hat{c}$};
\path[->,font=\scriptsize,>=angle 90]
(A) edge[bend left] node [above] {$f$}(B)
(A) edge[bend right] node [above] {$g$}(B)
(B) edge node [above] {$\phi$}(C)
(B) edge node [left] {$\psi$}(D)
(C) edge node [right] {$\exists ! \theta$}(D);
\end{tikzpicture}
\]
In other words, such an object $\textnormal{coeq}(f,g)$ and morphism $\theta$ are initial among such. An example of a coequalizer is in the category $\mathsf{Grp}$: given any group homomorphism $f \colon G \to H$, there is always a unique group homomorphism $0 \colon G \to H$ which sends every element of $G$ to the identity element of $H$, in which case $\textnormal{coeq}(f,0)=\ker(f)$.

\begin{definition}
A \define{span} in any category $\mathsf{C}$ is a diagram of the form:
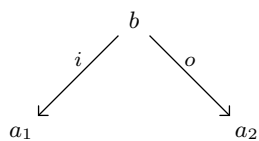
\[
\begin{tikzpicture}[scale=1.5]
\node (A) at (0,0) {$b$};
\node (B) at (-1,-1) {$a_1$};
\node (C) at (1,-1) {$a_2$};
\path[->,font=\scriptsize,>=angle 90]
(A) edge node [above] {$i$}(B)
(A) edge node [above] {$o$}(C);
\end{tikzpicture}
\]
A \define{pushout} is a colimit of a span, or equivalently, a colimit of a functor $F \colon \{ \star \leftarrow \star \rightarrow \star \} \to \mathsf{C}$ where $\{ \star \leftarrow \star \rightarrow \star \}$ denotes the category with three objects and two non-identity morphisms with a common source and distinct targets.
\end{definition}
Unraveling what this means, a pushout of the above span is an object $a_1 +_b a_2$ together with a pair of maps $j \colon a_1 \to a_1 +_b a_2$ and $k \colon a_2 \to a_1 +_b a_2$ making the induced square commute, meaning that $ji=ko$. Such an object and pair of maps are universal among such, meaning that given another object $q$ and maps $j' \colon a_1 \to q$ and $k' \colon a_2 \to q$ such that $j'i=k'o$, there exists a unique $\psi \colon a_1 +_b a_2 \to q$ such that $j'= \psi j$ and $k' = \psi k$.
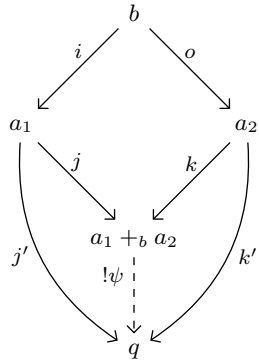
\[

\]
In other words, a pushout is initial among such triples $(j',k',q)$.

We compose cospans by taking pushouts. In other words, given two composable cospans 
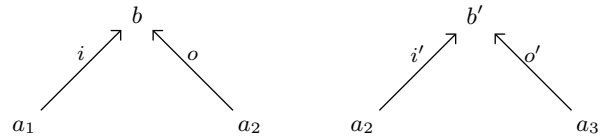
\[
%
\]
we take the pushout of the span formed by the morphisms $o$ and $i'$
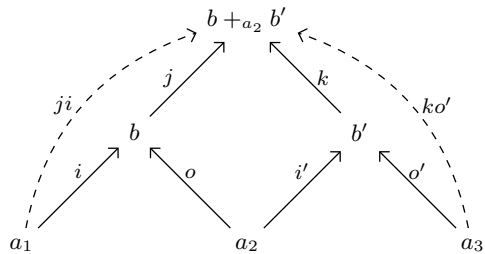
\[
%
\]
and then the resulting cospan is given by taking the composite of the outer morphisms leading up to the apex.
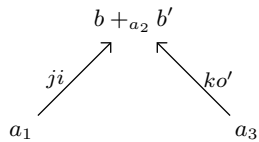
\[
%
\]

\section{Double categories} \label{DoubleCatAppendix}
\begin{definition}
Given a category $\mathsf{A}$ with finite limits, a \define{category internal to} $\mathsf{A}$ consists of:
\begin{enumerate}
\item{an object of objects $\mathsf{a}_0 \in \mathsf{A}$}
\item{an object of morphisms $\mathsf{a}_1 \in \mathsf{A}$}
\item{source and target assigning morphisms $s,t \colon \mathsf{a}_1 \to \mathsf{a}_0$}
\item{an identity assigning morphism $i \colon \mathsf{a}_0 \to \mathsf{a}_1$}
\item{a composition assigning morphism $c \colon \mathsf{a}_1 \times_{\mathsf{a}_0} \mathsf{a}_1 \to \mathsf{a}_1$}
\end{enumerate}
such that the following square is a pullback
\[

\]
and that the following diagrams commute:
\[
%
\]
which specifies the source and target of an identity morphism,
\[
%
\]
which say that the source and target of a composite of morphisms are the source and target of the first and second morphisms, respectively,
\[
%
\]
which says that composition of morphisms is strictly associative, and
\[
%
\]
which says how the left and right unit laws are compatible with composition.
\end{definition}
In the previous and following definitions, we do not really need all finite limits; it is enough for the stated pullbacks to exist.

\begin{definition}\label{pseudocategory}
Any 2-category (see Definition \ref{2-cat_definition}) has an underlying category with the same objects and morphisms, and we say that a 2-category \define{has finite limits} if its underlying category does. 
Given a 2-category $\mathbf{A}$ with finite limits, a \define{pseudocategory object} in $\mathbf{A}$ consists of the same data as a category object internal to the underlying category of $\mathbf{A}$, except that the following diagrams commute up to isomorphism.
\[
\begin{tikzpicture}[scale=1.5]
\node (A) at (0,0) {$\mathsf{a}_1 \times_{\mathsf{a}_0} \mathsf{a}_1 \times_{\mathsf{a}_0} \mathsf{a}_1$};
\node (C) at (2.5,0) {$\mathsf{a}_1 \times_{\mathsf{a}_0} \mathsf{a}_1$};
\node (A') at (0,-1) {$\mathsf{a}_1 \times_{\mathsf{a}_0} \mathsf{a}_1$};
\node (C') at (2.5,-1) {$\mathsf{a}_1$};
\node (B) at (1.25,-0.5) {$\alpha \NEarrow$};
\path[->,font=\scriptsize,>=angle 90]
(A) edge node[above]{$1 \times c$} (C)
(A) edge node [left]{$c \times 1$} (A')
(C)edge node[right]{$c$}(C')
(A')edge node [above] {$c$}(C');
\end{tikzpicture}
\]
\[
\begin{tikzpicture}[scale=1.5]
\node (A) at (0,0) {$\mathsf{a}_0 \times_{\mathsf{a}_0} \mathsf{a}_1$};
\node (C) at (2,0) {$\mathsf{a}_1 \times_{\mathsf{a}_0} \mathsf{a}_1$};
\node (A') at (4,0) {$\mathsf{a}_1 \times_{\mathsf{a}_0} \mathsf{a}_0$};
\node (C') at (2,-1) {$\mathsf{a}_1$};
\node (B) at (1.5,-0.35) {$\lambda \SWarrow$};
\node (B') at (2.5,-0.35) {$\rho \SEarrow$};
\path[->,font=\scriptsize,>=angle 90]
(A) edge node[above]{$i \times_{\mathsf{a}_0} 1$} (C)
(A') edge node [above]{$1 \times_{\mathsf{a}_0} i$} (C)
(C)edge node[left]{$c$}(C')
(A)edge node [below] {$p_2$}(C')
(A')edge node [below] {$p_1$}(C');
\end{tikzpicture}
\]
The isomorphisms $\alpha, \lambda$ and $\rho$ satisfy the pentagon and triangle identities of a monoidal category.
\end{definition}

\begin{definition}
A \define{strict double category} is a category object internal to $\mathsf{Cat}$ (which is a category with finite limits).
\end{definition}

\begin{definition}
A \define{(pseudo) double category} is a pseudocategory object internal to $\mathbf{Cat}$ (which is a 2-category with finite limits).
\end{definition}

In a nutshell, a strict double category is a category \emph{internal} to the category $\mathsf{Cat}$ of categories and functors, similar to how an ordinary small category is a category internal to the category $\mathsf{Set}$ of sets and functions. What this means is that instead of having a set of objects and a set of morphisms, we will instead have a \emph{category} of objects and a \emph{category} or morphisms. There are various kinds of double categories one can consider depending on how strict we are with the internalizations; whereas $\mathsf{Set}$ is a mere category, $\mathsf{Cat}$ is a 2-category which allows us to consider a triple composite of morphisms up to a 2-morphism. Internalizing a category object in the ordinary category $\mathsf{Cat}$ leads to what are typically known as \emph{strict} double categories, whereas internalizing a category object in $\mathsf{Cat}$ viewed as a 2-category, also known as a pseudocategory object, leads to \emph{pseudo} double categories, where the left and right unitors and associators no longer hold on-the-nose but only up to isomorphism. These latter pseudo double categories are the ones that we are primarily interested in. 

It is helpful to have the following picture in mind. A double category has 2-morphisms shaped like this:
\[
\begin{tikzpicture}[scale=1]
\node (D) at (-4,0.5) {$A$};
\node (E) at (-2,0.5) {$B$};
\node (F) at (-4,-1) {$C$};
\node (A) at (-2,-1) {$D$};
\node (B) at (-3,-0.25) {$\Downarrow a$};
\path[->,font=\scriptsize,>=angle 90]
(D) edge node [above]{$M$}(E)
(E) edge node [right]{$g$}(A)
(D) edge node [left]{$f$}(F)
(F) edge node [above]{$N$} (A);
\end{tikzpicture}
\]
We call $A, B, C$ and $D$ \define{objects} or \define{0-cells}, $f$ and $g$ \define{vertical 1-morphisms}, $M$ and $N$ \define{horizontal 1-cells} and $a$ a \define{2-morphism}. Note that a vertical 1-morphism is a morphism between 0-cells and a 2-morphism is a morphism between horizontal 1-cells. We denote both vertical 1-morphisms and horizontal 1-cells using single arrows, namely `$\to$'. We follow the notation of Shulman \cite{Shul} with the following definitions.

\begin{definition}\label{defn:double_category}
A \define{pseudo double category} $\mathbb{D}$, or $\textbf{double category}$ for short, consists of a category of objects $\mathbb{D}_0$ and a category of arrows $\mathbb{D}_1$ with the following functors
\begin{center}
$U\colon \mathbb{D}_0 \to \mathbb{D}_1$\\
$S,T \colon \mathbb{D}_1 \rightrightarrows \mathbb{D}_0$\\
$\odot \colon \mathbb{D}_1 \times_{\mathbb{D}_0} \mathbb{D}_1 \to \mathbb{D}_1$ (where the pullback is taken over $\mathbb{D}_1 \xrightarrow[]{T} \mathbb{D}_0 \xleftarrow[]{S} \mathbb{D}_1$) \\
\end{center}
 such that \\
\begin{center}
$S(U_{A})=A=T(U_{A})$\\
$S(M \odot N)=SN$\\
$T(M \odot N)=TM$\\
\end{center}
equipped with natural isomorphisms
\begin{center}

$\alpha \colon (M \odot N) \odot P \xrightarrow{\sim} M \odot (N \odot P)$\\
$\lambda \colon U_{B} \odot M \xrightarrow{\sim} M$\\
$\rho \colon M \odot U_{A} \xrightarrow{\sim} M$

\end{center}
such that $S(\alpha), S(\lambda), S(\rho), T(\alpha), T(\lambda)$ and $T(\rho)$ are all identities and that the coherence axioms of a monoidal category are satisfied. Following the notation of Shulman, objects of $\mathbb{D}_0$ are called $\textbf{0-cells}$ or \define{objects} and morphisms of $\mathbb{D}_0$ are called $\textbf{vertical 1-morphisms}$. Objects of $\mathbb{D}_1$ are called $\textbf{horizontal 1-cells}$ and morphisms of $\mathbb{D}_1$ are called $\textbf{2-morphisms}$. The morphisms of $\mathbb{D}_0$, which are vertical 1-morphisms, will be denoted $f \colon A \to C$ and we denote a horizontal 1-cell $M$ with $S(M)=A, T(M)=B$ by $M \colon A \to B$. Then a 2-morphism $a \colon M \to N$ of $\mathbb{D}_1$ with $S(a)=f,T(a)=g$ would look like:
\[
\begin{tikzpicture}[scale=1]
\node (D) at (-4,0.5) {$A$};
\node (E) at (-2,0.5) {$B$};
\node (F) at (-4,-1) {$C$};
\node (A) at (-2,-1) {$D$};
\node (B) at (-3,-0.25) {$\Downarrow a$};
\path[->,font=\scriptsize,>=angle 90]
(D) edge node [above]{$M$}(E)
(E) edge node [right]{$g$}(A)
(D) edge node [left]{$f$}(F)
(F) edge node [above]{$N$} (A);
\end{tikzpicture}
\]
The horizontal and vertical composition of 2-morphisms together obey a `middle-four' interchange law, or simply, interchange law, expressing the compatibility of horizontal and vertical composition with each other. Specifically, given four 2-morphisms as such:
\[
\begin{tikzpicture}[scale=1]
\node (D) at (-4,0.5) {$A$};
\node (E) at (-2,0.5) {$B$};
\node (F) at (-4,-1) {$C$};
\node (A) at (-2,-1) {$D$};
\node (D') at (0,0.5) {$B$};
\node (E') at (2,0.5) {$E$};
\node (F') at (0,-1) {$D$};
\node (A') at (2,-1) {$F$};
\node (B) at (-3,-0.25) {$\Downarrow a$};
\node (B') at (1,-0.25) {$\Downarrow b$};
\node (D'') at (-4,-2.5) {$C$};
\node (E'') at (-2,-2.5) {$D$};
\node (F'') at (-4,-4) {$G$};
\node (A'') at (-2,-4) {$H$};
\node (D''') at (0,-2.5) {$D$};
\node (E''') at (2,-2.5) {$F$};
\node (F''') at (0,-4) {$H$};
\node (A''') at (2,-4) {$I$};
\node (B'') at (-3,-3.25) {$\Downarrow a'$};
\node (B''') at (1,-3.25) {$\Downarrow b'$};
\path[->,font=\scriptsize,>=angle 90]
(D) edge node [above]{$M$}(E)
(E) edge node [right]{$g$}(A)
(D) edge node [left]{$f$}(F)
(F) edge node [above]{$N$} (A)
(D') edge node [above]{$O$}(E')
(E') edge node [right]{$h$}(A')
(D') edge node [left]{$g$}(F')
(F') edge node [above]{$P$} (A')
(D'') edge node [above]{$N$}(E'')
(E'') edge node [right]{$g'$}(A'')
(D'') edge node [left]{$f'$}(F'')
(F'') edge node [above]{$Q$} (A'')
(D''') edge node [above]{$P$}(E''')
(E''') edge node [right]{$h'$}(A''')
(D''') edge node [left]{$g'$}(F''')
(F''') edge node [above]{$R$} (A''');
\end{tikzpicture}
\]
the following equality holds, where $\odot$ denotes horizontal composition and juxtaposition denotes vertical composition. $$(a' \odot b')(a \odot b) = (a'a) \odot (b'b)$$
\end{definition}

The key difference between a strict double category and a pseudo double category is that in a pseudo double category, horizontal composition is associative and unital only up to natural isomorphism. The natural isomorphisms $\alpha, \lambda$ and $\rho$ are identities in a strict double category. Let us look at a few examples.

If $\mathsf{C}$ is any category, there exists a strict double category Sq$(\mathsf{C})$, where `Sq' denotes `square', which has:
\begin{enumerate}
\item{objects given by those of $\mathsf{C}$,}
\item{vertical 1-morphisms given by morphisms of $\mathsf{C}$,}
\item{horizontal 1-cells also given by morphisms of $\mathsf{C}$, and}
\item{2-morphisms as commutative squares in $\mathsf{C}$.}
\end{enumerate}
Composition of horizontal 1-cells coincides with composition of morphisms in $\mathsf{C}$ and both the horizontal and vertical composite of 2-morphisms is given by composing the edges of the commutative squares.

If $\mathsf{C}$ is a category with pushouts, then an example of a \emph{pseudo} double category, and probably the most important example of a double category in this thesis, is given by $\mathbb{C}\mathbf{sp}(\mathsf{C})$, where ``$\mathbb{C}\mathbf{sp}$" denotes ``cospan", which has:
\begin{enumerate}
\item{objects as those of $\mathsf{C}$,}
\item{vertical 1-morphisms as morphisms of $\mathsf{C}$,}
\item{horizontal 1-cells as cospans in $\mathsf{C}$, and}
\item{2-morphisms as maps of cospans in $\mathsf{C}$ which are given by commutative diagrams of the form:
\[

\]
}
\end{enumerate}
Composition of vertical 1-morphisms and the vertical composite of 2-morphisms is given by composition of morphisms in $\mathsf{C}$, and composition of horizontal 1-cells and the horizontal composite of 2-morphisms is given by pushouts in $\mathsf{C}$
\[
%
\]
where $\psi$ is the natural map into a coproduct and $J$ is the natural map from a coproduct to a pushout, for example, $\psi \colon b \to b+c$ and $J \colon b+c \to b+_{a_2} c$. More about this double category and others similar to it may be found in the work of Niefield \cite{Nie}.

The pseudo double categories that we are interested in all share a certain `lifting' property between the vertical 1-morphisms and horizontal 1-cells.
\begin{definition}
  Let $\mathbb{D}$ be a double category and $f\maps A\to B$ a vertical
  1-morphism.  A \define{companion} of $f$ is a horizontal 1-cell
  $\hat{f}\maps A\to B$ together with 2-morphisms
	\[
	\raisebox{-0.5\height}{

	}
	\]
  such that the following equations hold.
	\begin{equation}
	\label{eq:CompanionEq}
	\raisebox{-0.5\height}{
	%
	}
	\end{equation}
  A \define{conjoint} of $f$, denoted $\check{f} \maps B\to A$, is a
  companion of $f$ in the double category $\mathbb{D}^{h\cdot\mathrm{op}}$
  obtained by reversing the horizontal 1-cells, but not the vertical
  1-morphisms, of $\mathbb{D}$.
\end{definition}
\begin{definition}\label{def:isofibrant}
  We say that a double category is \define{fibrant} if every vertical
  1-morphism has both a companion and a conjoint and \define{isofibrant} if every vertical 1-isomorphism has both a companion and a conjoint.
\end{definition}
The property of isofibrancy in a double category is key as we are primarily interested in \emph{symmetric monoidal} double categories and bicategories, and it is precisely the property of isofibrancy that allows us to lift the portion of the monoidal structure of a symmetric monoidal double category that resides in the category of objects, such as the unitors, associators and braidings, to obtain a symmetric monoidal bicategory using a result of Shulman \cite{Shul}.

Next, we define the kinds of maps between double categories.
\begin{definition}
Let $\mathbb{A}$ and $\mathbb{B}$ be pseudo double categories. A \define{lax double functor} is a functor $\mathbb{F} \colon\mathbb{A}\to \mathbb{B}$ that takes items of $\mathbb{A}$ to items of $\mathbb{B}$ of the corresponding type, respecting vertical composition in the strict sense and the horizontal composition up to an assigned comparison $\phi$. This means that we have functors $\mathbb{F}_0 \colon \mathbb{A}_0 \to \mathbb{B}_0$ and $\mathbb{F}_1 \colon \mathbb{A}_1 \to \mathbb{B}_1$ such that the following equations are satisfied: $$S \circ \mathbb{F}_1 = \mathbb{F}_0 \circ S$$ $$T \circ \mathbb{F}_1 = \mathbb{F}_0 \circ T$$ Sometimes for brevity, we will omit the subscripts and simply say $\mathbb{F}$; as to whether we mean $\mathbb{F}_0$ or $\mathbb{F}_1$ will be clear from context. Furthermore, every object $a$ is equipped with a special globular 2-morphism $\phi_{a} \colon 1_{\mathbb{F}(a)} \to \mathbb{F}(1_{a})$ (the \define{identity comparison}), and every composable pair of horizontal 1-cells $N_{1} \odot N_{2}$ is equipped with a special globular 2-morphism $\phi(N_{1},N_{2}) \colon \mathbb{F}(N_{1}) \odot \mathbb{F}(N_{2}) \to \mathbb{F}(N_{1} \odot N_{2})$ (the \define{composition comparison}), in a coherent way. This means that the following diagrams commute.

\begin{enumerate}

\item For a horizontal composite, $\beta \odot \alpha$,

\begin{equation}\label{eq:square}
  \xymatrix@-.5pc{
    \mathbb{F}(A) \ar[r]|{|}^{\mathbb{F}(N_{2})}  \ar[d] \ar@{}[dr]|{\mathbb{F}(\alpha)}&
    \mathbb{F}(B) \ar[d] \ar[r]|{|}^{\mathbb{F}(N_{1})} \ar@{}[dr]|{\mathbb{F}(\beta)}&
    \mathbb{F}(C) \ar[d] &
     &
    \mathbb{F}(A) \ar[r]|{|}^{\mathbb{F}(N_{2})} \ar@{}[drr]|{\phi(N_{1},N_{2})} \ar[d]_{1} &
    \mathbb{F}(B) \ar[r]|{|}^{\mathbb{F}(N_{1})} &
    \mathbb{F}(C) \ar[d]^{1} \\
    \mathbb{F}(A') \ar[r]|{|}_{\mathbb{F}(N_{4})} \ar@{}[drr]|{\phi(N_{3},N_{4})} \ar[d]_{1}&
    \mathbb{F}(B') \ar[r]|{|}_{\mathbb{F}(N_{3})} &
    \mathbb{F}(C') \ar[d]^{1}&
    = &
    \mathbb{F}(A) \ar[rr]|{|}^{\mathbb{F}(N_{1} \odot N_{2})} \ar[d] \ar@{}[drr]|{\mathbb{F}(\beta \odot \alpha)}&
     &
    \mathbb{F}(C) \ar[d] \\
    \mathbb{F}(A') \ar[rr]|{|}_{\mathbb{F}(N_{3} \odot N_{4})} & 
     & 
    \mathbb{F}(C') &
     &
    \mathbb{F}(A') \ar[rr]|{|}_{\mathbb{F}(N_{3} \odot N_{4})} &
     &
    \mathbb{F}(C') \\
  }.
\end{equation}

\item For a horizontal 1-cell $N \colon A \to B$, the following diagrams are commutative (under horizontal composition).

\[
\begin{tikzpicture}[scale=1.5]
\node (A) at (1,1) {$\mathbb{F}(N) \odot 1_{\mathbb{F}(A)}$};
\node (C) at (3,1) {$\mathbb{F}(N)$};
\node (A') at (1,-1) {$\mathbb{F}(N) \odot \mathbb{F}(1_{A})$};
\node (C') at (3,-1) {$\mathbb{F}(N \odot 1_{A})$};
\node (B) at (5,1) {$1_{\mathbb{F}(B)} \odot \mathbb{F}(N)$};
\node (B') at (5,-1) {$\mathbb{F}(1_{B}) \odot \mathbb{F}(N)$};
\node (D) at (7,1) {$\mathbb{F}(N)$};
\node (D') at (7,-1) {$\mathbb{F}(1_{B} \odot N)$};
\path[->,font=\scriptsize,>=angle 90]
(A) edge node[left]{$1 \odot \phi_{A}$} (A')
(C') edge node[right]{$\mathbb{F} \rho$} (C)
(A) edge node[above]{$\rho_{\mathbb{F}(N)}$} (C)
(A') edge node[above]{$\phi(N,1_{A})$} (C')
(B) edge node[left]{$\phi_{B} \odot 1$} (B')
(B') edge node[above]{$\phi(1_{B},N)$} (D')
(B) edge node[above]{$\lambda_{\mathbb{F}(N)}$} (D)
(D') edge node[right]{$F \lambda$} (D);
\end{tikzpicture}
\]

\item For consecutive horizontal 1-cells $N_{1},N_{2}$ and $N_{3}$, the following diagram is commutative.

 \[\xymatrix{
    (\mathbb{F}(N_{1}) \odot \mathbb{F}(N_{2})) \odot \mathbb{F}(N_{3}) \ar[r]^{a'}\ar[d]_{\phi(N_{1},N_{2}) \odot 1}
    & \mathbb{F}(N_{1}) \odot (\mathbb{F}(N_{2}) \odot \mathbb{F}(N_{3})) \ar[d]^{1 \odot \phi(N_{2},N_{3})}\\
    \mathbb{F}(N_{1} \odot N_{2}) \odot \mathbb{F}(N_{3}) \ar[d]_{\phi(N_{1} \odot N_{2},N_{3})} &
    \mathbb{F}(N_{1}) \odot \mathbb{F}(N_{2} \odot N_{3}) \ar[d]^{\phi(N_{1},N_{2} \odot N_{3})}\\
    \mathbb{F}((N_{1} \odot N_{2}) \odot N_{3})\ar[r]^{\mathbb{F}a} &
    \mathbb{F}(N_{1} \odot (N_{2} \odot N_{3}))}\]
\end{enumerate}
We say the double functor $\mathbb{F}$ is \define{strict} if the comparison constraints $\phi_a$ and $\phi_{N_1,N_2}$ are identities, \define{strong} if the comparison constrains are globular isomorphisms, \define{pseudo} if the comparison constraints are isomorphisms, and \define{oplax} if the comparison constraints go in the opposite direction.
\end{definition}

We can also consider maps between maps of double functors, also known as double transformations. These are only used in Section \ref{SCMaps} of this thesis.

\begin{definition}\label{double_transformation}
A \define{double transformation} $\alpha \maps \mathbb{F} \Rightarrow \mathbb{G}$ between two double functors $\mathbb{F} \maps \mathbb{A} \to \mathbb{B}$ and $\mathbb{G} \maps \mathbb{A} \to \mathbb{B}$ consists of two natural transformations $\alpha_0 \maps \mathbb{F}_0 \Rightarrow \mathbb{G}_0$ and $\alpha_1 \maps \mathbb{F}_1 \Rightarrow \mathbb{G}_1$ such that for all horizontal 1-cells $M$ we have that $S({\alpha_1}_{M})={\alpha_0}_{S(M)}$ and $T({\alpha_1}_{M})={\alpha_0}_{T(M)}$ and for composable horizontal 1-cells $M$ and $N$, we have that
\[

\]
We call $\alpha_0$ the \define{object component} and $\alpha_1$ the \define{arrow component} of the double transformation $\alpha$.
\end{definition}

\subsection{Monoidal double categories}
Let $\mathbf{Dbl}$ denote the 2-category of double categories, double functors and double transformations. One can check that $\mathbf{Dbl}$ has finite products, and in any 2-category with finite products we can define a `pseudomonoid' or a `weak' monoid, which is a categorified analogue of a monoid in which the left and right unitors and associators are not identities but natural isomorphisms. It is the 2-categorical structure of $\mathbf{Dbl}$, or more generally, any 2-category with finite limits, that enables us to do this. For example, a pseudomonoid in $\mathbf{Cat}$ is a monoidal category. We are primarily concerned with the (weak) monoidal double categories in which the associators and left and right unitors are natural isomorphisms.

\begin{definition}
Let $(\mathsf{C},\otimes,I)$ be a monoidal category. A \define{monoid internal to} $\mathsf{C}$ consists of an object $M \in \mathsf{C}$ together with a morphism $m \colon M \otimes M \to M$ for multiplication and a morphism $i \colon I \to M$ for the multiplicative identity satisfying the associative law:
\[

\]
and left and right unit laws:
\[
%
\]
A \define{pseudomonoid} internal to a monoidal 2-category $(\mathbf{C}, \otimes, I)$ consists of an object $M$ together with a morphism $m: M \otimes M \to M$ and a morphism $i: I \to M$ such that the above diagrams commute up to specified 2-isomorphisms:
\[
%
\]
Furthermore, the 2-isomorphisms $A, L$ and $R$ are required to satisfy two equations which can be found in the work of Day and Street \cite{DayStreet}.
\end{definition}

\begin{definition}
A \define{braided pseudomonoid} is a pseudomonoid $M$ equipped with the extra structure of a braiding isomorphism $\beta \colon \otimes \cong \otimes \circ t$ where $t$ is the `twist' isomorphism $$t \colon M \otimes M \to M \otimes M$$ that together with the associators make the usual hexagons of a braided monoidal category commute. A \define{symmetric pseudomonoid} is a braided pseudomonoid such that the braiding isomorphism $\beta \colon \otimes \cong \otimes \circ t$ is self-inverse.
\end{definition}

\begin{definition}\label{defn:monoidal_double_category}
A \define{monoidal double category} is a pseudomonoid in the monoidal 2-category $\mathbf{Dbl}$. 

Explicitly, a monoidal double category is a double category equipped with double functors $\otimes \maps \mathbb{D} \times \mathbb{D} \to \mathbb{D}$ and $I \maps * \to \mathbb{D}$ where $*$ is the terminal double category, along with invertible double transformations called the \define{associator}:
\[  A \maps \otimes \, \circ \; (1_{\mathbb{D}} \times \otimes ) \Rightarrow \otimes \; \circ \; (\otimes \times 1_{\mathbb{D}}) ,\]
\define{left unitor}:
\[ L\maps \otimes \, \circ \; (1_{\mathbb{D}} \times I) \Rightarrow 1_{\mathbb{D}} ,\]
and \define{right unitor}:
\[ R \maps \otimes \,\circ\; (I \times 1_{\mathbb{D}}) \Rightarrow 1_{\mathbb{D}} \]
satisfying the pentagon axiom and triangle axioms of a monoidal category.
\end{definition}

This is a very nice and compact definition which encapsulates the structure of a monoidal double category. Unraveling this a bit, this means that:
\begin{enumerate}
\item $\mathbb{D}_0$ and $\mathbb{D}_1$ are both monoidal categories.
\item If $I$ is the monoidal unit of $\mathbb{D}_0$, then $U_I$ is (coherently isomorphic to) the
  monoidal unit of $\mathbb{D}_1$.
\item The functors $S$ and $T$ are strict monoidal, meaning that  $$S(M\ten N)
  = SM\ten SN$$ and $$T(M\ten N)=TM\ten TN$$ and $S$ and $T$ also
  preserve the associativity and unit constraints.
\item We have globular isomorphisms 
  \[\chi \maps (M_1\ten N_1)\odot (M_2\ten N_2)\xrightarrow{\sim} (M_1\odot M_2)\ten (N_1\odot N_2)\]
  and
  \[\mu\maps U_{A\ten B} \xrightarrow{\sim} (U_A \ten U_B)\] which arise from weakly-commuting squares:

\[

\]
expressing the \emph{weak} commutativity of $\otimes$ with the functors $U$ and $\odot$.

These globular isomorphisms $\chi$ and $\mu$ make the following diagrams commute:
		\item \label{diag:MonDblCat}
			The following diagrams commute expressing that $\otimes \colon \lD \times \lD \to \lD$ is a pseudo double functor.
			\[
			%
		\]
		\item The following diagrams commute expressing 
		the associativity isomorphism for $\otimes$ is a transformation of double categories.
		\[
		%
		\]
		\item The following diagrams commute expressing that 
		the unit isomorphisms for $\otimes$ are transformations of double categories. 
		\[
		%
		\]
		\newcounter{mondbl}
		\setcounter{mondbl}{\value{enumi}}
	\end{enumerate}
Thus we define a monoidal double category to be a pseudomonoid object weakly internal to the 2-category $\mathbf{Dbl}$ of double categories, double functors and double transformations. In other words, a monoidal double category is a pseudomonoid internal to categories weakly internal to $\mathbf{Cat}$. 
But beware: this is \emph{not} the same as a category weakly internal to the 2-category $\mathbf{MonCat}$ of monoidal categories, strong monoidal functors and monoidal natural transformations. In a monoidal double category, the functors $S$ and $T$ are strict monoidal.  In a category weakly internal to $\mathbf{MonCat}$, they would only need to be strong monoidal.


\begin{definition}
	\label{defn:symmetric_monoidal_double_category}
A \define{braided monoidal double category} is a braided pseudomonoid internal to $\mathbf{Dbl}$. 
\end{definition}
This means that a braided monoidal double category is a monoidal double category
category equipped with an invertible double transformation
\[ \beta \maps \otimes \Rightarrow \otimes \circ \tau \]
called the \define{braiding}, where $\tau \maps \mathbb{D} \times \mathbb{D} \to \mathbb{D} \times \mathbb{D}$ is the twist double functor sending pairs in the object and arrow categories to the same pairs in the opposite order. The braiding is required to satisfy the usual two hexagon identities \cite[Sec.\ XI.1]{ML}.  If the braiding is self-inverse we say that $\mathbb{D}$ is a symmetric pseudomonoid internal to $\mathbf{Dbl}$ and that $\mathbb{D}$ is a \define{symmetric monoidal double category}.

Unraveling this a bit, we get that a braided monoidal double category 
	is a monoidal double category 
	such that:
	\begin{enumerate}
		\setcounter{enumi}{\value{mondbl}}
		\item $\mathbb{D}_{0}$ and $\mathbb{D}_{1}$ are braided monoidal categories.
		\item The functors $S$ and $T$ are strict braided monoidal functors.
		\item The following diagrams commute expressing that the braiding is a transformation of double categories.
		\[
		\begin{tikzpicture}
			\node (A) at (0,1.5) {\footnotesize{$(M_1 \odot M_2) \otimes (N_1 \odot N_2)$}};
			\node (A') at (0,0) {\footnotesize{$(M_1\otimes N_1) \odot (M_2\otimes N_2)$}};
			\node (B) at (5,1.5) {\footnotesize{$(N_1\odot N_2) \otimes (M_1 \odot M_2)$}};
			\node (B') at (5,0) {\footnotesize{$(N_1 \otimes M_1) \odot (N_2 \otimes M_2)$}};
			\path[->,font=\scriptsize]
				(A) edge node[left]{$\chi$} (A')
				(A) edge node[above]{$\beta$} (B)
				(B) edge node[right]{$\chi$} (B')
				(A') edge node[above]{$\beta \odot \beta$} (B');
		\end{tikzpicture}
		\quad
		\begin{tikzpicture}
			\node (A) at (0,1.5) {\footnotesize{$U_A \otimes U_B$}};
			\node (A') at (0,0) {\footnotesize{$U_B\otimes U_A$}};
			\node (B) at (2,1.5) {\footnotesize{$U_{A\otimes B} $}};
			\node (B') at (2,0) {\footnotesize{$U_{B\otimes A}$}};
			\path[->,font=\scriptsize]
				(A) edge node[left]{$\beta$} (A')
				(B) edge node[above]{$\mu$} (A)
				(B) edge node[right]{$U_\beta$} (B')
				(B') edge node[above]{$\mu$} (A');
		\end{tikzpicture}
		\]
		\setcounter{mondbl}{\value{enumi}}
	\end{enumerate}
	Finally, a symmetric monoidal double category 
	is a braided monoidal double category $\mathbb{D}$ such that:
	\begin{enumerate}
		\setcounter{enumi}{\value{mondbl}}
		\item $\mathbb{D}_{0}$ and $\mathbb{D}_{1}$ are symmetric monoidal categories.
	\end{enumerate}

\subsection{Monoidal double functors and transformations}
We also have maps between symmetric monoidal double categories, which, just as maps between ordinary symmetric monoidal categories, can come in three flavors according to direction of the comparison maps $\phi_{ (-,-) }$.

\begin{definition}
\label{defn:monoidal_double_functor}
A \define{(strong) monoidal lax double functor} $\mathbb{F} \colon \mathbb{C} \to \mathbb{D}$ between monoidal double categories $\mathbb{C}$ and $\mathbb{D}$ is a lax double functor $\mathbb{F} \maps \mathbb{C} \to \mathbb{D}$ such that
	\begin{itemize}
		\item $\mathbb{F}_0$ and $\mathbb{F}_1$ are (strong) monoidal functors, meaning that there exists

\begin{enumerate}
\item{an isomorphism $\epsilon \colon 1_{\lD} \to \mathbb{F}(1_{\lC})$}
\item{a natural isomorphism $\theta_{A,B} \colon \mathbb{F}(A) \otimes \mathbb{F}(B) \to \mathbb{F}(A \otimes B)$ for all objects $A$ and $B$ of $\lC$}
\item{an isomorphism $\delta \colon U_{1_{\lD}} \to \mathbb{F}(U_{1_{\lC}})$}
\item{a natural isomorphism $\nu_{M,N} \colon \mathbb{F}(M) \otimes \mathbb{F}(N) \to \mathbb{F}(M \otimes N)$ for all horizontal 1-cells $N$ and $M$ of $\lC$}
\end{enumerate}
such that the following diagrams commute: for objects $A,B$ and $C$ of $\lC$,
 \[\xymatrix{
    (\mathbb{F}(A) \otimes \mathbb{F}(B)) \otimes \mathbb{F}(C) \ar[r]^{\alpha^\prime}\ar[d]_{\theta_{A,B} \otimes 1}
    & \mathbb{F}(A) \otimes (\mathbb{F}(B) \otimes \mathbb{F}(C)) \ar[d]^{1 \otimes \theta_{B,C}}\\
    \mathbb{F}(A \otimes B) \otimes \mathbb{F}(C) \ar[d]_{\theta_{A \otimes B,C}} &
    \mathbb{F}(A) \otimes \mathbb{F}(B \otimes C) \ar[d]^{\theta_{A,B \otimes C}}\\
    \mathbb{F}((A \otimes B) \otimes C)\ar[r]^{\mathbb{F}\alpha} &
    \mathbb{F}(A \otimes (B \otimes C))}\]
\[
\begin{tikzpicture}[scale=1.5]
\node (A) at (1,1) {$\mathbb{F}(A) \otimes 1_{\lD}$};
\node (C) at (3,1) {$\mathbb{F}(A)$};
\node (A') at (1,-1) {$\mathbb{F}(A) \otimes \mathbb{F}(1_{\lC})$};
\node (C') at (3,-1) {$\mathbb{F}(A \otimes 1_{\lC})$};
\node (B) at (5,1) {$1_{\lD} \otimes \mathbb{F}(A)$};
\node (B') at (5,-1) {$\mathbb{F}(1_{\lC}) \otimes \mathbb{F}(A)$};
\node (D) at (7,1) {$\mathbb{F}(A)$};
\node (D') at (7,-1) {$\mathbb{F}(1_{\lC} \otimes A)$};
\path[->,font=\scriptsize,>=angle 90]
(A) edge node[left]{$1 \otimes \epsilon$} (A')
(C') edge node[right]{$\mathbb{F}(r_{A})$} (C)
(A) edge node[above]{$r_{\mathbb{F}(A)}$} (C)
(A') edge node[above]{$\theta_{A,1_{\lC}}$} (C')
(B) edge node[left]{$\epsilon \otimes 1$} (B')
(B') edge node[above]{$\theta_{1_{\lC},A}$} (D')
(B) edge node[above]{$\ell_{\mathbb{F}(A)}$} (D)
(D') edge node[right]{$\mathbb{F}(\ell_{A})$} (D);
\end{tikzpicture}
\]
and for horizontal 1-cells $N_{1},N_{2}$ and $N_{3}$ of $\lC$,
 \[\xymatrix{
    (\mathbb{F}(N_{1}) \otimes \mathbb{F}(N_{2})) \otimes \mathbb{F}(N_{3}) \ar[r]^{\alpha^\prime}\ar[d]_{\nu_{N_{1},N_{2}} \otimes 1}
    & \mathbb{F}(N_{1}) \otimes (\mathbb{F}(N_{2}) \otimes \mathbb{F}(N_{3})) \ar[d]^{1 \otimes \nu_{N_{2},N_{3}}}\\
    \mathbb{F}(N_{1} \otimes N_{2}) \otimes \mathbb{F}(N_{3}) \ar[d]_{\nu_{N_{1} \otimes N_{2},N_{3}}} &
    \mathbb{F}(N_{1}) \otimes \mathbb{F}(N_{2} \otimes N_{3}) \ar[d]^{\nu_{N_{1},N_{2} \otimes N_{3}}}\\
    \mathbb{F}((N_{1} \otimes N_{2}) \otimes N_{3})\ar[r]^{\mathbb{F}\alpha} &
    \mathbb{F}(N_{1} \otimes (N_{2} \otimes N_{3}))}\]
\[

\]

		\item $S\mathbb{F}_1= \mathbb{F}_0S$ and $T\mathbb{F}_1 = \mathbb{F}_0T$ are equations between monoidal functors, and
		\item the composition and unit comparisons $\phi(N_1,N_2) \maps \mathbb{F}_1(N_1) \odot \mathbb{F}_1(N_2) \to \mathbb{F}_1(N_1\odot N_2)$ and $\phi_A \maps U_{\mathbb{F}_0 (A)} \to \mathbb{F}_1(U_A)$ are monoidal natural transformations.

\item{The following diagrams commute expressing that $\theta$ and $\nu$ together constitute a transformation of double categories:
		\[
		%
		\]
}
	\end{itemize}
The monoidal lax double functor is \define{braided} if $\mathbb{F}_0$ and $\mathbb{F}_1$ are braided monoidal functors and \define{symmetric} if they are symmetric monoidal functors, and $\define{lax} \textnormal{ } \define{monoidal}$ or $\define{oplax} \textnormal{ } \define{monoidal}$ if instead of the
isomorphisms and families of natural isomorphisms in items (1)-(4), we merely
have morphisms and natural transformations going in the appropriate directions.
If the double functor $\mathbb{F} \colon \mathbb{C} \to \mathbb{D}$ is a double functor between isofibrant symmetric monoidal double categories, also known as `symmetric monoidal framed bicategories' \cite{Shul2}, instead of $\theta$ and $\nu$ together constituting a transformation of double categories, it suffices that the comparison and unit constraints $\mathbb{F}_{M,N}$ and $\mathbb{F}_c$ be monoidal natural transformations.
\end{definition}
\begin{definition}\label{monoidal_double_transformation}
Given monoidal double functors $(\mathbb{F},\phi),(\mathbb{G},\psi) \colon \mathbb{C} \to \mathbb{D}$, a \define{monoidal double transformation} $\alpha \colon \mathbb{F} \Rightarrow \mathbb{G}$ is a double transformation $\alpha$ such that both the object component $\alpha_0 \colon \mathbb{F}_0 \Rightarrow \mathbb{G}_0$ and arrow component $\alpha_1 \colon \mathbb{F}_1 \Rightarrow \mathbb{G}_1$ are monoidal natural transformations. This means that the following equations hold:
\[

\]
\end{definition}

\section{Bicategories and 2-categories}\label{bicat_definitions}
\begin{definition}
A \define{bicategory} $\mathbf{C}$ is a double category (see Definition \ref{defn:double_category}) $\mathbf{C}=(\mathbf{C}_0,\mathbf{C}_1)$ such that the category of objects $\mathbf{C}_0$ is discrete, meaning that $\mathbf{C}_0$ contains only identity morphisms. In a bicategory, we refer to the objects of $\mathbf{C}_1$, which are horizontal 1-cells, as \define{morphisms}.
\end{definition}

Unraveling this a bit, a \define{bicategory} $\mathbf{C}$ consists of:
\begin{enumerate}
\item{a collection of objects $a,b,c,d,\ldots$,}
\item{for every pair of objects $a$ and $b$, a category $\hom_{\mathbf{C}}(a,b)$, called the \define{hom category of $a$ and $b$}, where objects are called $\define{morphisms}$ from $a$ to $b$ and whose morphisms are called $\define{2}$-$\define{morphisms}$,}
\item{for every object $a$, a functor $1_a \colon 1 \to \hom_\mathbf{C}(a,a)$ which picks out the identity morphism for the object $a$ and for every triple of objects $a,b$ and $c$, a functor $\circ \colon \hom_\mathbf{C}(a,b) \times \hom_\mathbf{C}(b,c) \to \hom_\mathbf{C}(a,c)$ for composition,}
\item{for every pair of objects $a$ and $b$ and morphism $f \in \hom_{\mathbf{C}}(a,b)$, a natural isomorphism $$\lambda \colon 1_b f \Rightarrow f$$called the \define{left unitor} and a natural isomorphism $$\rho \colon f 1_a \Rightarrow f$$called the \define{right unitor},}
\item{for every quadruple of objects $a,b,c$ and $d$, a natural isomorphism $$\alpha \colon \circ (1 \times \circ) \Rightarrow \circ (\circ \times 1)$$where$$\circ(1 \times \circ),\circ (\circ \times 1) \colon \hom_{\mathbf{C}}(a,b) \times \hom_{\mathbf{C}}(b,c) \times \hom_{\mathbf{C}}(c,d) \to \hom_{\mathbf{C}}(a,d)$$}
\end{enumerate}
such that the left and right unitors satisfy the triangle identity and the associator satisfies the pentagon identity.

\begin{definition}\label{2-cat_definition}
A \define{2-category} $\mathbf{C}$ is a bicategory $\mathbf{C}$ in which the left and right unitors and associators are identity 2-morphisms.
\end{definition}
Equivalently, a 2-category is a \emph{strict} double category in which the category of objects is discrete.

The primordial example of a 2-category is $\mathbf{Cat}$, the 2-category of categories, functors and natural transformations: natural transformations make up the morphisms in each hom category $\hom_\mathbf{C}(a,b)$.
A 2-category is sometimes called a `strict' 2-category and a bicategory a `weak' 2-category.
Strict 2-categories along with double categories were first discovered by Ehresmann \cite{Ehresmann63, Ehresmann65}, and bicategories are due to B\'enabou \cite{Be}.

\begin{definition}\label{definition:whiskering}
Given a 2-morphism $\alpha \colon f \Rightarrow g \colon c \to d$ and a morphism $h \colon b \to c$ in a 2-category $\mathbf{C}$:
\[
\begin{tikzpicture}[scale=1]
\node (B) at (0,0) {$b$};
\node (C) at (2,0) {$c$};
\node (D) at (4,0) {$d$};
\node (I) at (3,0) {$\Downarrow \alpha$};
\path[->,font=\scriptsize,>=angle 90]
(B) edge node [above]{$h$}(C)
(C) edge[bend left] node [above]{$f$}(D)
(C) edge[bend right] node [below]{$g$}(D);
\end{tikzpicture}
\]
the \define{left whiskering of $\alpha$ by $h$}, denoted by $1_h \odot \alpha$, is given by the horizontal composite of the 2-morphism $\alpha$ with the identity 2-morphism of $h$:
\[
\begin{tikzpicture}[scale=1]
\node (F) at (8,0) {$b$};
\node (G) at (10,0) {$c$};
\node (H) at (12,0) {$d$};
\node (J) at (9,0) {$\Downarrow 1_h$};
\node (K) at (11,0) {$\Downarrow \alpha$};
\path[->,font=\scriptsize,>=angle 90]
(F) edge[bend left] node [above]{$h$}(G)
(F) edge[bend right] node [below]{$h$} (G)
(G) edge[bend left] node [above]{$f$} (H)
(G) edge[bend right] node [below] {$g$} (H);
\end{tikzpicture}
\]
\define{Right whiskering} is defined analogously.
\end{definition}

\subsection{Pseudofunctors and pseudonatural transformations}
\begin{definition}\label{pseudofunctor_definition}
Given bicategories $\mathbf{C}$ and $\mathbf{D}$, a \define{pseudofunctor} $F \colon \mathbf{C} \to \mathbf{D}$ consists of:
\begin{enumerate}
\item for each object $c \in \mathbf{C}$, an object $F(c) \in \mathbf{D}$,
\item for each category $\mathbf{C}(c,c')$, a functor $F \colon \mathbf{C}(c,c') \to \mathbf{D}(F(c),F(c'))$,
\item for each object $c \in \mathbf{C}$, a 2-isomorphism $F_c \colon 1_{F(c)} \Rightarrow F(1_c)$
\item for every triple of objects $a,b,c \in \mathbf{C}$ and pair of composable morphisms $f \colon a \to b$ and $g \colon b \to c$ in $\mathbf{C}$, a 2-isomorphism $F_{f,g} \colon F(f)F(g) \Rightarrow F(fg)$ natural in $f$ and $g$
\end{enumerate}
such that the following diagrams commute:
\[
\begin{tikzpicture}[scale=1.5]
\node (A) at (0,0.5) {$(F(f)F(g))F(h)$};
\node (A') at (4.5,0.5) {$F(f)(F(g)F(h))$};
\node (B) at (0,-0.25) {$F(fg)F(h)$};
\node (C) at (4.5,-0.25) {$F(f)F(gh)$};
\node (C') at (0,-1) {$F((fg)h)$};
\node (D) at (4.5,-1) {$F(f(gh))$};
\path[->,font=\scriptsize,>=angle 90]
(A) edge node[above]{$a'$} (A')
(A) edge node[left]{$F_{f,g} \odot 1_{F(h)}$} (B)
(A') edge node[right]{$1_{F(f)} \odot F_{g,h}$} (C)
(B) edge node[left]{$F_{fg,h}$} (C')
(C) edge node [right] {$F_{f,gh}$} (D)
(C') edge node [above] {$F(a)$} (D);
\end{tikzpicture}
\]
\[
\begin{tikzpicture}[scale=1.5]
\node (A) at (1,1) {$F(f) 1_{F(a)}$};
\node (C) at (3,1) {$F(f)$};
\node (A') at (1,-1) {$F(f) F(1_a)$};
\node (C') at (3,-1) {$F(f 1_a)$};
\node (B) at (5,1) {$1_{F(b)}  F(f)$};
\node (B') at (5,-1) {$F(1_b) F(f)$};
\node (D) at (7,1) {$F(f)$};
\node (D') at (7,-1) {$F( 1_b f)$};
\path[->,font=\scriptsize,>=angle 90]
(A) edge node[left]{$1_{F(f)} \odot F_a$} (A')
(C') edge node[right]{$F(r_f)$} (C)
(A) edge node[above]{$r_{F(f)}'$} (C)
(A') edge node[above]{$F_{f,1_a}$} (C')
(B) edge node[left]{$F_b \odot  1_{F(f)}$} (B')
(B') edge node[above]{$F_{1_b,f}$} (D')
(B) edge node[above]{$\ell_{F(f)}'$} (D)
(D') edge node[right]{$F(\ell_f)$} (D);
\end{tikzpicture}
\]
Here, all of the arrows in the diagrams are given by 2-morphisms in $\mathbf{D}$, $a,\ell,r$ denote the associator, left and right unitors for morphism composition in $\mathbf{C}$, similarly $a', \ell^\prime,r^\prime$ denote the associator, left and right unitors for morphism composition in $\mathbf{D}$, juxtaposition is used to denote morphism composition in both $\mathbf{C}$ and $\mathbf{D}$ and $\odot$ denotes whiskering in $\mathbf{D}$ (see Definition \ref{definition:whiskering}).
\end{definition}

\begin{definition}
Given two pseudofunctors $F,G \colon \mathbf{A} \to \mathbf{B}$, a \define{pseudonatural transformation} $\sigma$ consists of:
\begin{enumerate}
\item{for each object $a \in \mathbf{A}$, a morphism $\sigma_a \colon F(a) \to G(a)$ in $\mathbf{B}$ and}
\item{for each morphism $f \colon a \to b$ in $\mathbf{A}$, an invertible natural 2-morphism $\sigma_f \colon G(f) \sigma_a \xrightarrow{\sim} \sigma_b F(f)$ in $\mathbf{B}$ which is compatible with composition and identities.}
\end{enumerate}
\end{definition}

Let $[\mathsf{A},\mathbf{Cat}]_\textnormal{ps}$ denote the 2-category of pseudofunctors, pseudonatural transformations and `modifications' from an ordinary category $\mathsf{A}$ viewed as a 2-category with trivial 2-morphisms. We call $[\mathsf{A},\mathbf{Cat}]_\textnormal{ps}$ the \define{2-category of opindexed categories}, as an indexed category is a contravariant pseudofunctor into $\mathbf{Cat}$. A \define{lax monoidal pseudofunctor} $F \colon \mathbf{A} \to \mathbf{B}$ between monoidal bicategories \cite{Stay} is then a pseudofunctor equipped with pseudonatural transformations with components $$\mu_{a,b} \colon F(a) \otimes F(b) \xrightarrow{\sim} F(a \otimes b)$$and$$\mu_0 \colon 1_\mathbf{B} \to F(1_\mathbf{A})$$together with coherent invertible modifications for associativity and unitality. This is also known as a \emph{weak} monoidal pseudofunctor. A \textbf{symmetric lax monoidal pseudofunctor} is then a lax monoidal pseudofunctor between symmetric monoidal bicategories together with invertible modifications $F(\beta) \mu_{a,b} \xrightarrow{\sim} \mu_{b,a} \beta^\prime$.

$\textnormal{ }$
}

{\ssp

}
\end{document}